\documentclass[11pt]{scrartcl}

\usepackage[utf8]{inputenc}

\usepackage{config}
\usepackage{titling}

\usepackage[a4paper, top=3cm, bottom=2cm, left=2.5cm, right=2.5cm, includefoot]{geometry}

\usepackage[inline]{enumitem}

\setlist[itemize]{topsep=0pt,partopsep=0pt,itemsep=0pt,parsep=0pt}
\setlist[itemize,1]{label={\small\textbullet}}
\setlist[itemize,2]{label={\tiny\textbullet}}
\setlist[itemize,3]{label=$\cdot$}
\setlist[enumerate]{topsep=0pt,partopsep=0pt,itemsep=0pt,parsep=0pt}
\setlist[enumerate,1]{label=\roman*)}
\setlist[enumerate,2]{label=\alph*)}
\setlist[enumerate,3]{label=\arabic*)}

\hypersetup{
	colorlinks=true,
	linkcolor=AO!65!black,
	citecolor=AO!65!black,
	urlcolor=AppleGreen!65!black,
	bookmarksopen=true,
	bookmarksnumbered,
	bookmarksopenlevel=2,
	bookmarksdepth=3
}

\title{Two Disjoint Alternating Paths in Bipartite Graphs}
\predate{}
\date{}
\postdate{}

\preauthor{}
\DeclareRobustCommand{\authorthing}{
	\begin{center}
		Archontia C.\@ Giannopoulou -- National and Kapodistrian University of Athens\\
		\href{mailto:archontia.giannopoulou@gmail.com}{archontia.giannopoulou@gmail.com}\\
		Sebastian Wiederrecht\thanks{Supported by the ANR project ESIGMA (ANR-17-CE23-0010) and by the ERC consolidator grant DISTRUCT-648527.} -- LIRMM, University of Montpellier\\
		\href{mailto:sebastian.wiederrecht@gmail.com}{sebastian.wiederrecht@gmail.com}
\end{center}}
\author{\authorthing}
\postauthor{}

\begin{document}
\maketitle

\begin{abstract}
A bipartite graph $B$ is called a brace if it is connected and every matching of size at most two in $B$ is contained in some perfect matching of $B$ and A cycle $C$ in $B$ is called conformal if $B-\V{C}$ has a perfect matching.
We show that there do not exist two disjoint alternating paths that form a cross over a conformal cycle $C$ in a brace $B$ if and only if one can reduce $B$, by an application of a matching theoretic analogue of small clique sums, to a planar brace $H$ in which $C$ bounds a face.
We then utilise this result and provide a polynomial time algorithm which solves the $2$-linkage problem for alternating paths in bipartite graphs with perfect matchings.

\end{abstract}

\section{Introduction}\label{sec:introduction}

Two of the key aspects in graph minor theory are those of embeddability into surfaces and linkedness, where the latter is the ability to link up a number of prescribed pairs of vertices by pairwise vertex disjoint paths.
In their series of papers on graph minors, Robertson and Seymour proved a general structure theorem that explains how graphs excluding a fixed minor $H$ can be constructed with a surprisingly comprehensive list of ingredients (see \cite{lovasz2006graph} for a summary).
Two of these ingredients are graphs embeddable in surfaces and the, so-called, clique sums.
For a non-negative integer $k$, a graph $G$ is a \emph{$k$-clique sum} of two graphs $G_1$ and $G_2$, if both $G_i$ contain a clique $Q_i$ of size $k$ and $G$ can be obtained from $G_1$ and $G_2$ by identifying $Q_1$ and $Q_2$ into a single clique $Q$ and possibly
forgetting some of the edges of $Q$.
Especially the concept of clique sums has proven itself in the search for exact characterisations of $H$-minor free graphs, at least for relatively small graphs $H$.
By adding graphs on surfaces, or planar graphs in this case, as the next ingredient, one reaches the notion of \emph{flatness} which plays a key role in a weaker version of the structure theorem for graphs excluding large complete graphs as minors known as the \emph{Flat Wall Theorem} \cite{robertson1995graph}.
In essence, a graph $H$ can be said to be \emph{flat} in some graph $G$, if $G$ can be constructed via (relatively) small clique sums from smaller graphs $G_1,\dots,G_n$, $n\geq 1$, such that the graph $G_i$ with $H\subseteq G_i$ is planar.
Please note that this is not the actual definition but rather a simplification to illustrate the flavour of the idea behind flatness.

An exceptional case occurs if we restrict $H$ to be some cycle $C$.
We say that $C$ is \emph{$C$-flat} in a graph $G$ if $G$ can be constructed via $k$-clique sums, where $k\leq 3$, from graphs $G_1,\dots,G_n$, where the graph $G_i$ that contains $C$ is planar, and $C$ bounds a face of $G_i$.
In the case of flat cycles, one immediately finds a flavour of linkedness.
A cycle $C$ is said to have a $C$-cross in $G$, if there exist distinct vertices $s_1,s_2,t_1,t_2$ that occur on $C$ in the order listed, and paths $P_1$ and $P_2$ such that:
Both $P_i$ are internally vertex disjoint from $C$, $P_1$ and $P_2$ are vertex disjoint and each $P_i$ has $s_i$ and $t_i$ as its endpoints.
A classic theorem, to which we will refer as the \emph{Two Paths Theorem}, links the notions of $C$-flatness and $C$-crosses.
The theorem has been obtained in many different forms with various techniques by a plethora of authors over time \cite{jung1970verallgemeinerung,seymour1980disjoint,shiloach1980polynomial,thomassen19802,robertson1990graph}.

\begin{theorem}[\emph{Two Paths Theorem}]\label{thm:twopaths}
	A cycle $C$ in a graph $G$ has \textbf{no} $C$-cross in $G$ if and only if it is $C$-flat in $G$.
\end{theorem}

By considering Menger's Theorem, one can get the impression that in terms of connectivity via paths, graphs and digraphs are not that different.
Indeed, if $D$ is an acyclic digraph, i.\@e.\@ a digraph without directed cycles, then the directed version of the $k$-Linkage Problem can be solved in polynomial time \cite{fortune1980directed}.
In the case $k=2$ there even exists a directed analogue for the Two Paths Theorem \cite{thomassen19852}.

However, once one allows for the existence of directed cycles, even the Directed $2$-Disjoint Paths Problem becomes NP-hard \cite{fortune1980directed}.
This hardness result has far-reaching consequences for structural digraph theory.
Deciding whether a given digraph $D$ contains a subdivision of a digraph $H$ can be NP-complete even for relatively small digraphs \cite{bang2015finding}, while the complexity of deciding whether $D$ contains $H$ as a \emph{butterfly minor}\footnote{A notion of minors suited for digraphs.} remains unknown for almost all digraphs $H$.
Similarly, attempts of generalising the Flat Wall Theorem to the world of digraphs \cite{giannopoulou2020directed} obtain weaker structural
insight due to the lack of a Two Paths Theorem.

So can we hope for a generalisation of the Two Paths Theorem, at least in some sense, for the directed case?
In this paper we present a possible answer to this question in terms of a matching theoretic result, how exactly this relates to digraphs will be discussed in the \hyperref[sec:conclusion]{conclusion}.
For this let us introduce some definitions and the structural result that acts as the key tool for the results of this paper.

Let $G$ be a graph, a set $F\subseteq\E{G}$ of pairwise disjoint edges is called a \emph{matching}, and a vertex $v\in\V{G}$ is said to be \emph{covered} by $F$ if $F$ contains an edge that has $v$ as an endpoint.
The set of all vertices covered by $F$ is denoted by $\V{F}$.
A matching $M$ is \emph{perfect} if $\V{M}=\V{G}$ and an edge $e\in\E{G}$ is \emph{admissible} if there exists a perfect matching $M'$ of $G$ with $e\in M'$.
A graph $G$ is \emph{matching covered} if it is connected and all of its edges are admissible.

A set $X\subseteq\V{G}$ is \emph{conformal} if $G-X$ has a perfect matching.
If $M$ is a perfect matching and $\E{G-X}\cap M$ is a perfect matching of $G-X$, we call $X$ \emph{$M$-conformal}.
A subgraph $H$ of $G$ for which $\V{H}$ is ($M$-)conformal is called \emph{($M$-)conformal}.
If $v\in\V{G}$ is a vertex of degree exactly two, we call the process of contracting both edges incident with $v$ and removing all loops and parallel edges afterwards the \emph{bicontraction} of $v$.
A graph $H$ that can be obtained from a conformal subgraph of $G$ by repeatedly applying bicontractions is called a \emph{matching minor}.
A path $P$ is \emph{$M$-alternating} if there exists a set $S\subseteq\V{P}$ of endpoints of $P$ such that $P-S$ is a conformal subgraph of $G$ and we say that $P$ is \emph{alternating} if there exists a perfect matching $M$ of $G$ such that $P$ is $M$-alternating.
$P$ is $M$-conformal if $S=\emptyset$ and $P$ is \emph{internally $M$-conformal} if $S$ contains both endpoints of $P$.

Let $G$ be a graph and $X\subseteq\V{G}$.
We denote by $\Cut{X}$ the set of edges in $G$ with exactly one endpoint in $X$ and call $\Cut{X}$ the \emph{edge cut} around $X$ in $G$.
Let us denote by $\Perf{G}$ the set of all perfect matchings in $G$.
An edge cut $\Cut{X}$ is \emph{tight} if $\Abs{\Cut{X}\cap M}=1$ for all perfect matchings $M\in\Perf{G}$.
If $\Cut{X}$ is a tight cut and $\Abs{X}\geq 2$, it is \emph{non-trivial}.
Identifying the \emph{shore} $X$ of a non-trivial tight cut $\Cut{X}$ into a single vertex is called a \emph{tight cut contraction} and the resulting graph $G'$ can easily be seen to be matching covered again.
A bipartite matching covered graph without non-trivial tight cuts is called a \emph{brace}.
It follows from a famous result of Lov\'asz \cite{lovasz1987matching} that the braces of a bipartite matching covered graph $B$ are uniquely determined.

Similar to how every $2$-connected minor of a graph $G$ must be a minor of one of its blocks, every brace that is a matching minor of some bipartite matching covered graph $B$ be a matching minor of some brace of $B$.
One brace in particular, namely $K_{3,3}$, has been at the centre of attention because of its connections to $K_{3,3}$-free Orientations (see \cite{mccuaig2004polya} for an overview on the topic).

\paragraph{Notational Conventions}
Let us also fix some conventions for this paper.
Given integers $i,j\in\Z$ we denote by $[i,j]$ the set $\CondSet{z\in\Z}{1\leq z\leq j}$.
This means in particular that $[i,j]=\emptyset$ in case $j<i$.
Given two sets $X$ and $Y$ we denote by $X\Delta Y$ their \emph{symmetric difference} $\Brace{X\setminus Y}\cup Y\setminus X$.

Since the whole paper mainly deals with bipartite graphs it is convenient for us to always assume that any bipartite graph $B$ comes together with a bipartition of its vertex set into two stable sets which we call $V_1$ and $V_2$.
These two sets are called the \emph{colour classes} of $B$ and in our figures $V_1$ is usually depicted as a family of black vertices, while the vertices of $V_2$ are depicted as white.
We sometimes write $\Vi{i}{B}$ in case we want to emphasise that we are speaking of the vertices of colour $i\in [1,2]$ in a specific bipartite graph $B$.

Given a bipartite graph $B$ with a perfect matching we say that $B$ \emph{contains} $K_{3,3}$ if it has a matching minor isomorphic to $K_{3,3}$.
If $B$ does not contain $K_{3,3}$ we say that $B$ is \emph{$K_{3,3}$-free}.

The following definition can be seen as a matching theoretic variant of clique sums.

\begin{definition}[$4$-Cycle Sum]\label{def:cyclesum}
For every $i\in\Set{1,2,3}$ let $B_i$ be a bipartite graph with a perfect matching and $C_i$ be a conformal cycle of length four in $B_i$.
A \emph{$4$-cycle-sum} of $B_1$ and $B_2$ at $C_1$ and $C_2$ is a graph $B'$ obtained by identifying $C_1$ and $C_2$ into the cycle $C'$ and possibly forgetting some of its edges.
If a bipartite graph $B''$ is a $4$-cycle-sum of $B'$ and some bipartite and matching covered graph $B_3$ at $C'$ and $C_3$, then $B''$ is called a \emph{trisum} of $B_1$, $B_2$ and $B_3$.
\end{definition}

The \emph{Heawood graph} is the bipartite graph associated with the incidence matrix of the Fano plane, see \cref{fig:heawood} for an illustration.
Including one exception in the form of the Heawood graph, the structure theorem for $K_{3,3}$-free braces bears a striking resemblance to Wagner's characterisation of $K_5$ minor free graphs \cite{wagner1937eigenschaft}.

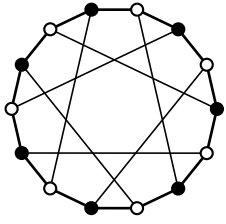
\begin{figure}[!h]
	\centering
	\begin{tikzpicture}[scale=0.9]
		\pgfdeclarelayer{background}
		\pgfdeclarelayer{foreground}
		\pgfsetlayers{background,main,foreground}
		
		\foreach \x in {2,4,6,8,10,12,14}
		{
			\node[v:main] () at (\x*25.71:15mm){};
		}
		
		\foreach \x in {1,3,5,7,9,11,13}
		{
			\node[v:mainempty] () at (\x*25.71:15mm){};
		}
		
		\begin{pgfonlayer}{background}
			\foreach \x in {2,4,6,8,10,12,14}
			{
				\draw[e:mainthin] (\x*25.71:15mm) to (128.55+\x*25.71:15mm);
			}
			
			\foreach \x in {1,...,14}
			{
				\draw[e:main] (\x*25.71:15mm) to (25.71+\x*25.71:15mm);
			}
		\end{pgfonlayer}
	\end{tikzpicture}
	\caption{The Heawood graph $H_{14}$.}
	\label{fig:heawood}
\end{figure}

\begin{theorem}\cite{mccuaig2004polya,robertson1999permanents}\label{thm:trisums}
	A brace is $K_{3,3}$-free if and only if it either is isomorphic to the Heawood graph, or it can be obtained from planar braces by repeated application of the \hyperref[def:cyclesum]{trisum} operation.
\end{theorem}

\begin{corollary}\cite{mccuaig2004polya,robertson1999permanents}\label{cor:pfaffianalg}
	There exists an algorithm that decides, given a brace $B$ as input, whether $B$ contains $K_{3,3}$ as a matching minor in time $\Fkt{\mathcal{O}}{\Abs{\V{B}}^3}$.
\end{corollary}

\subsection{A Matching Theoretic Two Paths Theorem}

The Two Paths Theorem describes the existence of crosses over cycles.
In the setting of matching theory, the cycles of a graph each fall into one of two categories: conformal cycles and non-conformal cycles.
The following lemma servers as an argument that, for the purpose of linking the existence of a cross to a topological property, it suffices to consider conformal cycles.

\begin{lemma}[\cite{mccuaig2004polya}]\label{lemma:conformalfaces}
	Let $B$ be a bipartite and planar matching covered graph, then every facial cycle of $B$ is conformal.
\end{lemma}

If we are interested in a statement like the one of \cref{thm:twopaths}, then we would expect the cycle which does not have a cross to bound a face in some kind of reduction of the original graph.
If this is the case, then \cref{lemma:conformalfaces} implies that our cycle is conformal.
Indeed, if our cycle is conformal in the end, our reductions should not have changed this and thus it should have been conformal even before applying any sort of reduction.
Hence it makes sense to only consider 'crosses' over conformal cycles.

Next we need a notion of reduction that is appropriate for the setting of bipartite graphs with perfect matchings.
We have already seen the use of the $4$-cycle sum in the \hyperref[thm:trisums]{characterisation} of bipartite graphs excluding $K_{3,3}$ as a matching minor.
The significance of the \hyperref[def:cyclesum]{trisum-operation} is that, besides one small exception, it provides a way to combine braces into larger braces.

Let $T_{10}$ be the $4$-cycle-sum of three $K_{3,3}$ at a $4$-cycle $C$ such that no edge of $C$ is in $\E{T_{10}}$.

\begin{lemma}[\cite{mccuaig2004polya}]\label{lemma:trisums}
	Let $k\geq 3$, and let $B,B_1,\dots,B_k$ be bipartite graphs such that $B$ is not isomorphic to $T_{10}$.
	Suppose $B$ is a $4$-cycle-sum of $B_1,\dots,B_n$ at the $4$-cycle $C$, then $G$ is a brace if and only if $B_1,\dots,B_n$ are braces.
\end{lemma}

The last piece we need is the definition of a 'cross' itself.

Let $B$ be a bipartite graph with a perfect matching $M$ and let $C$ be a conformal cycle in $B$.
Let $P$ be an $M$-alternating path in $B$.
If $P$ is internally $M$-conformal we say that $P$ is of \emph{type 1}, if $M$ is a perfect matching of $P$, we say $P$ is of \emph{type 2}, and otherwise exactly one of the end-edges of $P$ must belong to $M$, in this case $P$ is of \emph{type 3}.

\begin{definition}[Matching Cross]\label{def:matchingcross}
	Let $B$ be a bipartite graph with a perfect matching $M$ and let $C$ be a conformal cycle in $B$.
	The cycle $C$ is said to have a \emph{matching cross} if there exists a perfect matching $M$ and vertices $s_1$, $s_2$, $t_1$, $t_2$, called the \emph{pegs} of the cross, that appear on $C$ in the order listed such that there exist paths $P_1$ and $P_2$ satisfying the following properties:
	\begin{itemize}
		\item for each $i\in[1,2]$, $P_i$ has endpoints $s_i$ and $t_i$ and is otherwise disjoint from $C$,
		\item $P_1$ and $P_2$ are $M$-alternating, and
		\item $P_1$ and $P_2$ are vertex disjoint.
	\end{itemize}
	A matching cross over a conformal cycle $C$ is said to be \emph{strong} if it also meets the following requirements:
	\begin{itemize}
		\item  $\Abs{V_1\cap\Set{s_1,s_2,t_1,t_2}}=\Abs{V_2\cap\Set{s_1,s_2,t_1,t_2}}$, and
		\item $P_1$ and $P_2$ are of the same type
	\end{itemize}
	In case $C$ has a matching cross with paths $P_1$ and $P_2$ such that $C+P_1+P_2$ is a conformal subgraph of $B$ we say that $C$ has a \emph{conformal cross}.
\end{definition}

Please note that any path of type 1 or 2 must be of odd length and every path of type 3 is of even length.
Hence the two paths of a conformal cross are either both of type 1 or 2, or both of type 3.
Moreover, if exactly one of the two paths is of type 2, then this path together with one of the two subpaths of $C$ connecting its endpoints forms an alternating cycle.
By switching the perfect matching along this cycle we arrive at a perfect matching for which both paths are of the same type.
Therefore, if we are faced with a conformal cross we may assume this cross to be strong.

\begin{definition}\label{def:foreduction}
	Let $B$ be a brace and $C$ a conformal cycle in $B$.
	A brace $H$ is called a \emph{first order $C$-reduction} of $B$, if there exist braces $H$, $B_1,B_2$, and a $4$-cycle $K$ such that $B$ is a trisum of $H$, $B_1$, and $B_2$ at $K$, there exists $i\in[1,2]$ such that $\V{C}\cap\V{K}\subseteq V_i$, and $C\subseteq H$.
	A brace $H_{\ell}$ is called a \emph{$C$-reduction} of $B$ if there is a sequence of braces $H_1,\dots,H_{\ell}$ such that $B=H_1$ and $H_{i+1}$ is a first order $C$-reduction of $H_i$ for all $i\in[1,\ell-1]$.
\end{definition}

We can now state the main result of this paper, a Two Paths Theorem for braces.

\begin{theorem}\label{prop:twopathsinbraces}
	Let $B$ be a brace and $C$ a conformal cycle in $B$, then $C$ has \textbf{no} \hyperref[def:matchingcross]{matching cross} in $B$ if and only if $B$ does not contain $K_{3,3}$ as a matching minor and there exists a planar $C$-reduction of $B$ in which $C$ bounds a face.
\end{theorem}

Some of the intermediate results that lead to \cref{prop:twopathsinbraces} can be used to solve a slightly altered version of $2$-DAPP.

\begin{definition}[The (Bipartite) $k$-Matching Linkage Problem]
	Let $B$ be a bipartite graph with a perfect matching, $k\in\N$ a positive integer, and let $s_1,\dots,s_k\in V_1$ as well as $t_1,\dots,t_k\in V_2$ be $2k$ pairwise distinct vertices in $B$.
	A \emph{matching linkage} in $B$ for the \emph{terminals} $s_1,\dots,s_k,t_1,\dots,t_k$ is a perfect matching $M$ and a collection $P_1,\dots,P_k$ of pairwise disjoint and internally $M$-conformal paths such that $P_i$ has endpoints $s_i$ and $t_i$ for each $i\in[1,k]$.
	
	The \emph{(bipartite) $k$-Matching Linkage Problem} ($k$-MLP) is the question whether, given tuples $\Brace{s_1,\dots,s_k}$ and $\Brace{t_1,\dots,t_k}$ of vertices as above, there exists a matching linkage for the terminals $s_1,\dots,s_k,t_1,\dots,t_k$ in $B$.
\end{definition}

Please note that one can always turn an instance of $k$-DAPP into polynomially many instances of $k$-MLP by replacing vertices which appear several times as a terminal with a selection of their distance-$2$-neighbours\footnote{If $G$ is a graph and $v\in\V{G}$, then a \emph{distance $2$-neighbour} of $v$ is a vertex from $\NeighboursG{G}{\NeighboursG{G}{v}}\setminus\Brace{\NeighboursG{G}{v}\cup\Set{v}}$.}.
Hence $k$-DAPP and $k$-MLP are polynomial time equivalent, the difference is that $k$-MLP can be easier to work with since the possibility of several terminals being the same vertex does not have to be taken into account.

\begin{theorem}\label{prop:2linkage}
	Let $B$ be a bipartite graph with a perfect matching, and let $s_1,s_2\in V_1$ as well as $t_1,t_2\in V_2$ be four distinct vertices.
	There exists an algorithm that decides $2$-MLP for the terminals $s_1,s_2,t_1,t_2$ in time $\Fkt{\mathcal{O}}{\Abs{\V{B}}^5}$.
\end{theorem}

This result is somewhat surprising since, as we discuss in the \hyperref[sec:conclusion]{conclusion}, the variant of the $2$-MLP where we ask the paths to be alternating for a fixed perfect matching $M$ is polynomial time equivalent to the Directed $2$-Disjoint Paths Problem which is still NP-complete.

\paragraph{The project so far}

This paper is part of the larger project of extending the graph minors theory of Robertson and Seymour to bipartite graphs with perfect matchings.

A matching theoretic analogue of treewidth, called \emph{perfect matching width}, was introduced by Norine \cite{norine2005matching}.
Together with Hatzel and Rabinovich, the third author also derived a grid theorem for bipartite graphs with perfect matchings and perfect matching width from the related Directed Grid Theorem \cite{hatzel2019cyclewidth,rabinovich2019cyclewidth}.
This grid theorem was then further refined by the authors and Stephan Kreutzer in \cite{giannopoulou2021excluding} which allowed for a structural characterisation of all classes of bipartite graphs with perfect matchings of bounded perfect matching width:
the exclusion of a planar and matching covered matching minor.
In \cite{giannopoulou2021excluding} the complexity of $t$-DAPP on bipartite graphs of bounded perfect matching width was also discussed and a parametrised algorithm with $t$ and the perfect matching width as parameters was presented.

The current paper lays the necessary ground work towards a weak structure theorem or Flat Wall Theorem for bipartite graphs with perfect matchings.

\paragraph{Organisation of the Paper and Proof of \Cref{prop:twopathsinbraces}}\label{sec:mainthm}

Our proof of \cref{prop:twopathsinbraces} can be broken down into two essential pieces.
The first is \cref{prop:pfaffiancrosses} which characterises the existence of matching crosses over conformal cycles in $K_{3,3}$-free braces.
\Cref{sec:pfaffian} is dedicated to its proof, but the planar case, which is handled in \cref{sec:planar}, plays a major role.

\begin{proposition}\label{prop:pfaffiancrosses}
	Let $B$ be a $K_{3,3}$-free brace and $C$ a conformal cycle in $B$.
	Then there is \textbf{no} matching cross over $C$ in $B$ if and only if there exists a planar $C$-reduction of $B$ in which $C$ bounds a face.
\end{proposition}

The second part is \cref{prop:4cycleK33}, which guarantees conformal crosses over $4$-cycles in braces that contain $K_{3,3}$, this proposition is proved in \cref{sec:nonpfaffian}.

\begin{proposition}\label{prop:4cycleK33}
	Let $B$ be a brace containing $K_{3,3}$ and $C$ a $4$-cycle in $B$, then there exists a conformal bisubdivision of $K_{3,3}$ with $C$ as a subgraph.
\end{proposition}

In \cref{sec:nonplanar}, we establish some preliminary results which are needed for both the $K_{3,3}$-free case and the case where $B$ contains $K_{3,3}$, especially regarding paths and matching crosses over $4$-cycles.
An important role, in order to bridge between the existence of matching crosses and \cref{prop:4cycleK33} is held by the following lemma.

\begin{lemma}\label{lemma:goodcrossesmeanK33}
	Let $B$ be a brace and $C$ a $4$-cycle in $B$, then there is a conformal cross over $C$ in $B$ if and only if $C$ is contained in a conformal bisubdivision of $K_{3,3}$.	
\end{lemma}

As a last piece, we need to be able to make use of strong matching crosses over certain conformal cycles to find (not necessarily strong) matching crosses more easily.
This is especially useful when paired with \cref{prop:4cycleK33} but also finds applications in other places within this chapter. 

\begin{lemma}\label{lemma:extendingcrosses}
	Let $B$ be a brace and $C$ a conformal cycle.
	If $C'\neq C$ is a conformal cycle such that there exists an edge $e\notin\E{C}$ with both endpoints on $C$ with $C'\subseteq C+e$, and there is a perfect matching $M$ of $B$ and $M$-alternating paths $L$ and $R$ that form a matching cross over $C'$, then there exists a matching cross over $C$ in $B$ that does not use $e$.
\end{lemma}

\begin{proof}
	Since $C'\subseteq C+e$, $C'-e$ forms a subpath of $C$, hence the order of the endpoints of $L$ and $R$ on $C$ is the same as the order of these vertices on $C'$.
	So in case $L$ and $R$ are internally disjoint from $C$, they immediately form a matching cross over $C$ as well.
	Suppose exactly one of these paths, say $L$, intersects $C$.
	Let $v_R$ and $w_R$ be the endpoints of $R$ and $v_1$, $w_1$ be the endpoints of $L$.
	Then let $x_1$ be the last vertex of $L$ on $C$ we encounter when traversing along $L$ starting with $v_1$ such that $x_1$ is separated from $w_1$ on $C$ by $v_2$ and $w_2$.
	Next let $x_2$ be the first vertex of $L$ that lies on $C$ we encounter after $x_1$.
	Then $x_2$ must belong to the same component of $C-v_2-w_2$ as $w_1$ and thus $x_1$ and $x_2$ are separated on $C$ by $v_2$ and $w_2$.
	Since $L$ and $R$ are disjoint and $M$-alternating, so are $x_1Lx_2$ and $R$ and thus we have found a matching cross over $C$.
	So now assume that also $R$ intersects $C$.
	In this case, let $y_1$ be the last vertex we encounter when traversing along $R$ starting in $v_2$ such that $y_1$ and $w_2$ belong to different components of $C-x_1-x_2$.
	Then let $y_2$ be the first vertex of $C$ we encounter on $R$ after $y_1$.
	By choice of $y_1$, $y_1$ and $y_2$ must belong to different components of $C-x_1-x_2$ and $y_1Ry_2$ is internally disjoint from $C$.
	Hence $x_1Lx_2$ and $y_1Ry_2$ form a matching cross over $C$.
	Moreover, since $e$ is not contained in either $L$ or $R$, we have found a matching cross over $C$ which does not contain $e$.
\end{proof}

The four results above combined yield a short proof of \cref{prop:twopathsinbraces}.

\begin{proof}[Proof of \Cref{prop:twopathsinbraces}]
	Let $B$ be a brace and $C$ a conformal cycle in $B$.
	Suppose $B$ is Pfaffian, then \cref{prop:pfaffiancrosses} immediately yields both directions of our claim.
	Hence we may assume $B$ to be non-Pfaffian.
	If $C$ is a $4$-cycle, then \cref{prop:4cycleK33} guarantees the existence of a conformal cross over $C$ in $B$.
	So we may assume $C$ to have length at least six.
	Let $P$ be a subpath of $C$ of length three, so $P$ consists of exactly four vertices, two of each colour class.
	Let $a\in V_1$ and $b\in V_2$ be the endpoints of $P$.
	If the edge $ab$ does not exist in $B$ we introduce it, please note that introducing an edge does not change the status of $B$ being a brace, nor can $B+ab$ be $K_{3,3}$-free if $B$ is not.
	Hence $C'\coloneqq P+ab$ is a $4$-cycle in $B+ab$ such that $C'\subseteq C+ab$.
	By \cref{prop:4cycleK33} there is a conformal bisubdivision $L$ of $K_{3,3}$ in $B+ab$ that has $C'$ as a subgraph.
	By \cref{lemma:goodcrossesmeanK33} there is a conformal cross over $C'$.
	Please note that, with $L$ being a bisubdivision of $K_{3,3}$, we may choose a perfect matching $M$ of $B$ such that $L$ is $M$-conformal and $ab\notin M$.
	An application of \cref{lemma:extendingcrosses} now yields a matching cross over $C$ in $B+ab$ that does not contain $ab$.
	With $ab\notin M$ this means that there is a matching cross over $C$ in $B$.
\end{proof}

With \cref{prop:twopathsinbraces}, we have a tool that can help us to obtain an algorithmic solution of $2$-MLP.
In fact, it is \cref{prop:4cycleK33}, which provides the important insight.
\Cref{thm:twopaths} can be used to solve the $2$-Linkage Problem by introducing a small local construction.
In \cref{sec:algorithm}, we describe how a similar construction can be used for the $2$-MLP.

\section{Matching Crosses in Planar Braces}\label{sec:planar}

In this section we establish the base case of \cref{prop:twopathsinbraces} in form of an exact characterisation of the existence of matching crosses over conformal cycles in planar braces.

\begin{proposition}\label{prop:planarcrosses}
	Let $B$ be a brace and $C$ a conformal cycle in $B$, then there exists a strong matching cross over $C$ in $B$ if and only if $C$ does not bound a face.
\end{proposition}

Since every matching cross over $C$ in $B$ also is an ordinary\footnote{Ordinary here means a standard undirected cross in the sense of \cref{thm:twopaths}.} $C$-cross, the existence of such a cross immediately certifies that it is impossible to draw $B$ such that $C$ bounds a face.
We call a cycle $C$ in a planar graph $B$ \emph{separating} if $B$ cannot be drawn in a way such that $C$ bounds a face.
So we only need to show the reverse direction.

Before we can continue, some additional information about braces is necessary.
Let $G$ be a graph and $k$ be a positive integer.
We say that $G$ is \emph{$k$-extendible} if it is connected, has at least $2k+2$ vertices, and for every matching $F\subseteq\E{G}$ there exists a perfect matching $M$ of $G$ with $F\subseteq M$.

\begin{theorem}[\cite{lovasz2009matching}]\label{thm:braces}
	A bipartite graph $B$ is a brace if and only if it is either isomorphic to $C_4$, or it is $2$-extendible.
\end{theorem}

\begin{theorem}[\cite{plummer1980n}]\label{thm:extendibilitytoconnectivity}
	Let $k\in\N$ be a positive integer.
	Then every $k$-extendible graph is $\Brace{k+1}$-connected.
\end{theorem}

Extendibility in bipartite graphs can be expressed in many different ways.
Of particular interest for us will be the existence of disjoint alternating paths, a property similar to strong $k$-connectivity in digraphs.

\begin{theorem}[\cite{plummer1986matching,aldred2003m}]\label{thm:bipartiteextendibility}
	Let $B$ be a bipartite graph and $k\in\N$ a positive integer.
	The following statements are equivalent.
	\begin{enumerate}
		\item $B$ is $k$-extendible.
		\item $\Abs{V_1}=\Abs{V_2}$, and for all non-empty $S\subseteq V_1$, $\Abs{\NeighboursG{B}{S}}\geq \Abs{S}+k$.
		\item For all sets $S_1\subseteq V_1$ and $S_2\subseteq V_2$ with $\Abs{S_1}=\Abs{S_2}\leq k$ the graph $B-S_1-S_2$ has a perfect matching.
		\item There is a perfect matching $M\in\Perf{B}$ such that for every $v_1\in V_1$, every $v_2\in V_2$ there are $k$ pairwise internally disjoint internally $M$-conformal paths with endpoints $v_1$ and $v_2$.
		\item For every perfect matching $M\in\Perf{B}$, every $v_1\in V_1$, every $v_2\in V_2$ there are $k$ pairwise internally disjoint internally $M$-conformal paths with endpoints $v_1$ and $v_2$.
	\end{enumerate}
\end{theorem}

Since every brace is $3$-connected by \cref{thm:braces,thm:extendibilitytoconnectivity}, we can rely on the uniqueness of plane embeddings for $3$-connected graphs \cite{whitney1992congruent}.

Let $B$ be a brace and $C$ a conformal and separating cycle in $B$.
Then by Whitney's Theorem \cite{whitney1992congruent} the interior and the exterior of $C$ are the same in every drawing of $B$ (up to the choice of the outer face).
In what follows, we always assume $B$ to come with a fixed drawing to avoid ambiguity.
We denote the subgraph of $B$ induced by the interior of $C$ together with $C$ itself by $\Inner{B}{C}$ and the subgraph of $B$ induced by the exterior of $C$ together with $C$ is denoted by $\Outer{B}{C}$.
In both graphs $C$ bounds a face and, since $C$ is conformal, both graphs have a perfect matching.

\begin{lemma}\label{lemma:planarnonbraceparts}
	Let $B$ be a planar brace and $C$ a conformal and separating cycle in $B$, then $\Inner{B}{C}$ and $\Outer{B}{C}$ are matching covered.
\end{lemma}

\begin{proof}
	It suffices to show the claim for one of the two graphs.
	Moreover, by \cref{thm:bipartiteextendibility} it suffices to show for a single perfect matching $M$, that any pair $a\in V_1$, $b\in V_2$ of vertices is linked by an internally $M$-conformal path.
	So let us fix a perfect matching $M$ for which $C$ is $M$-conformal.
	Since $B$ is a brace, given $a\in \Vi{1}{\Inner{B}{C}}$ and $b\in \Vi{2}{\Inner{B}{C}}$, \cref{thm:bipartiteextendibility} guarantees the existence of two internally disjoint and internally $M$-conformal paths $P_1$, $P_2$ from $a$ to $b$.
	If one of these paths is disjoint from $C$, there is nothing to show.
	Hence we may assume that both meet $C$ and, by a similar argument, both of them need to contain an edge of $\Outer{B}{C}-\E{C}$.
	Let $b_1$ be the first vertex of $P_1$ on $C$ when traversing $P_1$ from $a$ towards $b$ and let $a_1$ be the last vertex of $P_1$ on $C$.
	Then $a_1$ and $b_1$ separate $C$ into two paths, one of them being $M$-conformal, let $P'$ be this path.
	Moreover, $P_1b_1$ and $a_1P_1$ are internally $M$-conformal, and all three paths are contained in $\Inner{B}{C}$.
	Hence $P_1b_1P'a_1P_1$ is an internally $M$-conformal $a$-$b$-path in $\Inner{B}{C}$, and we are done.
\end{proof}

Towards the next partial result we need to introduce some additional information on tight cuts in bipartite graphs.

Let $B$ be a bipartite graph, and $X\subseteq\V{G}$.
If $\Abs{X\cap V_1}=\Abs{X\cap V_2}$ we say that $X$ is \emph{balanced}, otherwise it is \emph{unbalanced}.
Suppose $X$ is unbalanced, then there are $i,j\in[1,2]$, and $k\in\N$ such that $\Abs{X\cap V_i}=\Abs{X\cap V_j}+k$.
In this case we call $X\cap V_i$ the \emph{majority} of $X$, denoted by $\Majority{X}$, and $X\cap V_j$ is the \emph{minority}, denoted by $\Minority{X}$.

\begin{lemma}[See the proof of Lemma 1.4 in \cite{lovasz1987matching}]\label{lemma:ktightmajorityminority}
	Let $B$ be a bipartite matching covered graph and $X\subseteq\V{B}$ where $\Abs{X}$ is odd.
	Then $\Cut{X}$ is a tight cut if and only if $\Abs{\Majority{X}}-\Abs{\Minority{X}}=1$ and $\Neighbours{\Minority{X}}\subseteq\Majority{X}$. 
\end{lemma}

\begin{lemma}\label{lemma:planartightcutsthroughsepcycles}
	Let $B$ be a planar brace, $C$ be a conformal and separating cycle in $B$ such that $B'\in\Set{\Inner{B}{C},\Outer{B}{C}}$ is not a brace, and $\CutG{B'}{X}$ be a non-trivial tight cut in $B'$.
	Then $\Abs{X\cap\V{C}}\geq 3$ and $\Abs{\Complement{X}\cap\V{C}}\geq 3$.
\end{lemma}

\begin{proof}
	Suppose $\Abs{X\cap\V{C}}\leq 1$.
	By symmetry, it suffices to treat this case.
	With \cref{lemma:ktightmajorityminority} we know that the minority of $X$ has no edge to a vertex of $\Complement{X}$, without loss of generality let us assume the majority of $\CutG{B}{X}$ to be in $V_1$.
	If $\V{C}\cap X=\emptyset$, then clearly $\CutG{B}{X}$ must be a non-trivial tight cut in $B$, which is impossible.
	Hence there must exist a unique vertex $a\in\V{C}\cap X$.
	Moreover, since this vertex has neighbours in $C$ which do not belong to $X$, $a\in V_1$.
	But also in this case $\CutG{B}{X}$ is non-trivially tight in $B$ and the claim follows.
\end{proof}

\begin{corollary}\label{cor:planar4cycledifference}
	Let $B$ be a planar brace, $C$ a conformal and separating $4$-cycle in $B$, then $\Inner{B}{C}$ and $\Outer{B}{C}$ are braces.
\end{corollary}

Please note that \cref{lemma:planarnonbraceparts,lemma:planartightcutsthroughsepcycles,cor:planar4cycledifference} can be extended to non-planar braces as well.
The general version of \cref{cor:planar4cycledifference} is due to McCuaig.

\begin{lemma}[\cite{mccuaig2004polya}]\label{lemma:4cyclesumwithtwosummands}
	Let $B_1$ and $B_2$ be bipartite graphs with a common $4$-cycle $C$ and otherwise disjoint. If $B_1+B_2$ is a brace, then so are $B_1$ and $B_2$.
\end{lemma}

The next lemma is a slightly restated version of Lemma 46 from \cite{mccuaig2004polya} which can be derived with the methods presented there.

\begin{lemma}[\cite{mccuaig2004polya}]\label{lemma:braceears}
	Let $B$ be a planar brace, $C$ be a facial cycle of $B$, $M$ be a perfect matching of $B$ for which $C$ is $M$-conformal, and $a\in \Vi{1}{C}$, $b\in \Vi{2}{C}$ be two vertices with $ab\notin\E{C}$.
	Then there exists an internally $M$-conformal $a$-$b$-path $P$ in $B$ which is internally disjoint from $C$.
\end{lemma}

As a special case, we first assume that we are interested in a (strong) matching cross over some separating cycle $C$ in a planar brace $B$ where both $\Inner{B}{C}$ and $\Outer{B}{C}$ are braces.
We also need some deeper insight in how conformal bisubdivisions of $K_{3,3}$ and the cube can appear in braces with respect to cycles of length four.

\begin{lemma}[\cite{mccuaig2004polya}]\label{lemma:cubeorK33original}
	Let $B$ be a brace and $C$ a $4$-cycle such that $B-\V{C}$ is connected, let $uv\in\E{C}$ and $x,y\in\V{B}\setminus\V{C}$ such that $ux,vy\in\E{B}$.
	Then $B$ contains a conformal bisubdivision of the cube with $C+ux+vy$ as a subgraph, or $B$ contains a conformal bisubdivision of $K_{3,3}$ with $C$ as a subgraph.
\end{lemma}

The following is a slight weakening of the lemma above.

\begin{corollary}[\cite{mccuaig2004polya}]\label{lemma:cubeorK33}
	Let $B$ be a brace and $C$ a $4$-cycle such that $B-\V{C}$ is connected, then $B$ contains a conformal bisubdivision of the cube or $K_{3,3}$ with $C$ as a subgraph.
\end{corollary}

Please note that a version of \cref{lemma:cubeorK33} can be found in \cite{robertson1999permanents}, where the containment of a $4$-cycle in a conformal bisubdivision of the cube is referred to as being 'fat' while being a subgraph of some conformal bisubdivision of $K_{3,3}$ is called having a '$C$-cross'.
Sadly these two are not mutually exclusive as one can see in the example in \cref{fig:basegraphs}.
However, there is a deeper connection between the existence of matching crosses, especially conformal ones, and the existence of conformal bisubdivisions of $K_{3,3}$ when it comes to $4$-cycles.
We revisit this topic in \cref{sec:nonplanar}.

\begin{figure}[h!]
	\begin{subfigure}{0.3\textwidth}
		\centering
		\begin{tikzpicture}[scale=0.9]
			\pgfdeclarelayer{background}
			\pgfdeclarelayer{foreground}
			\pgfsetlayers{background,main,foreground}
			
			\node[v:mainempty] () at (-0.7,-0.7){};
			\node[v:main] () at (-0.7,0.7){};
			\node[v:main] () at (0.7,-0.7){};
			\node[v:mainempty] () at (0.7,0.7){};
			\node[v:main] () at (-0.4,0.1){};
			\node[v:mainempty] at (0.4,0.1){};
			
			\begin{pgfonlayer}{background}
				\draw[e:main] (-0.7,-0.7) rectangle (0.7,0.7);
				\draw[e:main] (-0.4,0.1) -- (0.4,0.1);
				\draw[e:main] (-0.7,0.7)-- (0.4,0.1);
				\draw[e:main] (0.7,0.7)-- (-0.4,0.1);
				\draw[e:main] (-0.7,-0.7) -- (-0.4,0.1);
				\draw[e:main] (0.7,-0.7) -- (0.4,0.1);
			\end{pgfonlayer}
		\end{tikzpicture}
	\end{subfigure}
	\begin{subfigure}{0.3\textwidth}
		\centering
		\begin{tikzpicture}[scale=0.9]
			\pgfdeclarelayer{background}
			\pgfdeclarelayer{foreground}
			\pgfsetlayers{background,main,foreground}
			
			\node[v:mainempty] () at (-0.5,-0.5){};
			\node[v:main] () at (-0.5,0.5){};
			\node[v:main] () at (0.5,-0.5){};
			\node[v:mainempty] () at (0.5,0.5){};
			
			\node[v:main] () at (-1,-1){};
			\node[v:mainempty] () at (1,-1){};
			\node[v:mainempty] () at (-1,1){};
			\node[v:main] () at (1,1){};

			\begin{pgfonlayer}{background}
				\draw[e:main] (-1,-1) rectangle (1,1);
				\draw[e:main] (-0.5,-0.5) rectangle (0.5,0.5);
				\draw[e:main] (-0.5,-0.5) -- (-1,-1);
				\draw[e:main] (-0.5,0.5) -- (-1,1);
				\draw[e:main] (0.5,-0.5) -- (1,-1);
				\draw[e:main] (0.5,0.5) -- (1,1);
			\end{pgfonlayer}
			
		\end{tikzpicture}
	\end{subfigure}
	\begin{subfigure}{0.3\textwidth}
		\centering
		\begin{tikzpicture}[scale=0.9]
			
			\pgfdeclarelayer{background}
			\pgfdeclarelayer{foreground}
			\pgfsetlayers{background,main,foreground}
			
			\node[v:main] () at (-1,-1){};
			\node[v:mainempty] () at (1,-1){};
			\node[v:mainempty] () at (-1,1){};
			\node[v:main] () at (1,1){};
			
			\node[v:mainempty] () at (-0.7,-0.7){};
			\node[v:main] () at (-0.7,0.7){};
			\node[v:main] () at (0.7,-0.7){};
			\node[v:mainempty] () at (0.7,0.7){};
			\node[v:main] () at (-0.4,0.1){};
			\node[v:mainempty] at (0.4,0.1){};

			\begin{pgfonlayer}{background}
				\draw[e:marker] (-0.7,-0.7) rectangle (0.7,0.7);

				\draw[e:main] (-0.7,-0.7) rectangle (0.7,0.7);
				\draw[e:main] (-0.4,0.1) -- (0.4,0.1);
				\draw[e:main] (-0.7,0.7)-- (0.4,0.1);
				\draw[e:main] (0.7,0.7)-- (-0.4,0.1);
				\draw[e:main] (-0.7,-0.7) -- (-0.4,0.1);
				\draw[e:main] (0.7,-0.7) -- (0.4,0.1);
				
				\draw[e:main] (-1,-1) rectangle (1,1);
				\draw[e:main] (-0.7,-0.7) -- (-1,-1);
				\draw[e:main] (-0.7,0.7) -- (-1,1);
				\draw[e:main] (0.7,-0.7) -- (1,-1);
				\draw[e:main] (0.7,0.7) -- (1,1);
			\end{pgfonlayer}
			
		\end{tikzpicture}
	\end{subfigure}
	\caption{From left to right: $K_{3,3}$, the cube, and a brace with a $4$-cycle (the marked one) which is contained in both, a conformal bisubdivision of $K_{3,3}$ and a conformal bisubdivision of the cube.
		Please note that one could get rid of the fact that the marked cycle is separating by adding additional edges between vertices from different colour classes.}
	\label{fig:basegraphs}
\end{figure}
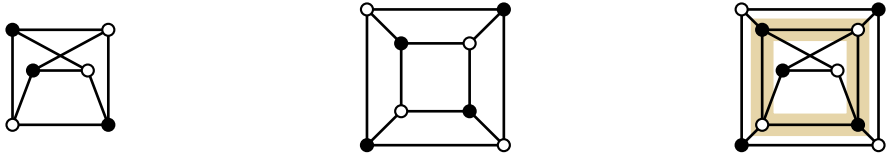

\begin{lemma}\label{lemma:braceparts}
	Let $B$ be a planar brace and $C$ a conformal and separating cycle in $B$ such that $\Inner{B}{C}$ and $\Outer{B}{C}$ both are braces.
	Then there exists a strong matching cross over $C$ in $B$.
\end{lemma}

\begin{proof}
	First, assume $C$ to have at least length $6$.
	In this case, we can select a perfect matching $M$ of $B$ such that $C$ is $M$-conformal, which in turn implies that $M$ contains perfect matchings $M_{\text{int}}$ and $M_{\text{out}}$ of $\Inner{B}{C}$ and $\Outer{B}{C}$ respectively.
	Now select vertices $s_1$, $s_2$, $t_1$, and $t_2$ such that they appear on $C$ in the order listed where $s_1,s_2\in V_1$, and $t_1,t_2\in V_2$.
	According to \cref{lemma:braceears}, we may choose an internally $M$-conformal $s_1$-$t_1$-path $P_1$ in $\Inner{B}{C}$ and an internally $M$-conformal $s_2$-$t_2$-path $P_2$ in $\Outer{B}{C}$ such that each of the $P_i$ is internally disjoint from $C$.
	Then $P_1$ and $P_2$ form a $C$-cross in $B$, and we are done.
	
	What remains is the case where $C=\Brace{s_1,t_1,s_2,t_2}$ is a $4$-cycle.
	By calling upon \cref{lemma:cubeorK33} we can find a conformal bisubdivision of the cube in each of the two braces such that each of these cubes contains $C$ as a subgraph.
	Let $H_1$ be a conformal bisubdivision of the cube in $\Inner{B}{C}$ and let $H_2$ be a conformal bisubdivision of the cube in $\Outer{B}{C}$.
	One can easily see, that $H\coloneqq H_1+H_2$ is a conformal subgraph of $B$.
	We can now use $H$ to find the required matching cross over $C$, as illustrated in \cref{fig:planecross}.
\end{proof}

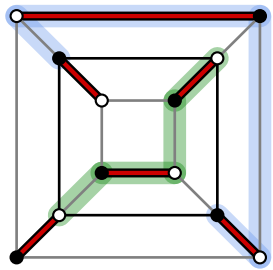
\begin{figure}[h!]
	\centering
	\begin{tikzpicture}[scale=0.8]
		\pgfdeclarelayer{background}
		\pgfdeclarelayer{foreground}
		\pgfsetlayers{background,main,foreground}
		
		\node[v:main] () at (-0.6,-0.6){};
		\node[v:mainempty] () at (-0.6,0.6){};
		\node[v:mainempty] () at (0.6,-0.6){};
		\node[v:main] () at (0.6,0.6){};
		
		\node[v:mainempty] () at (-1.3,-1.3){};
		\node[v:main] () at (-1.3,1.3){};
		\node[v:main] () at (1.3,-1.3){};
		\node[v:mainempty] () at (1.3,1.3){};
		
		\node[v:main] () at (-2,-2){};
		\node[v:mainempty] () at (-2,2){};
		\node[v:mainempty] () at (2,-2){};
		\node[v:main] () at (2,2){};   
		
		\begin{pgfonlayer}{background}
			\draw[e:marker,color=myLightBlue] (-1.3,1.3) -- (-2,2);
			\draw[e:marker,color=myLightBlue] (-2,2) -- (2,2);
			\draw[e:marker,color=myLightBlue] (2,2) -- (2,-2);
			\draw[e:marker,color=myLightBlue] (2,-2) -- (1.3,-1.3);
			
			\draw[e:marker,color=myGreen] (1.3,1.3) -- (0.6,0.6);
			\draw[e:marker,color=myGreen] (0.6,0.6) -- (0.6,-0.6);
			\draw[e:marker,color=myGreen] (0.6,-0.6) -- (-0.6,-0.6);
			\draw[e:marker,color=myGreen] (-0.6,-0.6) -- (-1.3,-1.3);
			
			\draw[e:main] (-1.3,-1.3) rectangle (1.3,1.3);
			
			\draw[e:main,color=gray] (-2,2) -- (-2,-2) -- (2,-2) -- (2,2);
			
			\draw[e:main,color=gray] (-0.6,-0.6) -- (-0.6,0.6) -- (0.6,0.6) -- (0.6,-0.6);
			
			\draw[e:coloredborder] (0.6,-0.6) -- (-0.6,-0.6);
			\draw[e:coloredthin,color=BostonUniversityRed] (0.6,-0.6) -- (-0.6,-0.6);
			
			\draw[e:coloredborder] (2,2) -- (-2,2);
			\draw[e:coloredthin,color=BostonUniversityRed] (2,2) -- (-2,2);
			
			\draw[e:coloredborder] (-1.3,-1.3) -- (-2,-2);
			\draw[e:coloredthin,color=BostonUniversityRed]  (-1.3,-1.3) -- (-2,-2);
			
			\draw[e:coloredborder] (1.3,-1.3) -- (2,-2);
			\draw[e:coloredthin,color=BostonUniversityRed]  (1.3,-1.3) -- (2,-2);
			
			\draw[e:main,color=gray] (-1.3,-1.3) -- (-0.6,-0.6);
			\draw[e:main,color=gray] (1.3,-1.3) -- (0.6,-0.6);
			
			\draw[e:coloredborder] (-1.3,1.3) -- (-0.6,0.6);
			\draw[e:coloredthin,color=BostonUniversityRed]  (-1.3,1.3) -- (-0.6,0.6);
			
			\draw[e:coloredborder] (1.3,1.3) -- (0.6,0.6);
			\draw[e:coloredthin,color=BostonUniversityRed]  (1.3,1.3) -- (0.6,0.6);
			
			\draw[e:main,color=gray] (-1.3,1.3) -- (-2,2);
			\draw[e:main,color=gray] (1.3,1.3) -- (2,2);

		\end{pgfonlayer}
	\end{tikzpicture}
	\caption{The conformal subgraph $H$ in the proof of \cref{lemma:braceparts} together with a strong matching cross over the separating $4$-cycle $C$.}
	\label{fig:planecross}
\end{figure}

\begin{lemma}\label{lemma:tightcutsthroughseparators}
	Let $B$ be a bipartite matching covered graph, $M$ a perfect matching of $B$, $e=ab\in M$ with $a\in V_1$, $b\in V_2$, $X\subseteq V_1\setminus\Set{a}$ and $Y\subseteq V_2\setminus\Set{b}$ such that every internally $M$-conformal $X$-$Y$-path in $B$ contains $e$.
	Then there exists a tight cut $\CutG{B}{Z}$ in $B$ with $X\subseteq Z$, $Y\subseteq\Complement{Z}$, and $e\in\CutG{B}{Z}$.
\end{lemma}

\begin{proof}
	By assumption, there is no internally $M$-conformal $X$-$Y$-path in $B-a-b$ and thus no vertex of $X$ can share an elementary component with a vertex of $Y$.
	Since, by definition, each elementary component would be matching covered and thus, \cref{thm:bipartiteextendibility} would guarantee the existence of such a path.
	Let $\Fkt{Up}{X}$ be the set of all vertices $w$ of $B-a-b$ such that there exist elementary components $K_X$ and $K_w$ with $K_X\leq_2 K_w$ where $K_X$ contains a vertex of $X$ and $w\in\V{K_w}$.
	Then $Y\cap\Fkt{Up}{X}=\emptyset$.
	Moreover, there is no $\Vi{1}{\Fkt{Up}{X}}$-$\Vi{2}{\Complement{\Fkt{Up}{X}}}$-path in $B-a-b$ at all.
	Hence in $B$ $\Fkt{Up}{X}\cup\Set{b}$ is a set of odd cardinality, where no vertex of $V_1$ has a neighbour outside of it and the difference between the number of $V_1$-vertices and the number of $V_1$-vertices is exactly one.
	So by \cref{lemma:ktightmajorityminority} $\CutG{B}{\Fkt{Up}{X}\cup\Set{b}}$ is a tight cut with $Y\subseteq\Complement{\Fkt{Up}{X}\cup\Set{b}}$.
\end{proof}

Together with the upcoming lemma, \cref{cor:planar4cycledifference,lemma:braceparts} imply \cref{prop:planarcrosses}.

\begin{lemma}\label{lemma:planarnonbracecrosses}
	Let $B$ be a planar brace, $C$ a conformal and separating cycle in $B$ such that $B'\in\Set{\Inner{B}{C},\Outer{B}{C}}$ is not a brace.
	Then there exists a strong matching cross over $C$ in $B$.
\end{lemma}

\begin{proof}
	By \cref{cor:planar4cycledifference} we may assume $\Abs{\V{C}}\geq 6$.
	Let $\CutG{B'}{X}$ be a non-trivial tight cut in $B'$ maximising $\Abs{\Complement{X}}$.
	Without loss of generality let us assume the majority of $X$ to be in $V_1$.
	With \cref{lemma:planartightcutsthroughsepcycles} and the fact that for every selection of three distinct edges of $C$, at least two of them belong to a common perfect matching, one can see that $\CutG{B'}{X}$ separates $C$ into two non-trivial paths $Q_1$ and $Q_2$ such that $\V{Q_1}\subseteq\Complement{X}$ and $\V{Q_2}\subseteq X$.
	Let $a_1, a_2\in X\cap V_1$ be the endpoints of $Q_2$ and let $b_1,b_2\in\Complement{X}\cap V_2$ be the endpoints of $Q_1$ such that $a_ib_i\in\E{C}$ for both $i\in[1,2]$.
	Next let $M$ be a perfect matching such that $C$ is $M$-conformal and $a_2b_2\in M$, moreover let $a'_1b_1, a_1b'_1\in M$.
	Since the majority of $X$ is in $V_1$, there cannot exist an internally $M$-conformal path starting at some vertex of $V_1\cap\Complement{X}$ and ending in a vertex of $V_2\cap X$ that avoids $a_2b_2$.
	However, with $B$ being a brace and \cref{thm:bipartiteextendibility} there must be two internally disjoint and internally $M$-conformal $a'_1$-$x$-paths for every $x\in X\cap V_2$ and one of them must avoid $a_2b_2$.
	Let us choose $b_X\in V_2\cap X$ such that there exists an internally $M$-conformal path $P$ from $a'_1$ to $b_X$ with the following properties:
	\begin{itemize}
		\item $a_{\Complement{X}}$ is the last vertex of $\V{C}\cap\Complement{X}$ along $P$ starting in $a'_1$, and
		\item $\V{P}\cap\V{Q_2}=\Set{b_X}$.
	\end{itemize} 
	Then $a_{\Complement{X}}\in V_1$ and $P_1\coloneqq a_{\Complement{X}}P$ is an internally $M$-conformal path which is internally disjoint from $C$.
	Moreover, $\V{P_1}\cap\V{B'}=\V{P_1}\cap\V{C}=\Set{a_{\Complement{X}},b_X}$.
	Let $Q_3$ be the component of $Q_1-a_{\Complement{X}}$ containing $b_1$, let $Y$ be the component of $C-a_{\Complement{X}}-b_X$ containing $Q_3$ and at last let $\Complement{Y}$ denote the other component of $C-a_{\Complement{X}}-b_X$.
	Let $e\in M$ be the edge covering $b_X$.
	What follows is a case distinction on the existence of some internally $M$-conformal path from $\Complement{Y}\cap V_1$ to $Y\cap V_2$.
	
	\textbf{Case 1:} There exists an internally $M$-conformal path $P'$ from $\Complement{Y}\cap V_1$ to $Y\cap V_2$ that avoids $e$.
	
	If this is the case, let $b_Y$ be the first vertex of $Y$ encountered while traversing along $P'$ starting in $\Complement{Y}\cap V_1$.
	Then let $a_{\Complement{Y}}$ be the last vertex of $P'$ in $\Complement{Y}\cap V_1$ encountered before $b_Y$.
	Now $P_2\coloneqq a_{\Complement{Y}}P'b_Y$ is an internally $M$-conformal path with no inner vertex on $C$ that is disjoint from $P_1$.
	Moreover, the vertices $a_{\Complement{Y}}$, $a_{\Complement{X}}$, $b_Y$, and $b_X$ appear on $C$ in the order listed and thus $P_1$ and $P_2$ form a strong matching cross over $C$ in $B$.
	
	\textbf{Case 2:} All internally $M$-conformal $(\Complement{Y}\cap V_1)$-$(Y\cap V_2)$-paths contain $e$.
	
	Then the deletion of both endpoints of $e$ in $B'$ leaves at least two elementary components, some containing vertices of $\Complement{Y}\cap V_1$ and some of the others containing vertices of $Y\cap V_2$ but never both.
	Thus, by \cref{lemma:tightcutsthroughseparators}, there exists a tight cut $\CutG{B'}{Z}$ with $\Complement{Y}\cap V_1\subseteq Z$, $Y\cap V_2\subseteq\Complement{Z}$, and $e\in\CutG{B}{Z}$.
	Indeed, the majority of $Z$ must be in $V_2$.
	Since $X$ is odd, one of the two sets $X\cap Z$ and $X\setminus Z$ must be odd and therefore, by \cref{lemma:tightcutuncrossing}, one of these sets defines a tight cut in $B'$.
	Clearly $a_2,b_X\in X\cap Z$ and thus $\Abs{X\cap Z}>1$.
	Since $\CutG{B}{Z}$ cannot contain more than two edges of $C$ and by choice of $\CutG{B}{X}$ there cannot be a non-trivial tight cut $\CutG{B}{X'}$ with $X'\subset X$ in $B'$, hence
	$X\setminus Z=\Set{a_1}$ and $a_1b_X\in M$.
	
	With an argument similar to the one in \textbf{Case 1} one can see that, in case there exists an internally $M$-conformal $(Y\cap V_1)$-$(\Complement{Y}\cap V_2)$-path avoiding the edge $e'\in M$ covering $a_{\Complement{X}}$, we are done again.
	So we may also assume that $e'$ meets all internally $M$-conformal $(Y\cap V_1)$-$(\Complement{Y}\cap V_2)$-paths.
	Then, with arguments as before, one derives the existence of a non-trivial tight cut $\CutG{B'}{Z'}$ such that $a_1\in Z'$ and $\Complement{Y}\cap B\subseteq\Complement{Z'}$.
	In the end we arrive at the conclusion that $X\cap Z'=\Set{a_2}$ and $e'\in\CutG{B'}{Z'}$.
	
	Moreover, this means that there is neither an internally $M$-conformal $(Y\cap V_1)$-$(\Complement{Y}\cap V_2)$-path, nor an internally $M$-conformal $(\Complement{Y}\cap V_1)$-$(Y\cap V_2)$-path in $B'$ after deleting the four vertices in $S\coloneqq e\cup e'$.
	Hence there cannot be any $Y$-$\Complement{Y}$-path in $B'-S$.
	This means there must be a face $C'$ of $B'$ containing both $a_{\Complement{X}}$ and $b_X$ that is distinct from $C$.
	Since $B$ itself is a brace, there must be an internally $M$-conformal path $P'_1$ from $Y\cap V_1$ to $\Complement{Y}\cap V_2$ that avoids $e'$ and this path may be chosen to be internally disjoint from $B'$.
	In particular, we may choose the endpoints of $P'_1$ to be disjoint from one of the two paths of $C'$, say $R$, connecting $a_{\Complement{X}}$ and $b_X$.
	This is due to the fact that every internal vertex of a path in $C\cap C'$ must be of degree two in $B'$, $B'\neq C$, and $a_{\Complement{X}}b_X\notin\E{C}$.
	Since $B'$ is matching covered by \cref{lemma:planarnonbraceparts}, \cref{lemma:conformalfaces} guarantees the existence of a perfect matching $M'$ of $B'$ for which $C'$ is $M'$-conformal.
	Let us choose $M'$ such that $R$ is an internally $M'$-conformal path and let $P'_2$ be an internally $M'$-conformal subpath of $R$ with both endpoints on $C$ and otherwise disjoint from $C$.
	All of these choices are possible since $a_{\Complement{X}}$ and $b_X$ belong to different colour classes of $B'$.
	At last let us set $M''\coloneqq M'\cup \Brace{M\setminus\E{B'}}$.
	Since $B'$ is an $M$-conformal subgraph of $B$, $M''$ is a perfect matching of $B$, and by construction, both $P'_1$ and $P'_2$ are internally $M''$-conformal paths.
	Now $P'_1$ and $P'_2$ form a strong matching cross over $C$ in $B$.
\end{proof}

We are ready to derive the main result of this section.

\begin{proof}[Proof of \cref{prop:planarcrosses}]
	Let $C$ be a conformal cycle in a planar brace.
	If $C$ does not bound a face, it is separating and thus \cref{lemma:braceparts,lemma:planarnonbracecrosses} guarantee the existence of a strong matching cross over $C$ in $B$.
	For the reverse suppose there exists a strong matching cross over $C$ in $B$, then exactly one path of the cross must lie in the interior of $C$ for every plane embedding of $B$.
	Hence $C$ does not bound a face in any plane embedding of $B$.
\end{proof}

\section{Paths and Matching Crosses through $4$-Cycle Sums}\label{sec:nonplanar}

Cycles of length four play a key role in many aspects of bipartite matching theory as we have seen in \cref{thm:trisums}.
In particular, the $4$-cycle sum operation and the fact, that no perfect matching $M$ for which an $M$-conformal matching cross over $C_4$ exists can contain a perfect matching of $C_4$ itself are things to be considered.
While tight cut contractions only preserve special types of matching crosses\footnote{It is possible to lose a matching cross if for example the tight cut separates the endpoints of both paths of the cross, in both tight cut contractions, what remains of the paths is now a set of two paths with a common endpoint.}, at some point they are not applicable any more to further decompose a given graph.

In the spirit of the Two Paths Theorem, we want to decompose our graph further while maintaining a fixed subgraph, in most cases the conformal cycle for which we seek a matching cross.
This section, and the following ones, exist to describe exactly the interaction between conformal cycles, matching crosses over these cycles, $4$-cycle sums, and matching crosses over $4$-cycles in both $K_{3,3}$-free braces which are non-planar and in braces containing $K_{3,3}$.

A brace $B$ is a \emph{maximal} $4$-cycle sum at the $4$-cycle $C$ of the braces $B_1,\dots,B_{\ell}$, $\ell\geq 3$, if there do not exist braces $H_1,\dots,H_n$, $n>\ell$ such that $B$ is a $4$-cycle sum of $H_1,\dots, H_m$ at $C$.

\begin{lemma}\label{lemma:pathsintrisums}
	Let $B$ be a $K_{3,3}$-free brace, $n\geq 3$, $B_1,\dots,B_{\ell}$ braces such that $B$ is a maximal $4$-cycle sum of $B_1,\dots,B_{\ell}$ at the $4$-cycle $C=\Brace{a_1,b_1,a_2,b_2}$ and let $M$ be a perfect matching of $B_1$.
	Then for every choice of $x\in\Set{a_1,a_2}$ and $y\in\Set{b_1,b_2}$ there exists a perfect matching $M'$ of $B$ such that $M\setminus\E{C}\subseteq M'$ and there are paths $P_1$ and $P_2$  in $B-\V{B_1-\V{C}}$ where
	\begin{itemize}
		\item $P_1$ has endpoints $x$ and $y$, while $P_2$ connects $\Set{a_1,a_2}\setminus\Set{x}$ to $\Set{b_1,b_2}\setminus\Set{y}$,
		\item $P_1$ and $P_2$ are disjoint, and
		\item $P_1+P_2+\CondSet{e\in M'}{e\cap \V{C}\neq\emptyset}$ is an $M'$-conformal subgraph of $B$.
	\end{itemize}	
\end{lemma}

\begin{proof}
	By \cref{lemma:cubeorK33} for every $i\in[2,\ell]$ $C$ is contained in a conformal bisubdivision of the cube, or of $K_{3,3}$.
	Since $B$ is $K_{3,3}$-free the later can never be true and thus for every $j\in[2,\ell]$, $C$ is contained in a conformal bisubdivision $H$ of the cube in $B_j$.
	As $H$ is conformal in $B_j$, every perfect matching of $H$ can be combined with a perfect matching of $B_j-\V{H}$ to a perfect matching of $B_j$.
	In general, let $M_H$ be a perfect matching of $H$ such that for every $e\in\E{C}\cap M_H$ both endpoints of $e$ are covered by edges of $M\setminus\E{C}$ and a vertex of $C$ is covered by a non-$\E{C}$-edge in $M_H$ if and only if it is covered by an edge of $M\cap\E{C}$.
	Moreover, let $M_j$ be a perfect matching of $B_j-\V{H}$, and for each $i\in2,\ell\setminus\Set{j}$ let $M_i$ be a perfect matching of $B_i-\V{C}$.
	The matching $M_i$ clearly exists since the $B_i$ are braces, and therefore they are $2$-extendible.
	Then $\Brace{M\setminus\E{C}}\cup\Brace{M_H\setminus\E{C}}\cup\bigcup_{i=1,i\neq j}^nM_i$ is a perfect matching of $B$.
	Hence it suffices to show that, for any $M$ we are given, we can choose the matching $M_j$ such that we are able to find the desired paths within $H$.
	What follows is a discussion of these paths depending on the number of edges in $M\cap\E{C}$.
	We present these matchings together with the paths in \cref{fig:cubepaths0edges,fig:cubepaths1edge}.
	As $H$ is a bisubdivision of the cube, each perfect matching of $H$ mirrors a perfect matching $M_H'$ of the cube in the sense that a bisubdivided edge of the cube is $M_H$-conformal if and only if the corresponding edge of the cube belongs to $M_H'$.
\end{proof}

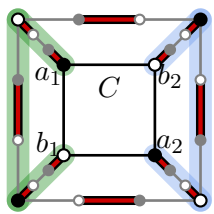
\begin{figure}[h!]
	\centering
	\begin{tikzpicture}
		\pgfdeclarelayer{background}
		\pgfdeclarelayer{foreground}
		\pgfsetlayers{background,main,foreground}
		
		\node[v:mainempty] () at (-0.6,-0.6){};
		\node[v:main] () at (-0.6,0.6){};
		\node[v:main] () at (0.6,-0.6){};
		\node[v:mainempty] () at (0.6,0.6){};
		\node[v:main] () at (-1.2,-1.2){};
		\node[v:mainempty] () at (-1.2,1.2){};
		\node[v:mainempty] () at (1.2,-1.2){};
		\node[v:main] () at (1.2,1.2){};
		
		\node[v:mainemptygray] () at (0.4,1.2){};
		\node[v:maingray] () at (0.4,-1.2){};
		\node[v:maingray] () at (-0.4,1.2){};
		\node[v:mainemptygray] () at (-0.4,-1.2){};
		
		\node[v:mainemptygray] () at (-1.2,-0.4){};
		\node[v:maingray] () at (-1.2,0.4){};
		\node[v:maingray] () at (1.2,-0.4){};
		\node[v:mainemptygray] () at (1.2,0.4){};
		
		\node[v:maingray] () at (0.8,0.8){};
		\node[v:mainemptygray] () at (1,1){};
		\node[v:maingray] () at (-0.8,0.8){};
		\node[v:mainemptygray] () at (-1,1){};
		\node[v:maingray] () at (0.8,-0.8){};
		\node[v:mainemptygray] () at (1,-1){};       
		\node[v:maingray] () at (-0.8,-0.8){};
		\node[v:mainemptygray] () at (-1,-1){};
		
		\node () at (0,0.3){$C$};
		\node () at (-0.8,0.45){$a_{1}$};
		\node () at (-0.8,-0.45){$b_{1}$};
		\node () at (0.8,0.45){$b_{2}$};
		\node () at (0.8,-0.45){$a_{2}$};
		
		\begin{pgfonlayer}{background}
			
			\draw[e:marker,color=myLightBlue] (0.6,0.6) -- (1.2,1.2);
			\draw[e:marker,color=myLightBlue] (1.2,1.2) -- (1.2,-1.2);
			\draw[e:marker,color=myLightBlue] (1.2,-1.2) -- (0.6,-0.6);
			\draw[e:marker,color=myGreen] (-0.6,0.6) -- (-1.2,1.2);
			\draw[e:marker,color=myGreen] (-1.2,1.2) -- (-1.2,-1.2);
			\draw[e:marker,color=myGreen] (-1.2,-1.2) -- (-0.6,-0.6);
			
			\draw[e:main] (-0.6,-0.6) rectangle (0.6,0.6);
			
			\draw[e:main,color=gray] (0.8,0.8) -- (1,1);
			\draw[e:main,color=gray] (-0.8,0.8) -- (-1,1);
			\draw[e:main,color=gray] (0.8,-0.8) -- (1,-1);
			\draw[e:main,color=gray] (-0.8,-0.8) -- (-1,-1);

			\draw[e:coloredborder] (1,1) -- (1.2,1.2);
			\draw[e:coloredthin,color=BostonUniversityRed] (1,1) -- (1.2,1.2);
			\draw[e:coloredborder] (0.6,0.6) -- (0.8,0.8);
			\draw[e:coloredthin,color=BostonUniversityRed] (0.6,0.6) -- (0.8,0.8);
			\draw[e:coloredborder] (1,-1) -- (1.2,-1.2);
			\draw[e:coloredthin,color=BostonUniversityRed] (1,-1) -- (1.2,-1.2);
			\draw[e:coloredborder] (0.6,-0.6) -- (0.8,-0.8);
			\draw[e:coloredthin,color=BostonUniversityRed] (0.6,-0.6) -- (0.8,-0.8);
			\draw[e:coloredborder] (-1,1) -- (-1.2,1.2);
			\draw[e:coloredthin,color=BostonUniversityRed] (-1,1) -- (-1.2,1.2);
			\draw[e:coloredborder] (-0.6,0.6) -- (-0.8,0.8);
			\draw[e:coloredthin,color=BostonUniversityRed] (-0.6,0.6) -- (-0.8,0.8);
			\draw[e:coloredborder] (-1,-1) -- (-1.2,-1.2);
			\draw[e:coloredthin,color=BostonUniversityRed] (-1,-1) -- (-1.2,-1.2);
			\draw[e:coloredborder] (-0.6,-0.6) -- (-0.8,-0.8);
			\draw[e:coloredthin,color=BostonUniversityRed] (-0.6,-0.6) -- (-0.8,-0.8);
			
			\draw[e:coloredborder] (-0.4,1.2) -- (0.4,1.2);
			\draw[e:coloredthin,color=BostonUniversityRed] (-0.4,1.2) -- (0.4,1.2);
			\draw[e:coloredborder] (-0.4,-1.2) -- (0.4,-1.2);
			\draw[e:coloredthin,color=BostonUniversityRed] (-0.4,-1.2) -- (0.4,-1.2);
			\draw[e:coloredborder] (-1.2,-0.4) -- (-1.2,0.4);
			\draw[e:coloredthin,color=BostonUniversityRed] (-1.2,-0.4) -- (-1.2,0.4);
			\draw[e:coloredborder] (1.2,-0.4) -- (1.2,0.4);
			\draw[e:coloredthin,color=BostonUniversityRed] (1.2,-0.4) -- (1.2,0.4);

			\draw[e:main,color=gray] (-1.2,1.2) -- (-0.4,1.2);
			\draw[e:main,color=gray] (1.2,1.2) -- (0.4,1.2);
			\draw[e:main,color=gray] (-1.2,-1.2) -- (-0.4,-1.2);
			\draw[e:main,color=gray] (1.2,-1.2) -- (0.4,-1.2);
			\draw[e:main,color=gray] (-1.2,-1.2) -- (-1.2,-0.4);
			\draw[e:main,color=gray] (-1.2,1.2) -- (-1.2,0.4);
			\draw[e:main,color=gray] (1.2,-1.2) -- (1.2,-0.4);
			\draw[e:main,color=gray] (1.2,1.2) -- (1.2,0.4);

		\end{pgfonlayer}
	\end{tikzpicture}
	\caption{A perfect matching $M_H$ of $H$ where all four vertices of $C$ are matched to vertices of $H-\V{C}$. Two paths, as requested by the assertion of \cref{lemma:pathsintrisums}, are marked.}
	\label{fig:cubepaths0edges}
\end{figure}

\begin{figure}
	\begin{subfigure}{0.24\textwidth}
		\centering
		\begin{tikzpicture}
			\pgfdeclarelayer{background}
			\pgfdeclarelayer{foreground}
			\pgfsetlayers{background,main,foreground}
			
			\node[v:mainempty] () at (-0.6,-0.6){};
			\node[v:main] () at (-0.6,0.6){};
			\node[v:main] () at (0.6,-0.6){};
			\node[v:mainempty] () at (0.6,0.6){};
			\node[v:main] () at (-1.2,-1.2){};
			\node[v:mainempty] () at (-1.2,1.2){};
			\node[v:mainempty] () at (1.2,-1.2){};
			\node[v:main] () at (1.2,1.2){};
			
			\node[v:mainemptygray] () at (0.4,1.2){};
			\node[v:maingray] () at (0.4,-1.2){};
			\node[v:maingray] () at (-0.4,1.2){};
			\node[v:mainemptygray] () at (-0.4,-1.2){};
			
			\node[v:mainemptygray] () at (-1.2,-0.4){};
			\node[v:maingray] () at (-1.2,0.4){};
			\node[v:maingray] () at (1.2,-0.4){};
			\node[v:mainemptygray] () at (1.2,0.4){};
			
			\node[v:maingray] () at (0.8,0.8){};
			\node[v:mainemptygray] () at (1,1){};
			\node[v:maingray] () at (-0.8,0.8){};
			\node[v:mainemptygray] () at (-1,1){};
			\node[v:maingray] () at (0.8,-0.8){};
			\node[v:mainemptygray] () at (1,-1){};       
			\node[v:maingray] () at (-0.8,-0.8){};
			\node[v:mainemptygray] () at (-1,-1){};
			
			\node () at (0,0.3){$C$};
			\node () at (-0.8,0.45){$a_{1}$};
			\node () at (-0.8,-0.45){$b_{1}$};
			\node () at (0.8,0.45){$b_{2}$};
			\node () at (0.8,-0.45){$a_{2}$};
			
			\begin{pgfonlayer}{background}
				
				\draw[e:marker,color=myGreen] (0.6,0.6) -- (1.2,1.2);
				\draw[e:marker,color=myGreen] (1.2,1.2) -- (1.2,-1.2);
				\draw[e:marker,color=myGreen] (1.2,-1.2) -- (0.6,-0.6);
				\draw[e:marker,color=myLightBlue] (-0.6,0.6) -- (-1.2,1.2);
				\draw[e:marker,color=myLightBlue] (-1.2,1.2) -- (-1.2,-1.2);
				\draw[e:marker,color=myLightBlue] (-1.2,-1.2) -- (-0.6,-0.6);

				\draw[e:main,color=gray] (-1.2,1.2) -- (-0.4,1.2);
				\draw[e:main,color=gray] (1.2,1.2) -- (0.4,1.2);
				\draw[e:main,color=gray] (-1.2,-1.2) -- (-0.4,-1.2);
				\draw[e:main,color=gray] (1.2,-1.2) -- (0.4,-1.2);
				\draw[e:main,color=gray] (-1.2,-0.4) -- (-1.2,0.4);
				\draw[e:main,color=gray] (1.2,-0.4) -- (1.2,0.4);
				
				\draw[e:main] (0.6,0.6) -- (0.6,-0.6) -- (-0.6,-0.6) -- (-0.6,0.6);
				
				\draw[e:main,color=gray] (1,1) -- (1.2,1.2);
				\draw[e:main,color=gray] (0.6,0.6) -- (0.8,0.8);
				\draw[e:main,color=gray] (1,-1) -- (1.2,-1.2);
				\draw[e:main,color=gray] (0.6,-0.6) -- (0.8,-0.8);
				\draw[e:main,color=gray] (-1,1) -- (-1.2,1.2);
				\draw[e:main,color=gray] (-0.6,0.6) -- (-0.8,0.8);
				\draw[e:main,color=gray] (-1,-1) -- (-1.2,-1.2);
				\draw[e:main,color=gray] (-0.6,-0.6) -- (-0.8,-0.8);
				
				\draw[e:coloredborder] (-0.4,1.2) -- (0.4,1.2);
				\draw[e:coloredthin,color=BostonUniversityRed] (-0.4,1.2) -- (0.4,1.2);
				\draw[e:coloredborder] (-0.4,-1.2) -- (0.4,-1.2);
				\draw[e:coloredthin,color=BostonUniversityRed] (-0.4,-1.2) -- (0.4,-1.2);
				
				\draw[e:coloredborder] (-0.6,0.6) -- (0.6,0.6);
				\draw[e:coloredthin,color=BostonUniversityRed] (-0.6,0.6) -- (0.6,0.6);
				
				\draw[e:coloredborder](-1.2,-1.2) -- (-1.2,-0.4);
				\draw[e:coloredthin,color=BostonUniversityRed](-1.2,-1.2) -- (-1.2,-0.4);
				
				\draw[e:coloredborder](-1.2,1.2) -- (-1.2,0.4);
				\draw[e:coloredthin,color=BostonUniversityRed](-1.2,1.2) -- (-1.2,0.4);
				
				\draw[e:coloredborder](1.2,-1.2) -- (1.2,-0.4);
				\draw[e:coloredthin,color=BostonUniversityRed](1.2,-1.2) -- (1.2,-0.4);
				
				\draw[e:coloredborder](1.2,1.2) -- (1.2,0.4);
				\draw[e:coloredthin,color=BostonUniversityRed](1.2,1.2) -- (1.2,0.4);
				
				\draw[e:coloredborder] (0.8,0.8) -- (1,1);
				\draw[e:coloredthin,color=BostonUniversityRed](0.8,0.8) -- (1,1);
				\draw[e:coloredborder] (-0.8,0.8) -- (-1,1);
				\draw[e:coloredthin,color=BostonUniversityRed](-0.8,0.8) -- (-1,1);
				\draw[e:coloredborder] (0.8,-0.8) -- (1,-1);
				\draw[e:coloredthin,color=BostonUniversityRed](0.8,-0.8) -- (1,-1);        
				\draw[e:coloredborder] (-0.8,-0.8) -- (-1,-1);
				\draw[e:coloredthin,color=BostonUniversityRed](-0.8,-0.8) -- (-1,-1);

			\end{pgfonlayer}
		\end{tikzpicture}
	\end{subfigure}
	\begin{subfigure}{0.24\textwidth}
		\centering
		\begin{tikzpicture}
			\pgfdeclarelayer{background}
			\pgfdeclarelayer{foreground}
			\pgfsetlayers{background,main,foreground}
			
			\node[v:mainempty] () at (-0.6,-0.6){};
			\node[v:main] () at (-0.6,0.6){};
			\node[v:main] () at (0.6,-0.6){};
			\node[v:mainempty] () at (0.6,0.6){};
			\node[v:main] () at (-1.2,-1.2){};
			\node[v:mainempty] () at (-1.2,1.2){};
			\node[v:mainempty] () at (1.2,-1.2){};
			\node[v:main] () at (1.2,1.2){};
			
			\node[v:mainemptygray] () at (0.4,1.2){};
			\node[v:maingray] () at (0.4,-1.2){};
			\node[v:maingray] () at (-0.4,1.2){};
			\node[v:mainemptygray] () at (-0.4,-1.2){};
			
			\node[v:mainemptygray] () at (-1.2,-0.4){};
			\node[v:maingray] () at (-1.2,0.4){};
			\node[v:maingray] () at (1.2,-0.4){};
			\node[v:mainemptygray] () at (1.2,0.4){};
			
			\node[v:maingray] () at (0.8,0.8){};
			\node[v:mainemptygray] () at (1,1){};
			\node[v:maingray] () at (-0.8,0.8){};
			\node[v:mainemptygray] () at (-1,1){};
			\node[v:maingray] () at (0.8,-0.8){};
			\node[v:mainemptygray] () at (1,-1){};       
			\node[v:maingray] () at (-0.8,-0.8){};
			\node[v:mainemptygray] () at (-1,-1){};
			
			\node () at (0,0.3){$C$};
			\node () at (-0.8,0.45){$a_{1}$};
			\node () at (-0.8,-0.45){$b_{1}$};
			\node () at (0.8,0.45){$b_{2}$};
			\node () at (0.8,-0.45){$a_{2}$};
			
			\begin{pgfonlayer}{background}
				
				\draw[e:marker,color=myLightBlue] (-0.6,0.6) -- (-1.2,1.2);
				\draw[e:marker,color=myLightBlue] (-1.2,1.2) -- (1.2,1.2);
				\draw[e:marker,color=myLightBlue] (1.2,1.2) -- (0.6,0.6);
				
				\draw[e:marker,color=myGreen] (-0.6,-0.6) -- (-1.2,-1.2);
				\draw[e:marker,color=myGreen] (-1.2,-1.2) -- (1.2,-1.2);
				\draw[e:marker,color=myGreen] (1.2,-1.2) -- (0.6,-0.6);
				
				\draw[e:main,color=gray] (-1.2,-1.2) -- (-0.4,-1.2);
				\draw[e:main,color=gray] (1.2,-1.2) -- (0.4,-1.2);
				
				\draw[e:main,color=gray] (-1.2,-1.2) -- (-1.2,0.4);
				\draw[e:main,color=gray] (-1.2,1.2) -- (-1.2,-0.4);
				\draw[e:main,color=gray] (1.2,-1.2) -- (1.2,0.4);
				\draw[e:main,color=gray] (1.2,1.2) -- (1.2,-0.4);

				\draw[e:main] (0.6,0.6) -- (0.6,-0.6) -- (-0.6,-0.6) -- (-0.6,0.6);
				
				\draw[e:main,color=gray] (1,1) -- (1.2,1.2);
				\draw[e:main,color=gray] (0.6,0.6) -- (0.8,0.8);
				
				\draw[e:main,color=gray] (-1,1) -- (-1.2,1.2);
				\draw[e:main,color=gray] (-0.6,0.6) -- (-0.8,0.8);
				
				\draw[e:main,color=gray] (-0.4,1.2) -- (0.4,1.2);
				
				\draw[e:coloredborder] (-0.6,0.6) -- (0.6,0.6);
				\draw[e:coloredthin,color=BostonUniversityRed] (-0.6,0.6) -- (0.6,0.6);
				
				\draw[e:coloredborder] (-0.4,-1.2) -- (0.4,-1.2);
				\draw[e:coloredthin,color=BostonUniversityRed] (-0.4,-1.2) -- (0.4,-1.2);
				
				\draw[e:coloredborder] (-1.2,-0.4) -- (-1.2,0.4);
				\draw[e:coloredthin,color=BostonUniversityRed] (-1.2,-0.4) -- (-1.2,0.4);
				\draw[e:coloredborder] (1.2,-0.4) -- (1.2,0.4);
				\draw[e:coloredthin,color=BostonUniversityRed] (1.2,-0.4) -- (1.2,0.4);

				\draw[e:coloredborder] (-1,-1) -- (-1.2,-1.2);
				\draw[e:coloredthin,color=BostonUniversityRed] (-1,-1) -- (-1.2,-1.2);
				\draw[e:coloredborder] (-0.6,-0.6) -- (-0.8,-0.8);
				\draw[e:coloredthin,color=BostonUniversityRed] (-0.6,-0.6) -- (-0.8,-0.8);
				\draw[e:coloredborder] (0.6,-0.6) -- (0.8,-0.8);
				\draw[e:coloredthin,color=BostonUniversityRed] (0.6,-0.6) -- (0.8,-0.8);
				\draw[e:coloredborder] (1,-1) -- (1.2,-1.2);
				\draw[e:coloredthin,color=BostonUniversityRed] (1,-1) -- (1.2,-1.2);
				
				\draw[e:coloredborder] (0.8,0.8) -- (1,1);
				\draw[e:coloredthin,color=BostonUniversityRed](0.8,0.8) -- (1,1);
				\draw[e:coloredborder] (-0.8,0.8) -- (-1,1);
				\draw[e:coloredthin,color=BostonUniversityRed](-0.8,0.8) -- (-1,1);
				
				\draw[e:coloredborder] (-1.2,1.2) -- (-0.4,1.2);
				\draw[e:coloredthin,color=BostonUniversityRed] (-1.2,1.2) -- (-0.4,1.2);
				\draw[e:coloredborder] (1.2,1.2) -- (0.4,1.2);
				\draw[e:coloredthin,color=BostonUniversityRed] (1.2,1.2) -- (0.4,1.2);

			\end{pgfonlayer}
		\end{tikzpicture}
	\end{subfigure}
	\begin{subfigure}{0.24\textwidth}
		\centering
		\begin{tikzpicture}
			\pgfdeclarelayer{background}
			\pgfdeclarelayer{foreground}
			\pgfsetlayers{background,main,foreground}
			
			\node[v:mainempty] () at (-0.6,-0.6){};
			\node[v:main] () at (-0.6,0.6){};
			\node[v:main] () at (0.6,-0.6){};
			\node[v:mainempty] () at (0.6,0.6){};
			\node[v:main] () at (-1.2,-1.2){};
			\node[v:mainempty] () at (-1.2,1.2){};
			\node[v:mainempty] () at (1.2,-1.2){};
			\node[v:main] () at (1.2,1.2){};
			
			\node[v:mainemptygray] () at (0.4,1.2){};
			\node[v:maingray] () at (0.4,-1.2){};
			\node[v:maingray] () at (-0.4,1.2){};
			\node[v:mainemptygray] () at (-0.4,-1.2){};
			
			\node[v:mainemptygray] () at (-1.2,-0.4){};
			\node[v:maingray] () at (-1.2,0.4){};
			\node[v:maingray] () at (1.2,-0.4){};
			\node[v:mainemptygray] () at (1.2,0.4){};
			
			\node[v:maingray] () at (0.8,0.8){};
			\node[v:mainemptygray] () at (1,1){};
			\node[v:maingray] () at (-0.8,0.8){};
			\node[v:mainemptygray] () at (-1,1){};
			\node[v:maingray] () at (0.8,-0.8){};
			\node[v:mainemptygray] () at (1,-1){};       
			\node[v:maingray] () at (-0.8,-0.8){};
			\node[v:mainemptygray] () at (-1,-1){};
			
			\node () at (0,0.3){$C$};
			\node () at (-0.8,0.45){$a_{1}$};
			\node () at (-0.8,-0.45){$b_{1}$};
			\node () at (0.8,0.45){$b_{2}$};
			\node () at (0.8,-0.45){$a_{2}$};
			
			\begin{pgfonlayer}{background}
				\draw[e:marker,color=myGreen] (0.6,0.6) -- (1.2,1.2);
				\draw[e:marker,color=myGreen] (1.2,1.2) -- (1.2,-1.2);
				\draw[e:marker,color=myGreen] (1.2,-1.2) -- (0.6,-0.6);
				\draw[e:marker,color=myLightBlue] (-0.6,0.6) -- (-1.2,1.2);
				\draw[e:marker,color=myLightBlue] (-1.2,1.2) -- (-1.2,-1.2);
				\draw[e:marker,color=myLightBlue] (-1.2,-1.2) -- (-0.6,-0.6);
				
				\draw[e:main,color=gray] (-1.2,1.2) -- (-0.4,1.2);
				\draw[e:main,color=gray] (1.2,1.2) -- (0.4,1.2);
				\draw[e:main,color=gray] (-1.2,-1.2) -- (-0.4,-1.2);
				\draw[e:main,color=gray] (1.2,-1.2) -- (0.4,-1.2);
				\draw[e:main,color=gray] (-1.2,-0.4) -- (-1.2,0.4);
				\draw[e:main,color=gray] (1.2,-0.4) -- (1.2,0.4);
				
				\draw[e:main] (0.6,0.6) -- (0.6,-0.6);
				\draw[e:main] (-0.6,-0.6) -- (-0.6,0.6);
				
				\draw[e:main,color=gray] (1,1) -- (1.2,1.2);
				\draw[e:main,color=gray] (0.6,0.6) -- (0.8,0.8);
				\draw[e:main,color=gray] (1,-1) -- (1.2,-1.2);
				\draw[e:main,color=gray] (0.6,-0.6) -- (0.8,-0.8);
				\draw[e:main,color=gray] (-1,1) -- (-1.2,1.2);
				\draw[e:main,color=gray] (-0.6,0.6) -- (-0.8,0.8);
				\draw[e:main,color=gray] (-1,-1) -- (-1.2,-1.2);
				\draw[e:main,color=gray] (-0.6,-0.6) -- (-0.8,-0.8);
				
				\draw[e:coloredborder] (-0.4,1.2) -- (0.4,1.2);
				\draw[e:coloredthin,color=BostonUniversityRed] (-0.4,1.2) -- (0.4,1.2);
				\draw[e:coloredborder] (-0.4,-1.2) -- (0.4,-1.2);
				\draw[e:coloredthin,color=BostonUniversityRed] (-0.4,-1.2) -- (0.4,-1.2);
				
				\draw[e:coloredborder] (-0.6,0.6) -- (0.6,0.6);
				\draw[e:coloredthin,color=BostonUniversityRed] (-0.6,0.6) -- (0.6,0.6);
				\draw[e:coloredborder] (-0.6,-0.6) -- (0.6,-0.6);
				\draw[e:coloredthin,color=BostonUniversityRed] (-0.6,-0.6) -- (0.6,-0.6);
				
				\draw[e:coloredborder](-1.2,-1.2) -- (-1.2,-0.4);
				\draw[e:coloredthin,color=BostonUniversityRed](-1.2,-1.2) -- (-1.2,-0.4);
				
				\draw[e:coloredborder](-1.2,1.2) -- (-1.2,0.4);
				\draw[e:coloredthin,color=BostonUniversityRed](-1.2,1.2) -- (-1.2,0.4);
				
				\draw[e:coloredborder](1.2,-1.2) -- (1.2,-0.4);
				\draw[e:coloredthin,color=BostonUniversityRed](1.2,-1.2) -- (1.2,-0.4);
				
				\draw[e:coloredborder](1.2,1.2) -- (1.2,0.4);
				\draw[e:coloredthin,color=BostonUniversityRed](1.2,1.2) -- (1.2,0.4);
				
				\draw[e:coloredborder] (0.8,0.8) -- (1,1);
				\draw[e:coloredthin,color=BostonUniversityRed](0.8,0.8) -- (1,1);
				\draw[e:coloredborder] (-0.8,0.8) -- (-1,1);
				\draw[e:coloredthin,color=BostonUniversityRed](-0.8,0.8) -- (-1,1);
				\draw[e:coloredborder] (0.8,-0.8) -- (1,-1);
				\draw[e:coloredthin,color=BostonUniversityRed](0.8,-0.8) -- (1,-1);        
				\draw[e:coloredborder] (-0.8,-0.8) -- (-1,-1);
				\draw[e:coloredthin,color=BostonUniversityRed](-0.8,-0.8) -- (-1,-1);

			\end{pgfonlayer}
		\end{tikzpicture}
	\end{subfigure}
	\begin{subfigure}{0.24\textwidth}
		\centering
		\begin{tikzpicture}
			\pgfdeclarelayer{background}
			\pgfdeclarelayer{foreground}
			\pgfsetlayers{background,main,foreground}
			
			\node[v:mainempty] () at (-0.6,-0.6){};
			\node[v:main] () at (-0.6,0.6){};
			\node[v:main] () at (0.6,-0.6){};
			\node[v:mainempty] () at (0.6,0.6){};
			\node[v:main] () at (-1.2,-1.2){};
			\node[v:mainempty] () at (-1.2,1.2){};
			\node[v:mainempty] () at (1.2,-1.2){};
			\node[v:main] () at (1.2,1.2){};
			
			\node[v:mainemptygray] () at (0.4,1.2){};
			\node[v:maingray] () at (0.4,-1.2){};
			\node[v:maingray] () at (-0.4,1.2){};
			\node[v:mainemptygray] () at (-0.4,-1.2){};
			
			\node[v:mainemptygray] () at (-1.2,-0.4){};
			\node[v:maingray] () at (-1.2,0.4){};
			\node[v:maingray] () at (1.2,-0.4){};
			\node[v:mainemptygray] () at (1.2,0.4){};
			
			\node[v:maingray] () at (0.8,0.8){};
			\node[v:mainemptygray] () at (1,1){};
			\node[v:maingray] () at (-0.8,0.8){};
			\node[v:mainemptygray] () at (-1,1){};
			\node[v:maingray] () at (0.8,-0.8){};
			\node[v:mainemptygray] () at (1,-1){};       
			\node[v:maingray] () at (-0.8,-0.8){};
			\node[v:mainemptygray] () at (-1,-1){};
			
			\node () at (0,0.3){$C$};
			\node () at (-0.8,0.45){$a_{1}$};
			\node () at (-0.8,-0.45){$b_{1}$};
			\node () at (0.8,0.45){$b_{2}$};
			\node () at (0.8,-0.45){$a_{2}$};
			
			\begin{pgfonlayer}{background}
				
				\draw[e:marker,color=myLightBlue] (-0.6,0.6) -- (-1.2,1.2);
				\draw[e:marker,color=myLightBlue] (-1.2,1.2) -- (1.2,1.2);
				\draw[e:marker,color=myLightBlue] (1.2,1.2) -- (0.6,0.6);
				
				\draw[e:marker,color=myGreen] (-0.6,-0.6) -- (-1.2,-1.2);
				\draw[e:marker,color=myGreen] (-1.2,-1.2) -- (1.2,-1.2);
				\draw[e:marker,color=myGreen] (1.2,-1.2) -- (0.6,-0.6);
				
				\draw[e:main] (0.6,0.6) -- (0.6,-0.6);
				\draw[e:main] (-0.6,-0.6) -- (-0.6,0.6);
				
				\draw[e:main,color=gray] (1,1) -- (1.2,1.2);
				\draw[e:main,color=gray] (0.6,0.6) -- (0.8,0.8);
				\draw[e:main,color=gray] (1,-1) -- (1.2,-1.2);
				\draw[e:main,color=gray] (0.6,-0.6) -- (0.8,-0.8);
				\draw[e:main,color=gray] (-1,1) -- (-1.2,1.2);
				\draw[e:main,color=gray] (-0.6,0.6) -- (-0.8,0.8);
				\draw[e:main,color=gray] (-1,-1) -- (-1.2,-1.2);
				\draw[e:main,color=gray] (-0.6,-0.6) -- (-0.8,-0.8);
				
				\draw[e:main,color=gray] (-0.4,1.2) -- (0.4,1.2);
				\draw[e:main,color=gray] (-0.4,-1.2) -- (0.4,-1.2);
				
				\draw[e:main,color=gray] (1.2,1.2) -- (1.2,0.4);
				\draw[e:main,color=gray] (1.2,-1.2) -- (1.2,-0.4);
				\draw[e:main,color=gray] (-1.2,-1.2) -- (-1.2,0.4);
				\draw[e:main,color=gray] (-1.2,-1.2) -- (-1.2,-0.4);
				
				\draw[e:coloredborder] (-0.6,0.6) -- (0.6,0.6);
				\draw[e:coloredthin,color=BostonUniversityRed] (-0.6,0.6) -- (0.6,0.6);
				\draw[e:coloredborder] (-0.6,-0.6) -- (0.6,-0.6);
				\draw[e:coloredthin,color=BostonUniversityRed] (-0.6,-0.6) -- (0.6,-0.6);

				\draw[e:coloredborder] (0.8,0.8) -- (1,1);
				\draw[e:coloredthin,color=BostonUniversityRed](0.8,0.8) -- (1,1);
				\draw[e:coloredborder] (-0.8,0.8) -- (-1,1);
				\draw[e:coloredthin,color=BostonUniversityRed](-0.8,0.8) -- (-1,1);
				\draw[e:coloredborder] (0.8,-0.8) -- (1,-1);
				\draw[e:coloredthin,color=BostonUniversityRed](0.8,-0.8) -- (1,-1);        
				\draw[e:coloredborder] (-0.8,-0.8) -- (-1,-1);
				\draw[e:coloredthin,color=BostonUniversityRed](-0.8,-0.8) -- (-1,-1);
				
				\draw[e:coloredborder] (-1.2,-0.4) -- (-1.2,0.4);
				\draw[e:coloredthin,color=BostonUniversityRed] (-1.2,-0.4) -- (-1.2,0.4);
				\draw[e:coloredborder] (1.2,-0.4) -- (1.2,0.4);
				\draw[e:coloredthin,color=BostonUniversityRed] (1.2,-0.4) -- (1.2,0.4);
				
				\draw[e:coloredborder] (-1.2,1.2) -- (-0.4,1.2);
				\draw[e:coloredthin,color=BostonUniversityRed] (-1.2,1.2) -- (-0.4,1.2);
				\draw[e:coloredborder] (1.2,1.2) -- (0.4,1.2);
				\draw[e:coloredthin,color=BostonUniversityRed] (1.2,1.2) -- (0.4,1.2);
				
				\draw[e:coloredborder] (-1.2,-1.2) -- (-0.4,-1.2);
				\draw[e:coloredthin,color=BostonUniversityRed] (-1.2,-1.2) -- (-0.4,-1.2);
				\draw[e:coloredborder] (1.2,-1.2) -- (0.4,-1.2);
				\draw[e:coloredthin,color=BostonUniversityRed] (1.2,-1.2) -- (0.4,-1.2);

			\end{pgfonlayer}
		\end{tikzpicture}
	\end{subfigure}
	
	\caption{A bisubdivision of the cube together with a perfect matching.
		Two paths, as requested by the assertion of \cref{lemma:pathsintrisums}, are marked.
		All other cases, in particular the ones regarding to the exact identity of $e$, can be derived from this by symmetry. }
	\label{fig:cubepaths1edge}
\end{figure}
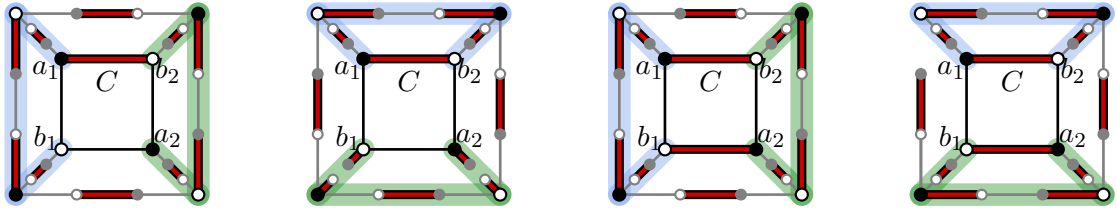

\begin{lemma}\label{lemma:matchingsingoodcrosses}
	Let $B$ be a brace and $C$ a $4$-cycle in $B$ as well as $P_1$, $P_2$ two paths that form a conformal cross over $C$.
	Then for every $e\in C$, $C+P_1+P_2$ has a perfect matching $M_e$ such that $\Set{e}= M_e\cap\E{C}$.
\end{lemma}

\begin{proof}
	By definition, since $P_1$ and $P_2$ form a conformal cross over $C$, there is a perfect matching $M$ of $H\coloneqq C+P_1+P_2$.
	Since $C$ is a $4$-cycle, $P_1$ and $P_2$ each must connect two vertices of the same colour on $C$ and thus $M$ must contain exactly one edge, say $e'$, of $C$ since the cross is conformal, hence $M_{e'}\coloneqq M$.
	Let $e\in\E{C}\setminus\Set{e'}$ be another edge of $C$.
	If $e$ and $e'$ are disjoint, then $P_1+P_2+e+e'$ is an $M$-conformal cycle which contains all edges of $M_{e'}$ and thus, $M_e\coloneqq \E{P_1+P_2+e+e'}\setminus M_{e'}$ is a perfect matching containing $e$.
	Otherwise, $e$ and $e'$ share exactly one endpoint and $e$ is incident with an endpoint of $P_i$ for some $i\in[1,2]$.
	Then $P_i+e+e'$ is an $M_{e'}$-conformal cycle and thus $M_e\coloneqq \Brace{M_{e'}\setminus\E{P_i+e+e'}}\cup\Brace{\E{P_i+e+e'}\setminus M_{e'}}$ is a perfect matching as required. 	
\end{proof}

A last and essential tool before we dive into the more specific cases of $K_{3,3}$-free and $K_{3,3}$-containing braces is the observation on conformal crosses over $4$-cycles in form of \cref{lemma:goodcrossesmeanK33}.

\begin{figure}[h!]
	\centering
	\begin{tikzpicture}
		\pgfdeclarelayer{background}
		\pgfdeclarelayer{foreground}
		\pgfsetlayers{background,main,foreground}
		
		\node[v:main] () at (-1.2,-1.2){};
		\node[v:mainempty] () at (-1.2,1.2){};
		
		\node[v:mainempty] () at (1.2,-1.2){};
		\node[v:main] () at (1.2,1.2){};
		
		\node () at (-1.45,1){$C$};
		\node[v:main] () at (0.6,0){};
		\node[v:mainempty] () at (-0.6,0){};
		
		\node[v:mainemptygray] () at (0.2,0){};
		\node[v:maingray] () at (-0.2,0){};
		
		\node[v:mainemptygray] () at (0.8,-0.4){};
		\node[v:maingray] () at (1,-0.8){};
		\node[v:maingray] () at (-0.8,-0.4){};
		\node[v:mainemptygray] () at (-1,-0.8){};
		
		\node[v:maingray] () at (-0.2,0.27){};
		\node[v:mainemptygray] () at (0.6,0.8){};
		\node[v:mainemptygray] () at (0.2,0.27){};
		\node[v:maingray] at (-0.6,0.8){};
		
		\begin{pgfonlayer}{background}
			\draw[e:marker,color=myLightBlue] (-1.2,1.2) -- (0.6,0);
			\draw[e:marker,color=myLightBlue] (0.6,0) -- (1.2,-1.2);
			
			\draw[e:marker,color=myGreen] (1.2,1.2) -- (-0.6,0);
			\draw[e:marker,color=myGreen ] (-0.6,0) -- (-1.2,-1.2);
			
			\draw[e:main] (-1.2,-1.2) -- (-1.2,1.2) -- (1.2,1.2) -- (1.2,-1.2);
			
			\draw[e:main,color=gray] (-1.2,-1.2) -- (-1,-0.8);
			\draw[e:main,color=gray] (1.2,-1.2) -- (1,-0.8);
			
			\draw[e:main,color=gray] (-0.8,-0.4) -- (-0.6,0);
			\draw[e:main,color=gray] (0.8,-0.4) -- (0.6,0);
			
			\draw[e:main,color=gray] (-0.6,0) -- (-0.2,0);
			\draw[e:main,color=gray] (0.6,0) -- (0.2,0);
			
			\draw[e:main,color=gray] (-0.2,0.27) -- (0.6,0.8);
			\draw[e:main,color=gray] (0.2,0.27) -- (-0.6,0.8);
			
			\draw[e:coloredborder] (-1.2,-1.2) -- (1.2,-1.2);
			\draw[e:coloredthin,color=BostonUniversityRed] (-1.2,-1.2) -- (1.2,-1.2);
			
			\draw[e:coloredborder] (-0.2,0) -- (0.2,0);
			\draw[e:coloredthin,color=BostonUniversityRed] (-0.2,0) -- (0.2,0);
			
			\draw[e:coloredborder] (-1,-0.8) -- (-0.8,-0.4);
			\draw[e:coloredthin,color=BostonUniversityRed] (-1,-0.8) -- (-0.8,-0.4);
			\draw[e:coloredborder] (1,-0.8) -- (0.8,-0.4);
			\draw[e:coloredthin,color=BostonUniversityRed] (1,-0.8) -- (0.8,-0.4);
			
			\draw[e:coloredborder] (0.6,0) -- (0.2,0.27);
			\draw[e:coloredthin,color=BostonUniversityRed] (0.6,0) -- (0.2,0.27);
			\draw[e:coloredborder] (-0.6,0) -- (-0.2,0.27);
			\draw[e:coloredthin,color=BostonUniversityRed] (-0.6,0) -- (-0.2,0.27);
			
			\draw[e:coloredborder] (0.6,0.8) -- (1.2,1.2);
			\draw[e:coloredthin,color=BostonUniversityRed] (0.6,0.8) -- (1.2,1.2);
			
			\draw[e:coloredborder] (-0.6,0.8) -- (-1.2,1.2);
			\draw[e:coloredthin,color=BostonUniversityRed] (-0.6,0.8) -- (-1.2,1.2);

		\end{pgfonlayer}
		
	\end{tikzpicture}
	\caption{A bisubdivision of $K_{3,3}$ together with a perfect matching $M$ and two disjoint $M$-alternating paths that form a conformal cross over the $4$-cycle $C$.}
	\label{fig:goodcrossinK33}
\end{figure}
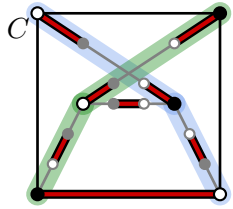

\begin{proof}[Proof of \cref{lemma:goodcrossesmeanK33}]
	First, assume that there exists a conformal bisubdivision $H$ of $K_{3,3}$ that has $C$ as a subgraph.
	Then we may choose a perfect matching $M$ of $B$ such that it contains a perfect matching of $H$ as seen in \cref{fig:goodcrossinK33}.
	The conformal cross over $C$ is also presented in the same figure.
	
	For the reverse let $M$ be a perfect matching of $B$ and $P_1$, $P_2$ two $M$-alternating paths that form an $M$-conformal cross over $C$.
	Since both, $P_1$ and $P_2$, must be of even length, they contain at least one inner vertex $v_1$ and $v_2$ respectively.
	Moreover, no endpoint of $P_1$ belongs to the same colour class of an endpoint of $P_2$, and thus $v_1$ and $v_2$ can be chosen such that they belong to different colour classes, and, for each $i$, the colour class of $v_i$ is different from the colour class of the endpoints of $P_i$.
	Let $e\in M\cap\E{C}$ be the unique edge of $C$ that belongs to $M$.
	With $B$ being a brace and \cref{thm:bipartiteextendibility} there must exist an internally $M$-conformal $v_1$-$v_2$ path $Q$ that avoids $e$.
	For each $i\in[1,2]$ let $x_i$ be the endpoint of $P_i$ that is not incident with $e$.
	We claim that $Q$ contains a subpath $R$ which is internally disjoint from $P_1$ and $P_2$, and for each $i\in[1,2]$ an endpoint of $R$ is an inner vertex of $P_i$.
	Let $u_1$ be the last vertex of $Q$ when traversing along $Q$ starting in $v_1$, which belongs to $P_1$.
	Then, as $P_1$ is $M$-alternating and the only vertex of $P_1$ not covered by $M$ within the path belongs to $e$, $u_1$ must be incident to an edge of $M\cap \E{P_1}$, and thus $P_1u_1$ is of even length.
	Hence $u_1$ and $v_1$ belong to the same colour class of $B$, which is different from the colour of $x_1$.
	Clearly, $u_1Q$ still contains an inner vertex of $P_2$, let $u_2$ be the first vertex of $P_2$ we encounter when traversing along $u_1Q$ starting in $u_1$.
	By the same arguments as above, the edge of $M$ that covers $u_2$ must belong to $P_2$ and thus $u_2$ and $v_2$ have the same colour, which is different from $v_2$.
	Consequently, $u_2\neq v_2$ and $R\coloneqq u_1Qu_2$ is an internally $M$-conformal path as required.
	Since $R$ is internally $M$-conformal and $C+P_1+P_2$ is $M$-conformal, $H\coloneqq C+P_1+P_2+R$ is also $M$-conformal.
	Moreover, for each $i\in[1,2]$, the vertex $u_i$ divides $P_i$ into two subpaths, and as $u_i$ has a different colour than any of the two endpoints of $P_i$, both of these paths must be of odd length.
	Hence $H$ is a conformal bisubdivision of $K_{3,3}$ in $B$.
\end{proof}

\section{Matching Crosses in $K_{3,3}$-free Braces}\label{sec:pfaffian}

To proceed towards the proof of \cref{prop:twopathsinbraces}, we need to describe how the structure of $K_{3,3}$-free braces, especially the non-planar $K_{3,3}$-free braces we obtain via the trisum operation from \cref{thm:trisums} by using planar braces as the base building blocks, behave regarding the existence of matching crosses.
The purpose of this section is to establish the $K_{3,3}$-free part of \cref{prop:twopathsinbraces} in the form of \cref{prop:pfaffiancrosses}.

In what follows we are concerned with $K_{3,3}$-free braces that are not planar.
By \cref{thm:trisums} there is a single exception to the $K_{3,3}$-free braces constructed from planar braces by the trisum operation, namely the Heawood graph.
While the Heawood graph does not contain a single $4$-cycle, in order to prove \cref{prop:pfaffiancrosses}, we have to discuss its cycles.

\begin{lemma}\label{lemma:heawoodcrosses}
	Let $C$ be a conformal cycle of the Heawood graph $H_{14}$, then there exists a conformal cross over $C$ in $H_{14}$.
\end{lemma}

\begin{figure}[!h]
	\centering
	\begin{subfigure}{0.3\textwidth}
		\centering
		\begin{tikzpicture}
			\pgfdeclarelayer{background}
			\pgfdeclarelayer{foreground}
			\pgfsetlayers{background,main,foreground}
			
			\foreach \x in {2,4,6,8,10,12,14}
			{
				\node[v:main] () at (\x*25.71:15mm){};
			}
			
			\foreach \x in {1,3,5,7,9,11,13}
			{
				\node[v:mainempty] () at (\x*25.71:15mm){};
			}
			
			\begin{pgfonlayer}{background}
				
				\draw[e:marker,myGreen] (257.1:15mm) -- (25.65:15mm);
				\draw[e:marker,myGreen] (257.1:15mm) -- (282.81:15mm);
				\draw[e:marker,myGreen] (282.81:15mm) -- (308.52:15mm);

				\draw[e:marker,myLightBlue] (0:15mm) -- (128.55:15mm);
				\draw[e:marker,myLightBlue] (77.13:15mm) -- (102.84:15mm); 
				\draw[e:marker,myLightBlue] (102.84:15mm) -- (128.55:15mm);
				\draw[e:marker] (308.52:15mm) -- (334.23:15mm);
				\draw[e:marker] (334.23:15mm) -- (0:15mm);
				\draw[e:marker] (0:15mm) -- (25.71:15mm);
				\draw[e:marker] (25.71:15mm) -- (51.42:15mm);
				\draw[e:marker] (51.42:15mm) -- (77.13:15mm);
				\draw[e:marker] (308.52:15mm) -- (77.13:15mm);
				
				\foreach \x in {2,4,6,8,10,12,14}
				{
					\draw[e:main] (\x*25.71:15mm) to (128.55+\x*25.71:15mm);
				}
				
				\foreach \x in {1,3,5,7,9,11,13}
				{
					\draw[e:main] (\x*25.71:15mm) to (25.71+\x*25.71:15mm);
				}
				
				\foreach \x in {2,4,6,8,10,12,14}
				{
					\draw[e:coloredborder] (\x*25.71:15mm) to (25.71+\x*25.71:15mm);
					\draw[e:coloredthin,color=BostonUniversityRed] (\x*25.71:15mm) to (25.71+\x*25.71:15mm);
				}

			\end{pgfonlayer}
		\end{tikzpicture}
	\end{subfigure}
	\begin{subfigure}{0.3\textwidth}
		\centering
		\begin{tikzpicture}
			\pgfdeclarelayer{background}
			\pgfdeclarelayer{foreground}
			\pgfsetlayers{background,main,foreground}
			
			\foreach \x in {2,4,6,8,10,12,14}
			{
				\node[v:main] () at (\x*25.71:15mm){};
			}
			
			\foreach \x in {1,3,5,7,9,11,13}
			{
				\node[v:mainempty] () at (\x*25.71:15mm){};
			}
			
			\begin{pgfonlayer}{background}
				\draw[e:marker,myGreen] (231.39:15mm) -- (257.1:15mm);
				\draw[e:marker,myGreen] (257.1:15mm) -- (282.81:15mm);
				\draw[e:marker,myGreen] (282.81:15mm) -- (308.52:15mm);
				
				\draw[e:marker,myLightBlue] (334.23:15mm) -- (0:15mm);
				\draw[e:marker,myLightBlue] (0:15mm) -- (25.71:15mm);
				\draw[e:marker,myLightBlue] (25.71:15mm) -- (51.42:15mm);
				
				\draw[e:marker] (308.52:15mm) -- (77.13:15mm);
				\draw[e:marker] (308.52:15mm) -- (334.23:15mm);
				\draw[e:marker] (334.23:15mm) -- (205.68:15mm);
				\draw[e:marker] (205.68:15mm) -- (231.39:15mm);
				\draw[e:marker] (231.39:15mm) -- (102.84:15mm);
				\draw[e:marker] (102.85:15mm) -- (128.55:15mm);
				\draw[e:marker] (128.55:15mm) -- (154.26:15mm);
				\draw[e:marker] (154.26:15mm) -- (179.97:15mm);
				\draw[e:marker] (179.97:15mm) -- (51.42:15mm);
				\draw[e:marker] (51.42:15mm) -- (77.13:15mm);
				
				\foreach \x in {2,4,6,8,10,12,14}
				{
					\draw[e:main] (\x*25.71:15mm) to (128.55+\x*25.71:15mm);
				}
				
				\foreach \x in {1,3,5,7,9,11,13}
				{
					\draw[e:main] (\x*25.71:15mm) to (25.71+\x*25.71:15mm);
				}
				
				\foreach \x in {2,4,6,8,10,12,14}
				{
					\draw[e:coloredborder] (\x*25.71:15mm) to (25.71+\x*25.71:15mm);
					\draw[e:coloredthin,color=BostonUniversityRed] (\x*25.71:15mm) to (25.71+\x*25.71:15mm);
				}

			\end{pgfonlayer}
		\end{tikzpicture}
	\end{subfigure}
	\begin{subfigure}{0.3\textwidth}
		\centering
		\begin{tikzpicture}
			\pgfdeclarelayer{background}
			\pgfdeclarelayer{foreground}
			\pgfsetlayers{background,main,foreground}
			
			\foreach \x in {2,4,6,8,10,12,14}
			{
				\node[v:main] () at (\x*25.71:15mm){};
			}
			
			\foreach \x in {1,3,5,7,9,11,13}
			{
				\node[v:mainempty] () at (\x*25.71:15mm){};
			}
			
			\begin{pgfonlayer}{background}
				
				\draw[e:marker,myLightBlue] (0:15mm) -- (128.55:15mm);
				\draw[e:marker,myGreen] (257.1:15mm) -- (25.65:15mm);
				
				\foreach \x in {1,...,14}
				{
					\draw[e:marker] (\x*25.71:15mm) to (25.71+\x*25.71:15mm);
				}
				
				\foreach \x in {2,4,6,8,10,12,14}
				{
					\draw[e:main] (\x*25.71:15mm) to (128.55+\x*25.71:15mm);
				}
				
				\foreach \x in {1,3,5,7,9,11,13}
				{
					\draw[e:main] (\x*25.71:15mm) to (25.71+\x*25.71:15mm);
				}
				
				\foreach \x in {2,4,6,8,10,12,14}
				{
					\draw[e:coloredborder] (\x*25.71:15mm) to (25.71+\x*25.71:15mm);
					\draw[e:coloredthin,color=BostonUniversityRed] (\x*25.71:15mm) to (25.71+\x*25.71:15mm);
				}

			\end{pgfonlayer}
		\end{tikzpicture}
	\end{subfigure}
	\caption{The Heawood graph $H_{14}$ together with a perfect matching and the three, up to symmetry, different conformal cycles in $H_{14}$.
		For each of these cycles, we provide a conformal cross.}
	\label{fig:heawoodcrosses}
\end{figure}
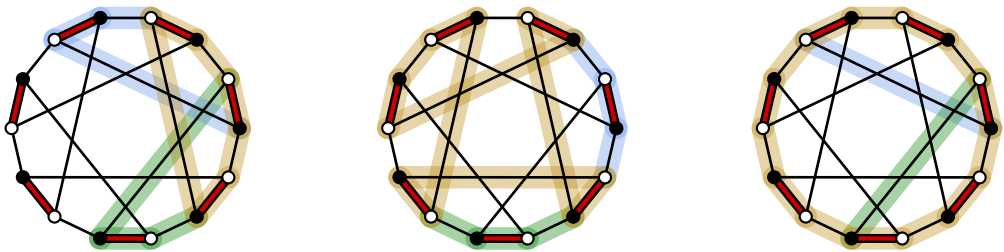

\begin{proof}
	It is known that the Heawood graph has, up to automorphisms, exactly one perfect matching, as, for example, the complement of every perfect matching of $H_{14}$ is a Hamilton cycle.
	See \cite{abreu2004graphs} for a discussion on the matter.
	Moreover, no conformal cycle in $H_{14}$ is of a length that is a multiple of four, for more details on that please consult \cite{mccuaig2000even}.
	Indeed, for every fixed length $\ell$ of a conformal cycle in $H_{14}$ it suffices to find a cross for one of them, as, again, the graph is highly symmetric.
	Hence in order to prove the assertion, it suffices to fix a perfect matching and check conformal cycles of length $14$, $10$, and $6$.
	This is done in \cref{fig:heawoodcrosses}.
\end{proof}

Let us first observe that the $4$-cycle itself, on which a trisum operation has been performed, must have a strong matching cross.

\begin{lemma}\label{lemma:badcrosses}
	Let $B$ be a $K_{3,3}$-free brace that is not the Heawood graph, $\ell\geq 3$, $B_1,\dots,B_{\ell}$ braces such that $B$ is a maximal $4$-cycle sum of $B_1,\dots,B_{\ell}$ at the $4$-cycle $C=\Brace{a_1,b_1,a_2,b_2}$.
	Then for any choice of two distinct values $i,j\in[1,\ell]$ and any choice of $x\in\Set{a_1,a_2}$ and $y\in\Set{b_1,b_2}$, there is a strong matching cross over $C$ with paths $P_1$ and $P_2$ in $B_i+B_j$ such that the $P_i$ are $M$-alternating for some perfect matching $M$ of $B$ for which $x$ and $y$ are the two vertices of $C$ which are covered by edges of $M\cap\Brace{\E{P_1}\cup\E{P_2}\cup\E{C}}$.
\end{lemma}

\begin{proof}
	The lemma is almost identical to \cref{lemma:braceparts}.
	In both $B_i$ and $B_j$ there exists a conformal bisubdivision of the cube which has $C$ as a subgraph by \cref{lemma:cubeorK33} since $B$ cannot contain a conformal bisubdivision of $K_{3,3}$.
	Hence $B_i+B_j$ contains a conformal bisubdivision of the graph $H_{12}$, which is obtained by identifying two cubes on one $4$-cycle.
	See \cref{fig:planecross} for an illustration of a bisubdivision of $H_{12}$.
	The figure also shows perfect matchings of the respective conformal bisubdivision one can find, together with two paths that make up a strong matching cross as desired.
	The exact matching cross depending on the choices of $x$ and $y$ can be obtained from the paths illustrated in \cref{fig:planecross} by symmetry.
	By adjusting the perfect matching such that exactly the requested edges belong to $M$, our proof is complete.
\end{proof}

In the following sections we sometimes encounter the situation of two $M$-alternating paths which are not disjoint.
However, one can observe that whenever two $M$-alternating paths meet such that their union contains an $M$-conformal cycle $C$, we may switch $M$ along this cycle $C$ to obtain a new perfect matching $N$ together with two $N$-alternating paths that intersect slightly less then before.
This relatively simple observation yields a general and powerful tool for bipartite graphs.

\begin{lemma}[Bipartite Untangling Lemma, \cite{mccuaig2001brace}]\label{lemma:untangletwopaths}
	Let $G$ be a bipartite graph with a perfect matching $M$ and $a_1,a_2\in V_1$ and $b_1b_1\in V_2$ four distinct vertices in $G$.
	Let further $P_1$ and $P_2$ be two internally $M$-conformal paths such that $P_i$ has endpoints $a_i$ and $b_i$.
	Then there exists a perfect matching $M'\in\Perf{G}$ together with two internally $M'$-conformal paths $P'_1$ and $P'_2$ such that:
	\begin{enumerate}
		\item $P'_i$ has endpoints $a_i$ and $b_i$ for both $i\in\Set{1,2}$,
		
		\item $P'_1+P'_2$ is a subgraph of $P_1+P_2$,
		
		\item $M\setminus\E{P_1+P_2}=M'\setminus\E{P_1+P_2}$, and
		
		\item either $P'_1\cap P'_2$ is an $M'$-conformal path or $P'_1$ and $P'_2$ are disjoint.
	\end{enumerate}
\end{lemma}

\begin{lemma}\label{lemma:easycrossesovercrossingcycles}
	Let $B$ be a $K_{3,3}$-free brace, $\ell\geq 3$, $B_1,\dots,B_{\ell}$ braces such that $B$ is a maximal $4$-cycle sum of $B_1,\dots,B_{\ell}$ at the $4$-cycle $C$ and let $C'$ be a conformal cycle in $B_1$ that also exists in $B$ such that $C'\cap C$ is a non-trivial path.
	Then there is a strong matching cross over $C'$ in $B$.
\end{lemma}

\begin{proof}
	Since $C\cap C'$ is a non-trivial path, let us call it $K$, the two cycles share at least two vertices, and $K$ contains at least one vertex of every colour class.
	
	Let us first assume, $K$ has length one.
	Let $M$ be a perfect matching of $B_1$ such that $C'$ is $M$-conformal and $M$ contains as many edges of $C$ as possible.
	Let $a\in V_1$ and $b\in V_2$ be the two vertices of $C-\V{K}$.
	We divide this case into two subcases, one, where $C$ is $M$-conformal as well and the second, where for every possible choice of $M$ we have $ab\notin M$.
	
	So first let us assume $ab\in M$.
	Then, since $B_1$ is a brace and by \cref{thm:bipartiteextendibility}, there exists an internally $M$-conformal path $L'$ with $a$ as one endpoint and a vertex $b'$ of $\Vi{2}{C'}$ as its other endpoints such that $L'$ is internally disjoint from $C'$ and avoids $K$.
	Similarly, there is an internally $M$-conformal path $R'$ with endpoints $b$ and $a'\in\Vi{1}{C'}\setminus\V{K}$.
	By \cref{lemma:untangletwopaths} there exists a perfect matching $M'$ which coincides with $M$ everywhere outside of $L'+R'$ together with internally $M'$-conformal paths $L$ and $R$ linking $a$ and $b'$, and $b$ and $a'$ respectively such that either $L$ and $R$ are disjoint, or $L\cap R$ is an $M'$-conformal path.
	In the later case, $ab$ together with the $a$-$b$-subpath of $L+R$ forms an $M'$-conformal cycle $O$.
	In a slight abuse of notation let us adjust $M'$\footnote{We still keep the name '$M'$' for better readability} in these cases to be the perfect matching $M'\Delta\E{O}$.
	Let $u$ and $v$ be the endpoints of $L\cap R$ such that $u$ appears on $L$ before $v$ when traversing along $L$ starting from $a$.
	We also might have to adjust our definition of $L$ and $R$ slightly, depending on the way, $L$ and $R$ currently connect to $C'$.
	In case we adjust the perfect matching, we also adjust $L$ to be the $M'$-alternating path $LuR$, while $R$ is adjusted to be the path $RvL$.
	
	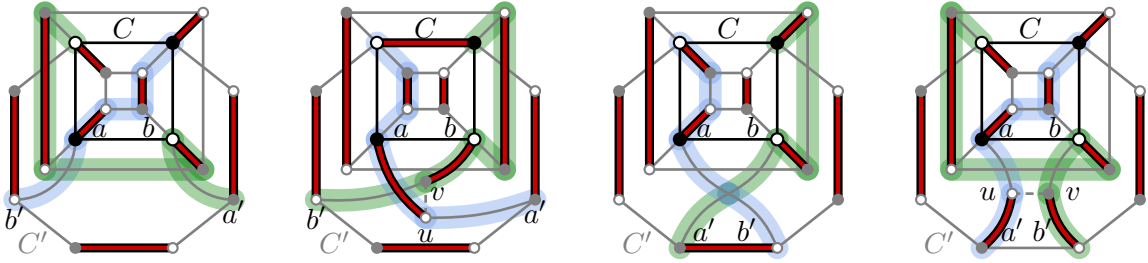
\begin{figure}[h!]
		\centering
		\begin{subfigure}{0.24\textwidth}
			\centering
			\begin{tikzpicture}[scale=0.8]
				\pgfdeclarelayer{background}
				\pgfdeclarelayer{foreground}
				\pgfsetlayers{background,main,foreground}
				
				\node () at (0,1.05){$C$};
				\node[gray] () at (-1.5,-2.5){$C'$};
				
				\node () at (-0.4,-0.65){$a$};
				\node () at (0.4,-0.6){$b$};
				\node() at (1.8,-2){$a'$};
				\node() at (-1.8,-2.1){$b'$};
				
				\node[v:mainemptygray] () at (-0.3,-0.3){};
				\node[v:maingray] () at (-0.3,0.3){};
				\node[v:maingray] () at (0.3,-0.3){};
				\node[v:mainemptygray] () at (0.3,0.3){};
				
				\node[v:main] () at (-0.8,-0.8){};
				\node[v:mainempty] () at (-0.8,0.8){};
				\node[v:mainempty] () at (0.8,-0.8){};
				\node[v:main] () at (0.8,0.8){};
				
				\node[v:mainemptygray] () at (-1.3,-1.3){};
				\node[v:maingray] () at (-1.3,1.3){};
				\node[v:maingray] () at (1.3,-1.3){};
				\node[v:mainemptygray] () at (1.3,1.3){};   
				
				\node[v:maingray] () at (-1.8,0){};
				\node[v:mainemptygray] () at (1.8,0){};
				
				\node[v:mainemptygray] () at (-1.8,-1.8){};
				\node[v:maingray] () at (1.8,-1.8){};
				
				\node[v:maingray] () at (-0.8,-2.6){};
				\node[v:mainemptygray] () at (0.8,-2.6){};
				
				\begin{pgfonlayer}{background}
					
					\draw[e:marker,myLightBlue] (0.8,0.8) -- (0.3,0.3);
					\draw[e:marker,myLightBlue] (0.3,0.3) -- (0.3,-0.3);
					\draw[e:marker,myLightBlue] (0.3,-0.3) -- (-0.3,-0.3);
					\draw[e:marker,myLightBlue] (-0.3,-0.3) -- (-0.8,-0.8);
					\draw[e:marker,bend left=45,myLightBlue] (-0.8,-0.8) to (-1.8,-1.8);
					
					\draw[e:marker,myGreen] (-0.8,0.8) -- (-1.3,1.3);
					\draw[e:marker,myGreen] (-1.3,1.3) -- (-1.3,-1.3);
					\draw[e:marker,myGreen] (-1.3,-1.3) -- (1.3,-1.3);
					\draw[e:marker,myGreen] (1.3,-1.3) -- (0.8,-0.8);
					\draw[e:marker,myGreen,bend right=45] (0.8,-0.8) to (1.8,-1.8);
					
					\draw[e:main] (-0.8,-0.8) rectangle (0.8,0.8);
					
					\draw[e:main,color=gray] (-1.3,1.3) -- (1.3,1.3) -- (1.3,-1.3) -- (-1.3,-1.3);
					
					\draw[e:main,color=gray] (0.3,0.3) -- (-0.3,0.3) -- (-0.3,-0.3) -- (0.3,-0.3);
					
					\draw[e:main,color=gray] (0.3,0.3) -- (0.8,0.8);
					\draw[e:main,color=gray] (0.3,-0.3) -- (0.8,-0.8);
					
					\draw[e:main,color=gray] (-0.8,0.8) -- (-1.3,1.3);
					\draw[e:main,color=gray] (-0.8,-0.8) -- (-1.3,-1.3);
					
					\draw[e:main,color=gray] (-0.8,0.8) -- (-1.8,0);
					\draw[e:main,color=gray] (0.8,0.8) -- (1.8,0);
					
					\draw[e:main,color=gray] (-1.8,-1.8) -- (-0.8,-2.6);
					\draw[e:main,color=gray] (1.8,-1.8) -- (0.8,-2.6);
					
					\draw[e:main,color=gray,bend left=45] (-0.8,-0.8) to (-1.8,-1.8);
					\draw[e:main,color=gray,bend right=45] (0.8,-0.8) to (1.8,-1.8);
					
					\draw[e:coloredborder] (0.3,-0.3) -- (0.3,0.3);
					\draw[e:coloredthin,color=BostonUniversityRed] (0.3,-0.3) -- (0.3,0.3);

					
					\draw[e:coloredborder] (-0.3,0.3) -- (-0.8,0.8);
					\draw[e:coloredthin,color=BostonUniversityRed] (-0.3,0.3) -- (-0.8,0.8);
					\draw[e:coloredborder] (-0.3,-0.3) -- (-0.8,-0.8);
					\draw[e:coloredthin,color=BostonUniversityRed] (-0.3,-0.3) -- (-0.8,-0.8);
					
					\draw[e:coloredborder] (-1.3,-1.3) -- (-1.3,1.3);
					\draw[e:coloredthin,color=BostonUniversityRed]  (-1.3,-1.3) -- (-1.3,1.3);
					
					\draw[e:coloredborder] (0.8,0.8) -- (1.3,1.3);
					\draw[e:coloredthin,color=BostonUniversityRed] (0.8,0.8) -- (1.3,1.3);
					\draw[e:coloredborder] (0.8,-0.8) -- (1.3,-1.3);
					\draw[e:coloredthin,color=BostonUniversityRed] (0.8,-0.8) -- (1.3,-1.3);
					
					\draw[e:coloredborder] (-1.8,0) -- (-1.8,-1.8);
					\draw[e:coloredthin,color=BostonUniversityRed] (-1.8,0) -- (-1.8,-1.8);
					\draw[e:coloredborder] (1.8,0) -- (1.8,-1.8);
					\draw[e:coloredthin,color=BostonUniversityRed] (1.8,0) -- (1.8,-1.8);
					
					\draw[e:coloredborder] (-0.8,-2.6) -- (0.8,-2.6);
					\draw[e:coloredthin,color=BostonUniversityRed] (-0.8,-2.6) -- (0.8,-2.6);

				\end{pgfonlayer}
			\end{tikzpicture}
		\end{subfigure}
		\begin{subfigure}{0.24\textwidth}
			\centering
			\begin{tikzpicture}[scale=0.8]
				\pgfdeclarelayer{background}
				\pgfdeclarelayer{foreground}
				\pgfsetlayers{background,main,foreground}
				
				\node () at (0,1.05){$C$};
				\node[gray] () at (-1.5,-2.5){$C'$};

				\node () at (-0.4,-0.65){$a$};
				\node () at (0.4,-0.6){$b$};
				
				\node() at (1.8,-2){$a'$};
				\node() at (-1.8,-2.1){$b'$};
				
				\node[v:mainemptygray] () at (-0.3,-0.3){};
				\node[v:maingray] () at (-0.3,0.3){};
				\node[v:maingray] () at (0.3,-0.3){};
				\node[v:mainemptygray] () at (0.3,0.3){};
				
				\node[v:main] () at (-0.8,-0.8){};
				\node[v:mainempty] () at (-0.8,0.8){};
				\node[v:mainempty] () at (0.8,-0.8){};
				\node[v:main] () at (0.8,0.8){};
				
				\node[v:mainemptygray] () at (-1.3,-1.3){};
				\node[v:maingray] () at (-1.3,1.3){};
				\node[v:maingray] () at (1.3,-1.3){};
				\node[v:mainemptygray] () at (1.3,1.3){}; 
				
				\node[v:maingray] () at (-1.8,0){};
				\node[v:mainemptygray] () at (1.8,0){};
				
				\node[v:mainemptygray] () at (-1.8,-1.8){};
				\node[v:maingray] () at (1.8,-1.8){};
				
				\node[v:maingray] () at (-0.8,-2.6){};
				\node[v:mainemptygray] () at (0.8,-2.6){};
				
				\node[v:maingray]() at (0,-1.5){};
				\node () at (0.2,-1.7){$v$};
				
				\node[v:mainemptygray] at (0,-2.1){};
				\node[] () at (0,-2.4){$u$};
				
				\begin{pgfonlayer}{background}
					
					\draw[e:marker,myGreen,bend right=13] (-1.8,-1.8) to (0,-1.5);
					\draw[e:marker,myGreen,bend right=20] (0,-1.5) to (0.8,-0.8);
					\draw[e:marker,myGreen] (0.8,-0.8) -- (1.3,-1.3);
					\draw[e:marker,myGreen] (1.3,-1.3) -- (1.3,1.3);
					\draw[e:marker,myGreen] (1.3,1.3) -- (0.8,0.8);
					
					\draw[e:marker,bend left=10,myLightBlue] (1.8,-1.8) to (0,-2.1);
					\draw[e:marker,bend left=20,myLightBlue] (0,-2.1) to (-0.8,-0.8);
					\draw[e:marker,myLightBlue] (-0.8,-0.8) -- (-0.3,-0.3);
					\draw[e:marker,myLightBlue] (-0.3,-0.3) -- (-0.3,0.3);
					\draw[e:marker,myLightBlue] (-0.3,0.3) -- (-0.8,0.8);
					
					\draw[e:main,color=gray] (-1.3,1.3) -- (1.3,1.3);
					\draw[e:main,color=gray] (1.3,-1.3) -- (-1.3,-1.3);
					
					\draw[e:main] (0.8,0.8) -- (0.8,-0.8) -- (-0.8,-0.8) -- (-0.8,0.8);
					
					\draw[e:main,color=gray] (0.3,0.3) -- (0.8,0.8);
					\draw[e:main,color=gray] (0.3,-0.3) -- (0.8,-0.8);
					
					\draw[e:main,color=gray] (0.3,0.3) -- (-0.3,0.3);
					\draw[e:main,color=gray] (-0.3,-0.3) -- (0.3,-0.3);
					
					\draw[e:main,color=gray] (-0.3,0.3) -- (-0.8,0.8);
					\draw[e:main,color=gray] (-0.3,-0.3) -- (-0.8,-0.8);
					
					\draw[e:main,color=gray] (-0.8,0.8) -- (-1.3,1.3);
					\draw[e:main,color=gray] (-0.8,-0.8) -- (-1.3,-1.3);
					
					\draw[e:main,color=gray] (0.8,0.8) -- (1.3,1.3);
					\draw[e:main,color=gray] (0.8,-0.8) -- (1.3,-1.3);
					
					\draw[e:main,color=gray] (-0.8,0.8) -- (-1.8,0);
					\draw[e:main,color=gray] (0.8,0.8) -- (1.8,0);
					
					\draw[e:main,color=gray] (-1.8,-1.8) -- (-0.8,-2.6);
					\draw[e:main,color=gray] (1.8,-1.8) -- (0.8,-2.6);
					
					\draw[e:main,color=gray,bend right=13] (-1.8,-1.8) to (0,-1.5);
					\draw[e:main,color=gray,bend left=10] (1.8,-1.8) to (0,-2.1);
					
					\draw[e:main,color=gray,densely dashed] (0,-1.5) -- (0,-2.1);
					
					\draw[e:coloredborder] (0.3,-0.3) -- (0.3,0.3);
					\draw[e:coloredthin,color=BostonUniversityRed] (0.3,-0.3) -- (0.3,0.3);
					
					\draw[e:coloredborder] (-0.3,0.3) -- (-0.3,-0.3);
					\draw[e:coloredthin,color=BostonUniversityRed] (-0.3,0.3) -- (-0.3,-0.3);
					
					\draw[e:coloredborder] (-0.8,0.8) -- (0.8,0.8);
					\draw[e:coloredthin,color=BostonUniversityRed] (-0.8,0.8) -- (0.8,0.8);
					
					\draw[e:coloredborder] (-1.3,-1.3) -- (-1.3,1.3);
					\draw[e:coloredthin,color=BostonUniversityRed]  (-1.3,-1.3) -- (-1.3,1.3);
					
					\draw[e:coloredborder] (1.3,1.3) -- (1.3,-1.3);
					\draw[e:coloredthin,color=BostonUniversityRed] (1.3,1.3) -- (1.3,-1.3);

					\draw[e:coloredborder] (-1.8,0) -- (-1.8,-1.8);
					\draw[e:coloredthin,color=BostonUniversityRed] (-1.8,0) -- (-1.8,-1.8);
					\draw[e:coloredborder] (1.8,0) -- (1.8,-1.8);
					\draw[e:coloredthin,color=BostonUniversityRed] (1.8,0) -- (1.8,-1.8);
					
					\draw[e:coloredborder] (-0.8,-2.6) -- (0.8,-2.6);
					\draw[e:coloredthin,color=BostonUniversityRed] (-0.8,-2.6) -- (0.8,-2.6);
					
					\draw[e:coloredborder,bend right=20] (0,-1.5) to (0.8,-0.8);
					\draw[e:coloredthin,color=BostonUniversityRed,bend right=20] (0,-1.5) to (0.8,-0.8);
					
					\draw[e:coloredborder,bend left=20] (0,-2.1) to (-0.8,-0.8);
					\draw[e:coloredthin,color=BostonUniversityRed,,bend left=20] (0,-2.1) to (-0.8,-0.8);

				\end{pgfonlayer}
			\end{tikzpicture}
		\end{subfigure}
		\begin{subfigure}{0.24\textwidth}
			\centering
			\begin{tikzpicture}[scale=0.8]
				\pgfdeclarelayer{background}
				\pgfdeclarelayer{foreground}
				\pgfsetlayers{background,main,foreground}
				
				\node () at (0,1.05){$C$};
				\node[gray] () at (-1.5,-2.5){$C'$};

				\node () at (-0.4,-0.65){$a$};
				\node () at (0.4,-0.6){$b$};
				
				\node[v:mainemptygray] () at (-0.3,-0.3){};
				\node[v:maingray] () at (-0.3,0.3){};
				\node[v:maingray] () at (0.3,-0.3){};
				\node[v:mainemptygray] () at (0.3,0.3){};
				
				\node[v:main] () at (-0.8,-0.8){};
				\node[v:mainempty] () at (-0.8,0.8){};
				\node[v:mainempty] () at (0.8,-0.8){};
				\node[v:main] () at (0.8,0.8){};
				
				\node[v:mainemptygray] () at (-1.3,-1.3){};
				\node[v:maingray] () at (-1.3,1.3){};
				\node[v:maingray] () at (1.3,-1.3){};
				\node[v:mainemptygray] () at (1.3,1.3){}; 
				
				\node[v:maingray] () at (-1.8,0){};
				\node[v:mainemptygray] () at (1.8,0){};
				
				\node[v:mainemptygray] () at (-1.8,-1.8){};
				\node[v:maingray] () at (1.8,-1.8){};
				
				\node[v:maingray] () at (-0.8,-2.6){};
				\node () at (-0.4,-2.3){$a'$};
				\node[v:mainemptygray] () at (0.8,-2.6){};
				\node() at (0.3,-2.3){$b'$};
				
				\begin{pgfonlayer}{background}
					
					\draw[e:marker,color=myGreen] (0.8,0.8) -- (1.3,1.3);
					\draw[e:marker,color=myGreen] (1.3,1.3) -- (1.3,-1.3);
					\draw[e:marker,color=myGreen] (1.3,-1.3) -- (0.8,-0.8);
					\draw[e:marker,color=myGreen,bend left=20] (0.8,-0.8) to (0,-1.7);
					\draw[e:marker,color=myGreen,bend right=20] (0,-1.7) to (-0.8,-2.6);
					
					\draw[e:marker,myLightBlue] (-0.8,0.8) -- (-0.3,0.3);
					\draw[e:marker,myLightBlue] (-0.3,0.3) -- (-0.3,-0.3);
					\draw[e:marker,myLightBlue] (-0.3,-0.3) -- (-0.8,-0.8);
					\draw[e:marker,bend right=20,myLightBlue] (-0.8,-0.8) to (0,-1.7);
					\draw[e:marker,bend left=20,myLightBlue] (0,-1.7) to (0.8,-2.6);
					
					\draw[e:main] (-0.8,-0.8) rectangle (0.8,0.8);
					
					\draw[e:main,color=gray] (-1.3,1.3) -- (1.3,1.3) -- (1.3,-1.3) -- (-1.3,-1.3);
					
					\draw[e:main,color=gray] (0.3,0.3) -- (0.8,0.8);
					\draw[e:main,color=gray] (0.3,-0.3) -- (0.8,-0.8);
					
					\draw[e:main,color=gray] (0.3,0.3) -- (-0.3,0.3) -- (-0.3,-0.3) -- (0.3,-0.3);
					
					\draw[e:main,color=gray] (-0.8,0.8) -- (-1.3,1.3);
					\draw[e:main,color=gray] (-0.8,-0.8) -- (-1.3,-1.3);
					
					\draw[e:main,color=gray] (-0.8,0.8) -- (-1.8,0);
					\draw[e:main,color=gray] (0.8,0.8) -- (1.8,0);
					
					\draw[e:main,color=gray] (-1.8,-1.8) -- (-0.8,-2.6);
					\draw[e:main,color=gray] (1.8,-1.8) -- (0.8,-2.6);
					
					\draw[e:main,color=gray,bend right=20] (-0.8,-0.8) to (0,-1.7);
					\draw[e:main,color=gray,bend left=20] (0,-1.7) to (0.8,-2.6);
					
					\draw[e:main,color=gray,bend left=20] (0.8,-0.8) to (0,-1.7);
					\draw[e:main,color=gray,bend right=20] (0,-1.7) to (-0.8,-2.6);

					\draw[e:coloredborder] (-0.3,0.3) -- (-0.8,0.8);
					\draw[e:coloredthin,color=BostonUniversityRed] (-0.3,0.3) -- (-0.8,0.8);
					\draw[e:coloredborder] (-0.3,-0.3) -- (-0.8,-0.8);
					\draw[e:coloredthin,color=BostonUniversityRed] (-0.3,-0.3) -- (-0.8,-0.8);
					
					\draw[e:coloredborder] (0.3,-0.3) -- (0.3,0.3);
					\draw[e:coloredthin,color=BostonUniversityRed] (0.3,-0.3) -- (0.3,0.3);
					
					\draw[e:coloredborder] (-1.3,-1.3) -- (-1.3,1.3);
					\draw[e:coloredthin,color=BostonUniversityRed]  (-1.3,-1.3) -- (-1.3,1.3);
					
					\draw[e:coloredborder] (0.8,0.8) -- (1.3,1.3);
					\draw[e:coloredthin,color=BostonUniversityRed] (0.8,0.8) -- (1.3,1.3);
					\draw[e:coloredborder] (0.8,-0.8) -- (1.3,-1.3);
					\draw[e:coloredthin,color=BostonUniversityRed] (0.8,-0.8) -- (1.3,-1.3);
					
					\draw[e:coloredborder] (-1.8,0) -- (-1.8,-1.8);
					\draw[e:coloredthin,color=BostonUniversityRed] (-1.8,0) -- (-1.8,-1.8);
					\draw[e:coloredborder] (1.8,0) -- (1.8,-1.8);
					\draw[e:coloredthin,color=BostonUniversityRed] (1.8,0) -- (1.8,-1.8);
					
					\draw[e:coloredborder] (-0.8,-2.6) -- (0.8,-2.6);
					\draw[e:coloredthin,color=BostonUniversityRed] (-0.8,-2.6) -- (0.8,-2.6);

				\end{pgfonlayer}
			\end{tikzpicture}
		\end{subfigure}
		\begin{subfigure}{0.24\textwidth}
			\centering
			\begin{tikzpicture}[scale=0.8]
				\pgfdeclarelayer{background}
				\pgfdeclarelayer{foreground}
				\pgfsetlayers{background,main,foreground}
				
				\node () at (0,1.05){$C$};
				\node[gray] () at (-1.5,-2.5){$C'$};

				\node () at (-0.4,-0.65){$a$};
				\node () at (0.4,-0.6){$b$};
				
				\node[v:mainemptygray] () at (-0.3,-0.3){};
				\node[v:maingray] () at (-0.3,0.3){};
				\node[v:maingray] () at (0.3,-0.3){};
				\node[v:mainemptygray] () at (0.3,0.3){};
				
				\node[v:main] () at (-0.8,-0.8){};
				\node[v:mainempty] () at (-0.8,0.8){};
				\node[v:mainempty] () at (0.8,-0.8){};
				\node[v:main] () at (0.8,0.8){};
				
				\node[v:mainemptygray] () at (-1.3,-1.3){};
				\node[v:maingray] () at (-1.3,1.3){};
				\node[v:maingray] () at (1.3,-1.3){};
				\node[v:mainemptygray] () at (1.3,1.3){};  
				
				\node[v:maingray] () at (-1.8,0){};
				\node[v:mainemptygray] () at (1.8,0){};
				
				\node[v:mainemptygray] () at (-1.8,-1.8){};
				\node[v:maingray] () at (1.8,-1.8){};
				
				\node[v:maingray] () at (-0.8,-2.6){};
				\node () at (-0.3,-2.3){$a'$};
				\node[v:mainemptygray] () at (0.8,-2.6){};
				\node() at (0.2,-2.3){$b'$};
				
				\node[v:mainemptygray] () at (-0.3,-1.7){};
				\node () at (-0.7,-1.7){$u$};
				\node[v:maingray] () at (0.3,-1.7){};
				\node () at (0.7,-1.7){$v$};
				
				\begin{pgfonlayer}{background}
					
					\draw[e:marker,myGreen] (-0.8,0.8) -- (-1.3,1.3);
					\draw[e:marker,myGreen] (-1.3,1.3) -- (-1.3,-1.3);
					\draw[e:marker,myGreen] (-1.3,-1.3) -- (1.3,-1.3);
					\draw[e:marker,myGreen] (1.3,-1.3) -- (0.8,-0.8);
					\draw[e:marker,myGreen,bend right=20] (0.8,-0.8) to (0.3,-1.7);
					\draw[e:marker,myGreen,bend right=20] (0.3,-1.7) to (0.8,-2.6);
					
					\draw[e:marker,myLightBlue] (0.8,0.8) -- (0.3,0.3);
					\draw[e:marker,myLightBlue] (0.3,0.3) -- (0.3,-0.3);
					\draw[e:marker,myLightBlue] (0.3,-0.3) -- (-0.3,-0.3);
					\draw[e:marker,myLightBlue] (-0.3,-0.3) -- (-0.8,-0.8);
					\draw[e:marker,bend left=20,myLightBlue] (-0.8,-0.8) to (-0.3,-1.7);
					\draw[e:marker,bend left=20,myLightBlue] (-0.3,-1.7) to (-0.8,-2.6);
					
					\draw[e:main] (-0.8,-0.8) rectangle (0.8,0.8);
					
					\draw[e:main,color=gray] (-1.3,1.3) -- (1.3,1.3) -- (1.3,-1.3) -- (-1.3,-1.3);
					
					\draw[e:main,color=gray] (0.3,0.3) -- (0.8,0.8);
					\draw[e:main,color=gray] (0.3,-0.3) -- (0.8,-0.8);
					
					\draw[e:main,color=gray] (0.3,0.3) -- (-0.3,0.3) -- (-0.3,-0.3) -- (0.3,-0.3);
					
					\draw[e:main,color=gray] (-0.8,0.8) -- (-1.3,1.3);
					\draw[e:main,color=gray] (-0.8,-0.8) -- (-1.3,-1.3);
					
					\draw[e:main,color=gray] (-0.8,0.8) -- (-1.8,0);
					\draw[e:main,color=gray] (0.8,0.8) -- (1.8,0);
					
					\draw[e:main,color=gray] (-1.8,-1.8) -- (-0.8,-2.6);
					\draw[e:main,color=gray] (1.8,-1.8) -- (0.8,-2.6);
					
					\draw[e:main,color=gray] (-0.8,-2.6) -- (0.8,-2.6);
					
					\draw[e:main,color=gray,densely dashed] (-0.3,-1.7) -- (0.3,-1.7);
					
					\draw[e:main,color=gray,bend left=20] (-0.8,-0.8) to (-0.3,-1.7);
					\draw[e:main,color=gray,bend right=20] (0.8,-0.8) to (0.3,-1.7);

					\draw[e:coloredborder] (-0.3,0.3) -- (-0.8,0.8);
					\draw[e:coloredthin,color=BostonUniversityRed] (-0.3,0.3) -- (-0.8,0.8);
					\draw[e:coloredborder] (-0.3,-0.3) -- (-0.8,-0.8);
					\draw[e:coloredthin,color=BostonUniversityRed] (-0.3,-0.3) -- (-0.8,-0.8);
					
					\draw[e:coloredborder] (0.3,-0.3) -- (0.3,0.3);
					\draw[e:coloredthin,color=BostonUniversityRed] (0.3,-0.3) -- (0.3,0.3);
					
					\draw[e:coloredborder] (-1.3,-1.3) -- (-1.3,1.3);
					\draw[e:coloredthin,color=BostonUniversityRed]  (-1.3,-1.3) -- (-1.3,1.3);
					
					\draw[e:coloredborder] (0.8,0.8) -- (1.3,1.3);
					\draw[e:coloredthin,color=BostonUniversityRed] (0.8,0.8) -- (1.3,1.3);
					\draw[e:coloredborder] (0.8,-0.8) -- (1.3,-1.3);
					\draw[e:coloredthin,color=BostonUniversityRed] (0.8,-0.8) -- (1.3,-1.3);
					
					\draw[e:coloredborder] (-1.8,0) -- (-1.8,-1.8);
					\draw[e:coloredthin,color=BostonUniversityRed] (-1.8,0) -- (-1.8,-1.8);
					\draw[e:coloredborder] (1.8,0) -- (1.8,-1.8);
					\draw[e:coloredthin,color=BostonUniversityRed] (1.8,0) -- (1.8,-1.8);
					
					\draw[e:coloredborder,bend left=20] (-0.3,-1.7) to (-0.8,-2.6);
					\draw[e:coloredthin,color=BostonUniversityRed,bend left=20] (-0.3,-1.7) to (-0.8,-2.6);
					
					\draw[e:coloredborder,bend right=20] (0.3,-1.7) to (0.8,-2.6);
					\draw[e:coloredthin,color=BostonUniversityRed,bend right=20] (0.3,-1.7) to (0.8,-2.6);

				\end{pgfonlayer}
			\end{tikzpicture}
		\end{subfigure}
		
		\caption{The four different ways to obtain a strong matching cross over $C'$ in the first subcase of the first case in the proof of \cref{lemma:easycrossesovercrossingcycles}.}
		\label{fig:pfaffiancase1crosses1}
	\end{figure}
	
	In either case, we can now use \cref{lemma:badcrosses} to find a perfect matching $M''$ of $B$ such that $\E{M'}\setminus\E{C}\subseteq M''$ together with two $M''$-alternating paths $P_1$ and $P_2$ such that these paths are disjoint and are completely contained in $B-\V{B_1-\V{C}}$.
	Indeed, $P_1$ and $P_2$ can be chosen such that $P_1L$ and $P_2R$ form a strong matching cross over $C'$ in $B$ as illustrated in \cref{fig:pfaffiancase1crosses1}.
	
	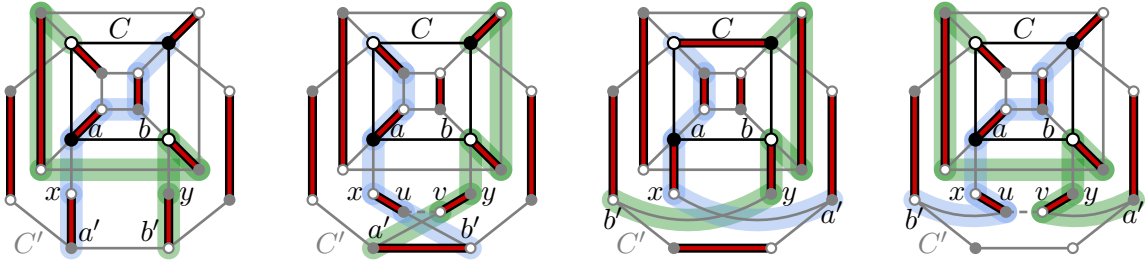
\begin{figure}[h!]
		\centering
		\begin{subfigure}{0.24\textwidth}
			\centering
			\begin{tikzpicture}[scale=0.8]
				\pgfdeclarelayer{background}
				\pgfdeclarelayer{foreground}
				\pgfsetlayers{background,main,foreground}
				
				\node () at (0,1.05){$C$};
				\node[gray] () at (-1.5,-2.5){$C'$};
				
				\node () at (-0.4,-0.65){$a$};
				\node () at (0.4,-0.6){$b$};
				\node() at (-0.5,-2.3){$a'$};
				\node() at (0.5,-2.3){$b'$};
				
				\node[v:mainemptygray] () at (-0.3,-0.3){};
				\node[v:maingray] () at (-0.3,0.3){};
				\node[v:maingray] () at (0.3,-0.3){};
				\node[v:mainemptygray] () at (0.3,0.3){};
				
				\node[v:main] () at (-0.8,-0.8){};
				\node[v:mainempty] () at (-0.8,0.8){};
				\node[v:mainempty] () at (0.8,-0.8){};
				\node[v:main] () at (0.8,0.8){};
				
				\node[v:mainemptygray] () at (-1.3,-1.3){};
				\node[v:maingray] () at (-1.3,1.3){};
				\node[v:maingray] () at (1.3,-1.3){};
				\node[v:mainemptygray] () at (1.3,1.3){};   
				
				\node[v:maingray] () at (-1.8,0){};
				\node[v:mainemptygray] () at (1.8,0){};
				
				\node[v:mainemptygray] () at (-1.8,-1.8){};
				\node[v:maingray] () at (1.8,-1.8){};
				
				\node[v:maingray] () at (-0.8,-2.6){};
				\node[v:mainemptygray] () at (0.8,-2.6){};
				
				\node[v:mainemptygray] () at (-0.8,-1.7){};
				\node () at (-1.1,-1.7){$x$};
				
				\node[v:maingray] () at (0.8,-1.7){};
				\node () at (1.1,-1.75){$y$};
				
				\begin{pgfonlayer}{background}
					
					\draw[e:marker,myLightBlue] (0.8,0.8) -- (0.3,0.3);
					\draw[e:marker,myLightBlue] (0.3,0.3) -- (0.3,-0.3);
					\draw[e:marker,myLightBlue] (0.3,-0.3) -- (-0.3,-0.3);
					\draw[e:marker,myLightBlue] (-0.3,-0.3) -- (-0.8,-0.8);
					\draw[e:marker,myLightBlue] (-0.8,-0.8) -- (-0.8,-1.7);
					\draw[e:marker,myLightBlue] (-0.8,-2.6) -- (-0.8,-1.7);

					\draw[e:marker,myGreen] (-0.8,0.8) -- (-1.3,1.3);
					\draw[e:marker,myGreen] (-1.3,1.3) -- (-1.3,-1.3);
					\draw[e:marker,myGreen] (-1.3,-1.3) -- (1.3,-1.3);
					\draw[e:marker,myGreen] (1.3,-1.3) -- (0.8,-0.8);
					\draw[e:marker,myGreen] (0.8,-0.8) -- (0.8,-1.7);
					\draw[e:marker,myGreen] (0.8,-2.6) -- (0.8,-1.7);
					
					\draw[e:main] (-0.8,-0.8) rectangle (0.8,0.8);
					
					\draw[e:main,color=gray] (-1.3,1.3) -- (1.3,1.3) -- (1.3,-1.3) -- (-1.3,-1.3);
					
					\draw[e:main,color=gray] (0.3,0.3) -- (-0.3,0.3) -- (-0.3,-0.3) -- (0.3,-0.3);
					
					\draw[e:main,color=gray] (0.3,0.3) -- (0.8,0.8);
					\draw[e:main,color=gray] (0.3,-0.3) -- (0.8,-0.8);
					
					\draw[e:main,color=gray] (-0.8,0.8) -- (-1.3,1.3);
					\draw[e:main,color=gray] (-0.8,-0.8) -- (-1.3,-1.3);
					
					\draw[e:main,color=gray] (-0.8,0.8) -- (-1.8,0);
					\draw[e:main,color=gray] (0.8,0.8) -- (1.8,0);
					
					\draw[e:main,color=gray] (-1.8,-1.8) -- (-0.8,-2.6);
					\draw[e:main,color=gray] (1.8,-1.8) -- (0.8,-2.6);
					
					\draw[e:main,color=gray] (-0.8,-2.6) -- (0.8,-2.6);
					
					\draw[e:main,color=gray] (-0.8,-0.8) -- (-0.8,-1.7);
					\draw[e:main,color=gray] (0.8,-0.8) -- (0.8,-1.7);

					\draw[e:coloredborder] (0.3,-0.3) -- (0.3,0.3);
					\draw[e:coloredthin,color=BostonUniversityRed] (0.3,-0.3) -- (0.3,0.3);

					
					\draw[e:coloredborder] (-0.3,0.3) -- (-0.8,0.8);
					\draw[e:coloredthin,color=BostonUniversityRed] (-0.3,0.3) -- (-0.8,0.8);
					\draw[e:coloredborder] (-0.3,-0.3) -- (-0.8,-0.8);
					\draw[e:coloredthin,color=BostonUniversityRed] (-0.3,-0.3) -- (-0.8,-0.8);
					
					\draw[e:coloredborder] (-1.3,-1.3) -- (-1.3,1.3);
					\draw[e:coloredthin,color=BostonUniversityRed]  (-1.3,-1.3) -- (-1.3,1.3);
					
					\draw[e:coloredborder] (0.8,0.8) -- (1.3,1.3);
					\draw[e:coloredthin,color=BostonUniversityRed] (0.8,0.8) -- (1.3,1.3);
					\draw[e:coloredborder] (0.8,-0.8) -- (1.3,-1.3);
					\draw[e:coloredthin,color=BostonUniversityRed] (0.8,-0.8) -- (1.3,-1.3);
					
					\draw[e:coloredborder] (-1.8,0) -- (-1.8,-1.8);
					\draw[e:coloredthin,color=BostonUniversityRed] (-1.8,0) -- (-1.8,-1.8);
					\draw[e:coloredborder] (1.8,0) -- (1.8,-1.8);
					\draw[e:coloredthin,color=BostonUniversityRed] (1.8,0) -- (1.8,-1.8);
					
					\draw[e:coloredborder] (0.8,-1.7) -- (0.8,-2.6);
					\draw[e:coloredthin,color=BostonUniversityRed] (0.8,-1.7) -- (0.8,-2.6);
					
					\draw[e:coloredborder] (-0.8,-1.7) -- (-0.8,-2.6);
					\draw[e:coloredthin,color=BostonUniversityRed] (-0.8,-1.7) -- (-0.8,-2.6);

				\end{pgfonlayer}
			\end{tikzpicture}
		\end{subfigure}
		\begin{subfigure}{0.24\textwidth}
			\centering
			\begin{tikzpicture}[scale=0.8]
				\pgfdeclarelayer{background}
				\pgfdeclarelayer{foreground}
				\pgfsetlayers{background,main,foreground}
				
				\node () at (0,1.05){$C$};
				\node[gray] () at (-1.5,-2.5){$C'$};

				\node () at (-0.4,-0.65){$a$};
				\node () at (0.4,-0.6){$b$};
				
				\node[v:mainemptygray] () at (-0.3,-0.3){};
				\node[v:maingray] () at (-0.3,0.3){};
				\node[v:maingray] () at (0.3,-0.3){};
				\node[v:mainemptygray] () at (0.3,0.3){};
				
				\node[v:main] () at (-0.8,-0.8){};
				\node[v:mainempty] () at (-0.8,0.8){};
				\node[v:mainempty] () at (0.8,-0.8){};
				\node[v:main] () at (0.8,0.8){};
				
				\node[v:mainemptygray] () at (-1.3,-1.3){};
				\node[v:maingray] () at (-1.3,1.3){};
				\node[v:maingray] () at (1.3,-1.3){};
				\node[v:mainemptygray] () at (1.3,1.3){}; 
				
				\node[v:maingray] () at (-1.8,0){};
				\node[v:mainemptygray] () at (1.8,0){};
				
				\node[v:mainemptygray] () at (-1.8,-1.8){};
				\node[v:maingray] () at (1.8,-1.8){};
				
				\node[v:maingray] () at (-0.8,-2.6){};
				\node () at (-0.7,-2.25){$a'$};
				\node[v:mainemptygray] () at (0.8,-2.6){};
				\node() at (0.8,-2.25){$b'$};
				
				\node[v:mainemptygray] () at (-0.8,-1.7){};
				\node () at (-1.1,-1.7){$x$};
				
				\node[v:maingray] () at (0.8,-1.7){};
				\node () at (1.1,-1.75){$y$};
				
				\node[v:maingray] at (-0.3,-2){};
				\node () at (-0.3,-1.7){$u$};
				
				\node[v:mainemptygray] at (0.3,-2){};
				\node () at (0.3,-1.7){$v$};
				
				\begin{pgfonlayer}{background}
					
					\draw[e:marker,color=myGreen] (0.8,0.8) -- (1.3,1.3);
					\draw[e:marker,color=myGreen] (1.3,1.3) -- (1.3,-1.3);
					\draw[e:marker,color=myGreen] (1.3,-1.3) -- (0.8,-0.8);
					\draw[e:marker,myGreen] (0.8,-0.8) -- (0.8,-1.7);
					\draw[e:marker,myGreen] (0.8,-1.7) -- (0.3,-2);
					\draw[e:marker,myGreen] (0.3,-2) -- (-0.8,-2.6);
					
					\draw[e:marker,myLightBlue] (-0.8,0.8) -- (-0.3,0.3);
					\draw[e:marker,myLightBlue] (-0.3,0.3) -- (-0.3,-0.3);
					\draw[e:marker,myLightBlue] (-0.3,-0.3) -- (-0.8,-0.8);
					\draw[e:marker,myLightBlue] (-0.8,-0.8) -- (-0.8,-1.7);
					\draw[e:marker,myLightBlue] (-0.8,-1.7) -- (-0.3,-2);
					\draw[e:marker,myLightBlue] (-0.3,-2) -- (0.8,-2.6);
					
					\draw[e:main] (-0.8,-0.8) rectangle (0.8,0.8);
					
					\draw[e:main,color=gray] (-1.3,1.3) -- (1.3,1.3) -- (1.3,-1.3) -- (-1.3,-1.3);
					
					\draw[e:main,color=gray] (0.3,0.3) -- (0.8,0.8);
					\draw[e:main,color=gray] (0.3,-0.3) -- (0.8,-0.8);
					
					\draw[e:main,color=gray] (0.3,0.3) -- (-0.3,0.3) -- (-0.3,-0.3) -- (0.3,-0.3);
					
					\draw[e:main,color=gray] (-0.8,0.8) -- (-1.3,1.3);
					\draw[e:main,color=gray] (-0.8,-0.8) -- (-1.3,-1.3);
					
					\draw[e:main,color=gray] (-0.8,0.8) -- (-1.8,0);
					\draw[e:main,color=gray] (0.8,0.8) -- (1.8,0);
					
					\draw[e:main,color=gray] (-1.8,-1.8) -- (-0.8,-2.6);
					\draw[e:main,color=gray] (1.8,-1.8) -- (0.8,-2.6);
					
					\draw[e:main,color=gray] (-0.8,-0.8) -- (-0.8,-1.7);
					\draw[e:main,color=gray] (0.8,-0.8) -- (0.8,-1.7);
					
					\draw[e:main,color=gray,densely dashed] (-0.3,-2) -- (0.3,-2);
					
					\draw[e:main,color=gray] (-0.3,-2) -- (0.8,-2.6);
					\draw[e:main,color=gray] (0.3,-2) -- (-0.8,-2.6);
					
					\draw[e:coloredborder] (-0.3,0.3) -- (-0.8,0.8);
					\draw[e:coloredthin,color=BostonUniversityRed] (-0.3,0.3) -- (-0.8,0.8);
					\draw[e:coloredborder] (-0.3,-0.3) -- (-0.8,-0.8);
					\draw[e:coloredthin,color=BostonUniversityRed] (-0.3,-0.3) -- (-0.8,-0.8);
					
					\draw[e:coloredborder] (0.3,-0.3) -- (0.3,0.3);
					\draw[e:coloredthin,color=BostonUniversityRed] (0.3,-0.3) -- (0.3,0.3);
					
					\draw[e:coloredborder] (-1.3,-1.3) -- (-1.3,1.3);
					\draw[e:coloredthin,color=BostonUniversityRed]  (-1.3,-1.3) -- (-1.3,1.3);
					
					\draw[e:coloredborder] (0.8,0.8) -- (1.3,1.3);
					\draw[e:coloredthin,color=BostonUniversityRed] (0.8,0.8) -- (1.3,1.3);
					\draw[e:coloredborder] (0.8,-0.8) -- (1.3,-1.3);
					\draw[e:coloredthin,color=BostonUniversityRed] (0.8,-0.8) -- (1.3,-1.3);
					
					\draw[e:coloredborder] (-1.8,0) -- (-1.8,-1.8);
					\draw[e:coloredthin,color=BostonUniversityRed] (-1.8,0) -- (-1.8,-1.8);
					\draw[e:coloredborder] (1.8,0) -- (1.8,-1.8);
					\draw[e:coloredthin,color=BostonUniversityRed] (1.8,0) -- (1.8,-1.8);
					
					\draw[e:coloredborder] (-0.8,-2.6) -- (0.8,-2.6);
					\draw[e:coloredthin,color=BostonUniversityRed] (-0.8,-2.6) -- (0.8,-2.6);
					
					\draw[e:coloredborder] (-0.8,-1.7) -- (-0.3,-2);
					\draw[e:coloredthin,color=BostonUniversityRed](-0.8,-1.7) -- (-0.3,-2);
					
					\draw[e:coloredborder] (0.8,-1.7) -- (0.3,-2);
					\draw[e:coloredthin,color=BostonUniversityRed](0.8,-1.7) -- (0.3,-2);

				\end{pgfonlayer}
			\end{tikzpicture}
		\end{subfigure}
		\begin{subfigure}{0.24\textwidth}
			\centering
			\begin{tikzpicture}[scale=0.8]
				\pgfdeclarelayer{background}
				\pgfdeclarelayer{foreground}
				\pgfsetlayers{background,main,foreground}
				
				\node () at (0,1.05){$C$};
				\node[gray] () at (-1.5,-2.5){$C'$};

				\node () at (-0.4,-0.65){$a$};
				\node () at (0.4,-0.6){$b$};
				
				\node() at (1.8,-2){$a'$};
				\node() at (-1.8,-2.1){$b'$};
				
				\node[v:mainemptygray] () at (-0.3,-0.3){};
				\node[v:maingray] () at (-0.3,0.3){};
				\node[v:maingray] () at (0.3,-0.3){};
				\node[v:mainemptygray] () at (0.3,0.3){};
				
				\node[v:main] () at (-0.8,-0.8){};
				\node[v:mainempty] () at (-0.8,0.8){};
				\node[v:mainempty] () at (0.8,-0.8){};
				\node[v:main] () at (0.8,0.8){};
				
				\node[v:mainemptygray] () at (-1.3,-1.3){};
				\node[v:maingray] () at (-1.3,1.3){};
				\node[v:maingray] () at (1.3,-1.3){};
				\node[v:mainemptygray] () at (1.3,1.3){}; 
				
				\node[v:maingray] () at (-1.8,0){};
				\node[v:mainemptygray] () at (1.8,0){};
				
				\node[v:mainemptygray] () at (-1.8,-1.8){};
				\node[v:maingray] () at (1.8,-1.8){};
				
				\node[v:maingray] () at (-0.8,-2.6){};
				\node[v:mainemptygray] () at (0.8,-2.6){};

				\node[v:mainemptygray] () at (-0.8,-1.7){};
				\node () at (-1.1,-1.7){$x$};
				
				\node[v:maingray] () at (0.8,-1.7){};
				\node () at (1.1,-1.75){$y$};
				
				\begin{pgfonlayer}{background}
					
					\draw[e:marker,myGreen] (0.8,-0.8) -- (1.3,-1.3);
					\draw[e:marker,myGreen] (1.3,-1.3) -- (1.3,1.3);
					\draw[e:marker,myGreen] (1.3,1.3) -- (0.8,0.8);
					\draw[e:marker,myGreen] (0.8,-0.8) -- (0.8,-1.7);
					\draw[e:marker,myGreen,bend left=30] (0.8,-1.7) to (-1.8,-1.8);

					\draw[e:marker,myLightBlue] (-0.8,-0.8) -- (-0.3,-0.3);
					\draw[e:marker,myLightBlue] (-0.3,-0.3) -- (-0.3,0.3);
					\draw[e:marker,myLightBlue] (-0.3,0.3) -- (-0.8,0.8);
					\draw[e:marker,myLightBlue] (-0.8,-0.8) -- (-0.8,-1.7);
					\draw[e:marker,bend right=30,myLightBlue] (-0.8,-1.7) to (1.8,-1.8);
					
					\draw[e:main,color=gray] (-1.3,1.3) -- (1.3,1.3);
					\draw[e:main,color=gray] (1.3,-1.3) -- (-1.3,-1.3);
					
					\draw[e:main] (0.8,0.8) -- (0.8,-0.8) -- (-0.8,-0.8) -- (-0.8,0.8);
					
					\draw[e:main,color=gray] (0.3,0.3) -- (0.8,0.8);
					\draw[e:main,color=gray] (0.3,-0.3) -- (0.8,-0.8);
					
					\draw[e:main,color=gray] (0.3,0.3) -- (-0.3,0.3);
					\draw[e:main,color=gray] (-0.3,-0.3) -- (0.3,-0.3);
					
					\draw[e:main,color=gray] (-0.3,0.3) -- (-0.8,0.8);
					\draw[e:main,color=gray] (-0.3,-0.3) -- (-0.8,-0.8);
					
					\draw[e:main,color=gray] (-0.8,0.8) -- (-1.3,1.3);
					\draw[e:main,color=gray] (-0.8,-0.8) -- (-1.3,-1.3);
					
					\draw[e:main,color=gray] (0.8,0.8) -- (1.3,1.3);
					\draw[e:main,color=gray] (0.8,-0.8) -- (1.3,-1.3);
					
					\draw[e:main,color=gray] (-0.8,0.8) -- (-1.8,0);
					\draw[e:main,color=gray] (0.8,0.8) -- (1.8,0);
					
					\draw[e:main,color=gray] (-1.8,-1.8) -- (-0.8,-2.6);
					\draw[e:main,color=gray] (1.8,-1.8) -- (0.8,-2.6);
					
					\draw[e:main,color=gray,bend right=30] (-0.8,-1.7) to (1.8,-1.8);
					\draw[e:main,color=gray,bend left=30] (0.8,-1.7) to (-1.8,-1.8);
					
					\draw[e:coloredborder] (0.3,-0.3) -- (0.3,0.3);
					\draw[e:coloredthin,color=BostonUniversityRed] (0.3,-0.3) -- (0.3,0.3);
					
					\draw[e:coloredborder] (-0.3,0.3) -- (-0.3,-0.3);
					\draw[e:coloredthin,color=BostonUniversityRed] (-0.3,0.3) -- (-0.3,-0.3);
					
					\draw[e:coloredborder] (-0.8,0.8) -- (0.8,0.8);
					\draw[e:coloredthin,color=BostonUniversityRed] (-0.8,0.8) -- (0.8,0.8);
					
					\draw[e:coloredborder] (-1.3,-1.3) -- (-1.3,1.3);
					\draw[e:coloredthin,color=BostonUniversityRed]  (-1.3,-1.3) -- (-1.3,1.3);
					
					\draw[e:coloredborder] (1.3,1.3) -- (1.3,-1.3);
					\draw[e:coloredthin,color=BostonUniversityRed] (1.3,1.3) -- (1.3,-1.3);

					\draw[e:coloredborder] (-1.8,0) -- (-1.8,-1.8);
					\draw[e:coloredthin,color=BostonUniversityRed] (-1.8,0) -- (-1.8,-1.8);
					\draw[e:coloredborder] (1.8,0) -- (1.8,-1.8);
					\draw[e:coloredthin,color=BostonUniversityRed] (1.8,0) -- (1.8,-1.8);
					
					\draw[e:coloredborder] (-0.8,-2.6) -- (0.8,-2.6);
					\draw[e:coloredthin,color=BostonUniversityRed] (-0.8,-2.6) -- (0.8,-2.6);
					
					\draw[e:coloredborder] (-0.8,-0.8) -- (-0.8,-1.7);
					\draw[e:coloredthin,color=BostonUniversityRed] (-0.8,-0.8) -- (-0.8,-1.7);
					
					\draw[e:coloredborder] (0.8,-0.8) -- (0.8,-1.7);
					\draw[e:coloredthin,color=BostonUniversityRed] (0.8,-0.8) -- (0.8,-1.7);

				\end{pgfonlayer}
			\end{tikzpicture}
		\end{subfigure}
		\begin{subfigure}{0.24\textwidth}
			\centering
			\begin{tikzpicture}[scale=0.8]
				\pgfdeclarelayer{background}
				\pgfdeclarelayer{foreground}
				\pgfsetlayers{background,main,foreground}
				
				\node () at (0,1.05){$C$};
				\node[gray] () at (-1.5,-2.5){$C'$};

				\node () at (-0.4,-0.65){$a$};
				\node () at (0.4,-0.6){$b$};
				
				\node() at (1.8,-2){$a'$};
				\node() at (-1.8,-2.1){$b'$};
				
				\node[v:mainemptygray] () at (-0.3,-0.3){};
				\node[v:maingray] () at (-0.3,0.3){};
				\node[v:maingray] () at (0.3,-0.3){};
				\node[v:mainemptygray] () at (0.3,0.3){};
				
				\node[v:main] () at (-0.8,-0.8){};
				\node[v:mainempty] () at (-0.8,0.8){};
				\node[v:mainempty] () at (0.8,-0.8){};
				\node[v:main] () at (0.8,0.8){};
				
				\node[v:mainemptygray] () at (-1.3,-1.3){};
				\node[v:maingray] () at (-1.3,1.3){};
				\node[v:maingray] () at (1.3,-1.3){};
				\node[v:mainemptygray] () at (1.3,1.3){};  
				
				\node[v:maingray] () at (-1.8,0){};
				\node[v:mainemptygray] () at (1.8,0){};
				
				\node[v:mainemptygray] () at (-1.8,-1.8){};
				\node[v:maingray] () at (1.8,-1.8){};
				
				\node[v:maingray] () at (-0.8,-2.6){};
				\node[v:mainemptygray] () at (0.8,-2.6){};
				
				\node[v:mainemptygray] () at (-0.8,-1.7){};
				\node () at (-1.1,-1.7){$x$};
				
				\node[v:maingray] () at (0.8,-1.7){};
				\node () at (1.1,-1.75){$y$};
				
				\node[v:maingray] at (-0.3,-2){};
				\node () at (-0.3,-1.7){$u$};
				
				\node[v:mainemptygray] at (0.3,-2){};
				\node () at (0.3,-1.7){$v$};
				
				\begin{pgfonlayer}{background}
					
					\draw[e:marker,myGreen] (-0.8,0.8) -- (-1.3,1.3);
					\draw[e:marker,myGreen] (-1.3,1.3) -- (-1.3,-1.3);
					\draw[e:marker,myGreen] (-1.3,-1.3) -- (1.3,-1.3);
					\draw[e:marker,myGreen] (1.3,-1.3) -- (0.8,-0.8);
					\draw[e:marker,myGreen] (0.8,-0.8) -- (0.8,-1.7);
					\draw[e:marker,myGreen] (0.8,-1.7) -- (0.3,-2);
					\draw[e:marker,myGreen,bend right=20] (0.3,-2) to (1.8,-1.8);

					\draw[e:marker,myLightBlue] (0.8,0.8) -- (0.3,0.3);
					\draw[e:marker,myLightBlue] (0.3,0.3) -- (0.3,-0.3);
					\draw[e:marker,myLightBlue] (0.3,-0.3) -- (-0.3,-0.3);
					\draw[e:marker,myLightBlue] (-0.3,-0.3) -- (-0.8,-0.8);
					\draw[e:marker,myLightBlue] (-0.8,-0.8) -- (-0.8,-1.7);
					\draw[e:marker,myLightBlue] (-0.8,-1.7) -- (-0.3,-2);
					\draw[e:marker, bend left=20,myLightBlue] (-0.3,-2) to (-1.8,-1.8);

					\draw[e:main] (-0.8,-0.8) rectangle (0.8,0.8);
					
					\draw[e:main,color=gray] (-1.3,1.3) -- (1.3,1.3) -- (1.3,-1.3) -- (-1.3,-1.3);
					
					\draw[e:main,color=gray] (0.3,0.3) -- (0.8,0.8);
					\draw[e:main,color=gray] (0.3,-0.3) -- (0.8,-0.8);
					
					\draw[e:main,color=gray] (0.3,0.3) -- (-0.3,0.3) -- (-0.3,-0.3) -- (0.3,-0.3);
					
					\draw[e:main,color=gray] (-0.8,0.8) -- (-1.3,1.3);
					\draw[e:main,color=gray] (-0.8,-0.8) -- (-1.3,-1.3);
					
					\draw[e:main,color=gray] (-0.8,0.8) -- (-1.8,0);
					\draw[e:main,color=gray] (0.8,0.8) -- (1.8,0);
					
					\draw[e:main,color=gray] (-1.8,-1.8) -- (-0.8,-2.6);
					\draw[e:main,color=gray] (1.8,-1.8) -- (0.8,-2.6);
					
					\draw[e:main,color=gray] (-0.8,-2.6) -- (0.8,-2.6);
					
					\draw[e:main,color=gray] (-0.8,-0.8) -- (-0.8,-1.7);
					\draw[e:main,color=gray] (0.8,-0.8) -- (0.8,-1.7);
					
					\draw[e:main,color=gray,densely dashed] (-0.3,-2) -- (0.3,-2);
					
					\draw[e:main,color=gray,bend left=20] (-0.3,-2) to (-1.8,-1.8);
					\draw[e:main,color=gray,bend right=20] (0.3,-2) to (1.8,-1.8);
					
					\draw[e:coloredborder] (-0.3,0.3) -- (-0.8,0.8);
					\draw[e:coloredthin,color=BostonUniversityRed] (-0.3,0.3) -- (-0.8,0.8);
					\draw[e:coloredborder] (-0.3,-0.3) -- (-0.8,-0.8);
					\draw[e:coloredthin,color=BostonUniversityRed] (-0.3,-0.3) -- (-0.8,-0.8);
					
					\draw[e:coloredborder] (0.3,-0.3) -- (0.3,0.3);
					\draw[e:coloredthin,color=BostonUniversityRed] (0.3,-0.3) -- (0.3,0.3);
					
					\draw[e:coloredborder] (-1.3,-1.3) -- (-1.3,1.3);
					\draw[e:coloredthin,color=BostonUniversityRed]  (-1.3,-1.3) -- (-1.3,1.3);
					
					\draw[e:coloredborder] (0.8,0.8) -- (1.3,1.3);
					\draw[e:coloredthin,color=BostonUniversityRed] (0.8,0.8) -- (1.3,1.3);
					\draw[e:coloredborder] (0.8,-0.8) -- (1.3,-1.3);
					\draw[e:coloredthin,color=BostonUniversityRed] (0.8,-0.8) -- (1.3,-1.3);
					
					\draw[e:coloredborder] (-1.8,0) -- (-1.8,-1.8);
					\draw[e:coloredthin,color=BostonUniversityRed] (-1.8,0) -- (-1.8,-1.8);
					\draw[e:coloredborder] (1.8,0) -- (1.8,-1.8);
					\draw[e:coloredthin,color=BostonUniversityRed] (1.8,0) -- (1.8,-1.8);
					
					\draw[e:coloredborder] (-0.8,-1.7) -- (-0.3,-2);
					\draw[e:coloredthin,color=BostonUniversityRed](-0.8,-1.7) -- (-0.3,-2);
					
					\draw[e:coloredborder] (0.8,-1.7) -- (0.3,-2);
					\draw[e:coloredthin,color=BostonUniversityRed](0.8,-1.7) -- (0.3,-2);

				\end{pgfonlayer}
			\end{tikzpicture}
		\end{subfigure}
		
		\caption{The four different ways to obtain a strong matching cross over $C'$ in the second subcase of the first case in the proof of \cref{lemma:easycrossesovercrossingcycles}.}
		\label{fig:pfaffiancase1crosses2}
	\end{figure}
	
	So now let us assume $ab\notin M$.
	Then there exist vertices $x,y\in\V{B}\setminus\Brace{\V{C}\cup\V{C'}}$ such that $ax,by\in M$.
	Let $P$ be an internally $M$-conformal path connecting $x$ to some vertex $a'$ of $\Vi{1}{C'}$ such that $P$ is internally disjoint from $C'$ and avoids $K$.
	Similarly we choose $Q$ to be an internally $M$-conformal path connecting $y$ to some vertex $b'\in\Vi{2}{C'}$ while avoiding $K$ and being internally disjoint from $C'$.
	Suppose one of the two paths contains the initial matching edge of the other.
	Since these cases are symmetric, it suffices to consider one of them, so let us assume $P$ contains $by$.
	Then $Py$ is an internally $M$-conformal path that is disjoint from $C'$ and does not meet $K$ at all.
	Hence $Pyba$ is an $M$-conformal cycle and thus $M\Delta\E{Pyba}$ is a perfect matching of $B_1$ for which $C'+K$ is conformal.
	Since we ruled this possibility out by the first case this cannot happen, and thus $P$ does not contain $by$, and neither does $Q$ contain the edge $ax$.
	By calling upon \cref{lemma:untangletwopaths} again, we find a perfect matching $M'$ and paths $L'$ and $R'$ such that $L'\cap R'$ is either empty or an $M'$-conformal path, $M'$ equals $M$ outside of $axP$ and $byQ$, and $R$ and $L$ connect $\Set{a,b}$ to $\Set{a',b'}$ while being internally disjoint from $C'$ and avoiding $K$.
	If case $L'\cap R'$ is empty let $L\coloneqq L'$ and $R\coloneqq R'$.
	Otherwise, let $u$ and $v$ be the endpoints of $L'\cap R'$ such that $u$ is the first vertex of $R'$ one encounters when traversing $L'$ starting in $a$.
	Then $ab$ together with the unique $a$-$b$-subpath of $L'+R'$ forms an $M'$-conformal cycle $O$.
	By adjusting $M'$ to be the perfect matching $M'\Delta\E{O}$ and setting $L\coloneqq L'uR'$, $R\coloneqq R'vL'$ we have found two disjoint $M'$-alternating paths that can be extended to form a strong matching cross over $C'$ in $B$ by using \cref{lemma:badcrosses} as before.
	Please note that this step might alter the perfect matching $M'$ again with regards to the edges of $C$.
	See \cref{fig:pfaffiancase1crosses2} for an illustration of the cases that might arise.
	
	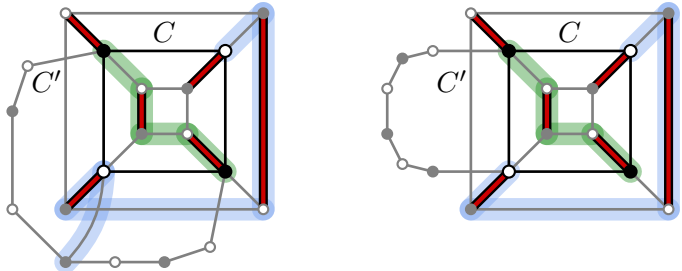
\begin{figure}[h!]
		\centering
		\begin{subfigure}[b]{0.3\textwidth}
			\begin{tikzpicture}
				\pgfdeclarelayer{background}
				\pgfdeclarelayer{foreground}
				\pgfsetlayers{background,main,foreground}
				
				\node[v:ghost] () at (0,-2.2){};
				\node () at (0,1.05){$C$};
				\node () at (-1.55,0.4){$C'$};
				
				\node[v:maingray] () at (-0.3,-0.3){};
				\node[v:mainemptygray] () at (-0.3,0.3){};
				\node[v:mainemptygray] () at (0.3,-0.3){};
				\node[v:maingray] () at (0.3,0.3){};
				
				\node[v:mainempty] () at (-0.8,-0.8){};
				\node[v:main] () at (-0.8,0.8){};
				\node[v:main] () at (0.8,-0.8){};
				\node[v:mainempty] () at (0.8,0.8){};
				
				\node[v:maingray] () at (-1.3,-1.3){};
				\node[v:mainemptygray] () at (-1.3,1.3){};
				\node[v:mainemptygray] () at (1.3,-1.3){};
				\node[v:maingray] () at (1.3,1.3){}; 
				
				\node[v:mainemptygray] () at (0.6,-1.8){};
				\node[v:maingray] () at (-0,-2){};
				\node[v:mainemptygray] () at (-0.65,-2){};
				\node[v:maingray] () at (-1.3,-2){};
				\node[v:mainemptygray] () at (-2,-1.3){};
				\node[v:maingray] () at (-2,0){};
				\node[v:mainemptygray] () at (-1.8,0.6){};
				
				\begin{pgfonlayer}{background}
					
					\draw[e:marker,myLightBlue] (0.8,0.8) -- (1.3,1.3);
					\draw[e:marker,myLightBlue] (1.3,1.3) -- (1.3,-1.3);
					\draw[e:marker,myLightBlue] (1.3,-1.3) -- (-1.3,-1.3);
					\draw[e:marker,myLightBlue] (-1.3,-1.3) -- (-0.8,-0.8);
					\draw[e:marker,bend left=20,myLightBlue] (-0.8,-0.8) to (-1.3,-2);

					\draw[e:marker,myGreen] (-0.8,0.8) -- (-0.3,0.3);
					\draw[e:marker,myGreen] (-0.3,0.3) -- (-0.3,-0.3);
					\draw[e:marker,myGreen] (-0.3,-0.3) -- (0.3,-0.3);
					\draw[e:marker,myGreen] (0.3,-0.3) -- (0.8,-0.8);
					
					\draw[e:main] (-0.8,-0.8) rectangle (0.8,0.8);
					
					\draw[e:main,color=gray]  (1.3,-1.3) -- (-1.3,-1.3) -- (-1.3,1.3) -- (1.3,1.3);
					
					\draw[e:main,color=gray] (-0.3,0.3) -- (0.3,0.3) -- (0.3,-0.3) -- (-0.3,-0.3);
					
					\draw[e:main,color=gray] (-0.3,0.3) -- (-0.8,0.8);
					\draw[e:main,color=gray] (-0.3,-0.3) -- (-0.8,-0.8);
					
					\draw[e:main,color=gray] (0.8,0.8) -- (1.3,1.3);
					\draw[e:main,color=gray] (0.8,-0.8) -- (1.3,-1.3);
					
					\draw[e:main,color=gray] (0.8,-0.8) -- (0.6,-1.8){} -- (-0,-2) -- (-0.65,-2) -- (-1.3,-2){} -- (-2,-1.3){} -- (-2,0) -- (-1.8,0.6) -- (-0.8,0.8);
					
					\draw[e:main,color=gray,bend left=20] (-0.8,-0.8) to (-1.3,-2);
					
					\draw[e:coloredborder] (1.3,1.3) -- (1.3,-1.3);
					\draw[e:coloredthin,color=BostonUniversityRed] (1.3,1.3) -- (1.3,-1.3);
					
					\draw[e:coloredborder] (-0.3,0.3) -- (-0.3,-0.3);
					\draw[e:coloredthin,color=BostonUniversityRed] (-0.3,0.3) -- (-0.3,-0.3);
					
					\draw[e:coloredborder] (0.3,0.3) -- (0.8,0.8);
					\draw[e:coloredthin,color=BostonUniversityRed] (0.3,0.3) -- (0.8,0.8);
					
					\draw[e:coloredborder] (0.3,-0.3) -- (0.8,-0.8);
					\draw[e:coloredthin,color=BostonUniversityRed] (0.3,-0.3) -- (0.8,-0.8);
					
					\draw[e:coloredborder] (-0.8,0.8) -- (-1.3,1.3);
					\draw[e:coloredthin,color=BostonUniversityRed] (-0.8,0.8) -- (-1.3,1.3);
					
					\draw[e:coloredborder] (-0.8,-0.8) -- (-1.3,-1.3);
					\draw[e:coloredthin,color=BostonUniversityRed] (-0.8,-0.8) -- (-1.3,-1.3);

				\end{pgfonlayer}
			\end{tikzpicture}
		\end{subfigure}
		\begin{subfigure}[b]{0.3\textwidth}
			\begin{tikzpicture}
				\pgfdeclarelayer{background}
				\pgfdeclarelayer{foreground}
				\pgfsetlayers{background,main,foreground}
				
				\node[v:ghost] () at (0,-2.2){};

				\node () at (0,1.05){$C$};
				\node () at (-1.55,0.4){$C'$};

				\node[v:maingray] () at (-0.3,-0.3){};
				\node[v:mainemptygray] () at (-0.3,0.3){};
				\node[v:mainemptygray] () at (0.3,-0.3){};
				\node[v:maingray] () at (0.3,0.3){};
				
				\node[v:mainempty] () at (-0.8,-0.8){};
				\node[v:main] () at (-0.8,0.8){};
				\node[v:main] () at (0.8,-0.8){};
				\node[v:mainempty] () at (0.8,0.8){};
				
				\node[v:maingray] () at (-1.3,-1.3){};
				\node[v:mainemptygray] () at (-1.3,1.3){};
				\node[v:mainemptygray] () at (1.3,-1.3){};
				\node[v:maingray] () at (1.3,1.3){}; 
				
				\node[v:mainemptygray] () at (-1.8,0.8){};
				\node[v:maingray] () at (-1.8,-0.8){};
				
				\node[v:maingray] () at (-2.2,0.7){};
				\node[v:mainemptygray] () at (-2.2,-0.7){};
				
				\node[v:mainemptygray] () at (-2.4,0.3){};
				\node[v:maingray] () at (-2.4,-0.3){};
				
				\begin{pgfonlayer}{background}
					
					\draw[e:marker,myLightBlue] (0.8,0.8) -- (1.3,1.3);
					\draw[e:marker,myLightBlue] (1.3,1.3) -- (1.3,-1.3);
					\draw[e:marker,myLightBlue] (1.3,-1.3) -- (-1.3,-1.3);
					\draw[e:marker,myLightBlue] (-1.3,-1.3) -- (-0.8,-0.8);
					
					\draw[e:marker,myGreen] (-0.8,0.8) -- (-0.3,0.3);
					\draw[e:marker,myGreen] (-0.3,0.3) -- (-0.3,-0.3);
					\draw[e:marker,myGreen] (-0.3,-0.3) -- (0.3,-0.3);
					\draw[e:marker,myGreen] (0.3,-0.3) -- (0.8,-0.8);
					
					\draw[e:main] (-0.8,-0.8) rectangle (0.8,0.8);
					
					\draw[e:main,color=gray]  (1.3,-1.3) -- (-1.3,-1.3) -- (-1.3,1.3) -- (1.3,1.3);
					
					\draw[e:main,color=gray] (-0.3,0.3) -- (0.3,0.3) -- (0.3,-0.3) -- (-0.3,-0.3);
					
					\draw[e:main,color=gray] (-0.3,0.3) -- (-0.8,0.8);
					\draw[e:main,color=gray] (-0.3,-0.3) -- (-0.8,-0.8);
					
					\draw[e:main,color=gray] (0.8,0.8) -- (1.3,1.3);
					\draw[e:main,color=gray] (0.8,-0.8) -- (1.3,-1.3);
					
					\draw[e:main,color=gray] (-0.8,0.8) -- (-1.8,0.8);
					\draw[e:main,color=gray] (-0.8,-0.8) -- (-1.8,-0.8);
					
					\draw[e:main,color=gray] (-1.8,0.8) -- (-2.2,0.7);
					\draw[e:main,color=gray] (-1.8,-0.8) -- (-2.2,-0.7);
					
					\draw[e:main,color=gray] (-2.2,-0.7) -- (-2.4,-0.3);
					\draw[e:main,color=gray] (-2.2,0.7) -- (-2.4,0.3);
					
					\draw[e:main,color=gray] (-2.4,-0.3) -- (-2.4,0.3);
					
					\draw[e:coloredborder] (1.3,1.3) -- (1.3,-1.3);
					\draw[e:coloredthin,color=BostonUniversityRed] (1.3,1.3) -- (1.3,-1.3);
					
					\draw[e:coloredborder] (-0.3,0.3) -- (-0.3,-0.3);
					\draw[e:coloredthin,color=BostonUniversityRed] (-0.3,0.3) -- (-0.3,-0.3);
					
					\draw[e:coloredborder] (0.3,0.3) -- (0.8,0.8);
					\draw[e:coloredthin,color=BostonUniversityRed] (0.3,0.3) -- (0.8,0.8);
					
					\draw[e:coloredborder] (0.3,-0.3) -- (0.8,-0.8);
					\draw[e:coloredthin,color=BostonUniversityRed] (0.3,-0.3) -- (0.8,-0.8);
					
					\draw[e:coloredborder] (-0.8,0.8) -- (-1.3,1.3);
					\draw[e:coloredthin,color=BostonUniversityRed] (-0.8,0.8) -- (-1.3,1.3);
					
					\draw[e:coloredborder] (-0.8,-0.8) -- (-1.3,-1.3);
					\draw[e:coloredthin,color=BostonUniversityRed] (-0.8,-0.8) -- (-1.3,-1.3);

				\end{pgfonlayer}
			\end{tikzpicture}
			
		\end{subfigure}
		\caption{Strong matching crosses over $C'$ in the second and third case in the proof of \cref{lemma:easycrossesovercrossingcycles}.}
		\label{fig:pfaffiancase1crosses3}
	\end{figure}
	So now let us assume $K$ to be of length two.
	In this case, there is a unique vertex $u\in\V{C}$ that does not belong to  $C'$.
	For this case, let us choose $M$ to be a perfect matching of $B_1$ for which $C$ is $M$-conformal.
	Since $B$ is a brace, by \cref{thm:bipartiteextendibility}, there exists an internally $M$-conformal path $P$ connecting $v$ to a vertex of $C'$ while avoiding any vertex in $\V{K}$.
	We then use \cref{lemma:badcrosses} to find paths $L'$ and $R$ together with a perfect matching $M'$ of $B$ such that $L\coloneqq L'P$ and $R$ form a strong matching cross over $C'$ as illustrated in \cref{fig:pfaffiancase1crosses3}.
	
	At last, consider the case where $K$ contains all of $C$.
	Here \cref{lemma:badcrosses} yields the strong matching cross over $C'$ in $B$ immediately.
\end{proof}

The above lemma illustrates why we cannot allow to always reduce a brace along a $4$-cycle sum.
In some cases, even the small separator given by the $4$-cycle is enough to provide a matching cross.
The following lemmas aim to make this observation more general and exact.

\begin{lemma}\label{lemma:crossesovercrossingcycles}
	Let $B$ be a $K_{3,3}$-free brace, $\ell\geq 3$, $B_1,\dots,B_{\ell}$ braces such that $B$ is a maximal $4$-cycle sum of $B_1,\dots,B_{\ell}$ at the $4$-cycle $C$ and let $C'$ be a conformal cycle in $B_1$ that also exists in $B$ such that $\Abs{\V{C'}\cap\V{C}}\geq 2$ and $\V{C}\cap\V{C'}$ contains vertices of both colour classes, then there is a matching cross over $C'$ in $B$.
\end{lemma}

\begin{proof}
	We divide this proof into three cases.
	\begin{enumerate}
		\item[\textbf{1:}] $\Abs{\V{C}\cap\V{C'}}=2$ and the vertices in $\V{C}\cap\V{C'}$ are adjacent on $C$.
		\item[\textbf{2:}] $\Abs{\V{C}\cap\V{C'}}=3$.
		\item[\textbf{3:}] $\Abs{\V{C}\cap\V{C'}}=4$.
	\end{enumerate}
	
	\textbf{Case 1:} $\Abs{\V{C}\cap\V{C'}}=2$ and the vertices in $\V{C}\cap\V{C'}$ are adjacent on $C$.
	
	In this case, let $x$ and $y$ be the two adjacent vertices of $C$ belong to $C'$.
	In case $xy\in\E{C'}$, we are done immediately by \cref{lemma:easycrossesovercrossingcycles}.
	If $xy\notin\E{C'}$, then $x$ and $y$ divide $C'$ into two paths of odd length, say $P_1$ and $P_2$, both with endpoints $x$ and $y$.
	If $M$ is a perfect matching of $B$ such that $C'$ is $M$-conformal, then exactly one of the two paths, say $P_1$ is also $M$-conformal and thus $P_1+xy$ is an $M$-conformal cycle as well.
	By \cref{lemma:easycrossesovercrossingcycles} $P_1+xy$ has a strong matching cross and thus, by \cref{lemma:extendingcrosses}, $C'$ must have a matching cross in $B$.
	
	\textbf{Case 2:} $\Abs{\V{C}\cap\V{C'}}=3$.
	
	In case $C\cap C'$ is a subpath of $C'$, we are done immediately by \cref{lemma:easycrossesovercrossingcycles}.
	Hence we may assume that this is not the case.
	Next, suppose $C'$ contains exactly one edge of $C$ and there is $xy\in\E{C}$ such that $x,y\in\V{C'}$, but $xy\notin\E{C'}$.
	Let $z$ be the remaining vertex of $C$ on $C'$, then $x$ and $y$ separate $C'$ into two paths, where one of them, say $P$, does not contain $z$.
	We may choose a perfect matching $M$ of $B$ such that $P$ is internally $M$-conformal.
	Then $K\coloneqq C'-P+xy$ is also an $M$-conformal cycle, and by our assumption, $K\cap C$ is a subpath of $C$.
	Hence we may apply \cref{lemma:easycrossesovercrossingcycles} together with \cref{lemma:extendingcrosses} to obtain a matching cross over $C'$.
	At last assume that $C'$ does not contain an edge of $C$.
	If we call the vertices of $C$ on $C'$ $x$, $y$, and $z$ again such that $x$ and $z$ belong to the same colour class, we again find the path $P$ avoiding $z$ but connecting $x$ and $y$ as before.
	But we also find a path $Q\subseteq C'$ that connects $z$ and $y$ and avoids $x$.
	By choosing a perfect matching $M$ of $B$ such that $C'$ is $M$-conformal and $Q$ is internally $M$-conformal, we have found a perfect matching for which $K'\coloneqq C'-Q+yz$ is an $M$-conformal cycle.
	For this cycle, we find a matching cross as discussed above, and by applying \cref{lemma:extendingcrosses} again, we obtain a matching cross for $C'$ as well.
	
	\textbf{Case 3:} $\Abs{\V{C}\cap\V{C'}}=4$.
	
	Let $C=\Brace{a_1,b_1,a_2,b_2}$.
	If the vertices of $C$ appear on $C'$ in the same order as they do on $C$, we can use \cref{lemma:badcrosses} to find a strong matching cross over $C$ whose paths are internally disjoint from $B_1$.
	Hence we have found a strong matching cross over $C'$ in $B$.
	
	Hence the vertices of $C$ do not appear on $C'$ in the order listed.
	The only way this is possible is, if they appear on $C'$ in the order $a_1$, $a_2$, $b_1$, $b_2$, or $a_1$, $a_2$, $b_2$, $b_1$.
	In both cases we can use \cref{lemma:pathsintrisums} to obtain a perfect matching $M$ of $B$ and internally $M$-conformal paths $P_1$ and $P_2$ which are internally disjoint from $B_1$ such that $P_1$ connects $a_1$ and $b_1$ while $P_2$ connects $a_2$ and $b_2$ in case the order of appearance is $a_1$, $a_2$, $b_1$, $b_2$.
	Otherwise, $P_1$ and $P_2$ may be chosen such that $P_1$ connects $a_1$ and $b_2$ while $P_2$ connects $a_2$ and $b_1$.
	Either way, the two paths form a strong matching cross over $C'$ in $B$.
\end{proof}

Let $B$ be a $K_{3,3}$-free brace, $C$ a $4$-cycle in $B$ such that $B-C$ is not connected, and $C'$ be any conformal cycle in $B$.
Suppose there is a matching cross over $C'$ in $B$ such that at least one of the paths of this cross uses $C$ and extends into another component of $B-C$.
Let $B'$ be a brace with $C'\subseteq B'$ such that $B$ is made from a set of $K_{3,3}$-free braces including $B'$ via a $4$-cycle sum at $C$.
Then some information on the matching cross over $C'$ in $B$ should also exist in $B'$.
Our goal is to use this information to show that even $B'$ cannot be planar while $C'$ bounds a face.

Let $B$ be a $K_{3,3}$-free brace, $C$ a $4$-cycle in $B$, $M$ a perfect matching and $C'$ a conformal cycle in $B$.
A tuple $\Brace{P_1,P_2,Q,M}$ is a \emph{diffuse $C'$-$M$-precross through $C$} if $P_1$, $P_2$ and $Q$ are pairwise disjoint, internally disjoint from $C'$, all are $M$-alternating paths\footnote{In particular, $P_1$ and $P_2$ are allowed to be single vertices of $C'$.}, and for each $i\in[1,2]$, $P_i$ is a $\V{C'}$-$\V{C}$-path such that the endpoints of the $P_i$ on $C'$ belong to different components of $C'-\V{Q}$.
Additionally, we require among the endpoints of the $P_i$ on $C$ at least one of each colour class to be covered by an edge of $\E{C}\cap M$.
A diffuse $C'$-$M$-precross through $C$ is \emph{daring} if the endpoints of the $P_i$ on $C$ belong to the same colour class.

\begin{lemma}\label{lemma:difuseprecrossesintrisums}
	Let $B$ be a $K_{3,3}$-free brace, $\ell\geq 3$, $B_1,\dots,B_{\ell}$ braces such that $B$ is a maximal $4$-cycle sum of $B_1,\dots,B_{\ell}$ at the $4$-cycle $C$ and let $C'$ be a conformal cycle in $B_1$ that also exists in $B$.
	If there is a diffuse $C'$-$M_1$-precross $H$ through $C$ in $B_1$, for some perfect matching $M_1$ of $B_1$, which is not daring, then there are matching crosses over $C'$ in $B_1$ and $B$.
\end{lemma}

\begin{proof}
	Let $H=\Brace{P_1,P_2,Q,M_1}$.
	Since $H$ is not daring, we may apply \cref{lemma:pathsintrisums} to find a perfect matching $M$ of $B$ with $M_1\setminus\E{C}\subseteq M$ together with a path $R$ in $B_j$ for some $j\in[2,\ell]$ such that $P_1RP_2^{-1}$ is $M$-alternating and disjoint from $Q$.
	Note that such a matching exists in particular because either $C$ is $M$-conformal, or the edge connecting the endpoints of the $P_i$ on $C$ must belong to $M$.
	Hence in this case we have found our matching cross over $C'$ in $B$.
	Since $R$ is a subpath of an $M$-alternating path, it itself is $M$-alternating.
	Moreover, with $H$ not being daring, the endpoints of $R$ belong to different colour classes, and thus $R$ is either $M$-conformal or internally $M$-conformal.
	Either way, the endpoints, let us call them $x$ and $y$, are adjacent on $C$ and there exists a perfect matching $M'$ of $B_1$ such that $\E{B_1-\V{C}}\cap M\subseteq M'$ and $xy\in M'$ if and only if $R$ is $M$-conformal.
	Hence $P_1xyP_2^{-1}$ is an $M'$-alternating path which forms, together with $Q$, a matching cross over $C'$ in $B_1$.
\end{proof}

Hence diffuse precrosses can always be extended to actual crosses if they are not daring.
Next we inspect daring precrosses more closely.

Let $B$ be a $K_{3,3}$-free brace, $C$ a $4$-cycle in $B$, $M$ a perfect matching and $C'$ an $M$-conformal cycle in $B$.
A diffuse $C'$-$M$-precross through $C$ $\Brace{P_1,P_2,Q,M}$ is \emph{successful} if it is daring, and either $P_1$ is internally $M$-conformal while $P_2$ is $M$-conformal or $P_1$ and $P_2$ are both of even length, and the endpoint of $P_1$ on $C'$ is covered by an edge of $\E{P_1}\cap M$ if and only if the endpoint of $P_2$ on $C'$ is \textbf{not} covered by an edge of $\E{P_2}\cap M$.

\begin{lemma}\label{lemma:badcrossesanddisjointcycles}
	Let $B$ be a $K_{3,3}$-free brace, $\ell\geq 3$, $B_1,\dots,B_{\ell}$ braces such that $B$ is a maximal $4$-cycle sum of $B_1,\dots,B_{\ell}$ at the $4$-cycle $C$ and let $C'$ be a conformal cycle in $B_1$ that also exists in $B$ such that there is $i\in[1,2]$ with $\V{C}\cap\V{C'}\subseteq V_i$.
	Then there is a matching cross over $C'$ in $B$ if and only if one of the following is true
	\begin{enumerate}
		\item there is a matching cross over $C'$ in $B_1$,
		\item there is a perfect matching $M_1$ of $B_1$ such that there exists a diffuse $C'$-$M_1$-precross through $C$, which is successful.
	\end{enumerate}
\end{lemma}

\begin{proof}
	Let us first prove that in case one of our conditions holds true we find a matching cross over $C'$ in $B$.
	First let $L$ and $R$ be the two paths of a matching cross over $C'$ in $B_1$ and let $M$ be a perfect matching such that $L$ and $R$ are $M$-alternating.
	If $L$ and $R$ also exist in $B$, we are done.
	So let us assume exactly one of them, say $L$, contains an edge of $C$ which does not exist in $B$, note that this is the only possibility why $L$ is not an alternating path in $B$.
	Let us traverse along $L$ starting in one of its endpoints, and let $x$ be the first vertex of $C$ we encounter, while $y$ is the last vertex of $C$ on $L$.
	In case $xy\in\E{C}$, we either have $xy\in M$, or $xy\notin M$, but $xLy$ is $M$-conformal, hence $xLyx$ is an $M$-conformal cycle, and we can adjust $M$ to include the edge $xy$.
	Then, after possibly adjusting the matching as above, $\Brace{Lx,yL,R,M}$ is a diffuse $C'$-$M$-precross through $C$ which is not daring.
	Hence we are done by \cref{lemma:difuseprecrossesintrisums}.
	So we may assume $x$ and $y$ not to be adjacent on $C$ and thus they belong to the same colour class.
	In this case, $xLy$ is an $M$-alternating path of even length and thus can neither be internally $M$-conformal, nor $M$-conformal.
	However, since $L$ contains an edge of $C$, there must be a vertex $z\in\V{C}\setminus\Set{x,y}$ such that one of the edges $xz$, $yz$ belongs to $L$.
	Without loss of generality let us assume $xz\in\E{L}$.
	In case $xz\in M$, $\Brace{Lx,yL,R,M}$ is a successful diffuse $C'$-$M$-precross through $C$ and we are done by \cref{lemma:difuseprecrossesintrisums}.
	So assume $xz\notin M$.
	Then $zLy$ must be an $M$-conformal path, and thus $zLyz$ is an $M$-conformal cycle.
	Hence we may adjust $M$ such that $yz\in M$ and again $\Brace{Lx,yL,R,M}$ is a successful diffuse $C'$-$M$-precross through $C$ and we are done by \cref{lemma:difuseprecrossesintrisums}.
	Hence we may now assume $L$ and $R$ to contain an edge of $C$ each.
	Since $C$ only has four vertices, this means there is a unique edge $x_Ly_L\in\E{L}\cap\E{C}$ and a unique edge $x_Ry_R\in\E{R}\cap\E{C}$.
	Now \cref{lemma:pathsintrisums} provides us with the two paths in $B_2$ that are necessary to be combined with $Lx_L$, $y_LL$ and $Rx_R$, $y_RR$ respectively in order to obtain a strong matching cross over $C'$ in $B$.
	
	So now we have to show that the existence of a matching cross in $B$ implies the existence of one of the two structures above.
	Let $R$ and $L$ be two $M$-alternating paths for some perfect matching $M$ that form a matching cross over $C'$.
	If neither $L$ nor $R$ contains a vertex of $C$, $L$ and $R$ must be completely contained in $B_1$ and thus form a strong matching cross over $C'$ in $B_1$ as well.
	So let us assume that exactly one of the two paths contains a vertex of $C$ and further assume, without loss of generality, that $L$ is that path.
	In case $L$ contains exactly one vertex of $C$, it cannot contain any vertex of $B_i-\V{C}$ for any $i\in[2,\ell]$ since $C$ separates the $B_j$ from each other.
	Hence in this case, $L$ and $R$ again also exist in $B_1$.
	So assume that $L$ contains exactly two vertices of $C$, say $v$ and $w$.
	Then either $vw\in\E{C}$, or $v$ and $w$ belong to the same colour class.
	In the first case, either $vw\in\E{L}$, or $vLwv$ is an $M$-conformal cycle, and we may adjust $M$ such that $LvwL$ is $M$-alternating.
	In any case, after possibly adjusting $M$, $LvwL$ and $R$ form a matching cross over $C'$ in $B_1$ for some perfect matching $M_1$ where $vw\in M_1$ if and only if $vLw$ is $M$-conformal.
	In the second case, $vLw$ must be of even length.
	In case exactly one of $Lv$ and $wL$ is of even length, there must be a vertex $u\in\V{C}\setminus\Set{v,w}$ such that $u$ is not covered by an edge of $M\cap\E{B_1}$, or $vLw=vuw$.
	Then we may choose a perfect matching $M_1$ of $B_1$ such that $M_1\setminus\E{C}\subseteq M$ and  $LvuwL$ is an $M_1$ alternating path of the same type as $L$ in $B_1$.
	Hence there is a matching cross over $C'$ in $B_1$.
	Thus we may assume $Lv$ and $wL$ to either both be of odd or both be of even length.
	In total, this means that $L$ is of even length.
	Hence, by choosing $M_1$ as before, $\Brace{Lv,wL,R,M_1}$ is a successful diffuse $C'$-$M_1$-precross over $C'$ in $B_1$.
	In case $L$ contains more than two vertices of $C$, let $v$ be the first vertex of $C$ on $L$ and $w$ be the last one.
	Let $Q$ be a shortest $v$-$w$-path on $C$, then $Q$ and $vLw$ are of the same parity, and we can choose a perfect matching $M_1$ such that both $R$ and $LvQwL$ are $M_1$ alternating in $B_1$, thus forming a matching cross over $C'$ in $B_1$.
	
	With this, we may now assume that both $L$ and $R$ contain vertices of $C$.
	If $R$ contains exactly one vertex of $C$ and the edge of $M$ covering this vertex belongs to $B_1$, this case can be handled the same way as the cases where $R$ does not contain any vertex of $C$.
	If, on the other hand, $R$ contains exactly one vertex of $C$ and the edge of $M$ covering this vertex does not belong to $B_1$ this means this vertex of $C$ is an endpoint of $R$ and thus belongs to $C'$.
	Let us assume that both $L$ and $R$ contain exactly one vertex of $C$ each and the edges of $M$ covering these vertices do not belong to $B_1$.
	This means that these endpoints of $L$ and $R$ on $C$ must belong to the same colour class by our assumption and thus no edge of $M$ that covers a vertex of $C$ can belong to $B_1$.
	But then we may choose a perfect matching $M_1$ of $B_1$ such that $M\cap\E{B_1}\subseteq M_1$ and $C$ is $M_1$-conformal, and then $L$ and $R$ sill are $M_1$-alternating paths in $B_1$.
	Hence we have found a matching cross over $C'$ in $B_1$.
	Now assume that $L$ contains more than one vertex of $C$, while $R$ still contains exactly one vertex, say $u$, of $C$ for which the edge of $M$ covering it does not belong to $B_1$.
	Let $x$ be the first vertex of $L$ on $C$ and $y$ be the last vertex.
	Then $xLy$ is of even length if and only if $x$ and $y$ belong to the same colour class.
	Suppose this is the case, then for one $z\in\Set{x,y}$ the edge of $M$ covering $z$ cannot belong to $B_1$ since $u$ cannot belong to the same colour class as $x$ and $y$.
	Hence there exists a perfect matching $M_1$ of $B_1$ with $M\cap\E{B_1}\subseteq M_1$ and $zu\in M_1$ and therefore $\Brace{xL,yL,R,M_1}$ is a successful $C'$-$M_1$-precross through $C$ in $B_1$.
	Otherwise, $x$ and $y$ belong to different colour classes and thus $xLy$ is either $M$-conformal or internally $M$-conformal.
	In the first case, $xLyx$ is an $M$-conformal cycle, and we can adjust $M$ such that $xy\in M$, in the second case, $LxyL$ already is an $M$-alternating path.
	Hence there is a perfect matching $M_1$ of $B_1$ such that $R$ and $LxyL$ form a matching cross over $C'$ in $B_1$. 
	Thus we may assume both $L$ and $R$ to contain exactly two vertices of $C$ each.
	Let $v_X$, $w_X$ be the two vertices of $V{C}\cap\V{X}$ for each $X\in\Set{L,R}$.
	In case $v_L$ and $w_L$ are adjacent on $C$, then so are $v_R$ and $w_R$, and we can choose a perfect matching $M_1$ such that $v_Xw_X\in M_1$ if and only if $v_XXw_X$ is an $M$-conformal path and $M_1\setminus\E{C}\subseteq M$.
	Then $Lv_Lw_LL$ and $Rv_Rw_RR$ form a matching cross over $C'$ in $B_1$.
	In case $v_L$ and $w_L$ belong to the same colour class, then so must $v_R$ and $w_R$.
	In this case, $v_LLw_L$ and $v_RRw_R$ form a matching cross over $C$ in $B$ and since these paths are alternating, each of them must contain an edge of $M$.
	Moreover, we can change $M$ to a perfect matching $M'$ such that $M'$ coincides with $M$ on $\E{B_i}\setminus\E{C}$ for all $i\in[2,\ell]$, and $C+v_LLw_L+v_RRw_R$ is $M'$-conformal.
	Indeed this means that $v_LLw_L$ and $v_RRw_R$ form a conformal cross over $C$ in $B$ and thus, by \cref{lemma:goodcrossesmeanK33}  $B$ cannot be $K_{3,3}$-free which contradicts our assumption.
\end{proof}

A graph that plays a huge role in non-planar $K_{3,3}$-free braces that are \textbf{not} the Heawood graph is the \emph{Rotunda}.
The \emph{Rotunda} is the graph obtained by performing the $4$-cycle sum operation on three cubes at a single common $4$-cycle $C$ and then forgetting all edges of $C$.
An important observation is the non-planarity of the Rotunda.

\begin{figure}[h!]
	\centering
	\begin{tikzpicture}[scale=1]
		\pgfdeclarelayer{background}
		\pgfdeclarelayer{foreground}
		\pgfsetlayers{background,main,foreground}
		
		\node[v:mainempty] () at (-0.5,2.5){};
		\node[v:main] () at (0.6,2.5){};
		\node[v:main] () at (-1.2,2){};
		\node[v:mainempty] () at (0,2){};
		
		\node[v:mainempty] () at (-0.5,0.3){};
		\node[v:main] () at (0.5,0.3){};
		\node[v:main] () at (-1,-0.3){};
		\node[v:mainempty] () at (0,-0.3){};
		
		\node[v:mainempty] () at (-0.5,-2){};
		\node[v:main] () at (0.6,-2){};
		\node[v:main] () at (-1.2,-2.5){};
		\node[v:mainempty] () at (-0,-2.5){};
		
		\node[v:main] () at (-0.9,0.8){};
		\node[v:mainempty] () at (1.7,0.8){};
		\node[v:mainempty] () at (-2.25,-0.8){};
		\node[v:main] () at (0.4,-0.8){};

		\begin{pgfonlayer}{background}
			\draw[e:main] (-0.5,2.5) -- (0.6,2.5) -- (0,2) -- (-1.2,2) -- (-0.5,2.5);    
			\draw[e:main] (-0.5,0.3) -- (0.5,0.3) -- (0,-0.3) -- (-1,-0.3) -- (-0.5,0.3);
			\draw[e:main] (-0.5,-2) -- (0.6,-2) -- (0,-2.5) -- (-1.2,-2.5) -- (-0.5,-2);

			\draw[e:main] (-0.5,2.5) -- (-0.9,0.8);
			\draw[e:main] (0.6,2.5) -- (1.7,0.8);
			\draw[e:main] (-1.2,2) -- (-2.25,-0.8);
			\draw[e:main] (0,2) -- (0.4,-0.8);   
			
			\draw[e:main] (-0.9,0.8) -- (-0.5,0.3);
			\draw[e:main] (0.5,0.3) -- (1.7,0.8);
			\draw[e:main] (-1,-0.3) -- (-2.25,-0.8);
			\draw[e:main] (0,-0.3) -- (0.4,-0.8);
			
			\draw[e:main] (-0.9,0.8) -- (-0.5,-2);
			\draw[e:main] (1.7,0.8) -- (0.6,-2);
			\draw[e:main] (-2.25,-0.8) -- (-1.2,-2.5);
			\draw[e:main] (0.4,-0.8) -- (-0,-2.5);

		\end{pgfonlayer}

	\end{tikzpicture}
	\caption{The smallest bipartite and non-planar $K_{3,3}$-free brace that is not isomorphic to the Heawood graph: The Rotunda.}
	\label{fig:rotunda}
\end{figure}
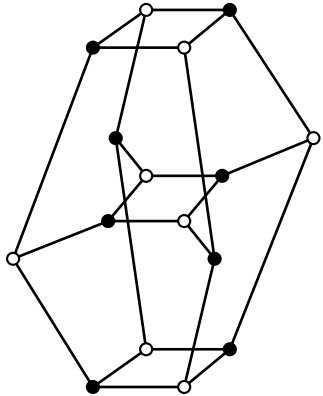

\begin{observation}\label{obs:rotundanonplanar}
	The Rotunda is \textbf{not} planar.
\end{observation}

From \cref{thm:trisums,lemma:cubeorK33} one can derive the following.

\begin{corollary}\label{cor:pfaffiantrisumshaverotunda}
	Let $B,B_1,\dots,B_{\ell}$, $\ell\geq 3$, be braces such that $B$ is $K_{3,3}$-free and a $4$-cycle sum of $B_1,\dots,B_{\ell}$ at a $4$-cycle $C$.
	Then $B$ contains a conformal bisubdivision of the Rotunda.
\end{corollary}

\begin{lemma}\label{lemma:bigintersectionscannotbereduced}
	Let $B,B_1,\dots,B_{\ell}$, $\ell\geq 3$, be braces such that $B$ is $K_{3,3}$-free and a maximal $4$-cycle sum of $B_1,\dots,B_{\ell}$ at the $4$-cycle $C$.
	Moreover, let $C'$ be a conformal cycle in $B_1$ that also exists in $B$.
	If $\V{C}\cap\V{C'}$ contains a vertex of each colour class, every $C'$-reduction of $B$ to some brace $H$ is \textbf{not} planar.
\end{lemma}

\begin{proof}
	Let $B$ be a minimal counterexample to the assertion, that is, the claim holds for every $C'$-reduction of any $C'$-reduction of $B$.
	Let $K$ be a $4$-cycle and $H_1,\dots,H_m$, $m\geq 3$, be braces such that $B$ is a maximal $4$-cycle sum of $H_1,\dots,H_m$  at $K$ with $C'\subseteq H_h$ for some $h\in[1,m]$.
	Moreover, let $K$ be chosen such that $\V{K}\cap\V{C'}$ contains vertices from at most one colour class of $B$, then $C\neq K$ and $C$ contains at least one vertex of $C'$ which does not belong to $K$.
	Hence $\Abs{\V{K}\cap\V{C}}\leq 3$
	
	Let us first observe that for any choice of $Z\in\Set{C,K}$, $B+\E{Z}$ still is a $K_{3,3}$-free brace as this does not change $B$ being a maximal $4$-cycle sum of the braces associated with $Z$ at the $4$-cycle $Z$, and by \cref{thm:trisums} all of these braces are $K_{3,3}$-free.
	Indeed, we claim that $B'\coloneqq B+\E{C}$ is a maximal $4$-cycle sum of $H'_1,\dots,H'_m$, where $H'_i\coloneqq H_i+\CondSet{xy\in\E{C}}{x,y\in\V{H_i}}$, at the $4$-cycle $K$.
	Suppose $K$ does not separate $C$, i.\@e.\@ $\V{C}\setminus\V{K}$ belongs to a unique component of $B'-\V{K}$, there is a unique $i\in[1,m]$ such that $\V{C}\subseteq\V{H_i}$ and in this case our claim holds true.
	Indeed that means if $\Abs{\V{C}\cap\V{K}}\leq 1$ or $\Abs{\V{C}\cap\V{K}}=3$ the claim follows immediately as in those cases $K$ does not separate vertices of $C$.
	So let us assume $\Abs{\V{K}\cap\V{C}}=2$ and $K$ separates $C$.
	Let $\Set{a_1,a_2}=\V{C}\cap\V{K}$, note that the $a_i$ belong to the same colour class, say $V_1$, of $B$.
	Let $\Set{b_1,b_2}\coloneqq\V{C}\setminus\V{K}$ and let $\Set{c_1,c_2}\coloneqq\V{K}\setminus\V{C}$.
	Without loss of generality let us assume $b_1\in\V{H'_1}$ and $b_2\in\V{H'_2}$.
	As $m\geq 3$, there is also some $H'_3$.
	For each $i\in[1,2]$ let $L_i$ be the $4$-cycle in $H'_i$ with vertex set $\Set{a_1,a_2,c_i,b_i}$.
	By \cref{lemma:cubeorK33original,thm:trisums}, since $B'$ is $K_{3,3}$-free, in $H'_1$ there is a conformal bisubdivision $R_1$ of the cube that contains $L_1+a_2c_2$ as a subgraph.
	Similarly, in $H'_2$ there is a conformal bisubdivision $R_2$ of the cube which contains $L_2+a_1c_1$ as a subgraph.
	Moreover, $H'_3$ has a conformal bisubdivision $R_3$ of the cube with $K$ as a subgraph.
	For each $i\in[1,3]$ let $R_i'$ be obtained from $R_i$ by removing all inner vertices of the paths that correspond to a bisubdivided edge of $K$.
	Then let $R\coloneqq R_1'+R_2'+R_3'$.
	By construction $R$ is a conformal subgraph of $B'$ and $C\subseteq R$.
	Careful inspection reveals, that there is a conformal cross over $C$ in $R$, see \cref{fig:thehorror}, and thus, by \cref{lemma:goodcrossesmeanK33} there must be a conformal bisubdivision of $K_{3,3}$ in $B'$.
	As $B'$ is $K_{3,3}$-free, this is a contradiction, and thus $K$ can never separate $C$.
	
	\begin{figure}[h!]
		\centering
		\begin{tikzpicture}[scale=0.7]
			\pgfdeclarelayer{background}
			\pgfdeclarelayer{foreground}
			\pgfsetlayers{background,main,foreground}
			
			\node[v:mainempty] at (-0.75,0.75){};
			\node[v:main] at (0.75,0.75){};
			\node[v:mainempty] at (0.75,-0.75){};
			\node[v:main] at (-0.75,-0.75){};
			
			\node[v:mainempty] at (-5,4.5){};
			\node[] at (-5.4,5.1){$b_{1}$};
			
			\node[v:mainempty] at (5,-4.5){};
			\node[] at (5.4,-5.1){$b_{2}$};

			\node[v:main] at (5,3){};
			\node at (5.5,3.5){$a_{2}$};
			
			\node[v:main] at (-5,-3){};
			\node at (-5.3,-3.5){$a_{1}$};
			
			\node[v:main] at (-4.25,3.42){};
			\node[v:main] at (4.25,-3.42){};
			
			\node[v:main] at (-2.5,1){};
			\node[v:main] at (2.5,-1){};
			
			\node[v:mainempty] at (4.25,1.8){};
			\node[v:mainempty] at (-4.25,-1.8){};
			
			\node[v:mainempty] at (0,2.5){};
			\node[] at (0.2,2.9){$c_{2}$};
			
			\node[v:mainempty] at (0,-2.5){};
			\node[] at (0,-3){$c_{1}$};
			
			\begin{pgfonlayer}{background}
				\draw[e:marker,myLightBlue] (-4.25,3.42) -- (-5,4.5);  
				\draw[e:marker,myLightBlue] (4.25,-3.42) -- (5,-4.5);
				\draw[e:marker,myLightBlue] (-4.25,3.42) -- (-4.25,-1.8);
				\draw[e:marker,myLightBlue,bend left=20] (0,-2.5) to (-2.5,1);
				\draw[e:marker,myLightBlue] (0,-2.5) -- (4.25,-3.42);
				\draw[e:marker,myLightBlue] (-2.5,1) -- (-4.25,-1.8);
				
				\draw[e:marker,myGreen] (-0.75,0.75) -- (-5,-3);
				\draw[e:marker,myGreen] (0.75,-0.75) -- (5,3);
				\draw[e:marker,myGreen] (0.75,-0.75) -- (-0.75,-0.75);
				\draw[e:marker,myGreen] (-0.75,-0.75) -- (-0.75,0.75);
				
				\draw[e:main] (-0.75,-0.75) rectangle (0.75,0.75);      
				\draw[e:main,DarkGoldenrod] (-5,4.5) -- (5,3) -- (5,-4.5) -- (-5,-3) -- (-5,4.5);
				
				\draw[e:main] (-4.25,3.42) -- (-5,4.5);  
				\draw[e:main] (4.25,-3.42) -- (5,-4.5);    
				
				\draw[e:main] (-0.75,0.75) -- (-5,-3);
				\draw[e:main] (0.75,-0.75) -- (5,3);
				
				\draw[e:main] (2.5,-1) -- (5,3);
				\draw[e:main] (-2.5,1) -- (-5,-3);
				
				\draw[e:main] (4.25,-3.42) -- (4.25,1.8);
				\draw[e:main] (-4.25,3.42) -- (-4.25,-1.8);
				
				\draw[e:main] (0,2.5) -- (-4.25,3.42);
				\draw[e:main] (0,2.5) -- (-2.5,1);
				\draw[e:main] (0,2.5) -- (0.75,0.75);
				\draw[e:main,bend left=20] (0,2.5) to (2.5,-1);
				
				\draw[e:main] (0,-2.5) -- (4.25,-3.42);
				\draw[e:main] (0,-2.5) -- (2.5,-1);
				\draw[e:main] (0,-2.5) -- (-0.75,-0.75);
				\draw[e:main,bend left=20] (0,-2.5) to (-2.5,1);
				
			\end{pgfonlayer}
		\end{tikzpicture}	    
		\caption{The graph $R$ from the proof of \cref{lemma:bigintersectionscannotbereduced} with the $4$-cycle $C$ as a subgraph and a conformal cross over $C$.} 
		\label{fig:thehorror}
	\end{figure}
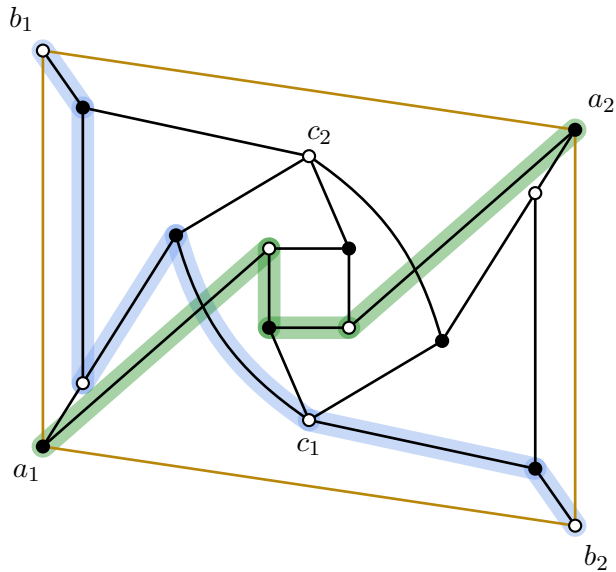
	
	Consequently the graph $B''\coloneqq B'+\E{K}$ is $K_{3,3}$-free and thus, with the same arguments as above, $B''$ is a $4$-cycle sum of $B_1',\dots,B_{\ell}'$ at $C$, where $B'_i\coloneqq B_i+\CondSet{xy\in\E{K}}{x,y\in\V{B_i}}$ for all $i\in[1,\ell]$.
	Indeed, from the discussion above one can derive that there are $i\in[1,\ell]$ and $j\in[1,m]$ such that $\bigcup_{k\in[1,m]\setminus\Set{j}}\V{H_k'}\subseteq\V{B_i}$.
	If $i\neq 1$, then, as $C'\subseteq H_h$, we must have $j=h$ and $H'_h$ still contains all $B_k$ for $k\in[1,\ell]\setminus\Set{i}$, as well as a $C$-reduction of $B_i$.
	So we may assume $i=1$.
	By assumption, we have that $\V{C}\cap\V{C'}$ contains a vertex of each colour class of $B$ and thus, in this case, $\bigcup_{k\in[2,\ell]}\V{B'_k}\subseteq\V{H_h'}$ implying that $H_h'$ is a $4$-cycle sum of at least $3$ braces at the cycle $C$.
	Consequently, by \cref{cor:pfaffiantrisumshaverotunda}, in both cases $H_h'$ contains a conformal bisubdivision of the Rotunda and thus is not planar.
	This contradicts $B$ being a minimal counterexample and thus completes our proof.
\end{proof}

We can now combine the above lemma with our previous observation on precrosses to rule out any planar $C'$-reductions if there exists a diffuse $C'$-$M$-precross though a $4$-cycle $C$ which shares vertices of at most one colour class with $C'$.

\begin{lemma}
	\label{lemma:diffuseprecrossesmeannonplanar}
	Let $B$ be a $K_{3,3}$-free brace, $M$ a perfect matching of $B$, $C$ a $4$-cycle, and $C'$ a conformal cycle for which $\V{C}\cap\V{C'}$ contains vertices from at most one colour class of $B$ such that there exists a diffuse $C'$-$M$-precross through $C$ in $B$, then there does \textbf{not} exist a $C$-reduction of $B$ to a brace $H$ such that $H$ is planar and $C$ bounds a face of $H$.
\end{lemma}

\begin{proof}
	Let $Q$ be the path of our $C'$-$M$-precross through $C$ and let $P_1$, $P_2$ be the two paths connecting $C'$ to $C$ such that their endpoints belong to different components of $C-\V{Q}$.
	
	Suppose there is a $4$-cycle $K$ in $B$ such that $B$ is a maximal $4$-cycle sum of the braces $H_1,\dots,H_m$, $m\geq 3$ at $K$, $H_1$ is a non-trivial $C'$-reduction of $H$, and $\V{C}\subseteq H_1$.
	Then there is a diffuse $C'$-$M'$-precross through $C$ in $H_1$ for some perfect matching $M'$, or there is at most one path $W\in\Set{P_1,P_2,Q}$ such that $\V{W}\cap\bigcup_{i=2}^m\V{H_i}\setminus\V{K}\neq\emptyset$.
	
	Suppose there are two paths $W_1,W_2\in\Set{P_1,P_2,Q}$ such that $\V{W_k}\cap\bigcup_{i=2}^m\V{H_i}\setminus\V{K}\neq\emptyset$ for both $k\in[1,2]$.
	Then let $W'_k$ be the subpath of $W_k$ in $\sum_{i=2}^m H_i$.
	If one of the $W'_k$ is of even length, then so is the other one.
	Indeed, if both are of even length, then each of them must have an edge incident to one of its endpoints in $\sum_{i=2}^m H_i$ that belongs to $M$.
	Moreover, as $H_1$ is a $C'$-reduction and therefore $\V{K}\cap\V{C'}$ contains vertices from at most one colour class of $B$, the edges of $M$ that are incident to the other endpoints of the $W_k'$ must belong to $H_1$.
	Hence there exists $h\in[2,m]$ such that $W_1'$ and $W_2'$ belong to $H_h$ and, in $H_h$ these paths form a conformal cross over $K$.
	Consequently, by \cref{lemma:goodcrossesmeanK33}, $H_h$ has a conformal bisubdivision of $K_{3,3}$ which, by \cref{thm:trisums}, contradicts $B$ being $K_{3,3}$-free.
	Hence the $W_k'$ are of odd length and thus are either internally $M$-conformal or $M$-conformal.
	In either case, for each $k\in[1,2]$ the endpoints $u_k$, $v_k$ of $W_k'$ are adjacent on $K$ and there exists a perfect matching $M_1'$ of $H_1$ with $M_1'\setminus\E{C}\subseteq M$ and $u_kv_k\in M_1'$ if and only if $W_k'$ is $M$-conformal.
	Thus there is a diffuse $C'$-$M'$-precross through $C$ in $H_1$.
	
	If there is a path $W\in\Set{P_1,P_2,Q}$ such that $\V{W}\cap\bigcup_{i=2}^m\V{H_i}\setminus\V{K}\neq\emptyset$, then either the endpoints of $W'$, which is the subpath of $W$ starting on the first vertex of $K$ and ending on the last vertex of $K$ when traversing along $W$, are adjacent on $K$, or at most one vertex of $\V{K}\setminus\V{W}$ belongs to another path from $\Set{P_1,P_2,Q}\setminus\Set{W}$.
	As we have seen above, no path besides $W$ may leave $H_1$ through $K$, hence all edges of the other paths in $\Set{P_1,P_2,Q}\setminus\Set{W}$ belong to $H_1$.
	This is particularly true for the edges of $M$ on these paths.
	Indeed, this means that  all four vertices of $K$ must be matched inside $H_1$ by $M$.
	However, $W'$ is a path of even length and therefore must contain an edge of $M$ that covers one of its endpoints.
	By definition and our assumption that $W'$ contains exactly two vertices of $K$, which are of the same colour, no edge of $W'$ belongs to $H_1$, which is impossible.
	
	A set $S$ of vertices with $\Abs{S\cap V_1}=\Abs{S\cap V_2}=2$ is called \emph{splitting} if there exist braces $L_1,\dots,Lq$, $q\geq 3$, such that $B$ is a $4$-cycle sum of $L_1,\dots,L_q$ at a $4$-cycle with vertex set $S$. 
	Let us call a set $S\subseteq\V{B}$ with $\Abs{S\cap V_1}=\Abs{S\cap V_2}=2$ \emph{well behaved}, if $B$ is a maximal $4$-cycle sum of the braces $H_1,\dots,H_m$, $m\geq 3$ at the $4$-cycle $K'$ with vertex set $S$, $H_1$ is a non-trivial $C'$-reduction of $H$, and $\V{C}\subseteq H_1$, or $S$ is not splitting.
	Let $H'$ be a $C'$-reduction of $B$ such that no splitting set $S$ in $H'$ is well behaved and let $K_1,\dots,K_{\ell}$ be the $4$-cycles used to reduce $B$ to $H'$.
	Let the $K_i$ be numbered in the order in which the $K_i$ were used to construct a non-trivial $C'$-reduction of $B$ to some brace $J_i$ in order to eventually reach $H'$.
	We claim that either $H'$ is non-planar, or $C'$ does not bound a face of $H'$.
	
	Since $B$ is $K_{3,3}$-free and contains a $4$-cycle it cannot be isomorphic to the Heawood graph.
	Suppose $H'$ still has a splitting set, then, by \cref{cor:pfaffiantrisumshaverotunda}, $H'$ is non-planar.
	Hence we may assume $H'$ to be planar for the sake of this claim.
	Next we iteratively construct paths $R_Q^i$, $R_{P_1}^i$, and  $R_{P_2}^i$ such that for each $W\in\Set{P_1,P_2,Q}$, $R_W^i$ is a path in $J_i$ and all three paths are disjoint.
	For each $W\in\Set{P_1,P_2,Q}$ let $R_W^1\coloneqq W$.
	The construction is pretty straight forward.
	Suppose in $J_i$, $i\in[1,\ell-1]$, the subpath of $R_W^i$ starting with its first vertex, $u_W^i$, on $K_{i+1}$ and ending on its last vertex, $v_W^i$, has edges that do not belong to $J_{i+1}$.
	Then either $u_W^i$ and $v_W^i$ are adjacent and we can set $R_W^{i+1}\coloneqq R_W^iu_W^iv_W^iR_W^i$, or they are not adjacent, in which case we have seen that there is a path $U$ of length two on $K_i$ such that $R_W^{i+1}\coloneqq R_W^iu_W^iUv_W^iR_W^i$ is a path and disjoint from the other two paths.
	If $R_W^i$ has no such subpath we simply set $R_W^{i+1}\coloneqq R_W^i$.
	
	Then the paths $R_W^{\ell}$, for $W\in\Set{P_1,P_2,Q}$, are pairwise disjoint paths in $H'$ such that $R_Q^{\ell}$ has both endpoints on $C'$ and $R_{P_1}^{\ell}$ and $R_{P_2}^{\ell}$ connect $C'$ to $C$.
	Moreover, the endpoints of $R_{P_1}^{\ell}$ and $R_{P_2}^{\ell}$ on $C'$ belong to different components of $C'-\V{R_Q^{\ell}}$.
	As all three paths are internally disjoint from $C$, we can now connect $R_{P_1}^{\ell}$ and $R_{P_2}^{\ell}$ on $C$ in order to create an ordinary cross\footnote{Ordinary in this context means that our paths are not necessarily alternating for any perfect matching.}.
	However, this means that $C'$ cannot bound a face of $H'$ by \cref{thm:twopaths}.
	
	To finalise the proof we have to show that, in case $H'$ is non-planar, we still cannot find a planar $C'$-reduction of $H'$ such that $C'$ bounds a face.
	For this note that, by \cref{lemma:bigintersectionscannotbereduced}, no splitting set $S$ can contain vertices of $C'$ from more than one colour class, or otherwise, the claim would follow immediately.
	Observe that every splitting set $S$ in $H'$ in this case must separate $C$ from $C'$.
	Clearly $P_1$ and $P_2$ are separated by $S$.
	
	If also $Q$ is separated by $S$, we have found two disjoint alternating paths that connect $C'$ to $S$ and that belong to a matching cross over $C'$.
	Suppose the two disjoint subpaths of $Q$ that link $C'$ to $S$ both have their endpoints in $S$ in the same colour class.
	Then, if we were to complete this matching cross we would, in particular, obtain a conformal cross over a $4$-cycle with vertex set $S$.
	As $S$ is splitting and $B$ $K_{3,3}$-free, by \cref{lemma:goodcrossesmeanK33} this is impossible.
	Hence, if $C_S$ is the $4$-cycle with vertex set $S$ and $H''$ is the $C'$-reduction of $H'$ at $S$, we either find a matching cross over $C'$, again implying that $H''$ is not planar, as otherwise, we would be done or find a diffuse $C'$-$M'$-precross through $C_S$.
	Hence, in either case, we simply re-enter a previously discussed case and thus our proof is complete.
\end{proof}

\begin{observation}\label{lemma:noreductionsofplanarbraces}
	Let $B$ be a brace and $C$ a conformal cycle in $B$.
	If $B$ is planar there does not exist a $C$-reduction of $B$.
\end{observation}

\begin{proof}
	We prove a stronger result, namely, that a planar brace cannot be a maximal $4$-cycle sum of three or more braces.
	Since $B$ is planar, it is $K_{3,3}$-free and thus does not contain conformal bisubdivisions of $K_{3,3}$.
	Let us assume $B$ is a maximal $4$-cycle sum of the braces $B_1,\dots,B_{\ell}$, $\ell\geq 3$, at the $4$-cycle $C$.
	Then the claim follows immediately from \cref{cor:pfaffiantrisumshaverotunda}.
\end{proof}

With this, everything is in place to prove \Cref{prop:pfaffiancrosses}.

\begin{proof}[Proof of \Cref{prop:pfaffiancrosses}]
	Let $B$ be a minimal counterexample to the forward direction of the assertion.
	So let us assume that $B$ is $K_{3,3}$-free, there is a conformal cycle $C$ in $B$ which has \textbf{no} cross, but every $C$-reduction of $B$ to a brace $H$ is either non-planar or $C$ bounds \textbf{no} face of $H$.
	Indeed, we may assume that $B$ is not isomorphic to the Heawood graph since every conformal cycle here has a matching cross \cref{lemma:heawoodcrosses}.
	However, in every $C$-reduction of $B$, our assertion holds.
	We claim that this means there is \textbf{no} $C$-reduction of $B$.
	
	Suppose there was one and let $H$ be a $C$-reduction of $B$.
	By the minimality of $B$ this means that either $H$ has a $C$-reduction to some brace $H'$ such that $C$ bounds a face of $H'$ or there is a matching cross over $C$ in $H$.
	In the first case, there exists a $C$-reduction of $B$ to $H'$ and thus we have a contradiction to $B$ being a counterexample.
	So we may consider the second case and assume that there is a matching cross over $C$ in $H$.
	Let $B$ be a maximal $4$-cycle sum of the braces $H,B_1,\dots,B_{\ell}$, $\ell\geq 2$ at the $4$-cycle $C'$.
	Then, since there is a matching cross over $C$ in $H$, \cref{lemma:crossesovercrossingcycles,lemma:badcrossesanddisjointcycles} imply that there also must exist a matching cross over $C$ in $B$ which again is a contradiction.
	
	Suppose $B$ is planar.
	By \cref{prop:planarcrosses}, this means that $C$ must either bound a face of $B$ or have a strong matching cross in $B$.
	Since neither is correct by our assumption, $B$ cannot be planar.
	So $B$ is neither planar nor does there exist a $C$-reduction.
	According to \cref{thm:trisums} $B$ must either be isomorphic to the Heawood graph or be a maximal $4$-cycle sum at some $4$-cycle $K$ of $K_{3,3}$-free braces $H_1,\dots,H_m$, $m\geq 3$.
	The first case is impossible by assumption.
	If $\V{K}\cap\V{C}$ contains vertices from at most one colour class of $B$, there would be a $C$-reduction in $B$ which we already ruled out, hence we must have $\Abs{\V{K}\cap\V{C}}\geq 2$ and $\V{K}\cap\V{C}$ contains a vertex of each of the two colour classes.
	Then \cref{lemma:crossesovercrossingcycles} implies the existence of a matching cross over $C$ in $B$.
	So, in either case, we reach a contradiction which means that there is no minimal counterexample and thus our proof is complete.
	
	For the reverse let $B$ be a minimal counterexample to the assertion such that $B$ is Pfaffian, there is a conformal cycle $C$ in $B$ which has a matching cross, but there is a $C$-reduction of $B$ to $H$ such that $H$ is planar and $C$ bounds a face.
	First, suppose $H$ is isomorphic to $B$.
	Then, since $C$ bounds a face of $B$, \cref{prop:planarcrosses} implies that there cannot be a matching cross over $C$ in $B$.
	Consequently, $B$ is non-planar.
	Since there is a $C$-reduction of $B$, $B$ is not the Heawood graph.
	Let $K$ be a $4$-cycle such that $B$ is a $4$-cycle sum of the braces $B_1,\dots,B_{\ell}$ at $K$ where $H$ is a $C$-reduction of $B_1$.
	With $H$ being a $C$-reduction of $B$, this must exist.
	Thus, as $B$ is a minimal counterexample and $H$ is planar such that $C$ bounds a face, there is no matching cross over $C$ in $B_1$.
	
	Since $B_1$ is a $C$-reduction of $B$, $\V{K}\cap\V{C}$ cannot contain vertices from both colour classes of $B$.
	Moreover, with $B$ being Pfaffian, by \cref{thm:trisums}, none of the $B_i$ can have a conformal bisubdivision of $K_{3,3}$.
	Hence, by \cref{lemma:badcrossesanddisjointcycles}, there exists a perfect matching $M_1$ of $B_1$ such that there is a diffuse $C$-$M_1$-precross through $K$ which is daring.
	However, in this case, \cref{lemma:diffuseprecrossesmeannonplanar} tells us that no $C$-reduction of $B_1$ can be planar such that $C$ bounds a face.
	As we assumed $H$ to be a $C$-reduction of $B_1$, this is a contradiction.
\end{proof}

\section{Matching Crosses in Braces containing $K_{3,3}$}\label{sec:nonpfaffian}

With \cref{prop:pfaffiancrosses}, we already have one half of \cref{prop:twopathsinbraces}.
To prove the $K_{3,3}$-containing part of our main result, we essentially need to strengthen \cref{lemma:cubeorK33} in the form of \cref{prop:4cycleK33}.

In light of \cref{lemma:goodcrossesmeanK33}, this means that every $4$-cycle in a $K_{3,3}$-containing brace has a conformal cross.
As a first step, we need to establish that we can always find a perfect matching $M$ in a $K_{3,3}$-containing brace such that a prescribed $4$-cycle $C$ is $M$-conformal and there exists an $M$-conformal bisubdivision of $K_{3,3}$ in $B$.
To do this, we make use of a helpful lemma of McCuaig once more.

Recall the definition of the odd m\"obius ladders.
For a m\"obius ladder $\mathscr{M}_{4k+2}$ with $k\geq 2$ we call an edge $e$ a \emph{rung} if it lies on two $4$-cycles.
The \emph{rungs} of a $\mathscr{M}_{4k+2}$ bisubdivision are the paths that correspond to the bisubdivided rungs of the Möbius ladder.
The \emph{base cycle} of $\mathscr{M}_{4k+2}$ is the Hamiltoncycle $C$ that consists entirely of non-rung edges.
The \emph{base cycle} of a $\mathscr{M}_{4k+2}$ bisubdivision is the cycle that consists entirely of the paths corresponding to the non-rung edges of $\mathscr{M}_{4k+2}$.

\begin{lemma}[\cite{mccuaig2004polya}]\label{lemma:conformalmoebiusstrips}
	Let $B$ be a bipartite graph with a conformal bisubdivision of $K_{3,3}$ and $M$ a perfect matching of $B$, then $B$ contains an $M$-conformal bisubdivision $L$ of $\mathscr{M}_{4k+2}$ for some $k\geq 1$.
	Furthermore, the rungs of $L$ are $M$-conformal in case $k\geq 2$.
\end{lemma}

\begin{lemma}\label{lemma:C4andK33}
	Let $B$ be a $K_{3,3}$-containing brace and $C$ a $4$-cycle in $B$, then there exists a perfect matching $M$ of $B$ such that $C$ is $M$-conformal and $B$ has an $M$-conformal bisubdivision of $K_{3,3}$.
\end{lemma}

\begin{proof}
	Let $M'$ be any perfect matching of $B$ for which $C$ is $M'$-conformal.
	By \cref{lemma:conformalmoebiusstrips} there exists an $M'$-conformal bisubdivision $L$ of $\mathscr{M}_{4k+2}$ for some $k\geq 1$.
	In case $k=1$ we are done, so assume $k\geq 2$.
	Let us choose $M'$ such that $k$ is as small as possible.
	We call a path $P$ in $L$ a \emph{bisubdivided edge} if $P$ corresponds to an edge of $\mathscr{M}_{4k+2}$.
	Note that since $C$ is $M'$-conformal, it contains exactly two edges of $M'$.
	Moreover, since $L$ is $M'$-conformal, if $L$ contains a vertex $x$ of $C$, then it also contains the vertex $y$ of $C$ with $xy\in M$.
	Indeed, if a bisubdivided edge $P$ of $L$ contains a vertex $x$ of $C$, then either $x$ is an endpoint of $P$, or $P$ contains the vertex $y$ of $C$ with $xy\in M$.
	Let $\Set{e_1,e_2}=\E{C}\cap M$ and let $P_1$, $P_2$ be the subdivided edges of $L$ such that $e_j\in\E{P_j}$ if $e_j\in\E{L}$.
	If $e_j\notin \E{L}$ for some $j$, let $P_j$ be chosen arbitrarily.
	In case $P_1=P_2$ let us choose $P_2$ to be any non-$P_1$-rung of $L$ instead.
	We show that there exists a perfect matching $N$ of $L$ such that $M\cap\Brace{\E{P_1}\cup\E{P_2}}\subseteq N$ and $L$ contains an $N$-conformal bisubdivision $L'$ of $\mathscr{M}_{4\Brace{k-1}+2}$ such that $P_1$ and $P_2$ are subdivided edges of $L'$.
	Since $M\coloneqq \Brace{M'\setminus \E{L}}\cup N$ is a perfect matching of $B$ for which $C$ is $M$-conformal, this is a contradiction to the choice of $M'$, and thus we must have had $k=1$ in the first place.
	
	Let $x_j$, $y_j$ be the endpoints of $P_j$.
	
	In case $P_1$ and $P_2$ are both rungs, we may assume that $x_1$, $x_2$, $y_1$, $y_2$ appear on $C'$ in the order listed, and $x_1\in V_1$.
	Then $C'$ is divided into four internally disjoint paths $Q_1,\dots,Q_4$ such that $Q_1$ connects $x_1$ to $x_2$, $Q_2$ connects $x_2$ to $y_1$, and so forth.
	Moreover, every rung of $L$ that has an endpoint on $Q_1$ also has an endpoint of $Q_3$,
	similarly for $Q_2$ and $Q_4$.
	Let us call the number of rungs that are different from $P_1$ and $P_2$ and have an endpoint on $Q_j$, the \emph{length} of $Q_j$.
	With $i\geq 2$ at least one of $Q_1$ and $Q_2$ has length at least two.
	Without loss of generality let us assume this to be true for $Q_1$ and thus also for $Q_3$.
	Let $R_1$ and $R_2$ be two rungs whose endpoints on $Q_1$ are internal vertices of $Q_1$ and consecutive, i.\@e.\@ no other rung has an endpoint on the subpath of $Q_1$ connecting $R_1$ to $R_2$.
	Let $Q_1'$ be this subpath and let $Q_3'$ be the corresponding subpath of $Q_3$.
	Then, since $R_1$ and $R_2$ are $M'$-conformal by \cref{lemma:conformalmoebiusstrips}, $K\coloneqq R_1Q_1'R_2Q_2'$ is an $M'$-conformal cycle.
	Let $N\coloneqq \Brace{\Brace{M\cap\E{L}}\setminus\Brace{\E{K}\setminus M}}\cup \Brace{\Brace{\E{K}}\setminus\Brace{M\cap\E{L}}}$, then $R_1$ and $R_2$ are internally $N$-conformal.
	Let $L'$ be the $N$-conformal subgraph of $L$ obtained by deleting all inner vertices of $R_1$ and $R_2$.
	Then $L'$ is a bisubdivision of $\mathscr{M}_{4\Brace{k-1}+2}$ as required.
	
	Now suppose, without loss of generality, that $P_1$ is not a rung of $L$, but $P_2$ is.
	Then one of $x_1$ and $y_1$ is an endpoint of a bisubdivided edge $Q$ of $L$ such that $Q\neq P_1$ and $Q$ is not a rung, but no rung of $L$ which shares an endpoint of $Q$ is $P_2$.
	Let $R_1$ and $R_2$ be those rungs and let $Q'$ be the subdivided edge of $L$ that connects the other two endpoints of $R_1$ and $R_2$.
	Not $K\coloneqq QR_1Q'R_2$ is an $M'$-conformal cycle that does not contain a vertex of $C$.
	We set $N\coloneqq \Brace{\Brace{M\cap\E{L}}\setminus\Brace{\E{K}\setminus M}}\cup \Brace{\Brace{\E{K}}\setminus\Brace{M\cap\E{L}}}$ and define $L'$ as the subgraph of $L$ we obtain by deleting the inner vertices of $R_1$ and $R_2$.
	Then $L'$ is again a bisubdivision of $\mathscr{M}_{4\Brace{k-1}+2}$ as required.
	
	If both $P_1$ and $P_2$ are subpaths of $C'$ we can again find some vertex $z\in\Set{x_1,x_2,y_2,y_2}$ such that $z$ is an endpoint of a bisubdivided edge $Q$ that is a subpath of $C'$ and different from $P_1$ and $P_2$.
	Let $R_1$ and $R_2$ be the two rungs that share endpoints with $Q$ and let $Q'$ be the bisubdivided edge of $L$ that connects the other two endpoints of $R_1$ and $R_2$.
	Note that $z$ can be chosen such that $Q'$ is also different from $P_1$ and $P_2$.
	We define $K$, $N$, and $L'$ as above, and thus the proof is complete.
\end{proof}

So for every $4$-cycle $C$ in a $K_{3,3}$-containing brace there is a perfect matching $M$ such that $C$ is $M$-conformal and there exists an $M$-conformal bisubdivision of $K_{3,3}$ in $B$.
The next step is to show that we may assume $\V{C}$ to be a subset of the vertices of this bisubdivision.

\begin{lemma}\label{lemma:C4inK33}
	Let $B$ be a $K_{3,3}$-containing brace and $C$ a $4$-cycle in $B$, then there exists a perfect matching $M$ of $B$ such that $C$ is $M$-conformal and there is an $M$-conformal bisubdivision $L$ of $K_{3,3}$ with $\V{C}\subseteq\V{L}$.
\end{lemma}

\begin{proof}
	By \cref{lemma:C4andK33} there exist a perfect matching $M'$ and an $M'$-conformal bisubdivision $L'$ of $K_{3,3}$ in $B$ such that $C$ is $M'$ conformal.
	In case $\V{C}\subseteq\V{L'}$ we are done.
	Next suppose $L'$ contains exactly one of the two edges in $M'\cap\E{C}$, let $xy$ be this edge.
	Then $C$ contains an internally $M'$-conformal path $P$ with $\V{P}=\V{C}$ with endpoints $x$ and $y$.
	Let $M\coloneqq M'\Delta \E{C}$, then $P$ is $M$-conformal and by replacing $xy$ in $L'$ with $P$ we obtain an $M$-conformal bisubdivision $L$ of $K_{3,3}$ as desired.
	So from now on, we may assume $C$ and $L'$ to be vertex disjoint.
	Let $ab\in\E{C}\setminus M'$.
	By using \cref{thm:bipartiteextendibility} and \cref{lemma:untangletwopaths}, we can find two internally $M'$-conformal paths $P_a$ and $P_b$ such that each $P_x$ has $x\in\Set{a,b}$ as an endpoint, has its other endpoint on $L'$ and is otherwise disjoint from $L'$ and $C$.
	Moreover, $P_a$ and $P_b$ are either disjoint, or $P_a\cap P_b$ is an $M'$-conformal path.
	What follows is a case distinction on how $P_a$ and $P_b$ connect $C$ to $L'$.
	For each $x\in\Set{a,b}$ let $s_x$ be the endpoint of $P_x$ on $L'$, and let $U$ be the $M'$-conformal path of length four on $C$ with endpoints $a$ and $b$.
	
	\textbf{Case 1}: $P_a$ and $P_b$ are disjoint and there exists a bisubdivided edge $Q$ of $L'$ containing both $s_a$ and $s_b$.
	
	Since $L'$ is an $M'$-conformal bisubdivision of $K_{3,3}$ and $s_a$ and $s_b$ belong to different colour classes, we can choose $M'$ such that the subpath connecting $s_a$ to $s_b$ on $Q$ is internally $M'$-conformal.
	Then we can simply replace $Q$ by $P_aUP_b$ in order to obtain $L$ as desired.
	
	\textbf{Case 2}: $P_a\cap P_b$ is an $M'$-conformal path, and there exists a bisubdivided edge $Q$ of $L'$ containing both $s_a$ and $s_b$.
	
	Since $L'$ is an $M'$-conformal bisubdivision of $K_{3,3}$ and $s_a$ and $s_b$ belong to different colour classes, we can choose $M'$ such that the subpath $R$ connecting $s_a$ to $s_b$ on $Q$ is internally $M'$-conformal.
	Let $W$ be the internally $M'$-conformal subpath of $P_a+P_b$ with endpoints $s_a$ and $s_b$, then we can replace $Q$ by $W$ in order to obtain a new $M'$-conformal bisubdivision $L''$ of $K_{3,3}$ which meets exactly the requirements of the previous case.
	So by reapplying the arguments from above we can find a conformal bisubdivision $L$ of $K_{3,3}$ as desired.
	
	\textbf{Case 3}: $P_a$ and $P_b$ are disjoint, and there exist bisubdivided edges $Q_a$ and $Q_b$ of $L'$ that share exactly one endpoint such that each $Q_x$ contains $s_x$ for $x\in\Set{a,b}$.
	
	Let $z$ be the common endpoint of $Q_a$ and $Q_b$ and let us assume, without loss of generality, that $z$ belongs to the same colour class as $s_a$.
	We may choose $M'$ such that $Q_a$ is $M'$-conformal.
	Let $R_1$ be the subpath of $Q_a$ that connects the non-$z$-endpoint of $Q_a$ to $s_a$, then $R_1$ is also $M'$-conformal.
	Let $u$ be the non-$z$-endpoint of $Q_b$ and $v$ be the non-$z$-endpoint of the third bisubdivided edge $Q'$ of $L'$ that has $z$ as an endpoint.
	Let $R_2\coloneqq s_aQ_azQ'v$, as well as $R_3$, be the path $P_aUP_bs_bQ_bu$.
	Now $R_2$ and $R_3$ are internally $M'$-conformal and by replacing $Q_a$, $Q_b$ and $Q'$ with $R_1$, $R_2$ and $R_3$ we have found our desired $M'$-conformal bisubdivision of $K_{3,3}$.
	
	\textbf{Case 4}: $P_a\cap P_b$ is an $M'$-conformal path, and there exist bisubdivided edges $Q_a$ and $Q_b$ of $L'$ that share exactly one endpoint such that each $Q_x$ contains $s_x$ for $x\in\Set{a,b}$.
	
	Let $W$ be the internally $M'$-conformal subpath of $P_a+P_b$ with endpoints $s_a$ and $s_b$.
	Let $z$ be the common endpoint of $Q_a$ and $Q_b$ and let us assume, without loss of generality, that $z$ belongs to the same colour class as $s_a$.
	We may choose $M'$ such that $Q_a$ is $M'$-conformal.
	Let $R_1$ be the subpath of $Q_a$ that connects the non-$z$-endpoint of $Q_a$ to $s_a$, then $R_1$ is also $M'$-conformal.
	Let $u$ be the non-$z$-endpoint of $Q_b$ and $v$ be the non-$z$-endpoint of the third bisubdivided edge $Q'$ of $L'$ that has $z$ as an endpoint.
	Let $R_2\coloneqq s_aQ_azQ'v$ as well as $R_3$ be the path $WQ_bu$.
	Now $R_2$ and $R_3$ are internally $M'$-conformal and by replacing $Q_a$, $Q_b$ and $Q'$ with $R_1$, $R_2$ and $R_3$ we have found an $M'$-conformal bisubdivision $L''$ of $K_{3,3}$ together with two disjoint internally $M'$-conformal paths, each linking a vertex of $\Set{a,b}$ to a common bisubdivided edge of $L''$.
	Hence by recurring to the first case, we can finish the argument.
	
	\textbf{Case 5}: $P_a$ and $P_b$ are disjoint and there exist bisubdivided edges $Q_a$ and $Q_b$ of $L'$ that vertex disjoint such that each $Q_x$ contains $s_x$ for $x\in\Set{a,b}$.
	
	Let $Q$ be the unique bisubdivided edge of $L'$ that shares an endpoint, say $v_a$, with $Q_a$ and an endpoint, let us call it $v_b$, with $Q_b$ such that for each $x\in\Set{a,b}$, $v_x$ and $x$ belong to the same colour class of $B$.
	For each $x\in\Set{a,b}$ let $F_x$ be the bisubdivided edge of $L'$ with endpoint $v_x$ that is neither $Q$ nor $Q_x$.
	Moreover, let $R^x_1$ be the subpath of $Q_x$ connecting $s_x$ to the non-$v_x$-endpoint of $Q_x$.
	Since $L'$ is an $M'$-conformal bisubdivision of $K_{3,3}$, we may choose $M'$ such that $Q_a$ and $Q_b$ both are $M'$-conformal.
	Then $R^a_1$ and $R^b_1$ are $M'$-conformal as well.
	For each $x\in\Set{a,b}$ let $R^x_2\coloneqq s_xQ_xv_xF_x$ and let $R$ be the path $P_aUP_b$.
	Then let $L$ be the graph obtained from $L'$ by replacing $Q_a,F_a$, $Q$, $Q_b$, and $F_b$ with the $R_i^x$, $i\in[1,2]$, $x\in\Set{a,b}$, and $R$.
	It is straight forward to check that $L$ is an $M'$-conformal bisubdivision of $K_{3,3}$ as required by the assertion.
	
	\textbf{Case 6}: $P_a\cap P_b$ is an $M'$-conformal path and there exist bisubdivided edges $Q_a$ and $Q_b$ of $L'$ that vertex disjoint such that each $Q_x$ contains $s_x$ for $x\in\Set{a,b}$.
	
	As before with the even numbered cases let $W$ be the internally $M'$-conformal subpath of $P_a+P_b$ with endpoints $s_a$ and $s_b$.
	We then repeat the construction from Case 5 in order to obtain an $M'$-conformal bisubdivision $L''$ of $K_{3,3}$ together with two disjoint internally $M'$-conformal paths that meet the requirements of the first case.
	By reapplying the arguments of the first case, we finally obtain $L$ as desired, and thus our proof is complete.
\end{proof}

Having established \cref{lemma:C4inK33}, the next step on the agenda is to analyse how the edges of $\E{C}\cap M$ can occur in the $M$-conformal bisubdivision $L$ of $K_{3,3}$.
The goal is to identify the cases where we immediately find a conformal cross over $C$, and which cases cannot occur in the first place.

Let $B$ be a $K_{3,3}$-containing brace, $C$ a $4$-cycle in $B$ and $M$ a perfect matching of $B$ such that there exists an $M$-conformal bisubdivision $L$ of $K_{3,3}$ in $B$ for which $\V{C}\subseteq\V{L}$.
Let $\Set{ab,a'b'}=\E{C}\cap M$ such that $a,a'\in V_1$ and let $P,Q$ be two odd length $M$-alternating paths where each $X\in\Set{P,Q}$ has endpoints $a_X,b_X$ such that $a_X\in V_1$.
We say that $e\in\Set{ab,a'b'}$ \emph{occurs} on $P$, if $e\in\E{P}$ and it \emph{occurs in reverse} on $P$ if $P-e$ consists of two paths of even length.
Please note that in case both of $ab$ and $a'b'$ occur on $P$ such that exactly one of them occurs in reverse, then no perfect matching $M'$ for which $P$ is $M'$-conformal or internally $M'$-conformal can contain both edges.
To see this simply observe that an edge occurring in reverse on an $M$-alternating path $P$ belongs to $M$ if and only if $P$ is $M$-conformal.

\begin{observation}\label{obs:homogenousoccurence}
	Let $B$ be a $K_{3,3}$-containing brace, $C$ a $4$-cycle in $B$ and $M$ a perfect matching of $B$ such that there exists an $M$-conformal bisubdivision $L$ of $K_{3,3}$ in $B$ for which $\V{C}\subseteq\V{L}$.
	Let $\Set{ab,a'b'}=\E{C}\cap M$ such that $a,a'\in V_1$ and let $P$ be a bisubdivided edge of $L$.
	If both $ab$ and $a'b'$ occur on $P$, then either both or none of them occurs in reverse.
\end{observation}

If $P$ and $Q$ have a common endpoint $z$ and are otherwise disjoint, we say that $ab$ and $a'b'$ are \emph{split} over $P$ and $Q$ if exactly one of $ab$ and $a'b'$ occurs on $P$ and the other one occurs on $Q$.
They are said to be \emph{split nicely} if the shortest $\Set{a,b}$-$\Set{a',b'}$-subpath $R$ of $PzQ$ has even length, and $z\in\V{R}$ does not share the colour of the endpoints of $R$.
In case $R$ is even, and $z$ belongs to the same colour class as the two endpoints of $R$ we say $ab$ and $a'b'$ are \emph{split completely} over $P$ and $Q$.
Please note that, if $ab$ and $a'b'$ are split completely over $P$ and $Q$, then each of the two edges occurs in reverse on its respective path.
Moreover, by the discussion above, $P$ and $Q$ would need both be $M$-conformal in order to guarantee $ab,a'b'\in M$.
Hence they can never be split completely over two bisubdivided edges of $L$ that share an endpoint.

\begin{observation}\label{obs:splitcompletely}
	Let $B$ be a $K_{3,3}$-containing brace, $C$ a $4$-cycle in $B$, and $M$ a perfect matching of $B$ such that there exists an $M$-conformal bisubdivision $L$ of $K_{3,3}$ in $B$ for which $\V{C}\subseteq\V{L}$.
	Let $\Set{ab,a'b'}=\E{C}\cap M$ such that $a,a'\in V_1$ and let $P,Q$ be two bisubdivided edges of $L$ sharing a single endpoint.
	If $ab$ and $a'b'$ are split over $P$ and $Q$, then they are not completely split.
\end{observation}

\begin{lemma}\label{lemma:easycrossesinK33}
	Let $B$ be a $K_{3,3}$-containing brace, $C$ a $4$-cycle in $B$ and $M$ a perfect matching of $B$ such that there exists an $M$-conformal bisubdivision $L$ of $K_{3,3}$ in $B$ for which $\V{C}\subseteq\V{L}$.
	Let $\Set{ab,a'b'}=\E{C}\cap M$ such that $a,a'\in V_1$ and let $P,Q$ be two bisubdivided edges of $L$ such that $ab$ occurs on $P$ and $a'b'$ occurs on $Q$.
	If $P$ and $Q$ are disjoint, or $ab$ and $a'b'$ are split nicely over $P$ and $Q$, there exist paths $R_1$ and $R_2$ in $L$ that form a conformal cross over $C$.
\end{lemma}

\begin{figure}[h!]
	\centering
	\begin{subfigure}{0.25\textwidth}
		\centering
		\begin{tikzpicture}
			\pgfdeclarelayer{background}
			\pgfdeclarelayer{foreground}
			\pgfsetlayers{background,main,foreground}
			
			\foreach \x in {0,2,4}
			{
				\node[v:mainempty] () at (\x*60:12mm){};
			}
			\foreach \x in {1,3,5}
			{
				\node[v:main] () at (\x*60:12mm){};
			}
			
			\node[v:mainemptygreen] () at (80:10.7mm){};
			\node () at (80:14mm){$b$};
			\node[v:maingreen] () at (100:10.7mm){};
			\node () at (100:13.5mm){$a$};
			\node[v:maingreen] () at (140:10.7mm){};
			\node () at (140:14mm){$a'$};
			\node[v:mainemptygreen] () at (160:10.7mm){};
			\node () at (160:13.5mm){$b'$};
			
			\node[blue] () at (90:7.8mm){$C$};

			\begin{pgfonlayer}{background}
				
				\draw[e:marker,myLightBlue] (80:10.7mm) -- (60:12mm); 
				\draw[e:marker,myLightBlue] (60:12mm) -- (240:12mm);
				\draw[e:marker,myLightBlue] (240:12mm) -- (180:12mm);
				\draw[e:marker,myLightBlue] (180:12mm) -- (160:10.7mm);
				
				\draw[e:marker,myGreen] (140:10.7mm) -- (120:12mm);
				\draw[e:marker,myGreen] (120:12mm) -- (100:10.7mm);
				
				\draw[e:main] (60:12mm) -- (120:12mm);
				
				\draw[e:main,gray] (0:12mm) -- (60:12mm);
				\draw[e:main,gray] (240:12mm) -- (300:12mm);
				\draw[e:main,gray] (180:12mm) -- (240:12mm);
				
				\draw[e:main,gray] (300:12mm) -- (120:12mm);
				
				\draw[e:main,gray] (0:12mm) -- (180:12mm);
				
				\draw[e:main,gray] (140:10.7mm) -- (160:10.7mm);
				
				\draw[e:main,gray] (60:12mm) -- (80:10.7mm);
				
				\draw[e:main,gray] (120:12mm) -- (100:10.7mm);
				
				\draw[e:main,blue,densely dashed] (140:10.7mm) -- (80:10.7mm);
				
				\draw[e:main,blue,densely dashed] (160:10.7mm) -- (100:10.7mm);

				\draw[e:coloredborder] (300:12mm) -- (0:12mm);
				\draw[e:coloredthin,color=BostonUniversityRed] (300:12mm) -- (0:12mm);
				
				\draw[e:coloredborder] (240:12mm) -- (60:12mm);
				\draw[e:coloredthin,color=BostonUniversityRed] (240:12mm) -- (60:12mm);
				
				\draw[e:coloredborder] (80:10.7mm) -- (100:10.7mm);
				\draw[e:coloredthin,color=BostonUniversityRed] (80:10.7mm) -- (100:10.7mm);
				
				\draw[e:coloredborder] (120:12mm) -- (140:10.7mm);
				\draw[e:coloredthin,color=BostonUniversityRed] (120:12mm) -- (140:10.7mm);
				
				\draw[e:coloredborder] (180:12mm) -- (160:10.7mm);
				\draw[e:coloredthin,color=BostonUniversityRed] (180:12mm) -- (160:10.7mm);

			\end{pgfonlayer}
		\end{tikzpicture}
	\end{subfigure}
	\begin{subfigure}{0.24\textwidth}
		\centering
		\begin{tikzpicture}
			\pgfdeclarelayer{background}
			\pgfdeclarelayer{foreground}
			\pgfsetlayers{background,main,foreground}
			
			\foreach \x in {0,2,4}
			{
				\node[v:mainempty] () at (\x*60:12mm){};
			}
			\foreach \x in {1,3,5}
			{
				\node[v:main] () at (\x*60:12mm){};
			}
			
			\node[v:mainemptygreen] () at (80:10.7mm){};
			\node () at (80:14mm){$b$};
			\node[v:maingreen] () at (100:10.7mm){};
			\node () at (100:13.5mm){$a$};
			
			\node[v:maingreen] () at (0:4mm){};
			\node () at (0.5,-0.2){$a'$};
			\node[v:mainemptygreen] () at (180:4mm){};
			\node () at (-0.5,-0.25){$b'$};
			
			\node[blue] () at (90:7.8mm){$C$};

			\begin{pgfonlayer}{background}
				
				\draw[e:marker,myLightBlue] (80:10.7mm) -- (60:12mm); 
				\draw[e:marker,myLightBlue] (60:12mm) -- (240:12mm);
				\draw[e:marker,myLightBlue] (240:12mm) -- (180:12mm);
				\draw[e:marker,myLightBlue] (180:12mm) -- (180:4mm);    
				
				\draw[e:marker,myGreen] (100:10.7mm) -- (120:12mm);
				\draw[e:marker,myGreen] (120:12mm) -- (300:12mm);
				\draw[e:marker,myGreen] (0:12mm) -- (300:12mm);
				\draw[e:marker,myGreen] (0:12mm) -- (0:4mm);
				
				\draw[e:main,gray] (0:12mm) -- (60:12mm);
				\draw[e:main,gray] (240:12mm) -- (300:12mm);
				\draw[e:main,gray] (180:12mm) -- (240:12mm);
				
				\draw[e:main,gray] (300:12mm) -- (0:12mm);
				
				\draw[e:main,gray] (120:12mm) -- (180:12mm);
				
				\draw[e:main,gray] (60:12mm) -- (80:10.7mm);
				
				\draw[e:main,gray] (120:12mm) -- (100:10.7mm);
				
				\draw[e:main,gray] (0:4mm) -- (180:4mm);
				
				\draw[e:main,blue,densely dashed] (100:10.7mm) -- (180:4mm);
				\draw[e:main,blue,densely dashed] (80:10.7mm) -- (0:4mm);

				\draw[e:coloredborder] (300:12mm) -- (120:12mm);
				\draw[e:coloredthin,color=BostonUniversityRed] (300:12mm) -- (120:12mm);
				
				\draw[e:coloredborder] (240:12mm) -- (60:12mm);
				\draw[e:coloredthin,color=BostonUniversityRed] (240:12mm) -- (60:12mm);
				
				\draw[e:coloredborder] (80:10.7mm) -- (100:10.7mm);
				\draw[e:coloredthin,color=BostonUniversityRed] (80:10.7mm) -- (100:10.7mm);
				
				\draw[e:coloredborder] (0:12mm) -- (0:4mm);
				\draw[e:coloredthin,color=BostonUniversityRed] (0:12mm) -- (0:4mm);
				
				\draw[e:coloredborder] (180:12mm) -- (180:4mm);
				\draw[e:coloredthin,color=BostonUniversityRed] (180:12mm) -- (180:4mm);

			\end{pgfonlayer}
		\end{tikzpicture}
	\end{subfigure}
	\begin{subfigure}{0.24\textwidth}
		\centering
		\begin{tikzpicture}
			\pgfdeclarelayer{background}
			\pgfdeclarelayer{foreground}
			\pgfsetlayers{background,main,foreground}
			
			\foreach \x in {0,2,4}
			{
				\node[v:mainempty] () at (\x*60:12mm){};
			}
			\foreach \x in {1,3,5}
			{
				\node[v:main] () at (\x*60:12mm){};
			}
			
			\node[v:mainemptygreen] () at (80:10.7mm){};
			\node () at (80:14mm){$b$};
			\node[v:maingreen] () at (100:10.7mm){};
			\node () at (100:13.5mm){$a$};
			
			\node[v:mainemptygreen] () at (0:4mm){};
			\node () at (0.5,-0.25){$b'$};
			\node[v:maingreen] () at (180:4mm){};
			\node () at (-0.5,-0.25){$a'$};
			
			\node[v:maingray] () at (0:8mm){};
			\node[v:mainemptygray] () at (180:8mm){};
			
			\node[blue] () at (0.25,0.8){$C$};
			
			\begin{pgfonlayer}{background}
				
				\draw[e:marker,myGreen] (100:10.7mm) -- (120:12mm);
				\draw[e:marker,myGreen] (180:12mm) -- (120:12mm);
				\draw[e:marker,myGreen] (180:12mm) -- (180:4mm);
				
				\draw[e:marker,myLightBlue] (80:10.7mm) -- (60:12mm);
				\draw[e:marker,myLightBlue] (60:12mm) -- (0:12mm);
				\draw[e:marker,myLightBlue] (0:12mm) -- (0:4mm);
				
				\draw[e:main,gray] (180:12mm) -- (240:12mm);
				
				\draw[e:main,gray] (300:12mm) -- (0:12mm);
				
				\draw[e:main,gray] (300:12mm) -- (120:12mm);
				
				\draw[e:main,gray] (240:12mm) -- (60:12mm);
				
				\draw[e:main,gray] (60:12mm) -- (80:10.7mm);
				
				\draw[e:main,gray] (120:12mm) -- (100:10.7mm);
				
				\draw[e:main,gray] (0:4mm) -- (180:4mm);
				
				\draw[e:main,gray] (0:8mm) -- (0:12mm);
				
				\draw[e:main,gray] (180:8mm) -- (180:12mm);   
				
				\draw[e:main,blue,densely dashed] (100:10.7mm) -- (0:4mm);
				\draw[e:main,blue,densely dashed] (80:10.7mm) -- (180:4mm);
				
				\draw[e:coloredborder] (0:12mm) -- (60:12mm);
				\draw[e:coloredthin,color=BostonUniversityRed] (0:12mm) -- (60:12mm);
				
				\draw[e:coloredborder] (240:12mm) -- (300:12mm);
				\draw[e:coloredthin,color=BostonUniversityRed] (240:12mm) -- (300:12mm);
				
				\draw[e:coloredborder] (120:12mm) -- (180:12mm);
				\draw[e:coloredthin,color=BostonUniversityRed] (120:12mm) -- (180:12mm);
				
				\draw[e:coloredborder] (80:10.7mm) -- (100:10.7mm);
				\draw[e:coloredthin,color=BostonUniversityRed] (80:10.7mm) -- (100:10.7mm);
				
				\draw[e:coloredborder] (0:4mm) -- (0:8mm);
				\draw[e:coloredthin,color=BostonUniversityRed] (0:4mm) -- (0:8mm);
				
				\draw[e:coloredborder] (180:4mm) -- (180:8mm);
				\draw[e:coloredthin,color=BostonUniversityRed] (180:4mm) -- (180:8mm);

			\end{pgfonlayer}
		\end{tikzpicture}
	\end{subfigure}
	\begin{subfigure}{0.24\textwidth}
		\centering
		\begin{tikzpicture}
			\pgfdeclarelayer{background}
			\pgfdeclarelayer{foreground}
			\pgfsetlayers{background,main,foreground}
			
			\foreach \x in {0,2,4}
			{
				\node[v:mainempty] () at (\x*60:12mm){};
			}
			\foreach \x in {1,3,5}
			{
				\node[v:main] () at (\x*60:12mm){};
			}
			
			\node[v:maingreen] () at (80:10.7mm){};
			\node () at (80:13.5mm){$a$};
			\node[v:mainemptygreen] () at (100:10.7mm){};
			\node () at (100:14mm){$b$};
			
			\node[v:mainemptygreen] () at (0:4mm){};
			\node () at (0.5,-0.25){$b'$};
			\node[v:maingreen] () at (180:4mm){};
			\node () at (-0.5,-0.25){$a'$};
			
			\node[v:maingray] () at (0:8mm){};
			\node[v:mainemptygray] () at (180:8mm){};
			
			\node[v:maingray] () at (110:11.2mm){};
			\node[v:mainemptygray] () at (70:11.2mm){};
			
			\node[blue] () at (90:7.8mm){$C$};
			
			\begin{pgfonlayer}{background}
				
				\draw[e:marker,myGreen] (80:10.7mm) -- (60:12mm); 
				\draw[e:marker,myGreen] (60:12mm) -- (240:12mm);
				\draw[e:marker,myGreen] (240:12mm) -- (180:12mm);
				\draw[e:marker,myGreen] (180:12mm) -- (180:4mm);    
				
				\draw[e:marker,myLightBlue] (100:10.7mm) -- (120:12mm);
				\draw[e:marker,myLightBlue] (120:12mm) -- (300:12mm);
				\draw[e:marker,myLightBlue] (0:12mm) -- (300:12mm);
				\draw[e:marker,myLightBlue] (0:12mm) -- (0:4mm);
				
				\draw[e:main,gray] (0:12mm) -- (60:12mm);
				\draw[e:main,gray] (240:12mm) -- (300:12mm);
				
				\draw[e:main,gray] (300:12mm) -- (120:12mm);
				
				\draw[e:main,gray] (240:12mm) -- (60:12mm);
				
				\draw[e:main,gray] (120:12mm) -- (180:12mm);
				
				\draw[e:main,gray] (0:4mm) -- (180:4mm);
				
				\draw[e:main,gray] (0:8mm) -- (0:12mm);
				
				\draw[e:main,gray] (180:8mm) -- (180:12mm);
				
				\draw[e:main,gray] (110:11.2mm) -- (100:10.7mm);
				\draw[e:main,gray] (70:11.2mm) -- (80:10.7mm);
				
				\draw[e:main,blue,densely dashed] (100:10.7mm) -- (180:4mm);
				\draw[e:main,blue,densely dashed] (80:10.7mm) -- (0:4mm);

				\draw[e:coloredborder] (180:12mm) -- (240:12mm);
				\draw[e:coloredthin,color=BostonUniversityRed] (180:12mm) -- (240:12mm);
				
				\draw[e:coloredborder] (300:12mm) -- (0:12mm);
				\draw[e:coloredthin,color=BostonUniversityRed] (300:12mm) -- (0:12mm);
				
				\draw[e:coloredborder] (80:10.7mm) -- (100:10.7mm);
				\draw[e:coloredthin,color=BostonUniversityRed] (80:10.7mm) -- (100:10.7mm);
				
				\draw[e:coloredborder] (0:4mm) -- (0:8mm);
				\draw[e:coloredthin,color=BostonUniversityRed] (0:4mm) -- (0:8mm);
				
				\draw[e:coloredborder] (180:4mm) -- (180:8mm);
				\draw[e:coloredthin,color=BostonUniversityRed] (180:4mm) -- (180:8mm);
				
				\draw[e:coloredborder] (120:12mm) -- (110:11.2mm);
				\draw[e:coloredthin,color=BostonUniversityRed] (120:12mm) -- (110:11.2mm);
				
				\draw[e:coloredborder] (60:12mm) -- (70:11.2mm);
				\draw[e:coloredthin,color=BostonUniversityRed] (60:12mm) -- (70:11.2mm);

			\end{pgfonlayer}
		\end{tikzpicture}
	\end{subfigure}
	
	\caption{The conformal crosses over the $4$-cycle $C$ in a bisubdivision of $K_{3,3}$ from \cref{lemma:easycrossesinK33}.} 
	\label{fig:easycrosses}
\end{figure}
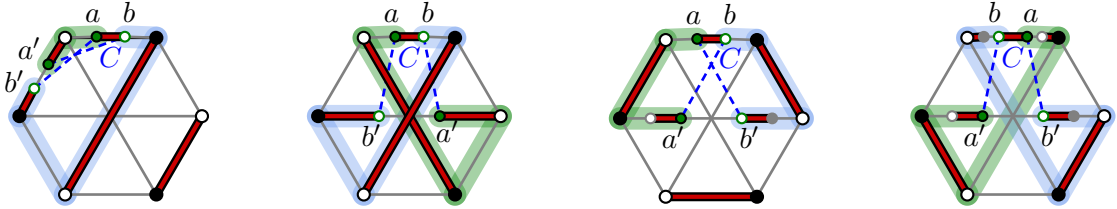

\begin{proof}
	The proof is essentially another case distinction over the following cases:
	\begin{enumerate}
		\item $ab$ and $a'b'$ are nicely split over $P$ and $Q$,
		\item $P$ and $Q$ are disjoint and neither $ab$ nor $a'b'$ occurs in reverse on its respective path,
		\item $P$ and $Q$ are disjoint and, without loss of generality, $a'b'$ occurs in reverse on $Q$, and
		\item $P$ and $Q$ are disjoint and both, $ab$ and $a'b'$, occur in reverse on their respective path.
	\end{enumerate}
	The perfect matchings of $L$ together with the paths $R_1$ and $R_2$ are illustrated in \cref{fig:easycrosses}.
	Please note that the copies of $K_{3,3}$ depicted in the figure are in fact bisubdivisions.
	Where necessary, additional subdivision vertices are drawn, but in general, the edges depicted in a light grey may be subdivided an arbitrary, but even, number of times.
	If in order to depict the respective perfect matching, a bisubdivided edge is marked and bold, this means that the respective path is $M$-conformal, while an unmarked bisubdivided edge represents an internally $M$-conformal path.
\end{proof}

\begin{figure}[h!]
	\centering
	\begin{subfigure}[b]{0.35\textwidth}
		\centering
		\begin{tikzpicture}
			\pgfdeclarelayer{background}
			\pgfdeclarelayer{foreground}
			\pgfsetlayers{background,main,foreground}
			
			\node () at (0,0){ordered};
			
			\node[blue] () at (0,0.8){$C$};
			\node[blue] () at (0,-2.2){$C$};

			\node[v:mainempty] () at (1.5,1.5){};
			\node[v:main] () at (-1.5,1.5){};
			\node[v:main] () at (0.8,0.5){};
			\node[v:mainempty] () at (-0.8,0.5){};
			\node[v:main] () at (0.8,2.5){};
			\node[v:mainempty] () at (-0.8,2.5){};
			
			\node[v:mainempty] () at (1.5,-1.5){};
			\node[v:main] () at (-1.5,-1.5){};
			\node[v:main] () at (0.8,-0.5){};
			\node[v:mainempty] () at (-0.8,-0.5){};
			\node[v:main] () at (0.8,-2.5){};
			\node[v:mainempty] () at (-0.8,-2.5){};
			
			\node[v:maingreen] () at (-0.5,1.5){};
			\node[myGreen] () at (-0.5,1.7){$a$};
			\node[v:mainemptygreen] () at (0.5,1.5){};
			\node[myGreen] () at (0.55,1.75){$b'$};
			\node[v:maingreen] () at (1,1.5){};
			\node[myGreen] () at (1.05,1.75){$a'$};
			\node[v:mainemptygreen] () at (-1,1.5){};
			\node[myGreen] () at (-1,1.75){$b$};
			
			\node[v:maingreen] () at (-0.5,-1.5){};
			\node[myGreen] () at (-0.5,-1.3){$a$};
			\node[v:mainemptygreen] () at (0.5,-1.5){};
			\node[myGreen] () at (0.55,-1.25){$b'$};
			\node[v:maingreen] () at (1,-1.5){};
			\node[myGreen] () at (1.05,-1.25){$a'$};
			\node[v:mainemptygreen] () at (-1,-1.5){};
			\node[myGreen] () at (-1,-1.25){$b$};
			
			\begin{pgfonlayer}{background}
				
				\draw[e:marker,BrightUbe,bend right=55] (-0.5,1.5) to (0.5,1.5);
				
				\draw[e:main] (-0.8,2.5) -- (0.8,0.5);
				\draw[e:main] (0.8,2.5) -- (-0.8,0.5);
				\draw[e:main] (1.5,1.5) -- (0.8,2.5) -- (-0.8,2.5) -- (-1.5,1.5) -- (-0.8,0.5) -- (0.8,0.5) -- (1.5,1.5);
				
				\draw[e:main] (-0.8,-2.5) -- (0.8,-0.5);
				\draw[e:main] (0.8,-2.5) -- (-0.8,-0.5);
				\draw[e:main] (1.5,-1.5) -- (0.8,-2.5) -- (-0.8,-2.5) -- (-1.5,-1.5) -- (-0.8,-0.5) -- (0.8,-0.5) -- (1.5,-1.5);
				
				\draw[e:main] (-1.5,1.5) -- (-1,1.5);
				\draw[e:main] (1.5,1.5) -- (1,1.5);
				\draw[e:main] (-0.5,1.5) -- (0.5,1.5);
				\draw[e:main,myGreen] (-0.5,1.5) -- (-1,1.5);
				\draw[e:main,myGreen] (0.5,1.5) -- (1,1.5);
				
				\draw[e:main] (-1.5,-1.5) -- (-1,-1.5);
				\draw[e:main] (1.5,-1.5) -- (1,-1.5);
				\draw[e:main,blue] (-0.5,-1.5) -- (0.5,-1.5);
				\draw[e:main,myGreen] (-0.5,-1.5) -- (-1,-1.5);
				\draw[e:main,myGreen] (0.5,-1.5) -- (1,-1.5);
				
				\draw[e:main,blue,densely dashed,bend right=60] (-1,1.5) to (1,1.5);
				\draw[e:main,blue,densely dashed,bend right=60] (-0.5,1.5) to (0.5,1.5);
				
				\draw[e:main,blue,densely dashed,bend right=60] (-1,-1.5) to (1,-1.5);
				
			\end{pgfonlayer}
		\end{tikzpicture}
	\end{subfigure}
	\begin{subfigure}[b]{0.3\textwidth}
		\centering
		\begin{tikzpicture}
			\pgfdeclarelayer{background}
			\pgfdeclarelayer{foreground}
			\pgfsetlayers{background,main,foreground}
			
			\node () at (0,0){reversed};
			
			\node[blue] () at (0,0.8){$C$};
			\node[blue] () at (0,-2.2){$C$};
			
			\node[v:mainempty] () at (1.5,1.5){};
			\node[v:main] () at (-1.5,1.5){};
			\node[v:main] () at (0.8,0.5){};
			\node[v:mainempty] () at (-0.8,0.5){};
			\node[v:main] () at (0.8,2.5){};
			\node[v:mainempty] () at (-0.8,2.5){};
			
			\node[v:mainempty] () at (1.5,-1.5){};
			\node[v:main] () at (-1.5,-1.5){};
			\node[v:main] () at (0.8,-0.5){};
			\node[v:mainempty] () at (-0.8,-0.5){};
			\node[v:main] () at (0.8,-2.5){};
			\node[v:mainempty] () at (-0.8,-2.5){};
			
			\node[v:mainemptygreen] () at (-0.5,1.5){};
			\node[myGreen] () at (-0.5,1.75){$b$};
			\node[v:maingreen] () at (0.5,1.5){};
			\node[myGreen] () at (0.55,1.75){$a'$};
			\node[v:mainemptygreen] () at (1,1.5){};
			\node[myGreen] () at (1.05,1.75){$b'$};
			\node[v:maingreen] () at (-1,1.5){};
			\node[myGreen] () at (-1,1.7){$a$};
			
			\node[v:mainemptygreen] () at (-0.5,-1.5){};
			\node[myGreen] () at (-0.5,-1.25){$b$};
			\node[v:maingreen] () at (0.5,-1.5){};
			\node[myGreen] () at (0.55,-1.25){$a'$};
			\node[v:mainemptygreen] () at (1,-1.5){};
			\node[myGreen] () at (1.05,-1.25){$b'$};
			\node[v:maingreen] () at (-1,-1.5){};
			\node[myGreen] () at (-1,-1.3){$a$};
			
			\node[v:mainemptygray] () at (-1.25,1.5){};
			\node[v:maingray] () at (1.25,1.5){};
			
			\node[v:mainemptygray] () at (-1.25,-1.5){};
			\node[v:maingray] () at (1.25,-1.5){};
			
			\begin{pgfonlayer}{background}
				
				\draw[e:marker,BrightUbe,bend right=55] (-0.5,1.5) to (0.5,1.5);
				
				\draw[e:main] (-0.8,2.5) -- (0.8,0.5);
				\draw[e:main] (0.8,2.5) -- (-0.8,0.5);
				\draw[e:main] (1.5,1.5) -- (0.8,2.5) -- (-0.8,2.5) -- (-1.5,1.5) -- (-0.8,0.5) -- (0.8,0.5) -- (1.5,1.5);
				
				\draw[e:main] (-0.8,-2.5) -- (0.8,-0.5);
				\draw[e:main] (0.8,-2.5) -- (-0.8,-0.5);
				\draw[e:main] (1.5,-1.5) -- (0.8,-2.5) -- (-0.8,-2.5) -- (-1.5,-1.5) -- (-0.8,-0.5) -- (0.8,-0.5) -- (1.5,-1.5);
				
				\draw[e:main] (-1.5,1.5) -- (-1,1.5);
				\draw[e:main] (1.5,1.5) -- (1,1.5);
				\draw[e:main] (-0.5,1.5) -- (0.5,1.5);
				\draw[e:main,myGreen] (-0.5,1.5) -- (-1,1.5);
				\draw[e:main,myGreen] (0.5,1.5) -- (1,1.5);
				
				\draw[e:main] (-1.5,-1.5) -- (-1,-1.5);
				\draw[e:main] (1.5,-1.5) -- (1,-1.5);
				\draw[e:main,blue] (-0.5,-1.5) -- (0.5,-1.5);
				\draw[e:main,myGreen] (-0.5,-1.5) -- (-1,-1.5);
				\draw[e:main,myGreen] (0.5,-1.5) -- (1,-1.5);
				
				\draw[e:main,blue,densely dashed,bend right=60] (-1,1.5) to (1,1.5);
				\draw[e:main,blue,densely dashed,bend right=60] (-0.5,1.5) to (0.5,1.5);
				
				\draw[e:main,blue,densely dashed,bend right=60] (-1,-1.5) to (1,-1.5);
				
			\end{pgfonlayer}
		\end{tikzpicture}
	\end{subfigure}
	\begin{subfigure}[b]{0.3\textwidth}
		\centering
		\begin{tikzpicture}
			\pgfdeclarelayer{background}
			\pgfdeclarelayer{foreground}
			\pgfsetlayers{background,main,foreground}
			
			\node () at (0,0){split};
			
			\node[blue] () at (0.6,1.7){$C$};
			\node[blue] () at (0.6,-1.2){$C$};
			
			\node[v:mainempty] () at (1.5,1.5){};
			\node[v:main] () at (-1.5,1.5){};
			\node[v:main] () at (0.8,0.5){};
			\node[v:mainempty] () at (-0.8,0.5){};
			\node[v:main] () at (0.8,2.5){};
			\node[v:mainempty] () at (-0.8,2.5){};
			
			\node[v:maingreen] () at (0.2,2.5){};
			\node[myGreen] () at (0.2,2.7){$a$};
			\node[v:mainemptygreen] () at (-0.2,2.5){};
			\node[myGreen] () at (-0.2,2.75){$b$};
			
			\node[v:maingray] () at (-0.5,2.5){};
			\node[v:mainemptygray] () at (0.5,2.5){};

			\node[v:mainempty] () at (1.5,-1.5){};
			\node[v:main] () at (-1.5,-1.5){};
			\node[v:maingreen] () at (0.8,-0.5){};
			\node[myGreen] () at (0.8,-0.3){$a$};
			\node[v:mainempty] () at (-0.8,-0.5){};
			\node[v:main] () at (0.8,-2.5){};
			\node[v:mainempty] () at (-0.8,-2.5){};
			
			\node[v:mainemptygreen] () at (0.2,-0.5){};
			\node[myGreen] () at (0.2,-0.25){$b$};
			\node[v:maingray] () at (-0.2,-0.5){};
			
			\node[v:maingreen] () at (1.29,1.8){};
			\node[myGreen] () at (1.45,2.1){$a'$};
			
			\node[v:mainemptygreen] () at (1,2.2){};
			\node[myGreen] () at (1.2,2.4){$b'$};
			
			\node[v:maingreen] () at (1.29,-1.2){};
			\node[myGreen] () at (1.45,-0.9){$a'$};
			
			\node[v:mainemptygreen] () at (1,-0.8){};
			\node[myGreen] () at (1.2,-0.6){$b'$};
			
			\begin{pgfonlayer}{background}
				
				\draw[e:marker,BrightUbe,bend right=20] (0.2,2.5) to (1,2.2);
				\draw[e:marker,myLightBlue] (0.2,2.5) -- (0.8,2.5);
				\draw[e:marker,myLightBlue] (0.8,2.5) -- (-0.8,0.5);
				
				\draw[e:main] (-0.8,2.5) -- (0.8,0.5);
				\draw[e:main] (0.8,2.5) -- (-0.8,0.5);
				\draw[e:main] (-1.5,1.5) -- (1.5,1.5);
				\draw[e:main] (1.5,1.5) -- (0.8,2.5) -- (-0.8,2.5) -- (-1.5,1.5) -- (-0.8,0.5) -- (0.8,0.5) -- (1.5,1.5);
				
				\draw[e:main] (-0.8,-2.5) -- (0.8,-0.5);
				\draw[e:main] (0.8,-2.5) -- (-0.8,-0.5);
				\draw[e:main] (-1.5,-1.5) -- (1.5,-1.5);
				
				\draw[e:main] (1.5,-1.5) -- (0.8,-2.5) -- (-0.8,-2.5) -- (-1.5,-1.5) -- (-0.8,-0.5) -- (0.8,-0.5) -- (1.5,-1.5);
				
				\draw[e:main,myGreen] (-0.2,2.5) -- (0.2,2.5){};
				\draw[e:main,myGreen] (1.29,1.8) -- (1,2.2);
				
				\draw[e:main,myGreen] (0.2,-0.5) -- (0.8,-0.5);
				\draw[e:main,myGreen] (1.29,-1.2) -- (1,-0.8);
				
				\draw[e:main,blue] (1,-0.8) -- (0.8,-0.5);
				
				\draw[e:main,blue,densely dashed,bend right=25] (0.2,-0.5) to (1.29,-1.2);
				
				\draw[e:main,blue,densely dashed,bend right=20] (0.2,2.5) to (1,2.2);
				\draw[e:main,blue,densely dashed,bend right=20] (-0.2,2.5) to (1.29,1.8);

			\end{pgfonlayer}
		\end{tikzpicture}
	\end{subfigure}
	\caption{The three possible configurations how the edges $ab$ and $a'b'$ may occur in a bisubdivision of $K_{3,3}$ without immediately yielding a conformal cross over $C$.
		The second line of figures shows how these cases can be reduced.} 
	\label{fig:threeconfigurations}
\end{figure}
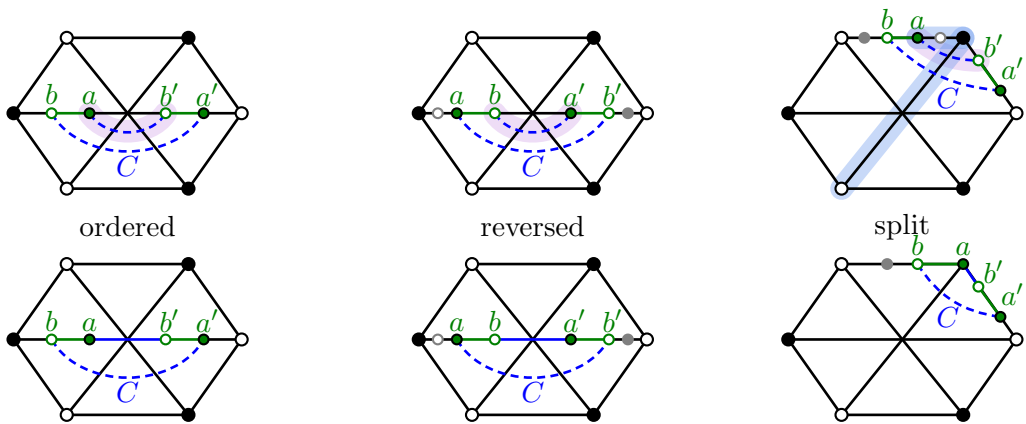

If we can find a conformal bisubdivision of $K_{3,3}$ that fits one of the cases in \cref{lemma:easycrossesinK33}, we are done immediately.
Hence what remains is a discussion of the three cases which are still left.
In \cref{fig:threeconfigurations}, these cases are illustrated.
In two of the three cases, the edges $ab$ and $a'b'$ occur on a single subdivided edge of $L$, while in the last case, the \emph{split case}, $ab$ and $a'b'$ are split over two subdivided edges in such a way that exactly one of $ab$ and $a'b'$ occurs in reverse on its respective path.
In each of the three cases, we can use an edge $e\in\E{C}\setminus \Set{ab,a'b'}$ in order to further reduce $L$ and make sure that we always find a conformal bisubdivision of $K_{3,3}$ which contains at least three edges of $C$.
Let $B$ be a brace, $C$ a $4$-cycle, $M$ a perfect matching of $B$ such that $C$ is $M$-conformal and $L$ an $M$-conformal bisubdivision of $K_{3,3}$ that contains the vertices of $C$.
We say that $L$ \emph{splits $C$} if the way the vertices of $C$ are distributed over the bisubdivided edges of $L$ as they are in the split case in \cref{fig:threeconfigurations}.
As an intermediate step, we want to show that we can always find a bisubdivision of $K_{3,3}$ that splits $C$.
Suppose $ab$ and $a'b'$ occur on a single subdivided edge $P$ of $L$ as in the \emph{ordered} or the \emph{reversed} case from \cref{fig:threeconfigurations}.
Let $u\in V_i$ be an endpoint of $P$ and let $Y$ be the \emph{bisubdivided claw} with centre $u$ in $L$ consisting of the three bisubdivided edges $P$, $Q_1$, and $Q_2$ of $L$ that have $u$ as an endpoint.
Let $T$ be the shortest $u$-$\V{C}$-subpath of $P$.
If there exists an internally $M$-conformal path $R$ that is internally disjoint from $L$ such that $R$ has an endpoint in $\Vi{3-i}{T}$ and its other endpoint lies in $\Vi{i}{L-Y}$ we say that $L$ has a \emph{$V_i$-jump over $C$}.

\begin{lemma}\label{lemma:tightcutsincofnromalsubgraphs}
	Let $B$ be a brace, $M$ a perfect matching of $B$, $H\subseteq B$ an $M$-conformal and matching covered subgraph, and $X\subseteq\V{H}$ such that $\CutG{H}{X}$ is a non-trivial tight cut in $H$.
	Then there exists an internally $M$-conformal path $P$ in $B$ such that $P$ is internally disjoint from $H$ and has its endpoints in the minorities of $X$ and $\V{H}\setminus X$.
\end{lemma}

\begin{proof}
	The claim follows immediately from \cref{thm:bipartiteextendibility}.
	With $B$ being a brace it is $2$-extendible.
	Let $e$ be the unique edge of $M$ in $\CutG{H}{X}$.
	Then there must be an internally $M$-conformal $\Minority{X}$-$\Minority{Y}$-path $P$ in $B$ that avoids $e$.
	If we choose $P$ to be as short as possible, it cannot contain any vertex of $\Minority{X}\cup \Minority{Y}$ as an inner vertex.
	Moreover, since $H$ is $M$-conformal, no vertex of $H$ can be an inner vertex of $P$.	
\end{proof}

\begin{lemma}\label{lemma:reducedK33}
	Let $B$ be a $K_{3,3}$-containing brace and $C$ a $4$-cycle in $B$ such that there is no conformal cross over $C$ in $B$.
	Then there exists a perfect matching $M$ of $B$ such that $C$ is $M$-conformal and there is an $M$-conformal bisubdivision $L$ of $K_{3,3}$ such that $L$ splits $C$, or $L$ has a $V_1$-jump over $C$.
\end{lemma}

\begin{proof}
	By \cref{lemma:C4inK33,lemma:easycrossesinK33} and the discussion above we know that there are a perfect matching $M$ of $B$ such that $C$ is $M$-conformal and an $M$-conformal bisubdivision $L$ of $K_{3,3}$ that contains the vertices of $C$ such that the way the vertices of $C$ occur in $L$ corresponds to one of the three cases depicted in \cref{fig:threeconfigurations}.
	If $L$ splits $C$ we are done already, so let us assume that there is a bisubdivided edge $P$ of $L$ such that the edges $ab,a'b\in\E{C}\cap M$ occur on $P$ as in the \emph{ordered} or the \emph{reversed} case from \cref{fig:threeconfigurations}.
	Let $u\in V_1$ be an endpoint of $P$.
	Consider the three bisubdivided edges $P$, $Q_1$, and $Q_2$ of $L$ that have $u$ as an endpoint.
	Let us choose $L$ such that the tuple $\Brace{\Abs{\E{P}},\Abs{\E{Q_1}\cup\E{Q_2}}}$ is lexicographically minimised.
	For each $Z\in\Set{P,Q_1,Q_2}$ let $v_Z\in V_2$ be the endpoint of $Z$ different from $u$ and let $Y\coloneqq \Brace{\V{P}\cup\V{Q_1}\cup\V{Q_2}}\setminus\Set{v_P,v_{Q_1},v_{Q_2}}$.
	Now every component of $\InducedSubgraph{L}{Y}-u$ is a path of odd length and thus $\Abs{V_1\cap Y}-\Abs{V_2\cap Y}=1$, moreover, no vertex of $Y\cap V_2$ has a neighbour in $L-Y$ within $L$ and $\Abs{Y\cap\V{C}}\geq 3$.
	Hence $\CutG{L}{Y}$ defines a non-trivial tight cut in $L$.
	For an illustration, see \cref{fig:tightcut}.
	
	\begin{figure}[h!]
		\centering
		\begin{tikzpicture}[scale=0.9]
			\pgfdeclarelayer{background}
			\pgfdeclarelayer{foreground}
			\pgfsetlayers{background,main,foreground}
			
			\node[blue] () at (3,2.4){$C$};
			\node[gray] () at (1,2.4){$P$};
			\node[gray] () at (5,0.5){$Q_{2}$};
			\node[gray] () at (3.5,0.5){$Q_{1}$};

			\node[v:main] () at (0,0.2){};
			
			\node[v:mainempty] () at (0,2){};
			\node () at (0,2.4){$v_{P}$};
			
			\node[v:main] () at (4.5,2){};
			\node () at (4.5,2.4){$u$};
			
			\node[v:mainempty] () at (4.5,-1){};
			\node () at (4.7,-1.4){$v_{Q_{2}}$};
			
			\node[v:mainempty] () at (1.2,-1){};
			\node () at (1.2,-1.4){$v_{Q_{1}}$};
			
			\node[v:main] () at (2.8,-1.5){};
			
			\node[v:maingray] () at (0.5,2){};
			\node[v:mainemptygray] () at (1,2){};
			
			\node[v:maingreen] () at (1.5,2){};
			\node[myGreen] () at (1.5,1.65){$a$};        
			
			\node[v:mainemptygreen] () at (2,2){};
			\node[myGreen] () at (2,1.7){$b$};
			
			\node[v:maingreen] () at (2.5,2){};
			\node[myGreen] () at (2.5,1.7){$a'$};
			
			\node[v:mainemptygreen] () at (3,2){};
			\node[myGreen] () at (3,1.7){$b'$};
			
			\node[v:maingray] () at (3.5,2){};
			\node[v:mainemptygray] () at (4,2){};
			
			\node[v:maingray] () at (4.5,0){};
			\node[v:mainemptygray] () at (4.5,1){};
			
			\node[v:mainemptygray] () at (0,0.8){};
			\node[v:maingray] () at (0,1.4){};
			
			\node[v:mainemptygray] () at (1.5,-0.2){};
			\node[v:maingray] () at (3,-0.6){};
			
			\node[v:maingray] () at (0.8,-0.6){};
			\node[v:mainemptygray] () at (0.4,-0.2){};
			
			\node[v:mainemptygray] () at (3.37,-1.33){};
			\node[v:maingray] () at (3.93,-1.16){};
			
			\node[v:maingray] () at (1.73,-1.16){};
			\node[v:mainemptygray] () at (2.23,-1.33){};
			
			\node[v:maingray] () at (2.3,0){};
			\node[v:mainemptygray] () at (3.4,1){};
			
			\node[v:maingray] () at (0.62,1.22){};
			\node[v:mainemptygray] () at (2.18,-0.72){};
			
			\begin{pgfonlayer}{background}
				
				\draw[e:marker] (0.5,2) -- (4.5,2);
				\draw[e:marker] (4.5,2) -- (4.5,0);
				\draw[e:marker] (4.5,2) -- (2.3,0);
				
				\draw[e:main,gray] (0,2) -- (0.5,2);
				\draw[e:main,gray] (1,2) -- (1.5,2);
				\draw[e:main,gray] (2,2) -- (2.5,2);
				\draw[e:main,gray] (3,2) -- (3.5,2);
				\draw[e:main,gray] (4,2) -- (4.5,2);
				
				\draw[e:main,gray] (0.5,2) -- (1,2);
				
				\draw[e:main,gray] (1.5,2) -- (2,2);
				
				\draw[e:main,gray] (2.5,2) -- (3,2);

				\draw[e:main,gray] (3.5,2) -- (4,2);
				
				\draw[e:main,blue,densely dashed,bend left=45] (1.5,2) to (3,2);
				
				\draw[e:main,gray] (0,2) -- (0,1.4);
				\draw[e:main,gray] (0,0.8) -- (0,0.2);
				\draw[e:main,gray] (0,0.8) -- (0,1.4);
				
				\draw[e:main,gray] (4.5,0) -- (4.5,1);
				\draw[e:main,gray](4.5,1) -- (4.5,2);
				\draw[e:main,gray](4.5,-1) -- (4.5,0);
				
				\draw[e:main,gray] (0,0.2) -- (1.5,-0.2);
				\draw[e:main,gray] (3,-0.6) -- (4.5,-1);
				\draw[e:main,gray] (1.5,-0.2) -- (3,-0.6);
				
				\draw[e:main,gray] (0,0.2) -- (0.4,-0.2);
				\draw[e:main,gray] (0.8,-0.6) -- (1.2,-1);
				\draw[e:main,gray] (0.4,-0.2) -- (0.8,-0.6); 
				
				\draw[e:main,gray] (2.8,-1.5) -- (3.37,-1.33);
				\draw[e:main,gray] (3.93,-1.16) -- (4.5,-1);
				\draw[e:main,gray] (3.37,-1.33) -- (3.93,-1.16);

				\draw[e:main,gray] (1.2,-1) -- (1.73,-1.16);
				\draw[e:main,gray] (2.23,-1.33) -- (2.8,-1.5);
				\draw[e:main,gray] (1.73,-1.16) -- (2.23,-1.33);
				
				\draw[e:main,gray] (3.4,1) -- (4.5,2);
				\draw[e:main,gray] (1.2,-1) -- (2.3,0);
				\draw[e:main,gray] (2.3,0) -- (3.4,1);
				
				\draw[e:main,gray] (0,2) -- (0.62,1.22);
				\draw[e:main,gray] (2.18,-0.72) -- (2.8,-1.5);
				\draw[e:main,gray] (0.62,1.22) -- (2.18,-0.72);

			\end{pgfonlayer}
			
		\end{tikzpicture}
		\caption{The non-trivial tight cut around the bisubdivided claw centred at $u$ in the proof of \cref{lemma:reducedK33}}
		\label{fig:tightcut}
	\end{figure}
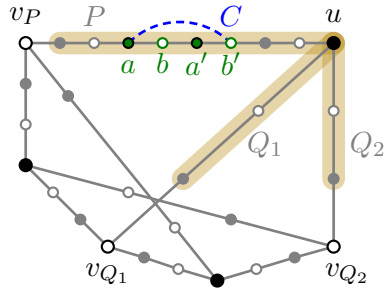
	
	By \cref{lemma:tightcutsincofnromalsubgraphs} there exists an internally $M$-conformal path $F$ in $B$ such that
	\begin{itemize}
		\item $F$ has an endpoint in $V_2\cap Y$,
		\item the other endpoint of $F$ lies in $\Vi{1}{L}\setminus Y$, and
		\item $F$ is internally disjoint from $L$.
	\end{itemize}
	Please note that we may change the perfect matching $M$ within the $M$-conformal subgraph $L$ of $B$ at will, without changing the fact that $F$ is an internally $M$-conformal path with the properties listed above.
	
	Let $y$ be the endpoint of $F$ in $Y$ and let $x$ be its other endpoint as well as $P_x$ be the subdivided edge of $L$ that contains $x$ in case $x$ is a vertex of degree two in $L$.
	What follows is a discussion on the possible positions of $x$ and $y$ in $L$.
	For an illustrative overview on the different cases that might appear consult \cref{fig:firstLreductions}.
	Let $T_{V_1}$ be the shortest subpath of $P$ with one endpoint in $\V{C}$ and $u$ as its other endpoint.
	Similarly, let $T_{V_2}$ be the shortest subpath of $P$ with one endpoint in $\V{C}$ and $v_P$ as its other endpoint.
	At last, let $w_1$ and $w_2$ be the two degree three vertices in $\Vi{1}{L}$ that are different from $u$, let $w_1$ be the endpoint of $P_x$ that lies in $V_1$.
	Given any two vertices $v_1$, $v_2$ of degree three in $L$ that belong to different colour classes let us denote by $E_{v_1v_2}=E_{v_2v_1}$ the subdivided edge of $L$ with endpoints $v_1$ and $v_2$.
	
	\begin{figure}[h!]
		\centering
		\begin{subfigure}[t]{0.3\textwidth}
			\begin{tikzpicture}[scale=0.7]
				\pgfdeclarelayer{background}
				\pgfdeclarelayer{foreground}
				\pgfsetlayers{background,main,foreground}
				
				\node[blue] () at (3,2.4){$C$};
				\node[gray] () at (1,2.4){$P$};
				\node[gray] () at (5,0.5){$Q_{2}$};
				\node[gray] () at (3.5,0.5){$Q_{1}$};
				
				\node () at (2.25,-2.5){$P$ gets shortened};
				\node () at (2.25,-3){\phantom{subdivision splits $C$}};

				\node[v:main] () at (0,0.2){};
				
				\node[v:mainempty] () at (0,2){};
				\node () at (0,2.4){$v_{P}$};
				
				\node[v:main] () at (4.5,2){};
				\node () at (4.5,2.4){$u$};
				
				\node[v:mainempty] () at (4.5,-1){};
				\node () at (4.7,-1.4){$v_{Q_{2}}$};
				
				\node[v:mainempty] () at (1.2,-1){};
				\node () at (1.2,-1.4){$v_{Q_{1}}$};
				
				\node[v:main] () at (2.8,-1.5){};
				
				\node[v:maingray] () at (0.5,2){};
				\node[v:mainemptygray] () at (1,2){};
				\node[v:maingreen] () at (1.5,2){};
				\node[v:mainemptygreen] () at (2,2){};
				\node[v:maingreen] () at (2.5,2){};
				\node[v:mainemptygreen] () at (3,2){};
				\node[v:maingray] () at (3.5,2){};
				\node[v:mainemptygray] () at (4,2){};
				
				\node[v:maingray] () at (4.5,0){};
				\node[v:mainemptygray] () at (4.5,1){};
				
				\node[v:mainemptygray] () at (0,0.8){};
				\node[v:maingray] () at (0,1.4){};

				\node[v:mainemptygray] () at (1.5,-0.2){};
				\node[v:maingray] () at (3,-0.6){};

				\node[v:maingray] () at (0.8,-0.6){};
				\node[v:mainemptygray] () at (0.4,-0.2){};

				\node[v:mainemptygray] () at (3.37,-1.33){};
				\node[v:maingray] () at (3.93,-1.16){};
				
				\node[v:maingray] () at (1.73,-1.16){};
				\node[v:mainemptygray] () at (2.23,-1.33){};

				\node[v:maingray] () at (2.3,0){};
				\node[v:mainemptygray] () at (3.4,1){};

				\node[v:maingray] () at (0.62,1.22){};
				\node[v:mainemptygray] () at (2.18,-0.72){};
				
				\node[Amethyst] () at (1.3,1.4){$F$};

				\begin{pgfonlayer}{background}

					\draw[e:marker,BrightUbe] (0,0.2) -- (0,2);
					\draw[e:marker,BrightUbe] (0,2) -- (4.5,2);
					\draw[e:marker,BrightUbe,bend left=30] (1,2) to (0.62,1.22);
					\draw[e:marker,BrightUbe] (0.62,1.22) to (2.8,-1.5);
					
					\draw[e:main,gray] (0,2) -- (0.5,2);
					\draw[e:main,gray] (1,2) -- (1.5,2);
					\draw[e:main,myGreen] (2,2) -- (2.5,2);
					\draw[e:main,gray] (3,2) -- (3.5,2);
					\draw[e:main,gray] (4,2) -- (4.5,2);
					
					\draw[e:coloredborder] (0.5,2) -- (1,2);
					\draw[e:coloredthin,color=BostonUniversityRed] (0.5,2) -- (1,2);
					
					\draw[e:coloredborder] (1.5,2) -- (2,2);
					\draw[e:coloredthin,color=BostonUniversityRed] (1.5,2) -- (2,2);
					
					\draw[e:coloredborder] (2.5,2) -- (3,2);
					\draw[e:coloredthin,color=BostonUniversityRed] (2.5,2) -- (3,2);

					\draw[e:coloredborder] (3.5,2) -- (4,2);
					\draw[e:coloredthin,color=BostonUniversityRed] (3.5,2) -- (4,2);
					
					\draw[e:main,blue,densely dashed,bend left=45] (1.5,2) to (3,2);
					
					\draw[e:coloredborder] (0,2) -- (0,1.4);
					\draw[e:coloredthin,color=BostonUniversityRed] (0,2) -- (0,1.4);
					\draw[e:coloredborder] (0,0.8) -- (0,0.2);
					\draw[e:coloredthin,color=BostonUniversityRed] (0,0.8) -- (0,0.2);
					\draw[e:main,gray] (0,0.8) -- (0,1.4);
					
					\draw[e:coloredborder] (4.5,0) -- (4.5,1);
					\draw[e:coloredthin,color=BostonUniversityRed] (4.5,0) -- (4.5,1);
					\draw[e:main,gray](4.5,1) -- (4.5,2);
					\draw[e:main,gray](4.5,-1) -- (4.5,0);
					
					\draw[e:main,gray] (0,0.2) -- (1.5,-0.2);
					\draw[e:main,gray] (3,-0.6) -- (4.5,-1);
					\draw[e:coloredborder] (1.5,-0.2) -- (3,-0.6);
					\draw[e:coloredthin,color=BostonUniversityRed] (1.5,-0.2) -- (3,-0.6);
					
					\draw[e:main,gray] (0,0.2) -- (0.4,-0.2);
					\draw[e:main,gray] (0.8,-0.6) -- (1.2,-1);
					\draw[e:coloredborder] (0.4,-0.2) -- (0.8,-0.6); 
					\draw[e:coloredthin,color=BostonUniversityRed] (0.4,-0.2) -- (0.8,-0.6); 
					
					\draw[e:coloredborder] (2.8,-1.5) -- (3.37,-1.33);
					\draw[e:coloredthin,color=BostonUniversityRed] (2.8,-1.5) -- (3.37,-1.33);
					\draw[e:coloredborder] (3.93,-1.16) -- (4.5,-1);
					\draw[e:coloredthin,color=BostonUniversityRed] (3.93,-1.16) -- (4.5,-1);
					\draw[e:main,gray] (3.37,-1.33) -- (3.93,-1.16);

					\draw[e:main,gray] (1.2,-1) -- (1.73,-1.16);
					\draw[e:main,gray] (2.23,-1.33) -- (2.8,-1.5);
					\draw[e:coloredborder] (1.73,-1.16) -- (2.23,-1.33);
					\draw[e:coloredthin,color=BostonUniversityRed] (1.73,-1.16) -- (2.23,-1.33);
					
					\draw[e:coloredborder] (3.4,1) -- (4.5,2);
					\draw[e:coloredthin,color=BostonUniversityRed] (3.4,1) -- (4.5,2);
					\draw[e:coloredborder] (1.2,-1) -- (2.3,0);
					\draw[e:coloredthin,color=BostonUniversityRed] (1.2,-1) -- (2.3,0);
					\draw[e:main,gray] (2.3,0) -- (3.4,1);
					
					\draw[e:main,gray] (0,2) -- (0.62,1.22);
					\draw[e:main,gray] (2.18,-0.72) -- (2.8,-1.5);
					\draw[e:coloredborder] (0.62,1.22) -- (2.18,-0.72);
					\draw[e:coloredthin,color=BostonUniversityRed] (0.62,1.22) -- (2.18,-0.72);
					
					\draw[e:main,Amethyst,bend left=30] (1,2) to (0.62,1.22);

				\end{pgfonlayer}
				
			\end{tikzpicture}
		\end{subfigure}
		\begin{subfigure}[t]{0.3\textwidth}
			\begin{tikzpicture}[scale=0.7]
				\pgfdeclarelayer{background}
				\pgfdeclarelayer{foreground}
				\pgfsetlayers{background,main,foreground}
				
				\node[blue] () at (3,2.4){$C$};
				\node[gray] () at (1,2.4){$P$};
				\node[gray] () at (5,0.5){$Q_{2}$};
				\node[gray] () at (3.5,0.5){$Q_{1}$};
				
				\node () at (2.25,-2.5){We find the path $R$};
				\node () at (2.25,-3){\phantom{subdivision splits $C$}};

				\node[v:main] () at (0,0.2){};
				
				\node[v:mainempty] () at (0,2){};
				\node () at (0,2.4){$v_{P}$};
				
				\node[v:main] () at (4.5,2){};
				\node () at (4.5,2.4){$u$};
				
				\node[v:mainempty] () at (4.5,-1){};
				\node () at (4.7,-1.4){$v_{Q_{2}}$};
				
				\node[v:mainempty] () at (1.2,-1){};
				\node () at (1.2,-1.4){$v_{Q_{1}}$};
				
				\node[v:main] () at (2.8,-1.5){};
				
				\node[v:maingray] () at (0.5,2){};
				\node[v:mainemptygray] () at (1,2){};
				\node[v:maingreen] () at (1.5,2){};
				\node[v:mainemptygreen] () at (2,2){};
				\node[v:maingreen] () at (2.5,2){};
				\node[v:mainemptygreen] () at (3,2){};
				\node[v:maingray] () at (3.5,2){};
				\node[v:mainemptygray] () at (4,2){};
				
				\node[v:maingray] () at (4.5,0){};
				\node[v:mainemptygray] () at (4.5,1){};
				
				\node[v:mainemptygray] () at (0,0.8){};
				\node[v:maingray] () at (0,1.4){};
				
				\node[v:mainemptygray] () at (1.5,-0.2){};
				\node[v:maingray] () at (3,-0.6){};
				
				\node[v:maingray] () at (0.8,-0.6){};
				\node[v:mainemptygray] () at (0.4,-0.2){};

				\node[v:mainemptygray] () at (3.37,-1.33){};
				\node[v:maingray] () at (3.93,-1.16){};
				
				\node[v:maingray] () at (1.73,-1.16){};
				\node[v:mainemptygray] () at (2.23,-1.33){};
				
				\node[v:maingray] () at (2.3,0){};
				\node[v:mainemptygray] () at (3.4,1){};
				
				\node[v:maingray] () at (0.62,1.22){};
				\node[v:mainemptygray] () at (2.18,-0.72){};
				
				\node[Amethyst] () at (3.6,-0.2){$F$};
				
				\begin{pgfonlayer}{background}
					\draw[e:main,gray] (0,2) -- (0.5,2);
					\draw[e:main,gray] (1,2) -- (1.5,2);
					\draw[e:main,myGreen] (2,2) -- (2.5,2);
					\draw[e:main,gray] (3,2) -- (3.5,2);
					\draw[e:main,gray] (4,2) -- (4.5,2);
					
					\draw[e:coloredborder] (0.5,2) -- (1,2);
					\draw[e:coloredthin,color=BostonUniversityRed] (0.5,2) -- (1,2);
					
					\draw[e:coloredborder] (1.5,2) -- (2,2);
					\draw[e:coloredthin,color=BostonUniversityRed] (1.5,2) -- (2,2);
					
					\draw[e:coloredborder] (2.5,2) -- (3,2);
					\draw[e:coloredthin,color=BostonUniversityRed] (2.5,2) -- (3,2);

					\draw[e:coloredborder] (3.5,2) -- (4,2);
					\draw[e:coloredthin,color=BostonUniversityRed] (3.5,2) -- (4,2);
					
					\draw[e:main,blue,densely dashed,bend left=45] (1.5,2) to (3,2);
					
					\draw[e:coloredborder] (0,0.8) -- (0,1.4);
					\draw[e:coloredthin,color=BostonUniversityRed] (0,0.8) -- (0,1.4);
					\draw[e:main,gray](0,2) -- (0,1.4);
					\draw[e:main,gray] (0,0.8) -- (0,0.2);
					
					\draw[e:coloredborder] (4.5,0) -- (4.5,1);
					\draw[e:coloredthin,color=BostonUniversityRed] (4.5,0) -- (4.5,1);
					\draw[e:main,gray](4.5,1) -- (4.5,2);
					\draw[e:main,gray](4.5,-1) -- (4.5,0);
					
					\draw[e:coloredborder] (0,0.2) -- (1.5,-0.2);
					\draw[e:coloredthin,color=BostonUniversityRed] (0,0.2) -- (1.5,-0.2);
					\draw[e:coloredborder] (3,-0.6) -- (4.5,-1);
					\draw[e:coloredthin,color=BostonUniversityRed] (3,-0.6) -- (4.5,-1);
					\draw[e:main,gray] (1.5,-0.2) -- (3,-0.6);
					
					\draw[e:main,gray] (0,0.2) -- (0.4,-0.2);
					\draw[e:main,gray] (0.8,-0.6) -- (1.2,-1);
					\draw[e:coloredborder] (0.4,-0.2) -- (0.8,-0.6); 
					\draw[e:coloredthin,color=BostonUniversityRed] (0.4,-0.2) -- (0.8,-0.6); 
					
					\draw[e:main,gray]   (2.8,-1.5) -- (3.37,-1.33);
					\draw[e:main,gray] (3.93,-1.16) -- (4.5,-1);
					\draw[e:coloredborder] (3.37,-1.33) -- (3.93,-1.16);
					\draw[e:coloredthin,color=BostonUniversityRed] (3.37,-1.33) -- (3.93,-1.16);

					\draw[e:main,gray] (1.2,-1) -- (1.73,-1.16);
					\draw[e:main,gray] (2.23,-1.33) -- (2.8,-1.5);
					\draw[e:coloredborder] (1.73,-1.16) -- (2.23,-1.33);
					\draw[e:coloredthin,color=BostonUniversityRed] (1.73,-1.16) -- (2.23,-1.33);   
					
					\draw[e:coloredborder] (3.4,1) -- (4.5,2);
					\draw[e:coloredthin,color=BostonUniversityRed] (3.4,1) -- (4.5,2);
					\draw[e:coloredborder] (1.2,-1) -- (2.3,0);
					\draw[e:coloredthin,color=BostonUniversityRed] (1.2,-1) -- (2.3,0);
					\draw[e:main,gray] (2.3,0) -- (3.4,1);
					
					\draw[e:coloredborder] (0,2) -- (0.62,1.22);
					\draw[e:coloredthin,color=BostonUniversityRed] (0,2) -- (0.62,1.22);
					\draw[e:coloredborder] (2.18,-0.72) -- (2.8,-1.5);
					\draw[e:coloredthin,color=BostonUniversityRed] (2.18,-0.72) -- (2.8,-1.5);
					\draw[e:main,gray] (0.62,1.22) -- (2.18,-0.72);
					
					\draw[e:main,Amethyst] (4,2) to (3.93,-1.16);
					
				\end{pgfonlayer}
				
			\end{tikzpicture}
		\end{subfigure}
		\begin{subfigure}[t]{0.3\textwidth}
			\begin{tikzpicture}[scale=0.7]
				\pgfdeclarelayer{background}
				\pgfdeclarelayer{foreground}
				\pgfsetlayers{background,main,foreground}
				
				\node[blue] () at (3,2.4){$C$};
				\node[gray] () at (1,2.4){$P$};
				\node[gray] () at (5,0.5){$Q_{2}$};
				\node[gray] () at (3.5,0.5){$Q_{1}$};
				
				\node () at (2.25,-2.5){The new $K_{3,3}$};
				\node () at (2.25,-3){subdivision splits $C$};

				\node[v:main] () at (0,0.2){};
				
				\node[v:mainempty] () at (0,2){};
				\node () at (0,2.4){$v_{P}$};
				
				\node[v:main] () at (4.5,2){};
				\node () at (4.5,2.4){$u$};
				
				\node[v:mainempty] () at (4.5,-1){};
				\node () at (4.7,-1.4){$v_{Q_{2}}$};
				
				\node[v:mainempty] () at (1.2,-1){};
				\node () at (1.2,-1.4){$v_{Q_{1}}$};
				
				\node[v:main] () at (2.8,-1.5){};
				
				\node[v:maingray] () at (0.5,2){};
				\node[v:mainemptygray] () at (1,2){};
				\node[v:maingreen] () at (1.5,2){};
				\node[v:mainemptygreen] () at (2,2){};
				\node[v:maingreen] () at (2.5,2){};
				\node[v:mainemptygreen] () at (3,2){};
				\node[v:maingray] () at (3.5,2){};
				\node[v:mainemptygray] () at (4,2){};
				
				\node[v:maingray] () at (4.5,0){};
				\node[v:mainemptygray] () at (4.5,1){};
				
				\node[v:mainemptygray] () at (0,0.8){};
				\node[v:maingray] () at (0,1.4){};
				
				\node[v:mainemptygray] () at (1.5,-0.2){};
				\node[v:maingray] () at (3,-0.6){};
				
				\node[Amethyst] () at (2.3,1){$F$};
				
				\node[v:maingray] () at (0.8,-0.6){};
				\node[v:mainemptygray] () at (0.4,-0.2){};

				\node[v:mainemptygray] () at (3.37,-1.33){};
				\node[v:maingray] () at (3.93,-1.16){};
				
				\node[v:maingray] () at (1.73,-1.16){};
				\node[v:mainemptygray] () at (2.23,-1.33){};
				
				\node[v:maingray] () at (2.3,0){};
				\node[v:mainemptygray] () at (3.4,1){};
				
				\node[v:maingray] () at (0.62,1.22){};
				\node[v:mainemptygray] () at (2.18,-0.72){};
				
				\begin{pgfonlayer}{background}
					
					\draw[e:marker,BrightUbe] (2.8,-1.5) -- (0,2);
					\draw[e:marker,BrightUbe] (0,2) -- (4.5,2);
					\draw[e:marker,BrightUbe,bend left=25] (2,2) to (3,-0.6);
					\draw[e:marker,BrightUbe] (0,0.2) -- (3,-0.6);
					
					\draw[e:main,gray] (0,2) -- (0.5,2);
					\draw[e:main,gray] (1,2) -- (1.5,2);
					\draw[e:main,myGreen] (2,2) -- (2.5,2);
					\draw[e:main,gray] (3,2) -- (3.5,2);
					\draw[e:main,gray] (4,2) -- (4.5,2);
					
					\draw[e:coloredborder] (0.5,2) -- (1,2);
					\draw[e:coloredthin,color=BostonUniversityRed] (0.5,2) -- (1,2);
					
					\draw[e:coloredborder] (1.5,2) -- (2,2);
					\draw[e:coloredthin,color=BostonUniversityRed] (1.5,2) -- (2,2);
					
					\draw[e:coloredborder] (2.5,2) -- (3,2);
					\draw[e:coloredthin,color=BostonUniversityRed] (2.5,2) -- (3,2);

					\draw[e:coloredborder] (3.5,2) -- (4,2);
					\draw[e:coloredthin,color=BostonUniversityRed] (3.5,2) -- (4,2);
					
					\draw[e:main,blue,densely dashed,bend left=45] (1.5,2) to (3,2);
					
					\draw[e:coloredborder] (0,0.8) -- (0,1.4);
					\draw[e:coloredthin,color=BostonUniversityRed] (0,0.8) -- (0,1.4);
					\draw[e:main,gray](0,2) -- (0,1.4);
					\draw[e:main,gray] (0,0.8) -- (0,0.2);
					
					\draw[e:coloredborder] (4.5,1) -- (4.5,2);
					\draw[e:coloredthin,color=BostonUniversityRed] (4.5,1) -- (4.5,2);
					\draw[e:coloredborder] (4.5,-1) -- (4.5,0);
					\draw[e:coloredthin,color=BostonUniversityRed] (4.5,-1) -- (4.5,0);
					\draw[e:main,gray] (4.5,0) -- (4.5,1);
					
					\draw[e:main,gray] (0,0.2) -- (1.5,-0.2);
					\draw[e:main,gray] (3,-0.6) -- (4.5,-1);
					\draw[e:coloredborder] (1.5,-0.2) -- (3,-0.6);
					\draw[e:coloredthin,color=BostonUniversityRed] (1.5,-0.2) -- (3,-0.6);
					
					\draw[e:coloredborder] (0,0.2) -- (0.4,-0.2);
					\draw[e:coloredthin,color=BostonUniversityRed] (0,0.2) -- (0.4,-0.2);
					\draw[e:coloredborder] (0.8,-0.6) -- (1.2,-1);
					\draw[e:coloredthin,color=BostonUniversityRed] (0.8,-0.6) -- (1.2,-1);
					\draw[e:main,gray] (0.4,-0.2) -- (0.8,-0.6); 
					
					\draw[e:main,gray]   (2.8,-1.5) -- (3.37,-1.33);
					\draw[e:main,gray] (3.93,-1.16) -- (4.5,-1);
					\draw[e:coloredborder] (3.37,-1.33) -- (3.93,-1.16);
					\draw[e:coloredthin,color=BostonUniversityRed] (3.37,-1.33) -- (3.93,-1.16);

					\draw[e:main,gray] (1.2,-1) -- (1.73,-1.16);
					\draw[e:main,gray] (2.23,-1.33) -- (2.8,-1.5);
					\draw[e:coloredborder] (1.73,-1.16) -- (2.23,-1.33);
					\draw[e:coloredthin,color=BostonUniversityRed] (1.73,-1.16) -- (2.23,-1.33);
					
					\draw[e:main,gray] (3.4,1) -- (4.5,2);
					\draw[e:main,gray] (1.2,-1) -- (2.3,0);
					\draw[e:coloredborder] (2.3,0) -- (3.4,1);
					\draw[e:coloredthin,color=BostonUniversityRed] (2.3,0) -- (3.4,1);   
					
					\draw[e:coloredborder] (0,2) -- (0.62,1.22);
					\draw[e:coloredthin,color=BostonUniversityRed] (0,2) -- (0.62,1.22);
					\draw[e:coloredborder] (2.18,-0.72) -- (2.8,-1.5);
					\draw[e:coloredthin,color=BostonUniversityRed] (2.18,-0.72) -- (2.8,-1.5);
					\draw[e:main,gray] (0.62,1.22) -- (2.18,-0.72);
					
					\draw[e:main,Amethyst,bend left=25] (2,2) to (3,-0.6);

				\end{pgfonlayer}
				
			\end{tikzpicture}
		\end{subfigure}
		
		\begin{subfigure}[t]{0.3\textwidth}
			\begin{tikzpicture}[scale=0.7]
				\pgfdeclarelayer{background}
				\pgfdeclarelayer{foreground}
				\pgfsetlayers{background,main,foreground}
				
				\node[blue] () at (3,2.4){$C$};
				\node[gray] () at (1,2.4){$P$};
				\node[gray] () at (5,0.5){$Q_{2}$};
				\node[gray] () at (3.5,0.5){$Q_{1}$};
				
				\node () at (2.25,-2.5){$Q_{1}$ gets shortened};
				
				\node[v:main] () at (0,0.2){};
				
				\node[v:mainempty] () at (0,2){};
				\node () at (0,2.4){$v_{P}$};
				
				\node[v:main] () at (4.5,2){};
				\node () at (4.5,2.4){$u$};
				
				\node[v:mainempty] () at (4.5,-1){};
				\node () at (4.7,-1.4){$v_{Q_{2}}$};
				
				\node[v:mainempty] () at (1.2,-1){};
				\node () at (1.2,-1.4){$v_{Q_{1}}$};
				
				\node[v:main] () at (2.8,-1.5){};
				
				\node[v:maingray] () at (0.5,2){};
				\node[v:mainemptygray] () at (1,2){};
				\node[v:maingreen] () at (1.5,2){};
				\node[v:mainemptygreen] () at (2,2){};
				\node[v:maingreen] () at (2.5,2){};
				\node[v:mainemptygreen] () at (3,2){};
				\node[v:maingray] () at (3.5,2){};
				\node[v:mainemptygray] () at (4,2){};
				
				\node[v:maingray] () at (4.5,0){};
				\node[v:mainemptygray] () at (4.5,1){};
				
				\node[v:mainemptygray] () at (0,0.8){};
				\node[v:maingray] () at (0,1.4){};
				
				\node[v:mainemptygray] () at (1.5,-0.2){};
				\node[v:maingray] () at (3,-0.6){};
				
				\node[v:maingray] () at (0.8,-0.6){};
				\node[v:mainemptygray] () at (0.4,-0.2){};

				\node[v:mainemptygray] () at (3.37,-1.33){};
				\node[v:maingray] () at (3.93,-1.16){};
				
				\node[v:maingray] () at (1.73,-1.16){};
				\node[v:mainemptygray] () at (2.23,-1.33){};
				
				\node[v:maingray] () at (2.3,0){};
				\node[v:mainemptygray] () at (3.4,1){};
				
				\node[v:maingray] () at (0.62,1.22){};
				\node[v:mainemptygray] () at (2.18,-0.72){};
				
				\node[Amethyst] () at (1.5,1.2){$F$};
				
				\begin{pgfonlayer}{background}
					
					\draw[e:marker,BrightUbe,bend left=25] (0,1.4) to (3.4,1);
					\draw[e:marker,myGreen,bend left=25] (0,1.4) to (3.4,1);
					
					\draw[e:marker,BrightUbe] (1.2,-1) -- (4.5,2);
					\draw[e:marker,BrightUbe] (1.2,-1) -- (2.8,-1.5);
					
					\draw[e:marker,myGreen] (0,2) -- (0,0.2);
					\draw[e:marker,myGreen] (0,0.2) -- (4.5,-1);
					
					\draw[e:main,gray] (0,2) -- (0.5,2);
					\draw[e:main,gray] (1,2) -- (1.5,2);
					\draw[e:main,myGreen] (2,2) -- (2.5,2);
					\draw[e:main,gray] (3,2) -- (3.5,2);
					\draw[e:main,gray] (4,2) -- (4.5,2);
					
					\draw[e:coloredborder] (0.5,2) -- (1,2);
					\draw[e:coloredthin,color=BostonUniversityRed] (0.5,2) -- (1,2);
					
					\draw[e:coloredborder] (1.5,2) -- (2,2);
					\draw[e:coloredthin,color=BostonUniversityRed] (1.5,2) -- (2,2);
					
					\draw[e:coloredborder] (2.5,2) -- (3,2);
					\draw[e:coloredthin,color=BostonUniversityRed] (2.5,2) -- (3,2);

					\draw[e:coloredborder] (3.5,2) -- (4,2);
					\draw[e:coloredthin,color=BostonUniversityRed] (3.5,2) -- (4,2);
					
					\draw[e:main,blue,densely dashed,bend left=45] (1.5,2) to (3,2);
					
					\draw[e:coloredborder] (0,0.8) -- (0,1.4);
					\draw[e:coloredthin,color=BostonUniversityRed] (0,0.8) -- (0,1.4);
					\draw[e:main,gray](0,2) -- (0,1.4);
					\draw[e:main,gray] (0,0.8) -- (0,0.2);
					
					\draw[e:coloredborder] (4.5,0) -- (4.5,1);
					\draw[e:coloredthin,color=BostonUniversityRed] (4.5,0) -- (4.5,1);
					\draw[e:main,gray](4.5,1) -- (4.5,2);
					\draw[e:main,gray](4.5,-1) -- (4.5,0);
					
					\draw[e:coloredborder] (0,0.2) -- (1.5,-0.2);
					\draw[e:coloredthin,color=BostonUniversityRed] (0,0.2) -- (1.5,-0.2);
					\draw[e:coloredborder] (3,-0.6) -- (4.5,-1);
					\draw[e:coloredthin,color=BostonUniversityRed] (3,-0.6) -- (4.5,-1);
					\draw[e:main,gray] (1.5,-0.2) -- (3,-0.6);
					
					\draw[e:main,gray] (0,0.2) -- (0.4,-0.2);
					\draw[e:main,gray] (0.8,-0.6) -- (1.2,-1);
					\draw[e:coloredborder] (0.4,-0.2) -- (0.8,-0.6); 
					\draw[e:coloredthin,color=BostonUniversityRed] (0.4,-0.2) -- (0.8,-0.6); 
					
					\draw[e:main,gray]   (2.8,-1.5) -- (3.37,-1.33);
					\draw[e:main,gray] (3.93,-1.16) -- (4.5,-1);
					\draw[e:coloredborder] (3.37,-1.33) -- (3.93,-1.16);
					\draw[e:coloredthin,color=BostonUniversityRed] (3.37,-1.33) -- (3.93,-1.16);
					
					\draw[e:main,gray] (1.2,-1) -- (1.73,-1.16);
					\draw[e:main,gray] (2.23,-1.33) -- (2.8,-1.5);
					\draw[e:coloredborder] (1.73,-1.16) -- (2.23,-1.33);
					\draw[e:coloredthin,color=BostonUniversityRed] (1.73,-1.16) -- (2.23,-1.33);    
					
					\draw[e:coloredborder] (3.4,1) -- (4.5,2);
					\draw[e:coloredthin,color=BostonUniversityRed] (3.4,1) -- (4.5,2);
					\draw[e:coloredborder] (1.2,-1) -- (2.3,0);
					\draw[e:coloredthin,color=BostonUniversityRed] (1.2,-1) -- (2.3,0);
					\draw[e:main,gray] (2.3,0) -- (3.4,1);
					
					\draw[e:coloredborder] (0,2) -- (0.62,1.22);
					\draw[e:coloredthin,color=BostonUniversityRed] (0,2) -- (0.62,1.22);
					\draw[e:coloredborder] (2.18,-0.72) -- (2.8,-1.5);
					\draw[e:coloredthin,color=BostonUniversityRed] (2.18,-0.72) -- (2.8,-1.5);
					\draw[e:main,gray] (0.62,1.22) -- (2.18,-0.72);
					
					\draw[e:main,Amethyst,bend left=25] (0,1.4) to (3.4,1);
					
				\end{pgfonlayer}

			\end{tikzpicture}
		\end{subfigure}
		\begin{subfigure}[t]{0.3\textwidth}
			\begin{tikzpicture}[scale=0.7]
				\pgfdeclarelayer{background}
				\pgfdeclarelayer{foreground}
				\pgfsetlayers{background,main,foreground}
				
				\node[blue] () at (3,2.4){$C$};
				\node[gray] () at (1,2.4){$P$};
				\node[gray] () at (5,0.5){$Q_{2}$};
				\node[gray] () at (3.5,0.5){$Q_{1}$};
				
				\node () at (2.25,-2.5){$Q_{1}$ gets shortened};

				\node[v:main] () at (0,0.2){};
				
				\node[v:mainempty] () at (0,2){};
				\node () at (0,2.4){$v_{P}$};
				
				\node[v:main] () at (4.5,2){};
				\node () at (4.5,2.4){$u$};
				
				\node[v:mainempty] () at (4.5,-1){};
				\node () at (4.7,-1.4){$v_{Q_{2}}$};
				
				\node[v:mainempty] () at (1.2,-1){};
				\node () at (1.2,-1.4){$v_{Q_{1}}$};
				
				\node[v:main] () at (2.8,-1.5){};

				\node[v:maingray] () at (0.5,2){};
				\node[v:mainemptygray] () at (1,2){};
				\node[v:maingreen] () at (1.5,2){};
				\node[v:mainemptygreen] () at (2,2){};
				\node[v:maingreen] () at (2.5,2){};
				\node[v:mainemptygreen] () at (3,2){};
				\node[v:maingray] () at (3.5,2){};
				\node[v:mainemptygray] () at (4,2){};
				
				\node[v:maingray] () at (4.5,0){};
				\node[v:mainemptygray] () at (4.5,1){};
				
				\node[v:mainemptygray] () at (0,0.8){};
				\node[v:maingray] () at (0,1.4){};
				
				\node[v:mainemptygray] () at (1.5,-0.2){};
				\node[v:maingray] () at (3,-0.6){};
				
				\node[v:maingray] () at (0.8,-0.6){};
				\node[v:mainemptygray] () at (0.4,-0.2){};

				\node[v:mainemptygray] () at (3.37,-1.33){};
				\node[v:maingray] () at (3.93,-1.16){};
				
				\node[v:maingray] () at (1.73,-1.16){};
				\node[v:mainemptygray] () at (2.23,-1.33){};
				
				\node[v:maingray] () at (2.3,0){};
				\node[v:mainemptygray] () at (3.4,1){};
				
				\node[v:maingray] () at (0.62,1.22){};
				\node[v:mainemptygray] () at (2.18,-0.72){};
				
				\node[Amethyst] () at (3.3,-0.1){$F$};
				
				\begin{pgfonlayer}{background}
					
					\draw[e:marker,BrightUbe,bend right=40] (3.4,1) to (3,-0.6);
					\draw[e:marker,myGreen,bend right=40] (3.4,1) to (3,-0.6);
					
					\draw[e:marker,BrightUbe] (1.2,-1) -- (4.5,2);
					\draw[e:marker,BrightUbe] (1.2,-1) -- (2.8,-1.5);
					
					\draw[e:marker,myGreen] (0,2) -- (0,0.2);
					\draw[e:marker,myGreen] (0,0.2) -- (4.5,-1);
					
					\draw[e:main,gray] (0,2) -- (0.5,2);
					\draw[e:main,gray] (1,2) -- (1.5,2);
					\draw[e:main,myGreen] (2,2) -- (2.5,2);
					\draw[e:main,gray] (3,2) -- (3.5,2);
					\draw[e:main,gray] (4,2) -- (4.5,2);
					
					\draw[e:coloredborder] (0.5,2) -- (1,2);
					\draw[e:coloredthin,color=BostonUniversityRed] (0.5,2) -- (1,2);
					
					\draw[e:coloredborder] (1.5,2) -- (2,2);
					\draw[e:coloredthin,color=BostonUniversityRed] (1.5,2) -- (2,2);
					
					\draw[e:coloredborder] (2.5,2) -- (3,2);
					\draw[e:coloredthin,color=BostonUniversityRed] (2.5,2) -- (3,2);

					\draw[e:coloredborder] (3.5,2) -- (4,2);
					\draw[e:coloredthin,color=BostonUniversityRed] (3.5,2) -- (4,2);
					
					\draw[e:main,blue,densely dashed,bend left=45] (1.5,2) to (3,2);
					
					\draw[e:coloredborder] (0,2) -- (0,1.4);
					\draw[e:coloredthin,color=BostonUniversityRed] (0,2) -- (0,1.4);
					\draw[e:coloredborder] (0,0.8) -- (0,0.2);
					\draw[e:coloredthin,color=BostonUniversityRed] (0,0.8) -- (0,0.2);
					\draw[e:main,gray] (0,0.8) -- (0,1.4);
					
					\draw[e:coloredborder] (4.5,0) -- (4.5,1);
					\draw[e:coloredthin,color=BostonUniversityRed] (4.5,0) -- (4.5,1);
					\draw[e:main,gray](4.5,1) -- (4.5,2);
					\draw[e:main,gray](4.5,-1) -- (4.5,0);
					
					\draw[e:main,gray] (0,0.2) -- (1.5,-0.2);
					\draw[e:main,gray] (3,-0.6) -- (4.5,-1);
					\draw[e:coloredborder] (1.5,-0.2) -- (3,-0.6);
					\draw[e:coloredthin,color=BostonUniversityRed] (1.5,-0.2) -- (3,-0.6);
					
					\draw[e:main,gray] (0,0.2) -- (0.4,-0.2);
					\draw[e:main,gray] (0.8,-0.6) -- (1.2,-1);
					\draw[e:coloredborder] (0.4,-0.2) -- (0.8,-0.6); 
					\draw[e:coloredthin,color=BostonUniversityRed] (0.4,-0.2) -- (0.8,-0.6); 
					
					\draw[e:coloredborder] (2.8,-1.5) -- (3.37,-1.33);
					\draw[e:coloredthin,color=BostonUniversityRed] (2.8,-1.5) -- (3.37,-1.33);
					\draw[e:coloredborder] (3.93,-1.16) -- (4.5,-1);
					\draw[e:coloredthin,color=BostonUniversityRed] (3.93,-1.16) -- (4.5,-1);
					\draw[e:main,gray] (3.37,-1.33) -- (3.93,-1.16);

					\draw[e:main,gray] (1.2,-1) -- (1.73,-1.16);
					\draw[e:main,gray] (2.23,-1.33) -- (2.8,-1.5);
					\draw[e:coloredborder] (1.73,-1.16) -- (2.23,-1.33);
					\draw[e:coloredthin,color=BostonUniversityRed] (1.73,-1.16) -- (2.23,-1.33);
					
					\draw[e:coloredborder] (3.4,1) -- (4.5,2);
					\draw[e:coloredthin,color=BostonUniversityRed] (3.4,1) -- (4.5,2);
					\draw[e:coloredborder] (1.2,-1) -- (2.3,0);
					\draw[e:coloredthin,color=BostonUniversityRed] (1.2,-1) -- (2.3,0);
					\draw[e:main,gray] (2.3,0) -- (3.4,1);
					
					\draw[e:main,gray] (0,2) -- (0.62,1.22);
					\draw[e:main,gray] (2.18,-0.72) -- (2.8,-1.5);
					\draw[e:coloredborder] (0.62,1.22) -- (2.18,-0.72);
					\draw[e:coloredthin,color=BostonUniversityRed] (0.62,1.22) -- (2.18,-0.72);
					
					\draw[e:main,Amethyst,bend right=40] (3.4,1) to (3,-0.6);
					
				\end{pgfonlayer}
				
			\end{tikzpicture}
		\end{subfigure}
		\begin{subfigure}[t]{0.3\textwidth}
			\begin{tikzpicture}[scale=0.7]
				\pgfdeclarelayer{background}
				\pgfdeclarelayer{foreground}
				\pgfsetlayers{background,main,foreground}
				
				\node[blue] () at (3,2.4){$C$};
				\node[gray] () at (1,2.4){$P$};
				\node[gray] () at (5,0.5){$Q_{2}$};
				\node[gray] () at (3.5,0.5){$Q_{1}$};
				
				\node () at (2.25,-2.5){$Q_{1}$ gets shortened};

				\node[v:main] () at (0,0.2){};
				
				\node[v:mainempty] () at (0,2){};
				\node () at (0,2.4){$v_{P}$};
				
				\node[v:main] () at (4.5,2){};
				\node () at (4.5,2.4){$u$};
				
				\node[v:mainempty] () at (4.5,-1){};
				\node () at (4.7,-1.4){$v_{Q_{2}}$};
				
				\node[v:mainempty] () at (1.2,-1){};
				\node () at (1.2,-1.4){$v_{Q_{1}}$};
				
				\node[v:main] () at (2.8,-1.5){};
				
				\node[v:maingray] () at (0.5,2){};
				\node[v:mainemptygray] () at (1,2){};
				\node[v:maingreen] () at (1.5,2){};
				\node[v:mainemptygreen] () at (2,2){};
				\node[v:maingreen] () at (2.5,2){};
				\node[v:mainemptygreen] () at (3,2){};
				\node[v:maingray] () at (3.5,2){};
				\node[v:mainemptygray] () at (4,2){};
				
				\node[v:maingray] () at (4.5,0){};
				\node[v:mainemptygray] () at (4.5,1){};
				
				\node[v:mainemptygray] () at (0,0.8){};
				\node[v:maingray] () at (0,1.4){};
				
				\node[v:mainemptygray] () at (1.5,-0.2){};
				\node[v:maingray] () at (3,-0.6){};
				
				\node[v:maingray] () at (0.8,-0.6){};
				\node[v:mainemptygray] () at (0.4,-0.2){};
				
				\node[Amethyst] () at (1.7,1.2){$F$};
				
				\node[v:mainemptygray] () at (3.37,-1.33){};
				\node[v:maingray] () at (3.93,-1.16){};
				
				\node[v:maingray] () at (1.73,-1.16){};
				\node[v:mainemptygray] () at (2.23,-1.33){};
				
				\node[v:maingray] () at (2.3,0){};
				\node[v:mainemptygray] () at (3.4,1){};
				
				\node[v:maingray] () at (0.62,1.22){};
				\node[v:mainemptygray] () at (2.18,-0.72){};
				
				\begin{pgfonlayer}{background}
					
					\draw[e:marker,BrightUbe] (0,0.2) -- (0.8,-0.7);
					\draw[e:marker,BrightUbe,bend right=40] (3.4,1) to (0.8,-0.6);
					\draw[e:marker,BrightUbe] (1.2,-1) -- (4.5,2);
					\draw[e:marker,BrightUbe] (1.2,-1) -- (2.8,-1.5);
					
					\draw[e:main,gray] (0,2) -- (0.5,2);
					\draw[e:main,gray] (1,2) -- (1.5,2);
					\draw[e:main,myGreen] (2,2) -- (2.5,2);
					\draw[e:main,gray] (3,2) -- (3.5,2);
					\draw[e:main,gray] (4,2) -- (4.5,2);
					
					\draw[e:coloredborder] (0.5,2) -- (1,2);
					\draw[e:coloredthin,color=BostonUniversityRed] (0.5,2) -- (1,2);
					
					\draw[e:coloredborder] (1.5,2) -- (2,2);
					\draw[e:coloredthin,color=BostonUniversityRed] (1.5,2) -- (2,2);
					
					\draw[e:coloredborder] (2.5,2) -- (3,2);
					\draw[e:coloredthin,color=BostonUniversityRed] (2.5,2) -- (3,2);

					\draw[e:coloredborder] (3.5,2) -- (4,2);
					\draw[e:coloredthin,color=BostonUniversityRed] (3.5,2) -- (4,2);
					
					\draw[e:main,blue,densely dashed,bend left=45] (1.5,2) to (3,2);
					
					\draw[e:coloredborder] (0,2) -- (0,1.4);
					\draw[e:coloredthin,color=BostonUniversityRed] (0,2) -- (0,1.4);
					\draw[e:coloredborder] (0,0.8) -- (0,0.2);
					\draw[e:coloredthin,color=BostonUniversityRed] (0,0.8) -- (0,0.2);
					\draw[e:main,gray] (0,0.8) -- (0,1.4);
					
					\draw[e:coloredborder] (4.5,0) -- (4.5,1);
					\draw[e:coloredthin,color=BostonUniversityRed] (4.5,0) -- (4.5,1);
					\draw[e:main,gray](4.5,1) -- (4.5,2);
					\draw[e:main,gray](4.5,-1) -- (4.5,0);
					
					\draw[e:main,gray] (0,0.2) -- (1.5,-0.2);
					\draw[e:main,gray] (3,-0.6) -- (4.5,-1);
					\draw[e:coloredborder] (1.5,-0.2) -- (3,-0.6);
					\draw[e:coloredthin,color=BostonUniversityRed] (1.5,-0.2) -- (3,-0.6);
					
					\draw[e:main,gray] (0,0.2) -- (0.4,-0.2);
					\draw[e:main,gray] (0.8,-0.6) -- (1.2,-1);
					\draw[e:coloredborder] (0.4,-0.2) -- (0.8,-0.6); 
					\draw[e:coloredthin,color=BostonUniversityRed] (0.4,-0.2) -- (0.8,-0.6); 
					
					\draw[e:coloredborder] (2.8,-1.5) -- (3.37,-1.33);
					\draw[e:coloredthin,color=BostonUniversityRed] (2.8,-1.5) -- (3.37,-1.33);
					\draw[e:coloredborder] (3.93,-1.16) -- (4.5,-1);
					\draw[e:coloredthin,color=BostonUniversityRed] (3.93,-1.16) -- (4.5,-1);
					\draw[e:main,gray] (3.37,-1.33) -- (3.93,-1.16);

					\draw[e:main,gray] (1.2,-1) -- (1.73,-1.16);
					\draw[e:main,gray] (2.23,-1.33) -- (2.8,-1.5);
					\draw[e:coloredborder] (1.73,-1.16) -- (2.23,-1.33);
					\draw[e:coloredthin,color=BostonUniversityRed] (1.73,-1.16) -- (2.23,-1.33);
					
					\draw[e:coloredborder] (3.4,1) -- (4.5,2);
					\draw[e:coloredthin,color=BostonUniversityRed] (3.4,1) -- (4.5,2);
					\draw[e:coloredborder] (1.2,-1) -- (2.3,0);
					\draw[e:coloredthin,color=BostonUniversityRed] (1.2,-1) -- (2.3,0);
					\draw[e:main,gray] (2.3,0) -- (3.4,1);
					
					\draw[e:main,gray] (0,2) -- (0.62,1.22);
					\draw[e:main,gray] (2.18,-0.72) -- (2.8,-1.5);
					\draw[e:coloredborder] (0.62,1.22) -- (2.18,-0.72);
					\draw[e:coloredthin,color=BostonUniversityRed] (0.62,1.22) -- (2.18,-0.72);
					
					\draw[e:main,Amethyst,bend right=40] (3.4,1) to (0.8,-0.6);
					
				\end{pgfonlayer}
				
			\end{tikzpicture}
		\end{subfigure}
		\caption{Examples of the cases occuring in the proof of \cref{lemma:reducedK33}} 
		\label{fig:firstLreductions}
	\end{figure}
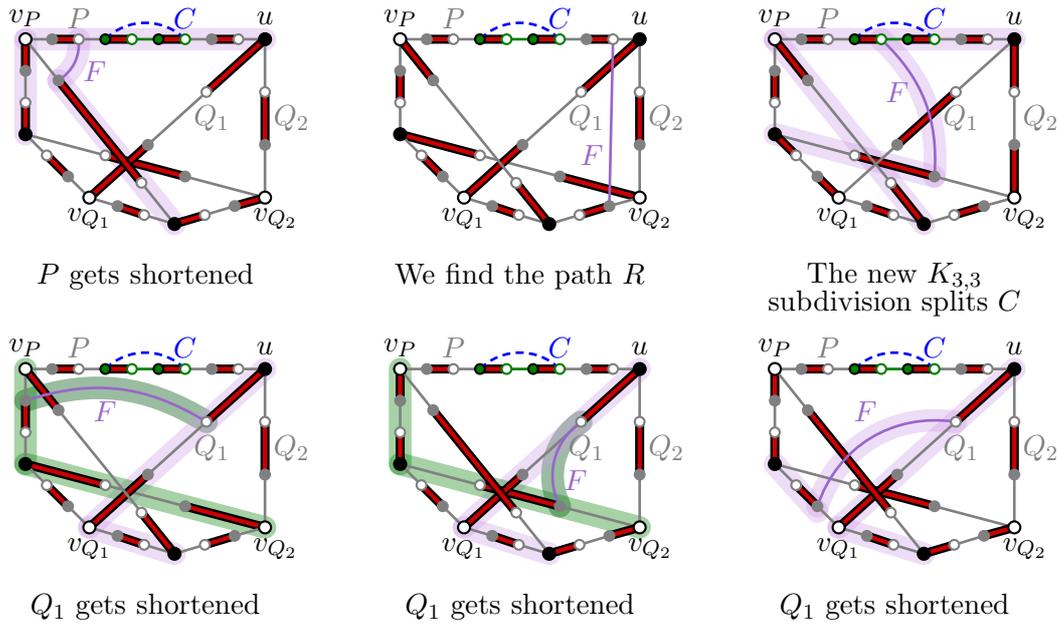
	
	\textbf{Case 1}: $y\in\V{T_{V_2}-C}$
	
	Suppose $P_x$ contains the vertex $v_P$, let $W$ be the third bisubdivided edge with endpoint $v_P$.
	Then choose $M$ such that both $P$ and $P_x$ are internally $M$-conformal.
	Now we may replace the three subdivided edges of $L$ with $v_P$ as an endpoint by the following three $M$-alternating paths in order to obtain an $M$-conformal bisubdivision $L'$ of $K_{3,3}$, where the subdivided edge $P'$ that contains $\V{C}$ is strictly shorter than $P$, thereby violating the minimal choice of $L$.
	We set $P'\coloneqq yPu$ and the other paths are $yPv_PW$ and $yFxP_xw_1$
	
	So we may assume $P_x$ does not contain $v_P$ which means that there is $i\in[1,2]$ such that $P_x$ has $v_{Q_i}$ as an endpoint.
	We now aim for a bisubdivision $L'$ of $K_{3,3}$ in which both $x$ and $y$ are vertices of degree three.
	As before, the bisubdivided edge of $L'$ that contains $\V{C}$ will be shorter than $P$ and thus provide a contradiction.
	We now replace the paths $P$, $P_x$, $E_{v_Pw_1}$, $E_{v_Pw_2}$, and $E_{w_1v_{Q_{3-i}}}$ by the paths $yPu$, $F$, $yPv_PE_{v_Pw_2}$, $xP_xv_{Q_i}$, and $xP_xw_1E_{w_1v_{Q_{3-i}}}$ to obtain the graph $L'$.
	Since $L$ is a bisubdivision of $K_{3,3}$, we may choose $M$ such that $L'$ is $M$-conformal and thus we are done with this case.
	
	\textbf{Case 2}: $y\in\V{T_A{V-1}-C}$
	
	In this case $F$ is a $V_1$-jump over $C$ and thus we are done immediately.
	
	\textbf{Case 3}: $y\in\V{C}$
	
	In essence, we can repeat the construction from the first case to obtain an $M$-conformal bisubdivision $L'$ of $K_{3,3}$.
	Since $y\in\V{C}$ we end up with some $L'$ in which the edges $ab$ and $a'b'$ occur on two different subdivided edges that share the endpoint $y$.
	Thus $L'$ splits $C$ and we can close this case.
	
	\textbf{Case 4}: $y\notin\V{P_x}$
	
	We may assume $y\in\V{Q_1}$ as $y\in\V{Q_2}$ can be handled analogously.
	Instead of $P$ as in the first case we reduce the length of $Q_1$ while maintaining the lengths of $P$ and $Q_2$ in order to obtain a contradiction.
	The main idea of the construction, however, remains the same as in the first case and thus we omit the exact construction here.
	In \cref{fig:firstLreductions}, the possible ways to obtain the new $K_{3,3}$-bisubdivision $L'$ are illustrated.
	
	Combining all of these cases, this means that $\CutG{L}{Y}$ cannot be a non-trivial tight cut.
	Since otherwise we are either done since we find a path $F$ that allows us to change $L$ into $L'$ which splits $C$, or $F$ is a $V_1$-jump over $C$.
	However, by construction $\Abs{\V{C}\cap Y}\geq 3$ and $\Abs{\V{L}\setminus Y}\geq 3$ and thus this is impossible.
	It follows that $L$ itself must already split $C$ and so we are done.
\end{proof}

\begin{lemma}\label{lemma:reducedK33withjump}
	Let $B$ be a $K_{3,3}$-containing brace and $C$ a $4$-cycle in $B$ such that there is no conformal cross over $C$ in $B$.
	If there exists a perfect matching $M$ of $B$ such that $C$ is $M$-conformal and there is an $M$-conformal bisubdivision $L$ of $K_{3,3}$ that has a $V_1$-jump over $C$, then there exists a perfect matching $M'$ of $B$ such that $C$ is $M'$-conformal and there is an $M'$-conformal bisubdivision $L'$ of $K_{3,3}$ that either splits $C$, or has both, a $V_1$-jump and a $V_2$-jump over $C$.
\end{lemma}

\begin{proof}
	The proof is a slight alteration of the proof of the previous lemma.
	Since $L$ has an $V_1$-jump over $C$, there exists a bisubdivided edge $P$ of $L$ such that $\V{C}\subseteq \V{P}$.
	Let $u\in V_1$ and $v\in V_2$ be the endpoints of $P$.
	Let $a_1$, $a_2$, $b_1$, and $b_2$ be the four degree three vertices of $L$ aside from $u$ and $v$ such that $a_1,a_2\in V_1$.
	Now let $Y\coloneqq \V{P-u}\cup \V{E_{va_1}-a_1}\cup\V{E_{va_2}-a_2}$.
	By the same arguments as in the previous lemma, $\CutG{L}{Y}$ must be a non-trivial tight cut.
	Similar to before we choose $L$ such that the tuple $\Brace{\Abs{\E{P}},\Abs{\E{E_{va_1}}\cup\E{E_{va_2}}}}$ is lexicographically minimised.
	By using the same case distinction as in the proof of \cref{lemma:reducedK33} we either reach a contradiction, find a conformal $K_{3,3}$-bisubdivision $L'$ that splits $C$, or the path $F$ yielded by \cref{lemma:tightcutsincofnromalsubgraphs} is a $V_2$-jump over $C$ in $L$.
	The major difference between this lemma and \cref{lemma:reducedK33} is, that we have to maintain the existence of a $V_1$-jump over $C$.
	In the technique from the proof of the previous lemma, there are two possible ways, the existence of an $V_1$-jump over $C$ in the newly constructed $K_{3,3}$-bisubdivision $L'$ is threatened\footnote{Note that in case $L'$ splits $C$ we are done.}.
	Let $R$ be an $V_1$-jump over $C$ for $L$.
	
	The easier to handle case is the one in which the newly found path $F$ in \textbf{Case 4} of the case distinction intersects $R$.
	However, since $R$ and $F$ are internally $M$-conformal, let $z$ be the first vertex of $R$ on $F$, then $Rz$ still is internally $M$-conformal, and thus in this case, $L'$ still has a $V_1$-jump over $C$.
	
	The more complicated case is a subcase of \textbf{Case 1}.
	Let $T_{V_1}$ be the shortest subpath of $P$ with one endpoint in $\V{C}$ and $u$ as its other endpoint.
	Similarly, let $T_{V_2}$ be the shortest subpath of $P$ with one endpoint in $\V{C}$ and $v$ as its other endpoint.
	If $F$ has its endpoint in $Y$ on the subpath of $T_{V_1}$ that connects $\V{C}$ to $R$, then no subpath of $R$ can be a $V_1$-jump over $C$ for $L'$.
	However, in this case, we have found a conformal $K_{3,3}$-bisubdivision where the path in which both $M$-edges of $C$ occur is shorter than in $L$.
	Among those bisubdivisions choose $L'$ to be one that lexicographically minimises $\Brace{\Abs{\E{P'}},\Abs{\E{E_{u'b'_1}}\cup\E{E_{u'b'_2}}}}$, where the vertices marked with a $'$ are those of $L'$ that naturally correspond to the vertices of $L$. similarly we define $P'$.
	By reapplying the case distinction of \cref{lemma:reducedK33} to $L'$ we either find a $K_{3,3}$-bisubdivision $L''$ that splits $C$, of we find a new $V_1$-jump over $C$ for $L'$ which would contradict our choice of $L$ in the first place since $\Abs{\E{P'}}<\Abs{\E{P}}$.
	Hence if we cannot find a conformal $K_{3,3}$-bisubdivision that splits $C$, we always find one that has both, a $V_1$-jump and a $V_2$-jump over $C$.
\end{proof}

\begin{lemma}\label{lemma:reducedK33mustsplit}
	Let $B$ be a $K_{3,3}$-containing brace and $C$ a $4$-cycle in $B$ such that there is no conformal cross over $C$ in $B$.
	Then there exists a perfect matching $M$ of $B$ such that $C$ is $M$-conformal and there is an $M$-conformal bisubdivision $L$ of $K_{3,3}$ that splits $C$.
\end{lemma}

\begin{proof}
	By \cref{lemma:reducedK33} we either find a conformal bisubdivision $L'$ of $K_{3,3}$ that splits $C$, in which case we are done, or we find one with a $V_1$-jump over $C$.
	Then \cref{lemma:reducedK33withjump} might again yield the existence of a conformal bisubdivision $L$ of $K_{3,3}$ that splits $C$ if it does not we find $M$ and $L$ such that $L$ has a $V_1$-jump $R_{V_1}$ and a $V_2$-jump $R_{V_2}$ over $C$.
	Let $P$ be the bisubdivided edge of $L$ that contains the vertices of $C$.
	We may assume $L$ to be a conformal $K_{3,3}$-bisubdivision that minimises the length of $P$ among all conformal bisubdivisions of $K_{3,3}$ for which $ab$ and $a'b'$ occur on a single bisubdivided edge $P$.
	By \cref{lemma:untangletwopaths} we may assume that $R_{V_1}$ and $R_{V_2}$ are either disjoint, or $R_{V_1}\cap R_{V_2}$ is an $M$-conformal path.
	For each $X\in\Set{V_1,V_2}$ let $v_X$ be the endpoint of $R_X$ that does not belong to the bisubdivided edge $P$.
	We have to consider the cases how $v_{V_1}$ and $v_{V_2}$ occur on the bisubdivided edges of $L$ and for each of these cases we need to look at $R_{V_1}$ and $R_{V_2}$ being disjoint or meeting in an $M$-conformal path.
	Let $u\in V_1$ and $v\in V_2$ be the endpoints of $P$ and let $a_1,a_2\in V_1$, $b_1,b_2\in V_2$ be the remaining four vertices of degree three in $L$.
	Then $R_{V_1}$ cannot have an endpoint on $E_{ub_1}$ or $E_{ub_2}$, while $R_{V_2}$ cannot have an endpoint on $E_{va_1}$ or $E_{va_2}$.
	Our goal is to show that $R_{V_1}$ and $R_{V_2}$ can be used to produce a contradiction to the choice of $L$ with respect to the minimality of $P$.
	
	Let us first consider the cases where at least one of $R_{V_1}$ and $R_{V_2}$ has an endpoint on one of the $E_{ub_i}$ or $E_{va_i}$.
	By symmetry, we just need to consider the case where $R_{V_1}$ meets $E_{va_1}$ and $R_{V_2}$ meets $E_{ub_1}$, and the case where $R_{V_1}$ meets $E_{va_1}$ while $R_{V_2}$ meets an arbitrary other bisubdivided edge of $L$, say $E_{a_1b_1}$.
	Please note that in all of these cases, it does not play a role whether $ab$ and $a'b'$ occur in reverse on $P$ or not.
	Hence we only treat the case where $ab$ and $a'b'$ are not reversed.
	In \cref{fig:shrinkingPno1} we give exemplary constructions of a new conformal $K_{3,3}$-bisubdivision $L'$ which still has a bisubdivided edge $P'$ containing $ab$ and $a'b'$, but with $\Abs{\E{P'}}<\Abs{\E{P}}$ this contradicts the choice of $L$.

	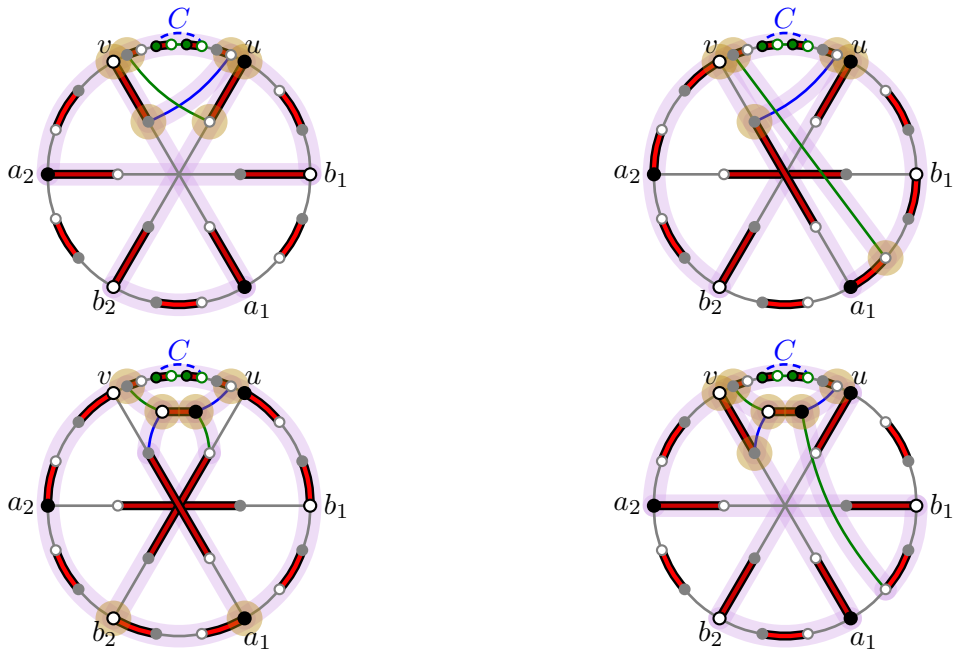
\begin{figure}[h!]
		\centering
		\begin{subfigure}{0.49\textwidth}
			\centering
			\begin{tikzpicture}[scale=1.15]
				\pgfdeclarelayer{background}
				\pgfdeclarelayer{foreground}
				\pgfsetlayers{background,main,foreground}
				
				\node[v:marked] () at (113.34:15mm){};
				\node[v:marked] () at (120:15mm){};
				\node[v:marked] () at (60:15mm){};
				\node[v:marked] () at (60:7mm){};
				\node[v:marked] () at (120:7mm){};
				\node[v:marked] () at (66.66:15mm){};
				
				\foreach \x in {0,2,4}
				{
					\node[v:mainempty] () at (\x*60:15mm){};
					\node[v:maingray] () at (\x*60:7mm){};
				}
				\foreach \x in {1,3,5}
				{
					\node[v:main] () at (\x*60:15mm){};  
					\node[v:mainemptygray] () at (\x*60:7mm){};
				}
				
				\node[v:maingray] () at (20:15mm){};
				\node[v:maingray] () at (140:15mm){};
				\node[v:maingray] () at (220:15mm){};
				\node[v:maingray] () at (260:15mm){};
				\node[v:maingray] () at (340:15mm){};
				
				\node[v:mainemptygray] () at (66.66:15mm){};
				\node[v:maingray] () at (73.32:15mm){};
				\node[v:mainemptygreen] () at (80.01:15mm){};
				\node[v:maingreen] () at (86.68:15mm){};
				\node[v:mainemptygreen] () at (93.35:15mm){};
				\node[v:maingreen] () at (100.01:15mm){};
				\node[v:mainemptygray] () at (106.68:15mm){};
				\node[v:maingray] () at (113.34:15mm){};
				
				\node[v:mainemptygray] () at (40:15mm){};
				\node[v:mainemptygray] () at (160:15mm){};
				\node[v:mainemptygray] () at (200:15mm){};
				\node[v:mainemptygray] () at (280:15mm){};
				\node[v:mainemptygray] () at (320:15mm){};
				
				\node () at (0:18mm){$b_{1}$};
				\node () at (60:17mm){$u$};
				\node () at (120:17mm){$v$};
				\node () at (180:18mm){$a_{2}$};
				\node () at (240:17mm){$b_{2}$};
				\node () at (300:18mm){$a_{1}$};
				
				\node[blue] () at (90:18mm){$C$};
				
				\begin{pgfonlayer}{background}
					
					\draw[e:marker,BrightUbe,bend right=15] (120:7mm) to (66.66:15mm);    
					\draw[e:marker,BrightUbe] (0:15mm) -- (180:15mm);
					\draw[e:marker,BrightUbe] (60:15mm) -- (240:15mm);
					\draw[e:marker,BrightUbe] (120:15mm) -- (300:15mm);
					\draw[e:marker,BrightUbe,bend right=28] (240:15mm) to (300:15mm);
					\draw[e:marker,BrightUbe,bend right=28] (0:15mm) to (60:15mm) to (120:15mm) to (180:15mm);
					
					\draw[e:main,gray] (60:15mm) to (66.66:15mm){};
					\draw[e:coloredborder] (66.66:15mm) to (73.32:15mm){};
					\draw[e:coloredthin,color=BostonUniversityRed] (66.66:15mm) to (73.32:15mm){};
					\draw[e:main,gray] (73.32:15mm) to (80.01:15mm){};
					\draw[e:coloredborder] (80.01:15mm) to (86.68:15mm){};
					\draw[e:coloredthin,color=BostonUniversityRed] (80.01:15mm) to (86.68:15mm){};
					\draw[e:main,myGreen] (86.68:15mm) to (93.35:15mm){};
					\draw[e:coloredborder] (93.35:15mm) to (100.01:15mm){};
					\draw[e:coloredthin,color=BostonUniversityRed] (93.35:15mm) to (100.01:15mm){};     
					\draw[e:main,gray] (100.01:15mm) to (106.68:15mm){};
					\draw[e:coloredborder] (106.68:15mm) to (113.34:15mm){};
					\draw[e:coloredthin,color=BostonUniversityRed] (106.68:15mm) to (113.34:15mm){}; 
					\draw[e:main,gray] (113.34:15mm) to (120:15mm){};
					
					\draw[e:main,blue,densely dashed,bend right=90] (80.01:15mm) to (100.01:15mm);
					
					\draw[e:main,gray,bend right=10] (0:15mm) to (20:15mm);
					\draw[e:coloredborder,bend right=10] (20:15mm) to (40:15mm);
					\draw[e:coloredthin,red,bend right=10] (20:15mm) to (40:15mm);
					\draw[e:main,gray,bend right=10] (40:15mm) to (60:15mm);
					
					\draw[e:main,gray,bend right=10] (120:15mm) to (140:15mm);
					\draw[e:coloredborder,bend right=10] (140:15mm) to (160:15mm);
					\draw[e:coloredthin,red,bend right=10] (140:15mm) to (160:15mm);
					\draw[e:main,gray,bend right=10] (160:15mm) to (180:15mm);
					
					\draw[e:main,gray,bend right=10] (180:15mm) to (200:15mm);
					\draw[e:coloredborder,bend right=10] (200:15mm) to (220:15mm);
					\draw[e:coloredthin,red,bend right=10] (200:15mm) to (220:15mm);
					\draw[e:main,gray,bend right=10] (220:15mm) to (240:15mm);
					
					\draw[e:main,gray,bend right=10] (240:15mm) to (260:15mm);
					\draw[e:coloredborder,bend right=10] (260:15mm) to (280:15mm);
					\draw[e:coloredthin,red,bend right=10] (260:15mm) to (280:15mm);
					\draw[e:main,gray,bend right=10] (280:15mm) to (300:15mm);
					
					\draw[e:main,gray,bend right=10] (300:15mm) to (320:15mm);
					\draw[e:coloredborder,bend right=10] (320:15mm) to (340:15mm);
					\draw[e:coloredthin,red,bend right=10] (320:15mm) to (340:15mm);
					\draw[e:main,gray,bend right=10] (340:15mm) to (0:15mm);
					
					\draw[e:coloredborder] (180:15mm) -- (180:7mm);
					\draw[e:coloredthin,color=BostonUniversityRed] (180:15mm) -- (180:7mm);
					\draw[e:coloredborder] (0:15mm) -- (0:7mm);
					\draw[e:coloredthin,color=BostonUniversityRed] (0:15mm) -- (0:7mm);
					\draw[e:main,gray] (180:7mm) -- (0:7mm);
					
					\draw[e:coloredborder] (240:15mm) -- (240:7mm);
					\draw[e:coloredthin,color=BostonUniversityRed] (240:15mm) -- (240:7mm);
					\draw[e:coloredborder] (60:15mm) -- (60:7mm);
					\draw[e:coloredthin,color=BostonUniversityRed] (60:15mm) -- (60:7mm);
					\draw[e:main,gray] (240:7mm) -- (60:7mm);
					
					\draw[e:coloredborder] (300:15mm) -- (300:7mm);
					\draw[e:coloredthin,color=BostonUniversityRed] (300:15mm) -- (300:7mm);
					\draw[e:coloredborder] (120:15mm) -- (120:7mm);
					\draw[e:coloredthin,color=BostonUniversityRed] (120:15mm) -- (120:7mm);
					\draw[e:main,gray] (300:7mm) -- (120:7mm);
					
					\draw[e:main,blue,bend right=15] (120:7mm) to (66.66:15mm);    
					\draw[e:main,myGreen,bend right=15] (113.34:15mm) to (60:7mm);

				\end{pgfonlayer}
			\end{tikzpicture}
		\end{subfigure}
		\begin{subfigure}{0.49\textwidth}
			\centering
			\begin{tikzpicture}[scale=1.15]
				\pgfdeclarelayer{background}
				\pgfdeclarelayer{foreground}
				\pgfsetlayers{background,main,foreground}
				
				\node[v:marked] () at (113.34:15mm){};
				\node[v:marked] () at (120:15mm){};
				\node[v:marked] () at (60:15mm){};
				\node[v:marked] () at (320:15mm){};
				\node[v:marked] () at (120:7mm){};
				\node[v:marked] () at (66.66:15mm){};
				
				\foreach \x in {0,2,4}
				{
					\node[v:mainempty] () at (\x*60:15mm){};
					\node[v:maingray] () at (\x*60:7mm){};
				}
				\foreach \x in {1,3,5}
				{
					\node[v:main] () at (\x*60:15mm){};  
					\node[v:mainemptygray] () at (\x*60:7mm){};
				}
				
				\node[v:maingray] () at (20:15mm){};
				\node[v:maingray] () at (140:15mm){};
				\node[v:maingray] () at (220:15mm){};
				\node[v:maingray] () at (260:15mm){};
				\node[v:maingray] () at (340:15mm){};
				
				\node[v:mainemptygray] () at (66.66:15mm){};
				\node[v:maingray] () at (73.32:15mm){};
				\node[v:mainemptygreen] () at (80.01:15mm){};
				\node[v:maingreen] () at (86.68:15mm){};
				\node[v:mainemptygreen] () at (93.35:15mm){};
				\node[v:maingreen] () at (100.01:15mm){};
				\node[v:mainemptygray] () at (106.68:15mm){};
				\node[v:maingray] () at (113.34:15mm){};
				
				\node[v:mainemptygray] () at (40:15mm){};
				\node[v:mainemptygray] () at (160:15mm){};
				\node[v:mainemptygray] () at (200:15mm){};
				\node[v:mainemptygray] () at (280:15mm){};
				\node[v:mainemptygray] () at (320:15mm){};
				
				\node () at (0:18mm){$b_{1}$};
				\node () at (60:17mm){$u$};
				\node () at (120:17mm){$v$};
				\node () at (180:18mm){$a_{2}$};
				\node () at (240:17mm){$b_{2}$};
				\node () at (300:18mm){$a_{1}$};
				
				\node[blue] () at (90:18mm){$C$};
				
				\begin{pgfonlayer}{background}
					
					\draw[e:marker,BrightUbe,bend right=15] (120:7mm) to (66.66:15mm);    
					\draw[e:marker,BrightUbe] (320:15mm) to (113.34:15mm);
					\draw[e:marker,BrightUbe] (60:15mm) -- (240:15mm);
					\draw[e:marker,BrightUbe] (120:15mm) -- (300:15mm);
					\draw[e:marker,BrightUbe,bend right=28] (300:15mm) to (0:15mm) to (60:15mm) to (120:15mm) to (180:15mm) to (240:15mm);
					
					\draw[e:main,gray] (60:15mm) to (66.66:15mm){};
					\draw[e:coloredborder] (66.66:15mm) to (73.32:15mm){};
					\draw[e:coloredthin,color=BostonUniversityRed] (66.66:15mm) to (73.32:15mm){};
					\draw[e:main,gray] (73.32:15mm) to (80.01:15mm){};
					\draw[e:coloredborder] (80.01:15mm) to (86.68:15mm){};
					\draw[e:coloredthin,color=BostonUniversityRed] (80.01:15mm) to (86.68:15mm){};
					\draw[e:main,myGreen] (86.68:15mm) to (93.35:15mm){};
					\draw[e:coloredborder] (93.35:15mm) to (100.01:15mm){};
					\draw[e:coloredthin,color=BostonUniversityRed] (93.35:15mm) to (100.01:15mm){};     
					\draw[e:main,gray] (100.01:15mm) to (106.68:15mm){};
					\draw[e:coloredborder] (106.68:15mm) to (113.34:15mm){};
					\draw[e:coloredthin,color=BostonUniversityRed] (106.68:15mm) to (113.34:15mm){}; 
					\draw[e:main,gray] (113.34:15mm) to (120:15mm){};
					
					\draw[e:main,blue,densely dashed,bend right=90] (80.01:15mm) to (100.01:15mm);
					
					\draw[e:main,gray,bend right=10] (0:15mm) to (20:15mm);
					\draw[e:coloredborder,bend right=10] (20:15mm) to (40:15mm);
					\draw[e:coloredthin,red,bend right=10] (20:15mm) to (40:15mm);
					\draw[e:main,gray,bend right=10] (40:15mm) to (60:15mm);
					
					\draw[e:coloredborder,bend right=10] (120:15mm) to (140:15mm);
					\draw[e:coloredthin,red,bend right=10] (120:15mm) to (140:15mm);
					\draw[e:main,gray,bend right=10] (140:15mm) to (160:15mm);
					\draw[e:coloredborder,bend right=10] (160:15mm) to (180:15mm);
					\draw[e:coloredthin,red,bend right=10] (160:15mm) to (180:15mm);
					
					\draw[e:main,gray,bend right=10] (180:15mm) to (200:15mm);
					\draw[e:coloredborder,bend right=10] (200:15mm) to (220:15mm);
					\draw[e:coloredthin,red,bend right=10] (200:15mm) to (220:15mm);
					\draw[e:main,gray,bend right=10] (220:15mm) to (240:15mm);
					
					\draw[e:main,gray,bend right=10] (240:15mm) to (260:15mm);
					\draw[e:coloredborder,bend right=10] (260:15mm) to (280:15mm);
					\draw[e:coloredthin,red,bend right=10] (260:15mm) to (280:15mm);
					\draw[e:main,gray,bend right=10] (280:15mm) to (300:15mm);
					
					\draw[e:coloredborder,bend right=10] (300:15mm) to (320:15mm);
					\draw[e:coloredthin,red,bend right=10] (300:15mm) to (320:15mm);
					\draw[e:main,gray,bend right=10] (320:15mm) to (340:15mm);
					\draw[e:coloredborder,bend right=10] (340:15mm) to (360:15mm);
					\draw[e:coloredthin,red,bend right=10] (340:15mm) to (360:15mm);
					
					\draw[e:coloredborder] (180:7mm) -- (0:7mm);
					\draw[e:coloredthin,color=BostonUniversityRed] (180:7mm) -- (0:7mm);
					\draw[e:main,gray] (180:15mm) -- (180:7mm);
					\draw[e:main,gray] (0:15mm) -- (0:7mm);
					
					\draw[e:coloredborder] (240:15mm) -- (240:7mm);
					\draw[e:coloredthin,color=BostonUniversityRed] (240:15mm) -- (240:7mm);
					\draw[e:coloredborder] (60:15mm) -- (60:7mm);
					\draw[e:coloredthin,color=BostonUniversityRed] (60:15mm) -- (60:7mm);
					\draw[e:main,gray] (240:7mm) -- (60:7mm);
					
					\draw[e:coloredborder] (300:7mm) -- (120:7mm);
					\draw[e:coloredthin,color=BostonUniversityRed] (300:7mm) -- (120:7mm);
					\draw[e:main,gray] (300:15mm) -- (300:7mm);
					\draw[e:main,gray] (120:15mm) -- (120:7mm);
					
					\draw[e:main,blue,bend right=15] (120:7mm) to (66.66:15mm);    
					\draw[e:main,myGreen] (320:15mm) to (113.34:15mm);

				\end{pgfonlayer}
			\end{tikzpicture}
		\end{subfigure}
		
		\begin{subfigure}{0.49\textwidth}
			\centering
			\begin{tikzpicture}[scale=1.15]
				\pgfdeclarelayer{background}
				\pgfdeclarelayer{foreground}
				\pgfsetlayers{background,main,foreground}
				
				\node[v:marked] () at (113.34:15mm){};
				\node[v:marked] () at (240:15mm){};
				\node[v:marked] () at (300:15mm){};
				\node[v:marked] () at (80:11mm){};
				\node[v:marked] () at (100:11mm){};
				\node[v:marked] () at (66.66:15mm){};
				
				\foreach \x in {0,2,4}
				{
					\node[v:mainempty] () at (\x*60:15mm){};
					\node[v:maingray] () at (\x*60:7mm){};
				}
				\foreach \x in {1,3,5}
				{
					\node[v:main] () at (\x*60:15mm){};  
					\node[v:mainemptygray] () at (\x*60:7mm){};
				}
				
				\node[v:maingray] () at (20:15mm){};
				\node[v:maingray] () at (140:15mm){};
				\node[v:maingray] () at (220:15mm){};
				\node[v:maingray] () at (260:15mm){};
				\node[v:maingray] () at (340:15mm){};
				
				\node[v:mainemptygray] () at (66.66:15mm){};
				\node[v:maingray] () at (73.32:15mm){};
				\node[v:mainemptygreen] () at (80.01:15mm){};
				\node[v:maingreen] () at (86.68:15mm){};
				\node[v:mainemptygreen] () at (93.35:15mm){};
				\node[v:maingreen] () at (100.01:15mm){};
				\node[v:mainemptygray] () at (106.68:15mm){};
				\node[v:maingray] () at (113.34:15mm){};
				
				\node[v:mainemptygray] () at (40:15mm){};
				\node[v:mainemptygray] () at (160:15mm){};
				\node[v:mainemptygray] () at (200:15mm){};
				\node[v:mainemptygray] () at (280:15mm){};
				\node[v:mainemptygray] () at (320:15mm){};
				
				\node[v:main] () at (80:11mm){};
				\node[v:mainempty] () at (100:11mm){};
				
				\node () at (0:18mm){$b_{1}$};
				\node () at (60:17mm){$u$};
				\node () at (120:17mm){$v$};
				\node () at (180:18mm){$a_{2}$};
				\node () at (240:17mm){$b_{2}$};
				\node () at (300:18mm){$a_{1}$};
				
				\node[blue] () at (90:18mm){$C$};
				
				\begin{pgfonlayer}{background}
					
					\draw[e:marker,BrightUbe,bend right=15] (80:11mm) to (66.66:15mm);
					\draw[e:marker,BrightUbe,bend right=15] (100:11mm) to (120:7mm);
					\draw[e:marker,BrightUbe,bend left=15] (100:11mm) to (113.34:15mm);
					\draw[e:marker,BrightUbe,bend left=15] (80:11mm) to (60:7mm);
					\draw[e:marker,BrightUbe] (80:11mm) -- (100:11mm);
					
					\draw[e:marker,BrightUbe] (60:7mm) -- (240:15mm);
					\draw[e:marker,BrightUbe] (120:7mm) -- (300:15mm);
					\draw[e:marker,BrightUbe,bend right=28] (300:15mm) to (0:15mm) to (60:15mm) to (120:15mm) to (180:15mm) to (240:15mm) to (300:15mm);
					
					\draw[e:main,gray] (60:15mm) to (66.66:15mm){};
					\draw[e:coloredborder] (66.66:15mm) to (73.32:15mm){};
					\draw[e:coloredthin,color=BostonUniversityRed] (66.66:15mm) to (73.32:15mm){};
					\draw[e:main,gray] (73.32:15mm) to (80.01:15mm){};
					\draw[e:coloredborder] (80.01:15mm) to (86.68:15mm){};
					\draw[e:coloredthin,color=BostonUniversityRed] (80.01:15mm) to (86.68:15mm){};
					\draw[e:main,myGreen] (86.68:15mm) to (93.35:15mm){};
					\draw[e:coloredborder] (93.35:15mm) to (100.01:15mm){};
					\draw[e:coloredthin,color=BostonUniversityRed] (93.35:15mm) to (100.01:15mm){};     
					\draw[e:main,gray] (100.01:15mm) to (106.68:15mm){};
					\draw[e:coloredborder] (106.68:15mm) to (113.34:15mm){};
					\draw[e:coloredthin,color=BostonUniversityRed] (106.68:15mm) to (113.34:15mm){}; 
					\draw[e:main,gray] (113.34:15mm) to (120:15mm){};
					
					\draw[e:main,blue,densely dashed,bend right=90] (80.01:15mm) to (100.01:15mm);
					
					\draw[e:coloredborder,bend right=10] (0:15mm) to (20:15mm);
					\draw[e:coloredthin,red,bend right=10] (0:15mm) to (20:15mm);
					\draw[e:main,gray,bend right=10] (20:15mm) to (40:15mm);
					\draw[e:coloredborder,bend right=10] (40:15mm) to (60:15mm);
					\draw[e:coloredthin,red,bend right=10] (40:15mm) to (60:15mm);
					
					\draw[e:coloredborder,bend right=10] (120:15mm) to (140:15mm);
					\draw[e:coloredthin,red,bend right=10] (120:15mm) to (140:15mm);
					\draw[e:main,gray,bend right=10] (140:15mm) to (160:15mm);
					\draw[e:coloredborder,bend right=10] (160:15mm) to (180:15mm);
					\draw[e:coloredthin,red,bend right=10] (160:15mm) to (180:15mm);
					
					\draw[e:main,gray,bend right=10] (180:15mm) to (200:15mm);
					\draw[e:coloredborder,bend right=10] (200:15mm) to (220:15mm);
					\draw[e:coloredthin,red,bend right=10] (200:15mm) to (220:15mm);
					\draw[e:main,gray,bend right=10] (220:15mm) to (240:15mm);
					
					\draw[e:coloredborder,bend right=10] (240:15mm) to (260:15mm);
					\draw[e:coloredthin,red,bend right=10] (240:15mm) to (260:15mm);
					\draw[e:main,gray,bend right=10] (260:15mm) to (280:15mm);
					\draw[e:coloredborder,bend right=10] (280:15mm) to (300:15mm);
					\draw[e:coloredthin,red,bend right=10] (280:15mm) to (300:15mm);
					
					\draw[e:main,gray,bend right=10] (300:15mm) to (320:15mm);
					\draw[e:coloredborder,bend right=10] (320:15mm) to (340:15mm);
					\draw[e:coloredthin,red,bend right=10] (320:15mm) to (340:15mm);
					\draw[e:main,gray,bend right=10] (340:15mm) to (0:15mm);
					
					\draw[e:coloredborder] (180:7mm) -- (0:7mm);
					\draw[e:coloredthin,color=BostonUniversityRed] (180:7mm) -- (0:7mm);
					\draw[e:main,gray] (180:15mm) -- (180:7mm);
					\draw[e:main,gray] (0:15mm) -- (0:7mm);
					
					\draw[e:coloredborder] (240:7mm) -- (60:7mm);
					\draw[e:coloredthin,color=BostonUniversityRed] (240:7mm) -- (60:7mm);
					\draw[e:main,gray] (240:15mm) -- (240:7mm);
					\draw[e:main,gray] (60:15mm) -- (60:7mm);
					
					\draw[e:coloredborder] (300:7mm) -- (120:7mm);
					\draw[e:coloredthin,color=BostonUniversityRed] (300:7mm) -- (120:7mm);
					\draw[e:main,gray] (300:15mm) -- (300:7mm);
					\draw[e:main,gray] (120:15mm) -- (120:7mm);
					
					\draw[e:coloredborder] (80:11mm) -- (100:11mm);
					\draw[e:coloredthin,color=BostonUniversityRed] (80:11mm) -- (100:11mm);
					
					\draw[e:main,blue,bend right=15] (80:11mm) to (66.66:15mm); \draw[e:main,blue,bend right=15] (100:11mm) to (120:7mm);
					\draw[e:main,myGreen,bend left=15] (100:11mm) to (113.34:15mm);
					\draw[e:main,myGreen,bend left=15] (80:11mm) to (60:7mm);

				\end{pgfonlayer}
			\end{tikzpicture}
		\end{subfigure}
		\begin{subfigure}{0.49\textwidth}
			\centering
			\begin{tikzpicture}[scale=1.15]
				\pgfdeclarelayer{background}
				\pgfdeclarelayer{foreground}
				\pgfsetlayers{background,main,foreground}
				
				\node[v:marked] () at (113.34:15mm){};
				\node[v:marked] () at (120:7mm){};
				\node[v:marked] () at (120:15mm){};
				\node[v:marked] () at (80:11mm){};
				\node[v:marked] () at (100:11mm){};
				\node[v:marked] () at (66.66:15mm){};
				
				\foreach \x in {0,2,4}
				{
					\node[v:mainempty] () at (\x*60:15mm){};
					\node[v:maingray] () at (\x*60:7mm){};
				}
				\foreach \x in {1,3,5}
				{
					\node[v:main] () at (\x*60:15mm){};  
					\node[v:mainemptygray] () at (\x*60:7mm){};
				}
				
				\node[v:maingray] () at (20:15mm){};
				\node[v:maingray] () at (140:15mm){};
				\node[v:maingray] () at (220:15mm){};
				\node[v:maingray] () at (260:15mm){};
				\node[v:maingray] () at (340:15mm){};
				
				\node[v:mainemptygray] () at (66.66:15mm){};
				\node[v:maingray] () at (73.32:15mm){};
				\node[v:mainemptygreen] () at (80.01:15mm){};
				\node[v:maingreen] () at (86.68:15mm){};
				\node[v:mainemptygreen] () at (93.35:15mm){};
				\node[v:maingreen] () at (100.01:15mm){};
				\node[v:mainemptygray] () at (106.68:15mm){};
				\node[v:maingray] () at (113.34:15mm){};
				
				\node[v:mainemptygray] () at (40:15mm){};
				\node[v:mainemptygray] () at (160:15mm){};
				\node[v:mainemptygray] () at (200:15mm){};
				\node[v:mainemptygray] () at (280:15mm){};
				\node[v:mainemptygray] () at (320:15mm){};
				
				\node[v:main] () at (80:11mm){};
				\node[v:mainempty] () at (100:11mm){};
				
				\node () at (0:18mm){$b_{1}$};
				\node () at (60:17mm){$u$};
				\node () at (120:17mm){$v$};
				\node () at (180:18mm){$a_{2}$};
				\node () at (240:17mm){$b_{2}$};
				\node () at (300:18mm){$a_{1}$};
				
				\node[blue] () at (90:18mm){$C$};
				
				\begin{pgfonlayer}{background}
					
					\draw[e:marker,BrightUbe,bend right=15] (80:11mm) to (66.66:15mm);
					\draw[e:marker,BrightUbe,bend right=15] (100:11mm) to (120:7mm);
					\draw[e:marker,BrightUbe,bend left=20] (100:11mm) to (113.34:15mm);
					\draw[e:marker,BrightUbe,bend right=15] (80:11mm) to (320:15mm);
					\draw[e:marker,BrightUbe] (80:11mm) -- (100:11mm);
					
					\draw[e:marker,BrightUbe] (0:15mm) -- (180:15mm);
					\draw[e:marker,BrightUbe] (60:15mm) -- (240:15mm);
					\draw[e:marker,BrightUbe] (120:15mm) -- (300:15mm);
					\draw[e:marker,BrightUbe,bend right=28] (60:15mm) to (120:15mm) to (180:15mm);
					\draw[e:marker,BrightUbe,bend right=28] (240:15mm) to (300:15mm);
					\draw[e:marker,BrightUbe,bend right=19] (320:15mm) to (0:15mm);
					
					\draw[e:main,gray] (60:15mm) to (66.66:15mm){};
					\draw[e:coloredborder] (66.66:15mm) to (73.32:15mm){};
					\draw[e:coloredthin,color=BostonUniversityRed] (66.66:15mm) to (73.32:15mm){};
					\draw[e:main,gray] (73.32:15mm) to (80.01:15mm){};
					\draw[e:coloredborder] (80.01:15mm) to (86.68:15mm){};
					\draw[e:coloredthin,color=BostonUniversityRed] (80.01:15mm) to (86.68:15mm){};
					\draw[e:main,myGreen] (86.68:15mm) to (93.35:15mm){};
					\draw[e:coloredborder] (93.35:15mm) to (100.01:15mm){};
					\draw[e:coloredthin,color=BostonUniversityRed] (93.35:15mm) to (100.01:15mm){};     
					\draw[e:main,gray] (100.01:15mm) to (106.68:15mm){};
					\draw[e:coloredborder] (106.68:15mm) to (113.34:15mm){};
					\draw[e:coloredthin,color=BostonUniversityRed] (106.68:15mm) to (113.34:15mm){}; 
					\draw[e:main,gray] (113.34:15mm) to (120:15mm){};
					
					\draw[e:main,blue,densely dashed,bend right=90] (80.01:15mm) to (100.01:15mm);
					
					\draw[e:main,gray,bend right=10] (0:15mm) to (20:15mm);
					\draw[e:coloredborder,bend right=10] (20:15mm) to (40:15mm);
					\draw[e:coloredthin,red,bend right=10] (20:15mm) to (40:15mm);
					\draw[e:main,gray,bend right=10] (40:15mm) to (60:15mm);
					
					\draw[e:main,gray,bend right=10] (120:15mm) to (140:15mm);
					\draw[e:coloredborder,bend right=10] (140:15mm) to (160:15mm);
					\draw[e:coloredthin,red,bend right=10] (140:15mm) to (160:15mm);
					\draw[e:main,gray,bend right=10] (160:15mm) to (180:15mm);
					
					\draw[e:main,gray,bend right=10] (180:15mm) to (200:15mm);
					\draw[e:coloredborder,bend right=10] (200:15mm) to (220:15mm);
					\draw[e:coloredthin,red,bend right=10] (200:15mm) to (220:15mm);
					\draw[e:main,gray,bend right=10] (220:15mm) to (240:15mm);
					
					\draw[e:main,gray,bend right=10] (240:15mm) to (260:15mm);
					\draw[e:coloredborder,bend right=10] (260:15mm) to (280:15mm);
					\draw[e:coloredthin,red,bend right=10] (260:15mm) to (280:15mm);
					\draw[e:main,gray,bend right=10] (280:15mm) to (300:15mm);
					
					\draw[e:main,gray,bend right=10] (300:15mm) to (320:15mm);
					\draw[e:coloredborder,bend right=10] (320:15mm) to (340:15mm);
					\draw[e:coloredthin,red,bend right=10] (320:15mm) to (340:15mm);
					\draw[e:main,gray,bend right=10] (340:15mm) to (0:15mm);
					
					\draw[e:coloredborder] (180:15mm) -- (180:7mm);
					\draw[e:coloredthin,color=BostonUniversityRed] (180:15mm) -- (180:7mm);
					\draw[e:coloredborder] (0:15mm) -- (0:7mm);
					\draw[e:coloredthin,color=BostonUniversityRed] (0:15mm) -- (0:7mm);
					\draw[e:main,gray] (180:7mm) -- (0:7mm);
					
					\draw[e:coloredborder] (240:15mm) -- (240:7mm);
					\draw[e:coloredthin,color=BostonUniversityRed] (240:15mm) -- (240:7mm);
					\draw[e:coloredborder] (60:15mm) -- (60:7mm);
					\draw[e:coloredthin,color=BostonUniversityRed] (60:15mm) -- (60:7mm);
					\draw[e:main,gray] (240:7mm) -- (60:7mm);
					
					\draw[e:coloredborder] (300:15mm) -- (300:7mm);
					\draw[e:coloredthin,color=BostonUniversityRed] (300:15mm) -- (300:7mm);
					\draw[e:coloredborder] (120:15mm) -- (120:7mm);
					\draw[e:coloredthin,color=BostonUniversityRed] (120:15mm) -- (120:7mm);
					\draw[e:main,gray] (300:7mm) -- (120:7mm);
					
					\draw[e:coloredborder] (80:11mm) -- (100:11mm);
					\draw[e:coloredthin,color=BostonUniversityRed] (80:11mm) -- (100:11mm);
					
					\draw[e:main,blue,bend right=15] (80:11mm) to (66.66:15mm); \draw[e:main,blue,bend right=15] (100:11mm) to (120:7mm);
					\draw[e:main,myGreen,bend left=20] (100:11mm) to (113.34:15mm);
					\draw[e:main,myGreen,bend right=15] (80:11mm) to (320:15mm);

				\end{pgfonlayer}
			\end{tikzpicture}
		\end{subfigure}
		\caption{The construction of the new conformal $K_{3,3}$-bisubdivision in the first case of the proof of \cref{lemma:reducedK33mustsplit}.} 
		\label{fig:shrinkingPno1}
	\end{figure}
	
	For the next case we assume $v_{V_1}$ and $v_{V_2}$ to be vertices of a common bisubdivided edge $Q$ of $L$.
	According to the previous discussion, $Q$ cannot share an endpoint with $P$ and by symmetry, it suffices to only consider one possible choice for $Q$, so let $Q\coloneqq E_{a_2b_1}$.
	The path $Q$ is split into three, possibly trivial, subpaths by the vertices $v_{V_1}$ and $v_{V_2}$.
	Since $Q$ is of odd length, either zero or exactly two of these subpaths are of even length, and these are exactly the two cases we need to distinguish.
	\Cref{fig:shrinkingPno2} shows how to construct the new conformal $K_{3,3}$-bisubdivision $L'$ which yields the desired contradiction.
	
	\begin{figure}[h!]
		\centering
		\begin{subfigure}{0.49\textwidth}
			\centering
			\begin{tikzpicture}[scale=1.15]
				\pgfdeclarelayer{background}
				\pgfdeclarelayer{foreground}
				\pgfsetlayers{background,main,foreground}
				
				\node[v:marked] () at (113.34:15mm){};
				\node[v:marked] () at (66.66:15mm){};
				\node[v:marked] () at (120:15mm){};
				\node[v:marked] () at (180:15mm){};
				\node[v:marked] () at (0:7mm){};
				\node[v:marked] () at (180:7mm){};
				
				\foreach \x in {0,2,4}
				{
					\node[v:mainempty] () at (\x*60:15mm){};
					\node[v:maingray] () at (\x*60:7mm){};
				}
				\foreach \x in {1,3,5}
				{
					\node[v:main] () at (\x*60:15mm){};  
					\node[v:mainemptygray] () at (\x*60:7mm){};
				}
				
				\node[v:maingray] () at (20:15mm){};
				\node[v:maingray] () at (140:15mm){};
				\node[v:maingray] () at (220:15mm){};
				\node[v:maingray] () at (260:15mm){};
				\node[v:maingray] () at (340:15mm){};
				
				\node[v:mainemptygray] () at (66.66:15mm){};
				\node[v:maingray] () at (73.32:15mm){};
				\node[v:mainemptygreen] () at (80.01:15mm){};
				\node[v:maingreen] () at (86.68:15mm){};
				\node[v:mainemptygreen] () at (93.35:15mm){};
				\node[v:maingreen] () at (100.01:15mm){};
				\node[v:mainemptygray] () at (106.68:15mm){};
				\node[v:maingray] () at (113.34:15mm){};
				
				\node[v:mainemptygray] () at (40:15mm){};
				\node[v:mainemptygray] () at (160:15mm){};
				\node[v:mainemptygray] () at (200:15mm){};
				\node[v:mainemptygray] () at (280:15mm){};
				\node[v:mainemptygray] () at (320:15mm){};
				
				\node () at (0:18mm){$b_{1}$};
				\node () at (60:17mm){$u$};
				\node () at (120:17mm){$v$};
				\node () at (180:18mm){$a_{2}$};
				\node () at (240:17mm){$b_{2}$};
				\node () at (300:18mm){$a_{1}$};
				
				\node[blue] () at (90:18mm){$C$};
				
				\begin{pgfonlayer}{background}
					
					\draw[e:marker,BrightUbe,bend right=15] (113.34:15mm) to (180:7mm);
					\draw[e:marker,BrightUbe,bend left=15] (66.66:15mm) to (0:7mm);
					\draw[e:marker,BrightUbe] (0:15mm) -- (180:15mm);
					\draw[e:marker,BrightUbe] (60:15mm) -- (240:15mm);
					\draw[e:marker,BrightUbe] (120:15mm) -- (300:15mm);
					
					\draw[e:marker,BrightUbe,bend right=28] (300:15mm) to (0:15mm);
					
					\draw[e:marker,BrightUbe,bend right=28] (60:15mm) to (120:15mm) to (180:15mm) to (240:15mm);
					
					\draw[e:main,gray] (60:15mm) to (66.66:15mm){};
					\draw[e:coloredborder] (66.66:15mm) to (73.32:15mm){};
					\draw[e:coloredthin,color=BostonUniversityRed] (66.66:15mm) to (73.32:15mm){};
					\draw[e:main,gray] (73.32:15mm) to (80.01:15mm){};
					\draw[e:coloredborder] (80.01:15mm) to (86.68:15mm){};
					\draw[e:coloredthin,color=BostonUniversityRed] (80.01:15mm) to (86.68:15mm){};
					\draw[e:main,myGreen] (86.68:15mm) to (93.35:15mm){};
					\draw[e:coloredborder] (93.35:15mm) to (100.01:15mm){};
					\draw[e:coloredthin,color=BostonUniversityRed] (93.35:15mm) to (100.01:15mm){};     
					\draw[e:main,gray] (100.01:15mm) to (106.68:15mm){};
					\draw[e:coloredborder] (106.68:15mm) to (113.34:15mm){};
					\draw[e:coloredthin,color=BostonUniversityRed] (106.68:15mm) to (113.34:15mm){}; 
					\draw[e:main,gray] (113.34:15mm) to (120:15mm){};
					
					\draw[e:main,blue,densely dashed,bend right=90] (80.01:15mm) to (100.01:15mm);
					
					\draw[e:main,gray,bend right=10] (0:15mm) to (20:15mm);
					\draw[e:coloredborder,bend right=10] (20:15mm) to (40:15mm);
					\draw[e:coloredthin,red,bend right=10] (20:15mm) to (40:15mm);
					\draw[e:main,gray,bend right=10] (40:15mm) to (60:15mm);
					
					\draw[e:coloredborder,bend right=10] (120:15mm) to (140:15mm);
					\draw[e:coloredthin,red,bend right=10] (120:15mm) to (140:15mm);
					\draw[e:main,gray,bend right=10] (140:15mm) to (160:15mm);
					\draw[e:coloredborder,bend right=10] (160:15mm) to (180:15mm);
					\draw[e:coloredthin,red,bend right=10] (160:15mm) to (180:15mm);
					
					\draw[e:main,gray,bend right=10] (180:15mm) to (200:15mm);
					\draw[e:coloredborder,bend right=10] (200:15mm) to (220:15mm);
					\draw[e:coloredthin,red,bend right=10] (200:15mm) to (220:15mm);
					\draw[e:main,gray,bend right=10] (220:15mm) to (240:15mm);
					
					\draw[e:main,gray,bend right=10] (240:15mm) to (260:15mm);
					\draw[e:coloredborder,bend right=10] (260:15mm) to (280:15mm);
					\draw[e:coloredthin,red,bend right=10] (260:15mm) to (280:15mm);
					\draw[e:main,gray,bend right=10] (280:15mm) to (300:15mm);
					
					\draw[e:coloredborder,bend right=10] (300:15mm) to (320:15mm);
					\draw[e:coloredthin,red,bend right=10] (300:15mm) to (320:15mm);
					\draw[e:main,gray,bend right=10] (320:15mm) to (340:15mm);
					\draw[e:coloredborder,bend right=10] (340:15mm) to (360:15mm);
					\draw[e:coloredthin,red,bend right=10] (340:15mm) to (360:15mm);
					
					\draw[e:coloredborder] (180:7mm) -- (0:7mm);
					\draw[e:coloredthin,color=BostonUniversityRed] (180:7mm) -- (0:7mm);
					\draw[e:main,gray] (180:15mm) -- (180:7mm);
					\draw[e:main,gray] (0:15mm) -- (0:7mm);
					
					\draw[e:coloredborder] (240:15mm) -- (240:7mm);
					\draw[e:coloredthin,color=BostonUniversityRed] (240:15mm) -- (240:7mm);
					\draw[e:coloredborder] (60:15mm) -- (60:7mm);
					\draw[e:coloredthin,color=BostonUniversityRed] (60:15mm) -- (60:7mm);
					\draw[e:main,gray] (240:7mm) -- (60:7mm);
					
					\draw[e:coloredborder] (300:7mm) -- (120:7mm);
					\draw[e:coloredthin,color=BostonUniversityRed] (300:7mm) -- (120:7mm);
					\draw[e:main,gray] (300:15mm) -- (300:7mm);
					\draw[e:main,gray] (120:15mm) -- (120:7mm);
					
					\draw[e:main,myGreen,bend right=15] (113.34:15mm) to (180:7mm);
					\draw[e:main,blue,bend left=15] (66.66:15mm) to (0:7mm);

				\end{pgfonlayer}
			\end{tikzpicture}
		\end{subfigure}
		\begin{subfigure}{0.49\textwidth}
			\centering
			\begin{tikzpicture}[scale=1.15]
				\pgfdeclarelayer{background}
				\pgfdeclarelayer{foreground}
				\pgfsetlayers{background,main,foreground}
				
				\node[v:marked] () at (113.34:15mm){};
				\node[v:marked] () at (66.66:15mm){};
				\node[v:marked] () at (120:15mm){};
				\node[v:marked] () at (60:15mm){};
				\node[v:marked] () at (0:4mm){};
				\node[v:marked] () at (180:4mm){};
				
				\foreach \x in {0,2,4}
				{
					\node[v:mainempty] () at (\x*60:15mm){};
					\node[v:maingray] () at (\x*60:7mm){};
				}
				\foreach \x in {1,3,5}
				{
					\node[v:main] () at (\x*60:15mm){};  
					\node[v:mainemptygray] () at (\x*60:7mm){};
				}
				
				\node[v:mainemptygray] () at (0:4mm){};
				\node[v:maingray] () at (180:4mm){};

				\node[v:maingray] () at (20:15mm){};
				\node[v:maingray] () at (140:15mm){};
				\node[v:maingray] () at (220:15mm){};
				\node[v:maingray] () at (260:15mm){};
				\node[v:maingray] () at (340:15mm){};
				
				\node[v:mainemptygray] () at (66.66:15mm){};
				\node[v:maingray] () at (73.32:15mm){};
				\node[v:mainemptygreen] () at (80.01:15mm){};
				\node[v:maingreen] () at (86.68:15mm){};
				\node[v:mainemptygreen] () at (93.35:15mm){};
				\node[v:maingreen] () at (100.01:15mm){};
				\node[v:mainemptygray] () at (106.68:15mm){};
				\node[v:maingray] () at (113.34:15mm){};
				
				\node[v:mainemptygray] () at (40:15mm){};
				\node[v:mainemptygray] () at (160:15mm){};
				\node[v:mainemptygray] () at (200:15mm){};
				\node[v:mainemptygray] () at (280:15mm){};
				\node[v:mainemptygray] () at (320:15mm){};
				
				\node () at (0:18mm){$b_{1}$};
				\node () at (60:17mm){$u$};
				\node () at (120:17mm){$v$};
				\node () at (180:18mm){$a_{2}$};
				\node () at (240:17mm){$b_{2}$};
				\node () at (300:18mm){$a_{1}$};
				
				\node[blue] () at (90:18mm){$C$};
				
				\begin{pgfonlayer}{background}
					
					\draw[e:marker,BrightUbe,bend left=15] (113.34:15mm) to (0:4mm);
					\draw[e:marker,BrightUbe,bend right=15] (66.66:15mm) to (180:4mm);
					\draw[e:marker,BrightUbe] (0:15mm) -- (180:15mm);
					\draw[e:marker,BrightUbe] (60:15mm) -- (240:15mm);
					\draw[e:marker,BrightUbe] (120:15mm) -- (300:15mm);
					
					\draw[e:marker,BrightUbe,bend right=28] (240:15mm) to (300:15mm);
					
					\draw[e:marker,BrightUbe,bend right=28] (0:15mm) to (60:15mm) to (120:15mm) to (180:15mm);
					
					\draw[e:main,gray] (60:15mm) to (66.66:15mm){};
					\draw[e:coloredborder] (66.66:15mm) to (73.32:15mm){};
					\draw[e:coloredthin,color=BostonUniversityRed] (66.66:15mm) to (73.32:15mm){};
					\draw[e:main,gray] (73.32:15mm) to (80.01:15mm){};
					\draw[e:coloredborder] (80.01:15mm) to (86.68:15mm){};
					\draw[e:coloredthin,color=BostonUniversityRed] (80.01:15mm) to (86.68:15mm){};
					\draw[e:main,myGreen] (86.68:15mm) to (93.35:15mm){};
					\draw[e:coloredborder] (93.35:15mm) to (100.01:15mm){};
					\draw[e:coloredthin,color=BostonUniversityRed] (93.35:15mm) to (100.01:15mm){};     
					\draw[e:main,gray] (100.01:15mm) to (106.68:15mm){};
					\draw[e:coloredborder] (106.68:15mm) to (113.34:15mm){};
					\draw[e:coloredthin,color=BostonUniversityRed] (106.68:15mm) to (113.34:15mm){}; 
					\draw[e:main,gray] (113.34:15mm) to (120:15mm){};
					
					\draw[e:main,blue,densely dashed,bend right=90] (80.01:15mm) to (100.01:15mm);
					
					\draw[e:main,gray,bend right=10] (0:15mm) to (20:15mm);
					\draw[e:coloredborder,bend right=10] (20:15mm) to (40:15mm);
					\draw[e:coloredthin,red,bend right=10] (20:15mm) to (40:15mm);
					\draw[e:main,gray,bend right=10] (40:15mm) to (60:15mm);
					
					\draw[e:main,gray,bend right=10] (120:15mm) to (140:15mm);
					\draw[e:coloredborder,bend right=10] (140:15mm) to (160:15mm);
					\draw[e:coloredthin,red,bend right=10] (140:15mm) to (160:15mm);
					\draw[e:main,gray,bend right=10] (160:15mm) to (180:15mm);
					
					\draw[e:main,gray,bend right=10] (180:15mm) to (200:15mm);
					\draw[e:coloredborder,bend right=10] (200:15mm) to (220:15mm);
					\draw[e:coloredthin,red,bend right=10] (200:15mm) to (220:15mm);
					\draw[e:main,gray,bend right=10] (220:15mm) to (240:15mm);
					
					\draw[e:main,gray,bend right=10] (240:15mm) to (260:15mm);
					\draw[e:coloredborder,bend right=10] (260:15mm) to (280:15mm);
					\draw[e:coloredthin,red,bend right=10] (260:15mm) to (280:15mm);
					\draw[e:main,gray,bend right=10] (280:15mm) to (300:15mm);
					
					\draw[e:main,gray,bend right=10] (300:15mm) to (320:15mm);
					\draw[e:coloredborder,bend right=10] (320:15mm) to (340:15mm);
					\draw[e:coloredthin,red,bend right=10] (320:15mm) to (340:15mm);
					\draw[e:main,gray,bend right=10] (340:15mm) to (0:15mm);
					
					\draw[e:coloredborder] (180:15mm) -- (180:7mm);
					\draw[e:coloredthin,color=BostonUniversityRed] (180:15mm) -- (180:7mm);
					\draw[e:coloredborder] (0:15mm) -- (0:7mm);
					\draw[e:coloredthin,color=BostonUniversityRed] (0:15mm) -- (0:7mm);
					\draw[e:main,gray] (180:7mm) -- (180:4mm);
					\draw[e:main,gray] (0:7mm) -- (0:4mm);
					\draw[e:coloredborder] (180:4mm) -- (0:4mm);
					\draw[e:coloredthin,color=BostonUniversityRed] (180:4mm) -- (0:4mm);

					\draw[e:coloredborder] (240:15mm) -- (240:7mm);
					\draw[e:coloredthin,color=BostonUniversityRed] (240:15mm) -- (240:7mm);
					\draw[e:coloredborder] (60:15mm) -- (60:7mm);
					\draw[e:coloredthin,color=BostonUniversityRed] (60:15mm) -- (60:7mm);
					\draw[e:main,gray] (240:7mm) -- (60:7mm);
					
					\draw[e:coloredborder] (300:15mm) -- (300:7mm);
					\draw[e:coloredthin,color=BostonUniversityRed] (300:15mm) -- (300:7mm);
					\draw[e:coloredborder] (120:15mm) -- (120:7mm);
					\draw[e:coloredthin,color=BostonUniversityRed] (120:15mm) -- (120:7mm);
					\draw[e:main,gray] (300:7mm) -- (120:7mm);
					
					\draw[e:main,myGreen,bend left=15] (113.34:15mm) to (0:4mm);
					\draw[e:main,blue,bend right=15] (66.66:15mm) to (180:4mm);        
					
				\end{pgfonlayer}
			\end{tikzpicture}
		\end{subfigure}

		\begin{subfigure}{0.49\textwidth}
			\centering
			\begin{tikzpicture}[scale=1.15]
				\pgfdeclarelayer{background}
				\pgfdeclarelayer{foreground}
				\pgfsetlayers{background,main,foreground}
				
				\node[v:marked] () at (113.34:15mm){};
				\node[v:marked] () at (66.66:15mm){};
				\node[v:marked] () at (100:11mm){};
				\node[v:marked] () at (80:11mm){};
				\node[v:marked] () at (0:7mm){};
				\node[v:marked] () at (180:7mm){};
				
				\foreach \x in {0,2,4}
				{
					\node[v:mainempty] () at (\x*60:15mm){};
					\node[v:maingray] () at (\x*60:7mm){};
				}
				\foreach \x in {1,3,5}
				{
					\node[v:main] () at (\x*60:15mm){};  
					\node[v:mainemptygray] () at (\x*60:7mm){};
				}
				
				\node[v:maingray] () at (20:15mm){};
				\node[v:maingray] () at (140:15mm){};
				\node[v:maingray] () at (220:15mm){};
				\node[v:maingray] () at (260:15mm){};
				\node[v:maingray] () at (340:15mm){};
				
				\node[v:mainemptygray] () at (66.66:15mm){};
				\node[v:maingray] () at (73.32:15mm){};
				\node[v:mainemptygreen] () at (80.01:15mm){};
				\node[v:maingreen] () at (86.68:15mm){};
				\node[v:mainemptygreen] () at (93.35:15mm){};
				\node[v:maingreen] () at (100.01:15mm){};
				\node[v:mainemptygray] () at (106.68:15mm){};
				\node[v:maingray] () at (113.34:15mm){};
				
				\node[v:mainemptygray] () at (40:15mm){};
				\node[v:mainemptygray] () at (160:15mm){};
				\node[v:mainemptygray] () at (200:15mm){};
				\node[v:mainemptygray] () at (280:15mm){};
				\node[v:mainemptygray] () at (320:15mm){};
				
				\node[v:main] () at (80:11mm){};
				\node[v:mainempty] () at (100:11mm){};
				
				\node () at (0:18mm){$b_{1}$};
				\node () at (60:17mm){$u$};
				\node () at (120:17mm){$v$};
				\node () at (180:18mm){$a_{2}$};
				\node () at (240:17mm){$b_{2}$};
				\node () at (300:18mm){$a_{1}$};
				
				\node[blue] () at (90:18mm){$C$};
				
				\begin{pgfonlayer}{background}
					
					\draw[e:marker,BrightUbe,bend right=15] (80:11mm) to (66.66:15mm);
					\draw[e:marker,BrightUbe,bend right=15] (100:11mm) to (0:7mm);
					\draw[e:marker,BrightUbe,bend left=20] (100:11mm) to (113.34:15mm);
					\draw[e:marker,BrightUbe,bend left=15] (80:11mm) to (180:7mm);
					\draw[e:marker,BrightUbe] (80:11mm) -- (100:11mm);
					
					\draw[e:marker,BrightUbe] (0:15mm) -- (180:15mm);
					\draw[e:marker,BrightUbe] (60:15mm) -- (240:15mm);
					\draw[e:marker,BrightUbe] (120:15mm) -- (300:15mm);
					\draw[e:marker,BrightUbe,bend right=28] (240:15mm) to (300:15mm);
					
					\draw[e:marker,BrightUbe,bend right=28] (0:15mm) to (60:15mm);
					\draw[e:marker,BrightUbe,bend right=28] (120:15mm) to (180:15mm);
					\draw[e:marker,BrightUbe,bend right=28] (66.66:15mm) to (113.34:15mm);
					
					\draw[e:main,gray] (60:15mm) to (66.66:15mm){};
					\draw[e:coloredborder] (66.66:15mm) to (73.32:15mm){};
					\draw[e:coloredthin,color=BostonUniversityRed] (66.66:15mm) to (73.32:15mm){};
					\draw[e:main,gray] (73.32:15mm) to (80.01:15mm){};
					\draw[e:coloredborder] (80.01:15mm) to (86.68:15mm){};
					\draw[e:coloredthin,color=BostonUniversityRed] (80.01:15mm) to (86.68:15mm){};
					\draw[e:main,myGreen] (86.68:15mm) to (93.35:15mm){};
					\draw[e:coloredborder] (93.35:15mm) to (100.01:15mm){};
					\draw[e:coloredthin,color=BostonUniversityRed] (93.35:15mm) to (100.01:15mm){};     
					\draw[e:main,gray] (100.01:15mm) to (106.68:15mm){};
					\draw[e:coloredborder] (106.68:15mm) to (113.34:15mm){};
					\draw[e:coloredthin,color=BostonUniversityRed] (106.68:15mm) to (113.34:15mm){}; 
					\draw[e:main,gray] (113.34:15mm) to (120:15mm){};
					
					\draw[e:main,blue,densely dashed,bend right=90] (80.01:15mm) to (100.01:15mm);
					
					\draw[e:coloredborder,bend right=10] (0:15mm) to (20:15mm);
					\draw[e:coloredthin,red,bend right=10] (0:15mm) to (20:15mm);
					\draw[e:main,gray,bend right=10] (20:15mm) to (40:15mm);
					\draw[e:coloredborder,bend right=10] (40:15mm) to (60:15mm);
					\draw[e:coloredthin,red,bend right=10] (40:15mm) to (60:15mm);
					
					\draw[e:coloredborder,bend right=10] (120:15mm) to (140:15mm);
					\draw[e:coloredthin,red,bend right=10] (120:15mm) to (140:15mm);
					\draw[e:main,gray,bend right=10] (140:15mm) to (160:15mm);
					\draw[e:coloredborder,bend right=10] (160:15mm) to (180:15mm);
					\draw[e:coloredthin,red,bend right=10] (160:15mm) to (180:15mm);
					
					\draw[e:main,gray,bend right=10] (180:15mm) to (200:15mm);
					\draw[e:coloredborder,bend right=10] (200:15mm) to (220:15mm);
					\draw[e:coloredthin,red,bend right=10] (200:15mm) to (220:15mm);
					\draw[e:main,gray,bend right=10] (220:15mm) to (240:15mm);
					
					\draw[e:coloredborder,bend right=10] (240:15mm) to (260:15mm);
					\draw[e:coloredthin,red,bend right=10] (240:15mm) to (260:15mm);
					\draw[e:main,gray,bend right=10] (260:15mm) to (280:15mm);
					\draw[e:coloredborder,bend right=10] (280:15mm) to (300:15mm);
					\draw[e:coloredthin,red,bend right=10] (280:15mm) to (300:15mm);
					
					\draw[e:main,gray,bend right=10] (300:15mm) to (320:15mm);
					\draw[e:coloredborder,bend right=10] (320:15mm) to (340:15mm);
					\draw[e:coloredthin,red,bend right=10] (320:15mm) to (340:15mm);
					\draw[e:main,gray,bend right=10] (340:15mm) to (0:15mm);
					
					\draw[e:coloredborder] (180:7mm) -- (0:7mm);
					\draw[e:coloredthin,color=BostonUniversityRed] (180:7mm) -- (0:7mm);
					\draw[e:main,gray] (180:15mm) -- (180:7mm);
					\draw[e:main,gray] (0:15mm) -- (0:7mm);
					
					\draw[e:coloredborder] (240:7mm) -- (60:7mm);
					\draw[e:coloredthin,color=BostonUniversityRed] (240:7mm) -- (60:7mm);
					\draw[e:main,gray] (240:15mm) -- (240:7mm);
					\draw[e:main,gray] (60:15mm) -- (60:7mm);
					
					\draw[e:coloredborder] (300:7mm) -- (120:7mm);
					\draw[e:coloredthin,color=BostonUniversityRed] (300:7mm) -- (120:7mm);
					\draw[e:main,gray] (300:15mm) -- (300:7mm);
					\draw[e:main,gray] (120:15mm) -- (120:7mm);
					
					\draw[e:coloredborder] (80:11mm) -- (100:11mm);
					\draw[e:coloredthin,color=BostonUniversityRed] (80:11mm) -- (100:11mm);
					
					\draw[e:main,blue,bend right=15] (80:11mm) to (66.66:15mm); \draw[e:main,blue,bend right=15] (100:11mm) to (0:7mm);
					\draw[e:main,myGreen,bend left=20] (100:11mm) to (113.34:15mm);
					\draw[e:main,myGreen,bend left=15] (80:11mm) to (180:7mm);

				\end{pgfonlayer}
			\end{tikzpicture}
		\end{subfigure}
		\begin{subfigure}{0.49\textwidth}
			\centering
			\begin{tikzpicture}[scale=1.15]
				\pgfdeclarelayer{background}
				\pgfdeclarelayer{foreground}
				\pgfsetlayers{background,main,foreground}
				
				\node[v:marked] () at (113.34:15mm){};
				\node[v:marked] () at (66.66:15mm){};
				\node[v:marked] () at (100:11mm){};
				\node[v:marked] () at (80:11mm){};
				\node[v:marked] () at (0:4mm){};
				\node[v:marked] () at (180:4mm){};
				
				\foreach \x in {0,2,4}
				{
					\node[v:mainempty] () at (\x*60:15mm){};
					\node[v:maingray] () at (\x*60:7mm){};
				}
				\foreach \x in {1,3,5}
				{
					\node[v:main] () at (\x*60:15mm){};  
					\node[v:mainemptygray] () at (\x*60:7mm){};
				}
				
				\node[v:mainemptygray] () at (0:4mm){};
				\node[v:maingray] () at (180:4mm){};
				
				\node[v:maingray] () at (20:15mm){};
				\node[v:maingray] () at (140:15mm){};
				\node[v:maingray] () at (220:15mm){};
				\node[v:maingray] () at (260:15mm){};
				\node[v:maingray] () at (340:15mm){};
				
				\node[v:mainemptygray] () at (66.66:15mm){};
				\node[v:maingray] () at (73.32:15mm){};
				\node[v:mainemptygreen] () at (80.01:15mm){};
				\node[v:maingreen] () at (86.68:15mm){};
				\node[v:mainemptygreen] () at (93.35:15mm){};
				\node[v:maingreen] () at (100.01:15mm){};
				\node[v:mainemptygray] () at (106.68:15mm){};
				\node[v:maingray] () at (113.34:15mm){};
				
				\node[v:mainemptygray] () at (40:15mm){};
				\node[v:mainemptygray] () at (160:15mm){};
				\node[v:mainemptygray] () at (200:15mm){};
				\node[v:mainemptygray] () at (280:15mm){};
				\node[v:mainemptygray] () at (320:15mm){};
				
				\node[v:main] () at (80:11mm){};
				\node[v:mainempty] () at (100:11mm){};
				
				\node () at (0:18mm){$b_{1}$};
				\node () at (60:17mm){$u$};
				\node () at (120:17mm){$v$};
				\node () at (180:18mm){$a_{2}$};
				\node () at (240:17mm){$b_{2}$};
				\node () at (300:18mm){$a_{1}$};
				
				\node[blue] () at (90:18mm){$C$};
				
				\begin{pgfonlayer}{background}
					
					\draw[e:marker,BrightUbe,bend right=15] (80:11mm) to (66.66:15mm);
					\draw[e:marker,BrightUbe,bend left=15] (100:11mm) to (180:4mm);
					\draw[e:marker,BrightUbe,bend left=20] (100:11mm) to (113.34:15mm);
					\draw[e:marker,BrightUbe,bend right=15] (80:11mm) to (0:4mm);
					\draw[e:marker,BrightUbe] (80:11mm) -- (100:11mm);
					
					\draw[e:marker,BrightUbe] (0:15mm) -- (180:15mm);
					\draw[e:marker,BrightUbe] (60:15mm) -- (240:15mm);
					\draw[e:marker,BrightUbe] (120:15mm) -- (300:15mm);
					
					\draw[e:marker,BrightUbe,bend right=28] (60:15mm) to (120:15mm);
					\draw[e:marker,BrightUbe,bend right=28] (180:15mm) to (240:15mm);
					\draw[e:marker,BrightUbe,bend right=28] (300:15mm) to (360:15mm);
					
					\draw[e:main,gray] (60:15mm) to (66.66:15mm){};
					\draw[e:coloredborder] (66.66:15mm) to (73.32:15mm){};
					\draw[e:coloredthin,color=BostonUniversityRed] (66.66:15mm) to (73.32:15mm){};
					\draw[e:main,gray] (73.32:15mm) to (80.01:15mm){};
					\draw[e:coloredborder] (80.01:15mm) to (86.68:15mm){};
					\draw[e:coloredthin,color=BostonUniversityRed] (80.01:15mm) to (86.68:15mm){};
					\draw[e:main,myGreen] (86.68:15mm) to (93.35:15mm){};
					\draw[e:coloredborder] (93.35:15mm) to (100.01:15mm){};
					\draw[e:coloredthin,color=BostonUniversityRed] (93.35:15mm) to (100.01:15mm){};     
					\draw[e:main,gray] (100.01:15mm) to (106.68:15mm){};
					\draw[e:coloredborder] (106.68:15mm) to (113.34:15mm){};
					\draw[e:coloredthin,color=BostonUniversityRed] (106.68:15mm) to (113.34:15mm){}; 
					\draw[e:main,gray] (113.34:15mm) to (120:15mm){};
					
					\draw[e:main,blue,densely dashed,bend right=90] (80.01:15mm) to (100.01:15mm);
					
					\draw[e:main,gray,bend right=10] (0:15mm) to (20:15mm);
					\draw[e:coloredborder,bend right=10] (20:15mm) to (40:15mm);
					\draw[e:coloredthin,red,bend right=10] (20:15mm) to (40:15mm);
					\draw[e:main,gray,bend right=10] (40:15mm) to (60:15mm);
					
					\draw[e:main,gray,bend right=10] (120:15mm) to (140:15mm);
					\draw[e:coloredborder,bend right=10] (140:15mm) to (160:15mm);
					\draw[e:coloredthin,red,bend right=10] (140:15mm) to (160:15mm);
					\draw[e:main,gray,bend right=10] (160:15mm) to (180:15mm);
					
					\draw[e:main,gray,bend right=10] (180:15mm) to (200:15mm);
					\draw[e:coloredborder,bend right=10] (200:15mm) to (220:15mm);
					\draw[e:coloredthin,red,bend right=10] (200:15mm) to (220:15mm);
					\draw[e:main,gray,bend right=10] (220:15mm) to (240:15mm);
					
					\draw[e:main,gray,bend right=10] (240:15mm) to (260:15mm);
					\draw[e:coloredborder,bend right=10] (260:15mm) to (280:15mm);
					\draw[e:coloredthin,red,bend right=10] (260:15mm) to (280:15mm);
					\draw[e:main,gray,bend right=10] (280:15mm) to (300:15mm);
					
					\draw[e:main,gray,bend right=10] (300:15mm) to (320:15mm);
					\draw[e:coloredborder,bend right=10] (320:15mm) to (340:15mm);
					\draw[e:coloredthin,red,bend right=10] (320:15mm) to (340:15mm);
					\draw[e:main,gray,bend right=10] (340:15mm) to (0:15mm);
					
					\draw[e:coloredborder] (180:15mm) -- (180:7mm);
					\draw[e:coloredthin,color=BostonUniversityRed] (180:15mm) -- (180:7mm);
					\draw[e:coloredborder] (0:15mm) -- (0:7mm);
					\draw[e:coloredthin,color=BostonUniversityRed] (0:15mm) -- (0:7mm);
					\draw[e:main,gray] (180:7mm) -- (180:4mm);
					\draw[e:main,gray] (0:7mm) -- (0:4mm);
					\draw[e:coloredborder] (180:4mm) -- (0:4mm);
					\draw[e:coloredthin,color=BostonUniversityRed] (180:4mm) -- (0:4mm);
					
					\draw[e:coloredborder] (240:15mm) -- (240:7mm);
					\draw[e:coloredthin,color=BostonUniversityRed] (240:15mm) -- (240:7mm);
					\draw[e:coloredborder] (60:15mm) -- (60:7mm);
					\draw[e:coloredthin,color=BostonUniversityRed] (60:15mm) -- (60:7mm);
					\draw[e:main,gray] (240:7mm) -- (60:7mm);
					
					\draw[e:coloredborder] (300:15mm) -- (300:7mm);
					\draw[e:coloredthin,color=BostonUniversityRed] (300:15mm) -- (300:7mm);
					\draw[e:coloredborder] (120:15mm) -- (120:7mm);
					\draw[e:coloredthin,color=BostonUniversityRed] (120:15mm) -- (120:7mm);
					\draw[e:main,gray] (300:7mm) -- (120:7mm);
					
					\draw[e:coloredborder] (80:11mm) -- (100:11mm);
					\draw[e:coloredthin,color=BostonUniversityRed] (80:11mm) -- (100:11mm);
					
					\draw[e:main,blue,bend right=15] (80:11mm) to (66.66:15mm); \draw[e:main,blue,bend left=15] (100:11mm) to (180:4mm);
					\draw[e:main,myGreen,bend left=20] (100:11mm) to (113.34:15mm);
					\draw[e:main,myGreen,bend right=15] (80:11mm) to (0:4mm);
					
				\end{pgfonlayer}
			\end{tikzpicture}
		\end{subfigure}
		\caption{The construction of the new conformal $K_{3,3}$-bisubdivision in the second case of the proof of \cref{lemma:reducedK33mustsplit}.} 
		\label{fig:shrinkingPno2}
	\end{figure}
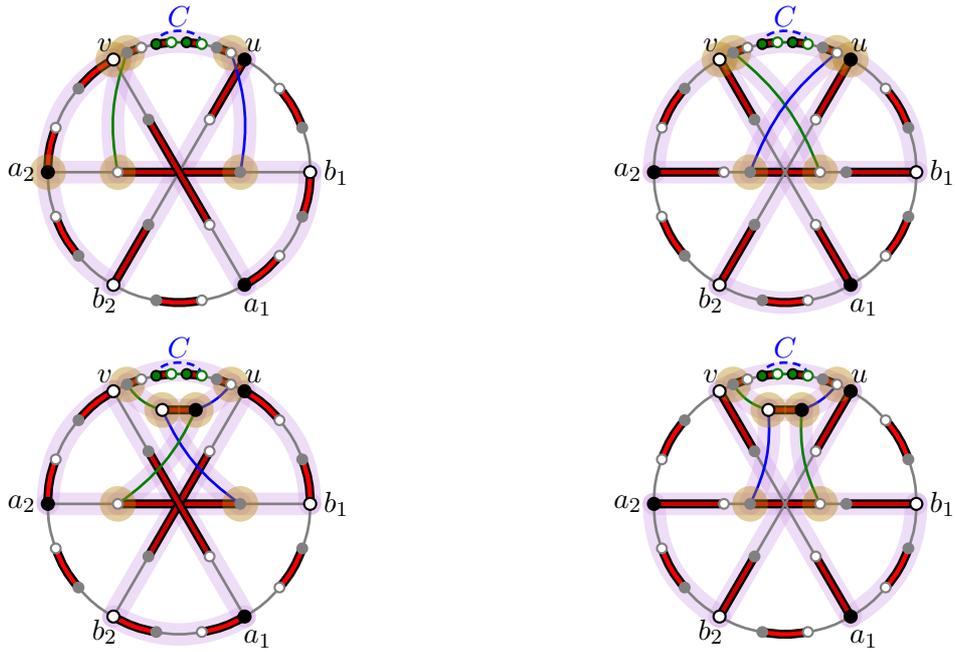
	
	For the last case, we may assume $v_{V_1}$ and $v_{V_2}$ to belong to different bisubdivided edges $Q_1$ and $Q_2$ such that neither $Q_i$ shares an endpoint with $P$.
	Here we need to distinguish between $Q_1$ and $Q_2$ sharing an endpoint and being disjoint.
	\Cref{fig:shrinkingPno3} illustrates the construction of $L'$.
	
	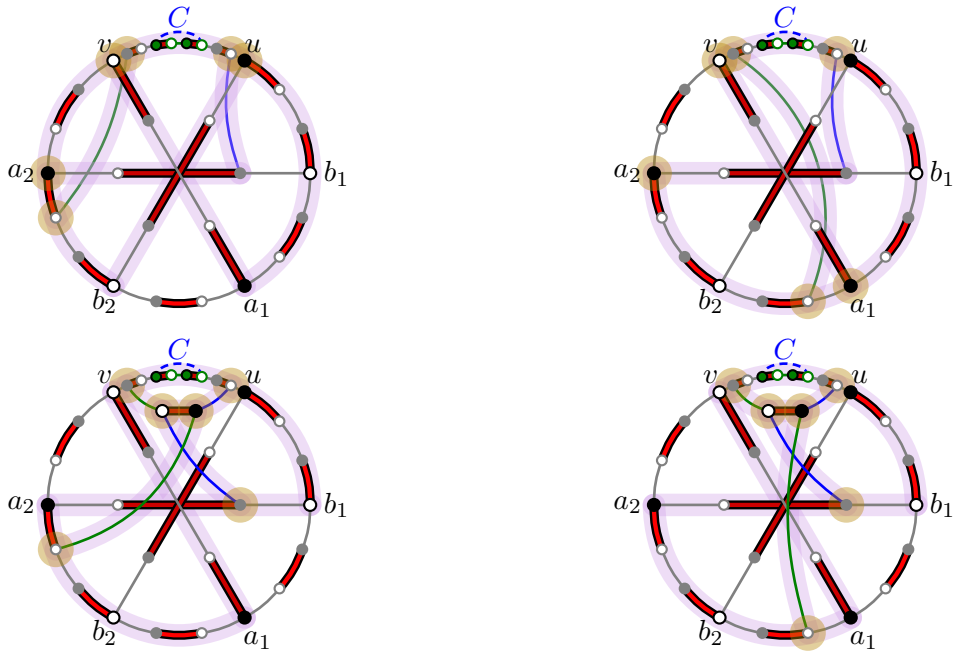
\begin{figure}[h!]
		\centering
		\begin{subfigure}{0.49\textwidth}
			\centering
			\begin{tikzpicture}[scale=1.15]
				\pgfdeclarelayer{background}
				\pgfdeclarelayer{foreground}
				\pgfsetlayers{background,main,foreground}
				
				\node[v:marked] () at (113.34:15mm){};
				\node[v:marked] () at (66.66:15mm){};
				\node[v:marked] () at (120:15mm){};
				\node[v:marked] () at (60:15mm){};
				\node[v:marked] () at (200:15mm){};
				\node[v:marked] () at (180:15mm){};
				
				\foreach \x in {0,2,4}
				{
					\node[v:mainempty] () at (\x*60:15mm){};
					\node[v:maingray] () at (\x*60:7mm){};
				}
				\foreach \x in {1,3,5}
				{
					\node[v:main] () at (\x*60:15mm){};  
					\node[v:mainemptygray] () at (\x*60:7mm){};
				}
				
				\node[v:maingray] () at (20:15mm){};
				\node[v:maingray] () at (140:15mm){};
				\node[v:maingray] () at (220:15mm){};
				\node[v:maingray] () at (260:15mm){};
				\node[v:maingray] () at (340:15mm){};
				
				\node[v:mainemptygray] () at (66.66:15mm){};
				\node[v:maingray] () at (73.32:15mm){};
				\node[v:mainemptygreen] () at (80.01:15mm){};
				\node[v:maingreen] () at (86.68:15mm){};
				\node[v:mainemptygreen] () at (93.35:15mm){};
				\node[v:maingreen] () at (100.01:15mm){};
				\node[v:mainemptygray] () at (106.68:15mm){};
				\node[v:maingray] () at (113.34:15mm){};
				
				\node[v:mainemptygray] () at (40:15mm){};
				\node[v:mainemptygray] () at (160:15mm){};
				\node[v:mainemptygray] () at (200:15mm){};
				\node[v:mainemptygray] () at (280:15mm){};
				\node[v:mainemptygray] () at (320:15mm){};
				
				\node () at (0:18mm){$b_{1}$};
				\node () at (60:17mm){$u$};
				\node () at (120:17mm){$v$};
				\node () at (180:18mm){$a_{2}$};
				\node () at (240:17mm){$b_{2}$};
				\node () at (300:18mm){$a_{1}$};
				
				\node[blue] () at (90:18mm){$C$};
				
				\begin{pgfonlayer}{background}
					\draw[e:main,myGreen,bend left=15] (113.34:15mm) to (200:15mm);
					\draw[e:main,blue,bend right=15] (66.66:15mm) to (0:7mm);        
					\draw[e:marker,BrightUbe,bend left=15] (113.34:15mm) to (200:15mm);
					\draw[e:marker,BrightUbe,bend right=15] (66.66:15mm) to (0:7mm);
					\draw[e:marker,BrightUbe] (0:7mm) -- (180:15mm);
					\draw[e:marker,BrightUbe] (60:15mm) -- (240:15mm);
					\draw[e:marker,BrightUbe] (120:15mm) -- (300:15mm);
					
					\draw[e:marker,BrightUbe,bend right=28] (300:15mm) to (0:15mm) to (60:15mm) to (120:15mm) to (180:15mm) to (240:15mm);
					\draw[e:main,gray] (60:15mm) to (66.66:15mm){};
					\draw[e:coloredborder] (66.66:15mm) to (73.32:15mm){};
					\draw[e:coloredthin,color=BostonUniversityRed] (66.66:15mm) to (73.32:15mm){};
					\draw[e:main,gray] (73.32:15mm) to (80.01:15mm){};
					\draw[e:coloredborder] (80.01:15mm) to (86.68:15mm){};
					\draw[e:coloredthin,color=BostonUniversityRed] (80.01:15mm) to (86.68:15mm){};
					\draw[e:main,myGreen] (86.68:15mm) to (93.35:15mm){};
					\draw[e:coloredborder] (93.35:15mm) to (100.01:15mm){};
					\draw[e:coloredthin,color=BostonUniversityRed] (93.35:15mm) to (100.01:15mm){};     
					\draw[e:main,gray] (100.01:15mm) to (106.68:15mm){};
					\draw[e:coloredborder] (106.68:15mm) to (113.34:15mm){};
					\draw[e:coloredthin,color=BostonUniversityRed] (106.68:15mm) to (113.34:15mm){}; 
					\draw[e:main,gray] (113.34:15mm) to (120:15mm){};
					
					\draw[e:main,blue,densely dashed,bend right=90] (80.01:15mm) to (100.01:15mm);
					
					\draw[e:coloredborder,bend right=10] (0:15mm) to (20:15mm);
					\draw[e:coloredthin,red,bend right=10] (0:15mm) to (20:15mm);
					\draw[e:main,gray,bend right=10] (20:15mm) to (40:15mm);
					\draw[e:coloredborder,bend right=10] (40:15mm) to (60:15mm);
					\draw[e:coloredthin,red,bend right=10] (40:15mm) to (60:15mm);
					
					\draw[e:main,gray,bend right=10] (120:15mm) to (140:15mm);
					\draw[e:coloredborder,bend right=10] (140:15mm) to (160:15mm);
					\draw[e:coloredthin,red,bend right=10] (140:15mm) to (160:15mm);
					\draw[e:main,gray,bend right=10] (160:15mm) to (180:15mm);
					
					\draw[e:coloredborder,bend right=10] (180:15mm) to (200:15mm);
					\draw[e:coloredthin,red,bend right=10] (180:15mm) to (200:15mm);
					\draw[e:main,gray,bend right=10] (200:15mm) to (220:15mm);
					\draw[e:coloredborder,bend right=10] (220:15mm) to (240:15mm);
					\draw[e:coloredthin,red,bend right=10] (220:15mm) to (240:15mm);

					\draw[e:main,gray,bend right=10] (240:15mm) to (260:15mm);
					\draw[e:coloredborder,bend right=10] (260:15mm) to (280:15mm);
					\draw[e:coloredthin,red,bend right=10] (260:15mm) to (280:15mm);
					\draw[e:main,gray,bend right=10] (280:15mm) to (300:15mm);

					\draw[e:main,gray,bend right=10] (300:15mm) to (320:15mm);
					\draw[e:coloredborder,bend right=10] (320:15mm) to (340:15mm);
					\draw[e:coloredthin,red,bend right=10] (320:15mm) to (340:15mm);
					\draw[e:main,gray,bend right=10] (340:15mm) to (0:15mm);
					
					\draw[e:coloredborder] (180:7mm) -- (0:7mm);
					\draw[e:coloredthin,color=BostonUniversityRed] (180:7mm) -- (0:7mm);
					\draw[e:main,gray] (180:15mm) -- (180:7mm);
					\draw[e:main,gray] (0:15mm) -- (0:7mm);
					
					\draw[e:coloredborder] (240:7mm) -- (60:7mm);
					\draw[e:coloredthin,color=BostonUniversityRed] (240:7mm) -- (60:7mm);
					\draw[e:main,gray] (240:15mm) -- (240:7mm);
					\draw[e:main,gray] (60:15mm) -- (60:7mm);
					
					\draw[e:coloredborder] (300:15mm) -- (300:7mm);
					\draw[e:coloredthin,color=BostonUniversityRed] (300:15mm) -- (300:7mm);
					\draw[e:coloredborder] (120:15mm) -- (120:7mm);
					\draw[e:coloredthin,color=BostonUniversityRed] (120:15mm) -- (120:7mm);
					\draw[e:main,gray] (300:7mm) -- (120:7mm);
					
				\end{pgfonlayer}
			\end{tikzpicture}
		\end{subfigure}
		\begin{subfigure}{0.49\textwidth}
			\centering
			\begin{tikzpicture}[scale=1.15]
				\pgfdeclarelayer{background}
				\pgfdeclarelayer{foreground}
				\pgfsetlayers{background,main,foreground}
				
				\node[v:marked] () at (113.34:15mm){};
				\node[v:marked] () at (66.66:15mm){};
				\node[v:marked] () at (120:15mm){};
				\node[v:marked] () at (300:15mm){};
				\node[v:marked] () at (280:15mm){};
				\node[v:marked] () at (180:15mm){};
				
				\foreach \x in {0,2,4}
				{
					\node[v:mainempty] () at (\x*60:15mm){};
					\node[v:maingray] () at (\x*60:7mm){};
				}
				\foreach \x in {1,3,5}
				{
					\node[v:main] () at (\x*60:15mm){};  
					\node[v:mainemptygray] () at (\x*60:7mm){};
				}
				
				\node[v:maingray] () at (20:15mm){};
				\node[v:maingray] () at (140:15mm){};
				\node[v:maingray] () at (220:15mm){};
				\node[v:maingray] () at (260:15mm){};
				\node[v:maingray] () at (340:15mm){};
				
				\node[v:mainemptygray] () at (66.66:15mm){};
				\node[v:maingray] () at (73.32:15mm){};
				\node[v:mainemptygreen] () at (80.01:15mm){};
				\node[v:maingreen] () at (86.68:15mm){};
				\node[v:mainemptygreen] () at (93.35:15mm){};
				\node[v:maingreen] () at (100.01:15mm){};
				\node[v:mainemptygray] () at (106.68:15mm){};
				\node[v:maingray] () at (113.34:15mm){};
				
				\node[v:mainemptygray] () at (40:15mm){};
				\node[v:mainemptygray] () at (160:15mm){};
				\node[v:mainemptygray] () at (200:15mm){};
				\node[v:mainemptygray] () at (280:15mm){};
				\node[v:mainemptygray] () at (320:15mm){};
				
				\node () at (0:18mm){$b_{1}$};
				\node () at (60:17mm){$u$};
				\node () at (120:17mm){$v$};
				\node () at (180:18mm){$a_{2}$};
				\node () at (240:17mm){$b_{2}$};
				\node () at (300:18mm){$a_{1}$};
				
				\node[blue] () at (90:18mm){$C$};
				
				\begin{pgfonlayer}{background}
					\draw[e:main,myGreen,bend left=40] (113.34:15mm) to (280:15mm);
					\draw[e:main,blue,bend right=15] (66.66:15mm) to (0:7mm);        
					\draw[e:marker,BrightUbe,bend left=40] (113.34:15mm) to (280:15mm);
					\draw[e:marker,BrightUbe,bend right=15] (66.66:15mm) to (0:7mm);
					\draw[e:marker,BrightUbe] (0:7mm) -- (180:15mm);
					\draw[e:marker,BrightUbe] (120:15mm) -- (300:15mm);
					
					\draw[e:marker,BrightUbe,bend right=28] (180:15mm) to (240:15mm) to (300:15mm) to (0:15mm) to (60:15mm) to (120:15mm);
					\draw[e:main,gray] (60:15mm) to (66.66:15mm){};
					\draw[e:coloredborder] (66.66:15mm) to (73.32:15mm){};
					\draw[e:coloredthin,color=BostonUniversityRed] (66.66:15mm) to (73.32:15mm){};
					\draw[e:main,gray] (73.32:15mm) to (80.01:15mm){};
					\draw[e:coloredborder] (80.01:15mm) to (86.68:15mm){};
					\draw[e:coloredthin,color=BostonUniversityRed] (80.01:15mm) to (86.68:15mm){};
					\draw[e:main,myGreen] (86.68:15mm) to (93.35:15mm){};
					\draw[e:coloredborder] (93.35:15mm) to (100.01:15mm){};
					\draw[e:coloredthin,color=BostonUniversityRed] (93.35:15mm) to (100.01:15mm){};     
					\draw[e:main,gray] (100.01:15mm) to (106.68:15mm){};
					\draw[e:coloredborder] (106.68:15mm) to (113.34:15mm){};
					\draw[e:coloredthin,color=BostonUniversityRed] (106.68:15mm) to (113.34:15mm){}; 
					\draw[e:main,gray] (113.34:15mm) to (120:15mm){};
					
					\draw[e:main,blue,densely dashed,bend right=90] (80.01:15mm) to (100.01:15mm);
					
					\draw[e:coloredborder,bend right=10] (0:15mm) to (20:15mm);
					\draw[e:coloredthin,red,bend right=10] (0:15mm) to (20:15mm);
					\draw[e:main,gray,bend right=10] (20:15mm) to (40:15mm);
					\draw[e:coloredborder,bend right=10] (40:15mm) to (60:15mm);
					\draw[e:coloredthin,red,bend right=10] (40:15mm) to (60:15mm);
					
					\draw[e:main,gray,bend right=10] (120:15mm) to (140:15mm);
					\draw[e:coloredborder,bend right=10] (140:15mm) to (160:15mm);
					\draw[e:coloredthin,red,bend right=10] (140:15mm) to (160:15mm);
					\draw[e:main,gray,bend right=10] (160:15mm) to (180:15mm);
					
					\draw[e:coloredborder,bend right=10] (180:15mm) to (200:15mm);
					\draw[e:coloredthin,red,bend right=10] (180:15mm) to (200:15mm);
					\draw[e:main,gray,bend right=10] (200:15mm) to (220:15mm);
					\draw[e:coloredborder,bend right=10] (220:15mm) to (240:15mm);
					\draw[e:coloredthin,red,bend right=10] (220:15mm) to (240:15mm);

					\draw[e:main,gray,bend right=10] (240:15mm) to (260:15mm);
					\draw[e:coloredborder,bend right=10] (260:15mm) to (280:15mm);
					\draw[e:coloredthin,red,bend right=10] (260:15mm) to (280:15mm);
					\draw[e:main,gray,bend right=10] (280:15mm) to (300:15mm);

					\draw[e:main,gray,bend right=10] (300:15mm) to (320:15mm);
					\draw[e:coloredborder,bend right=10] (320:15mm) to (340:15mm);
					\draw[e:coloredthin,red,bend right=10] (320:15mm) to (340:15mm);
					\draw[e:main,gray,bend right=10] (340:15mm) to (0:15mm);
					
					\draw[e:coloredborder] (180:7mm) -- (0:7mm);
					\draw[e:coloredthin,color=BostonUniversityRed] (180:7mm) -- (0:7mm);
					\draw[e:main,gray] (180:15mm) -- (180:7mm);
					\draw[e:main,gray] (0:15mm) -- (0:7mm);
					
					\draw[e:coloredborder] (240:7mm) -- (60:7mm);
					\draw[e:coloredthin,color=BostonUniversityRed] (240:7mm) -- (60:7mm);
					\draw[e:main,gray] (240:15mm) -- (240:7mm);
					\draw[e:main,gray] (60:15mm) -- (60:7mm);
					
					\draw[e:coloredborder] (300:15mm) -- (300:7mm);
					\draw[e:coloredthin,color=BostonUniversityRed] (300:15mm) -- (300:7mm);
					\draw[e:coloredborder] (120:15mm) -- (120:7mm);
					\draw[e:coloredthin,color=BostonUniversityRed] (120:15mm) -- (120:7mm);
					\draw[e:main,gray] (300:7mm) -- (120:7mm);
					
				\end{pgfonlayer}
			\end{tikzpicture}
		\end{subfigure}

		\begin{subfigure}{0.49\textwidth}
			\centering
			\begin{tikzpicture}[scale=1.15]
				\pgfdeclarelayer{background}
				\pgfdeclarelayer{foreground}
				\pgfsetlayers{background,main,foreground}
				
				\node[v:marked] () at (113.34:15mm){};
				\node[v:marked] () at (66.66:15mm){};
				\node[v:marked] () at (100:11mm){};
				\node[v:marked] () at (80:11mm){};
				\node[v:marked] () at (0:7mm){};
				\node[v:marked] () at (200:15mm){};
				
				\foreach \x in {0,2,4}
				{
					\node[v:mainempty] () at (\x*60:15mm){};
					\node[v:maingray] () at (\x*60:7mm){};
				}
				\foreach \x in {1,3,5}
				{
					\node[v:main] () at (\x*60:15mm){};  
					\node[v:mainemptygray] () at (\x*60:7mm){};
				}
				
				\node[v:maingray] () at (20:15mm){};
				\node[v:maingray] () at (140:15mm){};
				\node[v:maingray] () at (220:15mm){};
				\node[v:maingray] () at (260:15mm){};
				\node[v:maingray] () at (340:15mm){};
				
				\node[v:mainemptygray] () at (66.66:15mm){};
				\node[v:maingray] () at (73.32:15mm){};
				\node[v:mainemptygreen] () at (80.01:15mm){};
				\node[v:maingreen] () at (86.68:15mm){};
				\node[v:mainemptygreen] () at (93.35:15mm){};
				\node[v:maingreen] () at (100.01:15mm){};
				\node[v:mainemptygray] () at (106.68:15mm){};
				\node[v:maingray] () at (113.34:15mm){};
				
				\node[v:mainemptygray] () at (40:15mm){};
				\node[v:mainemptygray] () at (160:15mm){};
				\node[v:mainemptygray] () at (200:15mm){};
				\node[v:mainemptygray] () at (280:15mm){};
				\node[v:mainemptygray] () at (320:15mm){};
				
				\node[v:main] () at (80:11mm){};
				\node[v:mainempty] () at (100:11mm){};
				
				\node () at (0:18mm){$b_{1}$};
				\node () at (60:17mm){$u$};
				\node () at (120:17mm){$v$};
				\node () at (180:18mm){$a_{2}$};
				\node () at (240:17mm){$b_{2}$};
				\node () at (300:18mm){$a_{1}$};
				
				\node[blue] () at (90:18mm){$C$};
				
				\begin{pgfonlayer}{background}
					\draw[e:marker,BrightUbe,bend right=15] (80:11mm) to (66.66:15mm);
					\draw[e:marker,BrightUbe,bend right=15] (100:11mm) to (0:7mm);
					\draw[e:marker,BrightUbe,bend left=20] (100:11mm) to (113.34:15mm);
					\draw[e:marker,BrightUbe,bend left=30] (80:11mm) to (200:15mm);
					\draw[e:marker,BrightUbe] (80:11mm) -- (100:11mm);
					
					\draw[e:marker,BrightUbe] (0:15mm) -- (180:15mm);
					\draw[e:marker,BrightUbe] (120:15mm) -- (300:15mm);
					
					\draw[e:marker,BrightUbe,bend right=28] (0:15mm) to (60:15mm) to (120:15mm);
					\draw[e:marker,BrightUbe,bend right=28] (180:15mm) to (240:15mm) to (300:15mm);
					
					\draw[e:main,gray] (60:15mm) to (66.66:15mm){};
					\draw[e:coloredborder] (66.66:15mm) to (73.32:15mm){};
					\draw[e:coloredthin,color=BostonUniversityRed] (66.66:15mm) to (73.32:15mm){};
					\draw[e:main,gray] (73.32:15mm) to (80.01:15mm){};
					\draw[e:coloredborder] (80.01:15mm) to (86.68:15mm){};
					\draw[e:coloredthin,color=BostonUniversityRed] (80.01:15mm) to (86.68:15mm){};
					\draw[e:main,myGreen] (86.68:15mm) to (93.35:15mm){};
					\draw[e:coloredborder] (93.35:15mm) to (100.01:15mm){};
					\draw[e:coloredthin,color=BostonUniversityRed] (93.35:15mm) to (100.01:15mm){};     
					\draw[e:main,gray] (100.01:15mm) to (106.68:15mm){};
					\draw[e:coloredborder] (106.68:15mm) to (113.34:15mm){};
					\draw[e:coloredthin,color=BostonUniversityRed] (106.68:15mm) to (113.34:15mm){}; 
					\draw[e:main,gray] (113.34:15mm) to (120:15mm){};
					
					\draw[e:main,blue,densely dashed,bend right=90] (80.01:15mm) to (100.01:15mm);
					
					\draw[e:coloredborder,bend right=10] (0:15mm) to (20:15mm);
					\draw[e:coloredthin,red,bend right=10] (0:15mm) to (20:15mm);
					\draw[e:main,gray,bend right=10] (20:15mm) to (40:15mm);
					\draw[e:coloredborder,bend right=10] (40:15mm) to (60:15mm);
					\draw[e:coloredthin,red,bend right=10] (40:15mm) to (60:15mm);
					
					\draw[e:main,gray,bend right=10] (120:15mm) to (140:15mm);
					\draw[e:coloredborder,bend right=10] (140:15mm) to (160:15mm);
					\draw[e:coloredthin,red,bend right=10] (140:15mm) to (160:15mm);
					\draw[e:main,gray,bend right=10] (160:15mm) to (180:15mm);
					
					\draw[e:coloredborder,bend right=10] (180:15mm) to (200:15mm);
					\draw[e:coloredthin,red,bend right=10] (180:15mm) to (200:15mm);
					\draw[e:main,gray,bend right=10] (200:15mm) to (220:15mm);
					\draw[e:coloredborder,bend right=10] (220:15mm) to (240:15mm);
					\draw[e:coloredthin,red,bend right=10] (220:15mm) to (240:15mm);

					\draw[e:main,gray,bend right=10] (240:15mm) to (260:15mm);
					\draw[e:coloredborder,bend right=10] (260:15mm) to (280:15mm);
					\draw[e:coloredthin,red,bend right=10] (260:15mm) to (280:15mm);
					\draw[e:main,gray,bend right=10] (280:15mm) to (300:15mm);

					\draw[e:main,gray,bend right=10] (300:15mm) to (320:15mm);
					\draw[e:coloredborder,bend right=10] (320:15mm) to (340:15mm);
					\draw[e:coloredthin,red,bend right=10] (320:15mm) to (340:15mm);
					\draw[e:main,gray,bend right=10] (340:15mm) to (0:15mm);
					
					\draw[e:coloredborder] (180:7mm) -- (0:7mm);
					\draw[e:coloredthin,color=BostonUniversityRed] (180:7mm) -- (0:7mm);
					\draw[e:main,gray] (180:15mm) -- (180:7mm);
					\draw[e:main,gray] (0:15mm) -- (0:7mm);
					
					\draw[e:coloredborder] (240:7mm) -- (60:7mm);
					\draw[e:coloredthin,color=BostonUniversityRed] (240:7mm) -- (60:7mm);
					\draw[e:main,gray] (240:15mm) -- (240:7mm);
					\draw[e:main,gray] (60:15mm) -- (60:7mm);
					
					\draw[e:coloredborder] (300:15mm) -- (300:7mm);
					\draw[e:coloredthin,color=BostonUniversityRed] (300:15mm) -- (300:7mm);
					\draw[e:coloredborder] (120:15mm) -- (120:7mm);
					\draw[e:coloredthin,color=BostonUniversityRed] (120:15mm) -- (120:7mm);
					\draw[e:main,gray] (300:7mm) -- (120:7mm);
					
					\draw[e:coloredborder] (80:11mm) -- (100:11mm);
					\draw[e:coloredthin,color=BostonUniversityRed] (80:11mm) -- (100:11mm);
					
					\draw[e:main,blue,bend right=15] (80:11mm) to (66.66:15mm); \draw[e:main,blue,bend right=15] (100:11mm) to (0:7mm);
					\draw[e:main,myGreen,bend left=20] (100:11mm) to (113.34:15mm);
					\draw[e:main,myGreen,bend left=30] (80:11mm) to (200:15mm);
					
				\end{pgfonlayer}
			\end{tikzpicture}
		\end{subfigure}
		\begin{subfigure}{0.49\textwidth}
			\centering
			\begin{tikzpicture}[scale=1.15]
				\pgfdeclarelayer{background}
				\pgfdeclarelayer{foreground}
				\pgfsetlayers{background,main,foreground}
				
				\node[v:marked] () at (113.34:15mm){};
				\node[v:marked] () at (66.66:15mm){};
				\node[v:marked] () at (100:11mm){};
				\node[v:marked] () at (80:11mm){};
				\node[v:marked] () at (0:7mm){};
				\node[v:marked] () at (280:15mm){};
				
				\foreach \x in {0,2,4}
				{
					\node[v:mainempty] () at (\x*60:15mm){};
					\node[v:maingray] () at (\x*60:7mm){};
				}
				\foreach \x in {1,3,5}
				{
					\node[v:main] () at (\x*60:15mm){};  
					\node[v:mainemptygray] () at (\x*60:7mm){};
				}
				
				\node[v:maingray] () at (20:15mm){};
				\node[v:maingray] () at (140:15mm){};
				\node[v:maingray] () at (220:15mm){};
				\node[v:maingray] () at (260:15mm){};
				\node[v:maingray] () at (340:15mm){};
				
				\node[v:mainemptygray] () at (66.66:15mm){};
				\node[v:maingray] () at (73.32:15mm){};
				\node[v:mainemptygreen] () at (80.01:15mm){};
				\node[v:maingreen] () at (86.68:15mm){};
				\node[v:mainemptygreen] () at (93.35:15mm){};
				\node[v:maingreen] () at (100.01:15mm){};
				\node[v:mainemptygray] () at (106.68:15mm){};
				\node[v:maingray] () at (113.34:15mm){};
				
				\node[v:mainemptygray] () at (40:15mm){};
				\node[v:mainemptygray] () at (160:15mm){};
				\node[v:mainemptygray] () at (200:15mm){};
				\node[v:mainemptygray] () at (280:15mm){};
				\node[v:mainemptygray] () at (320:15mm){};
				
				\node[v:main] () at (80:11mm){};
				\node[v:mainempty] () at (100:11mm){};
				
				\node () at (0:18mm){$b_{1}$};
				\node () at (60:17mm){$u$};
				\node () at (120:17mm){$v$};
				\node () at (180:18mm){$a_{2}$};
				\node () at (240:17mm){$b_{2}$};
				\node () at (300:18mm){$a_{1}$};
				
				\node[blue] () at (90:18mm){$C$};
				
				\begin{pgfonlayer}{background}
					
					\draw[e:marker,BrightUbe,bend right=15] (80:11mm) to (66.66:15mm);
					\draw[e:marker,BrightUbe,bend right=15] (100:11mm) to (0:7mm);
					\draw[e:marker,BrightUbe,bend left=20] (100:11mm) to (113.34:15mm);
					\draw[e:marker,BrightUbe,bend right=15] (80:11mm) to (280:15mm);
					\draw[e:marker,BrightUbe] (80:11mm) -- (100:11mm);
					
					\draw[e:marker,BrightUbe] (0:15mm) -- (180:15mm);
					\draw[e:marker,BrightUbe] (120:15mm) -- (300:15mm);
					
					\draw[e:marker,BrightUbe,bend right=28] (0:15mm) to (60:15mm) to (120:15mm);
					\draw[e:marker,BrightUbe,bend right=28] (180:15mm) to (240:15mm) to (300:15mm);
					
					\draw[e:main,gray] (60:15mm) to (66.66:15mm){};
					\draw[e:coloredborder] (66.66:15mm) to (73.32:15mm){};
					\draw[e:coloredthin,color=BostonUniversityRed] (66.66:15mm) to (73.32:15mm){};
					\draw[e:main,gray] (73.32:15mm) to (80.01:15mm){};
					\draw[e:coloredborder] (80.01:15mm) to (86.68:15mm){};
					\draw[e:coloredthin,color=BostonUniversityRed] (80.01:15mm) to (86.68:15mm){};
					\draw[e:main,myGreen] (86.68:15mm) to (93.35:15mm){};
					\draw[e:coloredborder] (93.35:15mm) to (100.01:15mm){};
					\draw[e:coloredthin,color=BostonUniversityRed] (93.35:15mm) to (100.01:15mm){};     
					\draw[e:main,gray] (100.01:15mm) to (106.68:15mm){};
					\draw[e:coloredborder] (106.68:15mm) to (113.34:15mm){};
					\draw[e:coloredthin,color=BostonUniversityRed] (106.68:15mm) to (113.34:15mm){}; 
					\draw[e:main,gray] (113.34:15mm) to (120:15mm){};
					
					\draw[e:main,blue,densely dashed,bend right=90] (80.01:15mm) to (100.01:15mm);
					
					\draw[e:coloredborder,bend right=10] (0:15mm) to (20:15mm);
					\draw[e:coloredthin,red,bend right=10] (0:15mm) to (20:15mm);
					\draw[e:main,gray,bend right=10] (20:15mm) to (40:15mm);
					\draw[e:coloredborder,bend right=10] (40:15mm) to (60:15mm);
					\draw[e:coloredthin,red,bend right=10] (40:15mm) to (60:15mm);
					
					\draw[e:main,gray,bend right=10] (120:15mm) to (140:15mm);
					\draw[e:coloredborder,bend right=10] (140:15mm) to (160:15mm);
					\draw[e:coloredthin,red,bend right=10] (140:15mm) to (160:15mm);
					\draw[e:main,gray,bend right=10] (160:15mm) to (180:15mm);
					
					\draw[e:coloredborder,bend right=10] (180:15mm) to (200:15mm);
					\draw[e:coloredthin,red,bend right=10] (180:15mm) to (200:15mm);
					\draw[e:main,gray,bend right=10] (200:15mm) to (220:15mm);
					\draw[e:coloredborder,bend right=10] (220:15mm) to (240:15mm);
					\draw[e:coloredthin,red,bend right=10] (220:15mm) to (240:15mm);

					\draw[e:main,gray,bend right=10] (240:15mm) to (260:15mm);
					\draw[e:coloredborder,bend right=10] (260:15mm) to (280:15mm);
					\draw[e:coloredthin,red,bend right=10] (260:15mm) to (280:15mm);
					\draw[e:main,gray,bend right=10] (280:15mm) to (300:15mm);

					\draw[e:main,gray,bend right=10] (300:15mm) to (320:15mm);
					\draw[e:coloredborder,bend right=10] (320:15mm) to (340:15mm);
					\draw[e:coloredthin,red,bend right=10] (320:15mm) to (340:15mm);
					\draw[e:main,gray,bend right=10] (340:15mm) to (0:15mm);
					
					\draw[e:coloredborder] (180:7mm) -- (0:7mm);
					\draw[e:coloredthin,color=BostonUniversityRed] (180:7mm) -- (0:7mm);
					\draw[e:main,gray] (180:15mm) -- (180:7mm);
					\draw[e:main,gray] (0:15mm) -- (0:7mm);
					
					\draw[e:coloredborder] (240:7mm) -- (60:7mm);
					\draw[e:coloredthin,color=BostonUniversityRed] (240:7mm) -- (60:7mm);
					\draw[e:main,gray] (240:15mm) -- (240:7mm);
					\draw[e:main,gray] (60:15mm) -- (60:7mm);
					
					\draw[e:coloredborder] (300:15mm) -- (300:7mm);
					\draw[e:coloredthin,color=BostonUniversityRed] (300:15mm) -- (300:7mm);
					\draw[e:coloredborder] (120:15mm) -- (120:7mm);
					\draw[e:coloredthin,color=BostonUniversityRed] (120:15mm) -- (120:7mm);
					\draw[e:main,gray] (300:7mm) -- (120:7mm);
					
					\draw[e:coloredborder] (80:11mm) -- (100:11mm);
					\draw[e:coloredthin,color=BostonUniversityRed] (80:11mm) -- (100:11mm);
					
					\draw[e:main,blue,bend right=15] (80:11mm) to (66.66:15mm); \draw[e:main,blue,bend right=15] (100:11mm) to (0:7mm);
					\draw[e:main,myGreen,bend left=20] (100:11mm) to (113.34:15mm);
					\draw[e:main,myGreen,bend right=15] (80:11mm) to (280:15mm);
					
				\end{pgfonlayer}
			\end{tikzpicture}
		\end{subfigure}
		\caption{The construction of the new conformal $K_{3,3}$-bisubdivision in the third case of the proof of \cref{lemma:reducedK33mustsplit}.} 
		\label{fig:shrinkingPno3}
	\end{figure}
	
	So whenever we find both a $V_1$-jump and a $V_2$-jump in $L$, we are able to find a conformal bisubdivision $L'$ of $K_{3,3}$ with a bisubdivided edge $P'$ that contains all of $C$ but is shorter than $P$ in the previous bisubdivision. 
	Thus by choosing $L$ with minimal $P$, we still find a non-trivial tight cut as in the proofs of \cref{lemma:reducedK33,lemma:reducedK33withjump}, but neither of these tight cuts may yield a $V_1$-jump.
	Hence we must be able to construct a conformal bisubdivision of $K_{3,3}$ that splits $C$.
\end{proof}

With this we are ready to close this section with the proof of \cref{prop:4cycleK33}.

\begin{proof}[Proof of \Cref{prop:4cycleK33}]
	Suppose $B$ is a counterexample, so there exists a $4$-cycle $C$ in $B$ such that $C$ is not a subgraph of a conformal $K_{3,3}$-bisubdivision.
	By \cref{lemma:goodcrossesmeanK33} this means that there is no conformal cross over $C$ in $B$ and thus, by \cref{lemma:reducedK33mustsplit} there exists a conformal bisubdivision $L$ of $K_{3,3}$ that splits $C$.
	As we have seen in \cref{fig:threeconfigurations} we may choose $L$ such that one of the degree three vertices in $L$ belongs to $C$, let us call that vertex $u$.
	Let $P_1$, $P_2$, and $P_3$ be the bisubdivided edges of $L$ that have $u$ as an endpoint and let $v_i$ be the other endpoint of $P_i$ for all $i\in [1,3]$.
	Let $Y\coloneqq \bigcup_{i=1}^3\V{P_i-v_i}$ and let us choose $L$ among all conformal bisubdivisions of $K_{3,3}$ in $B$ that split $C$ to be one where $\Abs{Y}$ is minimal.
	As $L$ splits $C$ we still have $\Abs{Y\cap\V{C}}\geq 3$ and $\Abs{\V{L}\setminus Y}\geq 3$ and thus $Y$ is, as we have seen before, a non-trivial tight cut whose majority is exactly the colour class $u$ belongs to.
	Without loss of generality let us assume the minority of $Y$ to be in $V_1$.
	Observe that $C\cap L$ forms an $M$-conformal path that must contain internal vertices of two different bisubdivided edges of $L$.
	By \cref{lemma:tightcutsincofnromalsubgraphs} there must exist an internally $M$-conformal path $F$ that has one endpoint in $Y\cap V_1$ and the other one in $\Vi{2}{L-Y}$ such that $F$ is internally disjoint from $L$.
	Recall the constructions illustrated in \cref{fig:firstLreductions} and suppose the endpoint of $F$ in $Y$ is an interior vertex of $C\cap Y$.
	If this is the case, we find a conformal bisubdivision $L'$ of $K_{3,3}$ in which $ab$ and $a'b'$ belong to two different bisubdivided edges which do not share an endpoint.
	By \cref{lemma:easycrossesinK33} this means we find a conformal bisubdivision of $K_{3,3}$ which contains $C$ as a subgraph, contradicting $B$ being a counterexample.
	Hence $F$ cannot contain an inner vertex of $C\cap L$.
	But in this case, we can find a conformal $K_{3,3}$-bisubdivision $L'$ that splits $C$ such that $Y'$, which is defined for $L'$ in the same way as $Y$ is defined for $L$, contains fewer vertices than $Y$ which contradicts our choice of $L$.
	So either way we reach a contradiction and thus there cannot be a counterexample to our claim.
\end{proof}

\section{An Algorithm for $2$-MLP}\label{sec:algorithm}

To obtain an algorithmic solution for $2$-MLP, we use \cref{prop:4cycleK33} together with \cref{cor:pfaffianalg}.
On a high level, we run into the following problems:
First, we do not know for which perfect matching $M$ of $B$ we might be able to find a solution for $2$-MLP and since there is a potentially exponential number of perfect matchings in $B$ it clearly does not suffice to simply test all of them.
Indeed, such an approach is doomed from the beginning since trying to solve $2$-MLP for a fixed perfect matching is equivalent to the Directed $2$-Disjoint Paths Problem.
So we take a slightly different approach.
Let $a_1,a_2,b_1,b_2$ be the four vertices of an instance of $2$-MLP.
Let $\Abs{\V{B}}=n$, then $B$ contains $\frac{n}{2}$ vertices of each colour.
For each $x\in\Set{a_1,a_2,b_1,b_2}$ we may choose from among the $\frac{n}{2}$ vertices of the opposite colour in order to find a neighbour that might be matched to $x$ by some perfect matching of $B$.
In total this means there are at most
\begin{align*}
	2\Choose{\frac{n}{2}}{2}=2\frac{\frac{n}{2}\Brace{\frac{n}{2}-1}}{2}\in\Fkt{\mathcal{O}}{n^2}
\end{align*}
many choices of edges that might cover our four terminal vertices in a perfect matching of $B$.
Let $F\subseteq\E{B}$ be a set of at most four edges such that each vertex from among $a_1,a_2,b_1$, and $b_2$ is covered by an edge of $F$.
Next we need to decide whether $F$ is contained in a perfect matching of $B$, which can be done by the Hopcroft-Karp algorithm in time $\Fkt{\mathcal{O}}{n^{\frac{5}{2}}}$ \cite{hopcroft1973n}.
In case such a perfect matching exists, we then alter the graph $B$ locally which takes up constant time.
The main concern of this section is to introduce this local construction and to show that the existence of a conformal cross over a well-chosen $4$-cycle certifies the existence of the desired linkage in a way that makes use of the matching edges in $F$.
The key to deciding whether a conformal cross over our $4$-cycle exists is \cref{prop:4cycleK33} in combination with \cref{cor:pfaffianalg}.
In total, this approach decides $2$-MLP in time $\Fkt{\mathcal{O}}{n^5}$.

Let $B$ be a bipartite graph with a perfect matching, $F\subseteq\E{B}$ and $X\subseteq\V{B}$.
The set $F$ is said to be an \emph{$X$-cover}, if every edge in $F$ contains a vertex of $X$ and every vertex in $X$ is covered by an edge of $F$.
If $F$ is an extendible set of edges in $B$ and $M$ is a perfect matching $B$ with $F\subseteq M$, $M$ is said to \emph{extend} $F$.
From the discussion above it is clear that there are $\Fkt{\mathcal{O}}{\Abs{\V{B}}^2}$ many $X$-covers in $B$ for any set $X\subseteq\V{B}$ with $\Abs{X\cap V_1}=\Abs{X\cap V_2}=2$.
Given distinct vertices $a_1,a_2\in V_1$, $b_1,b_2\in V_2$, and an extendible $\Set{a_1,a_2,b_1,b_2}$-cover $F\subseteq\E{B}$ we say that $B$ is an \emph{$F$-instance} of $2$-MLP for $\Brace{a_1,a_2}$ and $\Brace{b_1,b_2}$ if there exists a perfect matching $M$ of $B$ that extends $F$ such that there are two disjoint internally $M$-conformal paths $P_1$ and $P_2$ such that $P_i$ has endpoints $a_i$ and $b_i$ for each $i\in[1,2]$.

\begin{definition}
	Let $B$ be a bipartite graph with a perfect matching, $a_1,a_2\in V_1$ and $b_1,b_2\in V_2$ four distinct vertices of $B$, and $F$ an extendible $\Set{a_1,a_2,b_1,b_2}$-cover of size four.
	Let $ua_2,vb_1\in F$.
	We define the following transformation of $B$ with respect to $F$, $\Brace{a_1,a_2}$, and $\Brace{b_1,b_2}$.
	\begin{align*}
		\MLPInstance{3}{B,F,\Set{a_1,a_2,b_1,b_2}}&\coloneqq B-u-v+a_2b_1\\
		\MLPCover{3}{B,F,\Set{a_1,a_2,b_1,b_2}}&\coloneqq \Brace{F\setminus\Set{ua_2,vb_1}}\cup\Set{a_2b_1}
	\end{align*}
\end{definition}

\begin{lemma}\label{lemma:reducefrom4to3}
	Let $B$ be a bipartite graph with a perfect matching, $a_1,a_2\in V_1$ and $b_1,b_2\in V_2$ four distinct vertices of $B$, and $F$ an extendible $\Set{a_1,a_2,b_1,b_2}$-cover of size four.
	Then $B$ is an $F$-instance of $2$-MLP for $\Brace{a_1,a_2}$ and $\Brace{b_1,b_2}$ if and only if $\MLPInstance{3}{B,F,\Set{a_1,a_2,b_1,b_2}}$ is a $\MLPCover{3}{B,F,\Set{a_1,a_2,b_1,b_2}}$-instance of $2$-MLP for $\Brace{a_1,a_2}$ and $\Brace{b_1,b_2}$.
\end{lemma}

\begin{proof}
	If $B$ is an $F$-instance of $2$-MLP for $\Brace{a_1,a_2}$ and $\Brace{b_1,b_2}$ there is a perfect matching $M$ that extends $F$ such that there exist disjoint and internally $M$-conformal paths $P_1$ and $P_2$ where $P_i$ has endpoints $a_i$ and $b_i$ for each $i\in[1,2]$.
	Let $ua_2,vb_1\in F$, then $\Set{u,v}\cap\Brace{\V{P_1}\cup\V{P_2}}=\emptyset$.
	Let us add the edges $uv$ and $a_2b_1$ to $B$, then $C\coloneqq ua_2b_1vu$ is an $M$-conformal $4$-cycle in $B$.
	We set $M'\coloneqq \Brace{M'\setminus\E{C}}\cup\Set{uv,a_2b_1}$, then $P_1$ and $P_2$ are internally $M'$-conformal paths that still exist in $\MLPInstance{3}{B,F,\Set{a_1,a_2,b_1,b_2}}$ and $M'\setminus\Set{uv}$ is a perfect matching of $\MLPInstance{3}{B,F,\Set{a_1,a_2,b_1,b_2}}$ that extends $\MLPCover{3}{B,F,\Set{a_1,a_2,b_1,b_2}}$.
	Hence $\MLPInstance{3}{B,F,\Set{a_1,a_2,b_1,b_2}}$ is a $\MLPCover{3}{B,F,\Set{a_1,a_2,b_1,b_2}}$-instance of $2$-MLP for $\Brace{a_1,a_2}$ and $\Brace{b_1,b_2}$.
	
	Now assume that $\MLPInstance{3}{B,F,\Set{a_1,a_2,b_1,b_2}}$ is a $\MLPCover{3}{B,F,\Set{a_1,a_2,b_1,b_2}}$-instance of $2$-MLP for $\Brace{a_1,a_2}$ and $\Brace{b_1,b_2}$.
	As before let $M$ be a perfect matching extending $\MLPCover{3}{B,F,\Set{a_1,a_2,b_1,b_2}}$ and let $P_1$, $P_2$ be the corresponding internally $M$-conformal paths.
	Let $ua_2,vb_1\in F$, then $M\cup\Set{uv}$ is a perfect matching of $B+a_2b_1+uv$ and $C\coloneqq ua_2b_1vu$ is an $M\cup\Set{uv}$-conformal $4$-cycle in $B+a_2b_1+uv$.
	We set $M'\coloneqq \Brace{M'\setminus\E{C}}\cup\Set{ua_2,vb_1}$, then $P_1$ and $P_2$ are internally $M'$-conformal paths in $B+a_2b_1+uv$ that still exist in $B$ and $M'$ is also a perfect matching of $B$.
	Thus $B$ is an $F$-instance of $2$-MLP for $\Brace{a_1,a_2}$ and $\Brace{b_1,b_2}$.
\end{proof}

\begin{definition}
	Let $B$ be a bipartite graph with a perfect matching, $a_1,a_2\in V_1$ and $b_1,b_2\in V_2$ four distinct vertices of $B$, $S\coloneqq\Set{a_1,a_2,b_1,b_2}$, and $F$ an extendible $S$-cover of size at least three such that $a_2b_1\in F$ if and only if $\Abs{F}=3$.
	Let $ua_1,vb_2\in F$.
	We define the following transformation of $B$ with respect to $F$, $\Brace{a_1,a_2}$, and $\Brace{b_1,b_2}$.
	If $\Abs{F}=3$ use the following construction:
	\begin{align*}
		\MLPInstance{2}{B,F,S}&\coloneqq B-u-v+a_1b_2\\
		\MLPCover{2}{B,F,S}&\coloneqq \Brace{F\setminus\Set{ua_1,vb_2}}\cup\Set{a_1b_2}
	\end{align*}
	Otherwise, we can first obtain an instance where our extendible cover has size three as required:
	\begin{align*}
		\MLPInstance{2}{B,F,S}&\coloneqq\MLPInstance{2}{\MLPInstance{3}{B,F,S},\MLPCover{3}{B,F,S},S}\\
		\MLPCover{2}{B,F,S}&\coloneqq \MLPCover{2}{\MLPInstance{3}{B,F,S},\MLPCover{3}{B,F,S},S}
	\end{align*}
\end{definition}

\begin{lemma}\label{lemma:reducefrom3to2}
	Let $B$ be a bipartite graph with a perfect matching, $a_1,a_2\in V_1$ and $b_1,b_2\in V_2$ four distinct vertices of $B$, $S\coloneqq\Set{a_1,a_2,b_1,b_2}$, and $F$ an extendible $S$-cover of size at least three.
	Then $B$ is an $F$-instance of $2$-MLP for $\Brace{a_1,a_2}$ and $\Brace{b_1,b_2}$ if and only if $\MLPInstance{2}{B,F,S}$ is a $\MLPCover{2}{B,F,S}$-instance of $2$-MLP for $\Brace{a_1,a_2}$ and $\Brace{b_1,b_2}$.
\end{lemma}

\begin{proof}
	We only have to consider the case $\Abs{F}=3$, since the case $\Abs{F}=4$ follows, by \cref{lemma:reducefrom4to3}, with the same arguments.
	
	If $B$ is an $F$-instance of $2$-MLP for $\Brace{a_1,a_2}$ and $\Brace{b_1,b_2}$ there is a perfect matching $M$ that extends $F$ such that there exist disjoint and internally $M$-conformal paths $P_1$ and $P_2$ where $P_i$ has endpoints $a_i$ and $b_i$ for each $i\in[1,2]$.
	Let $ua_1,vb_2\in F$, then $\Set{u,v}\cap\Brace{\V{P_1}\cup\V{P_2}}=\emptyset$.
	Let us add the edges $uv$ and $a_1b_2$ to $B$, then $C\coloneqq ua_1b_2vu$ is an $M$-conformal $4$-cycle in $B$.
	We set $M'\coloneqq \Brace{M'\setminus\E{C}}\cup\Set{uv,a_1b_2}$, then $P_1$ and $P_2$ are internally $M'$-conformal paths that still exist in $\MLPInstance{2}{B,F,S}$ and $M'\setminus\Set{uv}$ is a perfect matching of $\MLPInstance{2}{B,F,S}$ that extends $\MLPCover{2}{B,F,S}$.
	Hence $\MLPInstance{2}{B,F,S}$ is a $\MLPCover{2}{B,F,S}$-instance of $2$-MLP for $\Brace{a_1,a_2}$ and $\Brace{b_1,b_2}$.
	
	Now assume that $\MLPInstance{2}{B,F,S}$ is a $\MLPCover{2}{B,F,S}$-instance of $2$-MLP for $\Brace{a_1,a_2}$ and $\Brace{b_1,b_2}$.
	As before let $M$ be a perfect matching extending $\MLPCover{2}{B,F,S}$ and let $P_1$, $P_2$ be the corresponding internally $M$-conformal paths.
	Let $ua_1,vb_2\in F$, then $M\cup\Set{uv}$ is a perfect matching of $B+a_1b_2+uv$ and $C\coloneqq ua_1b_2vu$ is an $M\cup\Set{uv}$-conformal $4$-cycle in $B+a_1b_2+uv$.
	We set $M'\coloneqq \Brace{M'\setminus\E{C}}\cup\Set{ua_1,vb_2}$, then $P_1$ and $P_2$ are internally $M'$-conformal paths in $B+a_1b_2+uv$ that still exist in $B$ and $M'$ is also a perfect matching of $B$.
	Thus $B$ is an $F$-instance of $2$-MLP for $\Brace{a_1,a_2}$ and $\Brace{b_1,b_2}$.
\end{proof}

\begin{definition}
	Let $B$ be a bipartite graph with a perfect matching, $a_1,a_2\in V_1$ and $b_1,b_2\in V_2$ four distinct vertices of $B$, $S\coloneqq\Set{a_1,a_2,b_1,b_2}$, and $F$ an extendible $S$-cover such that $a_2b_1\in F$ if and only if $\Abs{F}\leq 3$.
	We define the following transformation of $B$ with respect to $F$, $\Brace{a_1,a_2}$, and $\Brace{b_1,b_2}$.
	If $\Abs{F}=2$, and therefore $F=\Set{a_1b_2,a_2b_1}$, use the following construction:
	Let $x,y$ be two distinct vertices that do not belong to $B$.
	\begin{align*}
		\MLPInstance{}{B,F,S}&\coloneqq B+x+y+\Set{xy,xb_1,xb_2,ya_1,ya_2}
	\end{align*}
	Otherwise, we can first obtain an instance where our extendible cover has size two as required above:
	\begin{align*}
		\MLPInstance{}{B,F,S}&\coloneqq\MLPInstance{}{\MLPInstance{2}{B,F,S},\MLPCover{2}{B,F,S},S}
	\end{align*}
	In either case let $\MLPCycle{B,F,S}$ be the $4$-cycle $a_1yxb_2$.
\end{definition}

\begin{lemma}\label{reducefrom2toK33}
	Let $B$ be a bipartite graph with a perfect matching, $a_1,a_2\in V_1$ and $b_1,b_2\in V_2$ four distinct vertices of $B$, $S\coloneqq\Set{a_1,a_2,b_1,b_2}$, and $F$ an extendible $S$-cover such that $a_2b_1\in F$ if and only if $\Abs{F}\leq 3$.
	Then $B$ is an $F$-instance of $2$-MLP for $\Brace{a_1,a_2}$ and $\Brace{b_1,b_2}$ if and only if there exists a conformal cross over $\MLPCycle{B,F,S}$ in $\MLPInstance{}{B,F,S}$.
\end{lemma}

\begin{proof}
	In case $\Abs{F}\geq 3$ we may replace $B$ and $F$ by $\MLPInstance{2}{B,F,S}$ and $\MLPCover{2}{B,F,S}$ without influencing the fact whether $B$ is an $F$-instance of $2$-MLP for $\Brace{a_1,a_2}$ and $\Brace{b_1,b_2}$ by \cref{lemma:reducefrom3to2}.
	Hence, without loss of generality, we may assume $\Abs{F}=2$.
	
	If $B$ is an $F$-instance of $2$-MLP for $\Brace{a_1,a_2}$ and $\Brace{b_1,b_2}$ there is a perfect matching $M$ that extends $F$ such that there exist disjoint and internally $M$-conformal paths $P_1$ and $P_2$ where $P_i$ has endpoints $a_i$ and $b_i$ for each $i\in[1,2]$.
	Then $M'\coloneqq M\cup\Set{xy}$ is a perfect matching of $\MLPInstance{}{B,F,S}$ and $P_1$, $P_2$ are internally $M'$-conformal paths in $\MLPInstance{}{B,F,S}$ that, in particular, avoid the vertices $x$ and $y$.
	Hence $H\coloneqq P_1+P_2+\InducedSubgraph{\MLPInstance{}{B,F,S}}{S\cup\Set{x,y}}$ is an $M'$-conformal subgraph of $\MLPInstance{}{B,F,S}$.
	It is straightforward to see that $H$ is indeed a bisubdivision of $K_{3,3}$ that contains $\MLPCycle{B,F,S}$ as a subgraph, see \cref{fig:K33bisubdivisioninReduction} for an illustration.
	By \cref{lemma:goodcrossesmeanK33} this means that there is a conformal cross over $\MLPCycle{B,F,S}$ in $\MLPInstance{}{B,F,S}$ and thus we are done with the forward direction.
	
	\begin{figure}[h!]
		\centering
		\begin{tikzpicture}
			\pgfdeclarelayer{background}
			\pgfdeclarelayer{foreground}
			\pgfsetlayers{background,main,foreground}
			
			\node[v:main] () at (-2,1){};
			\node () at (-2,1.3){$a_{1}$};
			
			\node[v:main] () at (2,1){};
			\node () at (2,1.3){$a_{2}$};
			
			\node[v:maingray] () at (0,-1){};
			\node[Amethyst] () at (0,-1.3){$y$};
			
			\node[v:mainemptygray] () at (0,1){};
			\node[Amethyst] () at (0,1.3){$x$};
			
			\node[v:mainempty] () at (-2,-1){};
			\node () at (-2,-1.3){$b_{2}$};
			
			\node[v:mainempty] () at (2,-1){};
			\node () at (2,-1.3){$b_{1}$};
			
			\node[myGreen] () at (-3,1.7){$P_{1}$};
			\node[myGreen] () at (-3,-1.7){$P_{2}$};
			
			\node[Amethyst] () at (-1,0.5){$C(B,F,S)$};
			
			\begin{pgfonlayer}{background}
				
				\draw[e:marker,BrightUbe] (-2,-1) -- (0,-1);
				\draw[e:marker,BrightUbe] (0,-1) -- (0,1);
				\draw[e:marker,BrightUbe] (0,1) -- (-2,1);
				\draw[e:marker,BrightUbe] (-2,1) -- (-2,-1);
				
				\draw[e:main,gray] (-2,-1) -- (2,-1);
				\draw[e:main,gray] (-2,1) -- (2,1);
				
				\draw[e:coloredborder] (2,1) -- (2,-1);
				\draw[e:coloredthin,color=BostonUniversityRed] (2,1) -- (2,-1);
				\draw[e:coloredborder] (-2,1) -- (-2,-1);
				\draw[e:coloredthin,color=BostonUniversityRed] (-2,1) -- (-2,-1);
				\draw[e:coloredborder] (0,1) -- (0,-1);
				\draw[e:coloredthin,color=BostonUniversityRed] (0,1) -- (0,-1);

				\draw[e:main,myGreen,bend left=45] (-2,1) to (-2.6, 1.7) to (-2,2.4);
				\draw[e:main,myGreen,decorate] (-2,2.4) -- (2,2.4);
				\draw[e:main,myGreen, bend left=45] (2,2.4) to (3.5,0.7) to (2,-1); 
				
				\draw[e:main,myGreen,bend right=45] (-2,-1) to (-2.6, -1.7) to (-2,-2.4);
				\draw[e:main,myGreen] (-2,-2.4) -- (2,-2.4);
				\draw[e:main,myGreen, bend right=45] (2,-2.4) to (3.5,-0.7) to (2,1); 
			\end{pgfonlayer}
			
		\end{tikzpicture}    
		\caption{A conformal bisubdivision of $K_{3,3}$ containing the $4$-cycle $\MLPCycle{B,F,S}$.} 
		\label{fig:K33bisubdivisioninReduction}
	\end{figure}
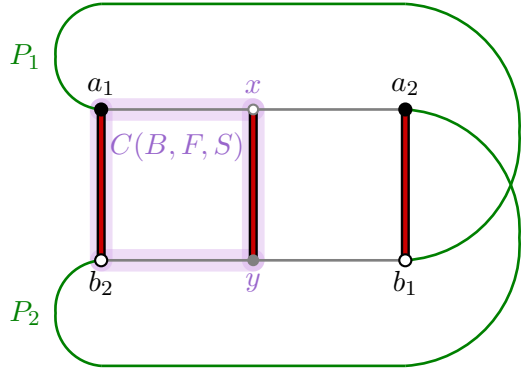
	
	For the reverse direction let $P_1$ and $P_2$ be the two alternating paths that form the conformal cross over $\MLPCycle{B,F,S}$ in $\MLPInstance{}{B,F,S}$ such that $P_1$ has $a_1$ as an endpoint while $P_2$ as $b_2$ as an endpoint.
	Then, in particular, $P_1$ and $P_2$ are of even length.
	Since $x$ and $y$ both are of degree exactly three in $\MLPInstance{}{B,F,S}$, $P_1+P_2$ must contain all neighbours of $x$ and $y$ and thus $S\subseteq\V{P_1+P_2}$.
	Since $P_1$ and $P_2$ form a conformal cross over $\MLPCycle{B,F,S}$, $H\coloneqq \MLPCycle{B,F,S}+P_1+P_2+a_2b_1$ is a conformal subgraph of $\MLPInstance{}{B,F,S}$.
	Indeed, $H$ is a bisubdivision of $K_{3,3}$ and thus there exists a perfect matching $M$ of $\MLPInstance{}{B,F,S}$ such that $a_1b_2,xy,a_2b_1\in M$ and $H$ is $M$-conformal.
	Let $P_1'\coloneqq a_1P_1b_1$ and $P_2'\coloneqq b_2P_2a_2$, then the $P_i'$ are disjoint and  internally $M'$-conformal paths.
	Moreover, $M'\setminus\Set{xy}$ is a perfect matching of $B$ that extends $F$, and thus $B$ is a $F$-instance of $2$-MLP for $\Brace{a_1,a_2}$ and $\Brace{b_1,b_2}$.
\end{proof}

Our goal is to reduce $2$-MLP to the detection of $K_{3,3}$-free braces.
For this we need to make sure that, in case we are dealing with a 'Yes'-instance, the bisubdivision of $K_{3,3}$ cannot vanish somehow.

\begin{lemma}\label{lemma:pathsovertightcuts}
	Let $B$ be a bipartite matching covered graph, $\CutG{B}{X}$ a non-trivial tight cut in $B$, and $M$ a perfect matching in $B$.
	If $P$ is an internally $M$-conformal path with both endpoints in $X$ but $\E{P}\cap\CutG{B}{X}\neq\emptyset$, then $\E{P}\cap\CutG{B}{X}\cap M\neq\emptyset$ and $\Abs{\E{P}\cap\CutG{B}{X}}=2$.
\end{lemma}

\begin{proof}
	Let $a\in V_1$ and $b\in V_2$ be the two endpoints of $P$ and let us traverse $P$ from $a$ towards $b$.
	Let $e_1$ be the first edge of $\E{P}\cap\CutG{B}{X}$ we encounter this way and let $e_2$ be the second edge.
	Moreover let $x_i$ be the endpoint of $e_i$ in $\Complement{X}$ and suppose $\Set{e_1,e_2}\cap M=\emptyset$.
	By choice of $e_1$ and $e_2$ the path $x_1Px_2$ lies completely in $\Complement{X}$ and is $M$-conformal.
	Thus $x_1Px_2$ must be of odd length and therefore $x_1$ and $x_2$ must be of different colour.
	Hence both $X\cap V_1$ and $X\cap V_2$ must have a neighbour in $\Complement{X}$, this, however, contradicts \cref{lemma:ktightmajorityminority}, and thus one of the two edges must be an edge of $M$.
	
	Suppose $P$ has more than two edges in $\CutG{B}{X}$.
	If the majority of $X$ is in $V_1$, then the second endpoint, say $y_2$, of $e_1$ must be a vertex of $V_1$ as well and $e_2\in M$.
	In this case let $a'\coloneqq y_2$.
	If on the other hand the majority of $X$ is in $V_2$, then $y_2\in V_2$ and thus $e_1\in M$ implying $e_2\notin M$.
	Hence $y_2$ must be covered by an edge $e'\in M\cap\E{P}$ with second endpoint $a'$.
	In either case, $a'P$ is an internally $M$-conformal path with both endpoints in $X$ and an edge in $\CutG{B}{X}$.
	By the arguments above, this means that $\E{a'P}\cap\CutG{B}{X}\cap M\neq\emptyset$, but this means $\Abs{M\cap\CutG{B}{X}}\geq 2$ contradicting $\CutG{B}{X}$ being a tight cut.
	Hence $\Abs{\E{P}\cap\CutG{B}{X}}=2$.
\end{proof}

\begin{lemma}\label{lemma:conformaldomino}
	Let $B$ be a bipartite graph with a perfect matching and $H\subseteq B$ a conformal subgraph $B$ such that $\Vi{1}{H}=\Set{a_1,a_2,y}$, $\Vi{2}{H}=\Set{b_1,b_2,x}$, $\E{H}=\Set{a_1b_2,a_2b_1,xy,a_1x,a_2x,b_1y,b_2y}$, and $\DegG{B}{x}=\DegG{B}{y}=3$.
	Let $C\coloneqq a_1xyb_2a_1$, then there is a conformal cross over $C$ in $B$ if and only if $B$ has a brace $J$ such that $H\subseteq J$ and $J$ contains $K_{3,3}$.
\end{lemma}

\begin{proof}
	Let $S\coloneqq\Set{a_1,a_2,b_2,b_2}$.
	Suppose there is a conformal cross over $C$ in $B$.
	This case starts out similar to the reverse direction of the previous lemma.
	Let $P_1$ and $P_2$ be the two alternating paths that form the conformal cross over $C$ in $B$ such that $P_1$ has $a_1$ as an endpoint while $P_2$ as $b_2$ as an endpoint.
	Then, in particular, $P_1$ and $P_2$ are of even length.
	Since $x$ and $y$ both are of degree exactly three in $B$ by our assumption, $P_1+P_2$ must contain all neighbours of $x$ and $y$ and thus $S\subseteq\V{P_1+P_2}$.
	Since $P_1$ and $P_2$ form a conformal cross over $C$, $H'\coloneqq C+P_1+P_2+a_2b_1$ is a conformal subgraph of $B$.
	Indeed, $H'$ is a bisubdivision of $K_{3,3}$
	
	Let us choose $B$ to be a minimal counterexample.
	In case $B$ is brace, it must contain $K_{3,3}$ since it contains a conformal bisubdivision of $K_{3,3}$ and thus we are done.
	Hence we may assume that there is a non-trivial tight cut $\CutG{B}{X}$ in $B$.
	If $X$, or $\Complement{X}$, is disjoint from $H'$, one of the two tight cut contractions of $\CutG{B}{X}$, let us call it $B'$, still contains $H'$ as a conformal subgraph and by choice of $B$, the assertion holds true for $B'$ and we find a brace $J$ of $B'$ as desired.
	By \cref{thm:tightcutuniqueness}, this means that $J$ is a brace of $B$ and thus $B$ cannot be a counterexample.
	Hence both $X$ and $\Complement{X}$ must contain vertices of $H'$.
	Observe that $\CutG{H'}{X}$ is also a tight cut of $H'$.
	Now $H'$ has exactly one $K_{3,3}$ and, possibly, a bunch of $C_4$ as its list of braces.
	Moreover, the brace $J'$ of $H'$ that is isomorphic to $K_{3,3}$ must contain all six degree three vertices of $H'$, or remainders of them.
	Indeed, this means that either $X$ or $\Complement{X}$ contains at least five vertices of $H$.
	By \cref{lemma:pathsovertightcuts} this means that at most one of the two paths $P_1$ and $P_2$ may have an edge in $\CutG{B}{X}$.
	If none of the two paths has an edge in $\CutG{B}{X}$, then one of the two tight cut contractions of $\CutG{B}{X}$ contains all of $H'$ as a conformal subgraph, contradicting $B$ being a minimal counterexample as before.
	Hence we may assume $\E{P_1}\cap\CutG{B}{X}\neq\emptyset$.
	First, assume $\Abs{\E{P_1}\cap\CutG{B}{X}}\geq 2$.
	We claim that both endpoints of $P_1$ belong to one of the two shores, say $X$ and $\Abs{\E{P_1}\cap\CutG{B}{X}}= 2$.
	To see this let $Q_1,\dots,Q_{\ell}$, $\ell\geq 2$ be the components of $P_1-\CutG{B}{X}$ with vertex sets in $\Complement{X}$.
	By \cref{lemma:ktightmajorityminority} each $Q_j$ must have both endpoints in the same colour class and thus is of even length.
	Thus for each $Q_j$ there exists an edge in $M\cap\CutG{B}{X}\cap\E{P_1}$ covering an endpoint of $Q_j$.
	Consequently, with $\ell\geq 2$ this contradicts $\CutG{B}{X}$ being tight.
	Hence  $\Abs{\E{P_1}\cap\CutG{B}{X}}\geq 2$.
	However, if $P_1$ would have an endpoint in both $X$ and $\Complement{X}$, then $\Abs{\CutG{B}{X}}$ would be odd.
	Also note that in case both endpoints of $P_1$ are in $X$, then all of $H$ must be in $X$ since otherwise, we could choose a perfect matching of $H'$ with at least two edges in $\CutG{B}{X}$.
	Hence after contracting the shore that does not contain an endpoint of $P_1$, we obtain a matching covered graph that contains a conformal $K_{3,3}$-bisubdivision with $H$ as a subgraph.
	In case $\Abs{\E{P_1}\cap\CutG{B}{X}}=1$ exactly one endpoint of $P_1$ must be contained in, say, $\Complement{X}$, while the rest of $H$ belongs to $X$.
	Again, after contracting the shore that does not contain an endpoint of $P_1$ we obtain a matching covered graph that contains a conformal $K_{3,3}$-bisubdivision with $H$ as a subgraph.
	Hence in neither case $B$ can be a minimal counterexample, and thus no such $B$ can exist.
	
	The reverse follows among similar lines.
	If there is a $K_{3,3}$-containing brace $J$ of $B$ such that $H\subseteq J$, then, by \cref{prop:4cycleK33} there must be a conformal bisubdivision $L$ of $K_{3,3}$ in $J$ that contains $C$ as a subgraph.
	Indeed, as we have seen before, we can choose $L$ such that $H\subseteq L$ and thus there must be a conformal bisubdivision $L'$ of $K_{3,3}$ in $B$ such that $H\subseteq L'$.
	According to \cref{lemma:goodcrossesmeanK33}, this means that there is a conformal cross over $C$ in $B$.
\end{proof}

\begin{lemma}[\cite{robertson1999permanents}]\label{lemma:tightcutdecomposition}
	There exists an algorithm that, given a bipartite and matching covered graph $B$ as input, computes a list of all braces of $B$ in time $\Fkt{\mathcal{O}}{\Abs{\V{B}}\Abs{\E{B}}}$.
\end{lemma}

With this, we are finally ready for the proof of \cref{prop:2linkage}

\begin{proof}[Proof of \Cref{prop:2linkage}]
	Let $B$ be a bipartite graph with a perfect matching and $S\coloneqq\Set{a_1,a_2,b_1,b_2}$ be the set of terminals we received as input for the $2$-MLP.
	By the discussion at the start of this section we only have to check for each of the at most $\Abs{\V{B}}^2$ $S$-covers $F$ whether they are extendible and whether $B$ is an $F$-instance of $2$-MLP for $\Brace{a_1,a_2}$ and $\Brace{b_1,b_2}$.
	To check whether $F$ is extendable we have to check whether $B-\V{F}$ has a perfect matching which can be done by the Hopcroft-Karp algorithm in time $\Fkt{\mathcal{O}}{n^{\frac{5}{2}}}$ \cite{hopcroft1973n}.
	So we may assume $F$ to be extendible.
	
	In case $F=\Set{a_1b_1,a_2b_2}$ we can stop immediately and return the answer 'Yes'.
	
	If $\Abs{F\cap \Set{a_1b_1,a_2b_2}}=1$ we can reduce the problem of finding our $2$-linkage to the reachability problem in digraphs as follows.
	Without loss of generality let us assume $a_1a_2\in F$ and let $M$ be a perfect matching of $B$ that extends $F$.
	Moreover, let $e_a\in M$ be the edge covering $a_2$ while $e_b$ is the edge of $M$ covering $b_2$.
	There exists a perfect matching $M'$ of $B$ that extends $F$ such that there is an internally $M'$-conformal path with endpoints $a_2$ and $b_2$ in $B-a_1-b_1$ if and only if there is an internally $M$-conformal path $P$ with endpoints $a_2$ and $b_2$ in $B-a_1-b_1$ by \cref{thm:bipartiteextendibility}.
	Testing whether such a path exists is equivalent to the reachability problem in digraphs, i.\@e.\@ the problem of deciding whether there exists a directed $s$-$t$-path for given vertices $s$ and $t$.
	Hence this can be done in polynomial time.
	
	Thus we may assume $F\cap\Set{a_1b_1,a_2b_2}=\emptyset$.
	By \cref{reducefrom2toK33} we can translate the problem into the decision problem, whether there is a conformal cross over the $4$-cycle $\MLPCycle{B,F,S}$ in $\MLPInstance{}{B,F,S}$.
	Let $H$ be the subgraph of $\MLPInstance{}{B,F,S}$ induced by $S\cup\Set{x,y}$, then \cref{lemma:conformaldomino} allows us to return 'Yes' if and only if $\MLPInstance{}{B,F,S}$ has a $K_{3,3}$-containing brace $J$ with $H\subseteq J$.
	\Cref{lemma:tightcutdecomposition} finds all braces of $\MLPInstance{}{B,F,S}$ in time $\Fkt{\mathcal{O}}{\Abs{\V{B}}^3}$ and if there is a brace $J$ with $H\subseteq J$ we can use \cref{cor:pfaffianalg} to decide in time $\Fkt{\mathcal{O}}{\Abs{\V{B}}^3}$ whether $J$ is $K_{3,3}$-free.
\end{proof}

\section{Conclusion}\label{sec:conclusion}

This paper is part of an ongoing series of papers which tries to extend the graph minors theory of Robertson and Seymour to matching minors in bipartite graphs.
While matching minors have already proven to be a powerful tool in the resolution of the Pfaffian Recognition Problem \cite{mccuaig2004polya,robertson1999permanents}, so far few attempts have been made to further understand the properties of matching minors in bipartite graphs.
In the previous paper of the series \cite{giannopoulou2021excluding} it was proven that the exclusion of a planar matching minor in a bipartite graph leads to a relatively well behaved tree-structure, similar to how excluding a planar minor leads to bounded treewidth.
In the same paper it was also shown that excluding a planar matching minor leads to an $\XP$-algorithm for the bipartite $t$-DAPP.
In the context of these findings, this paper is a consequential next step towards a matching theoretic version of the Flat Wall Theorem, which is the content of the next paper in line.
Our hope is that positive results like these spark an increase of interest in the topic and act as a base for deeper structural insights for bipartite matching covered graphs.

\subsection{Relations with Digraphs}

As mentioned in the introduction, there exists a tight interaction of digraphs and bipartite graphs with perfect matchings.
This becomes particularly obvious by considering the following operation:
Let $B$ be a bipartite graph with a perfect matching $M$.
Then let $\vec{B}$ be the orientation of $B$ obtained by orienting every edge of $B$ from $V_1$ to $V_2$.
Finally let $\DirM{B}{M}$ be the digraph obtained from $\vec{B}$ by contracting every edge of $M$.
See \cref{fig:Mdirection} for an illustration.

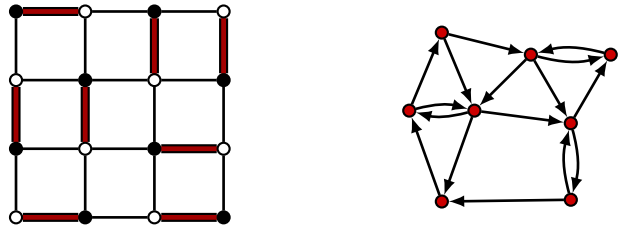
\begin{figure}[h!]
	\begin{center}
		\begin{tikzpicture}[scale=0.7]
			
			\pgfdeclarelayer{background}
			\pgfdeclarelayer{foreground}
			
			\pgfsetlayers{background,main,foreground}

			\begin{pgfonlayer}{main}
				
				\node (C) [] {};
				
				\node (C1) [v:ghost, position=180:40mm from C] {};
				
				\node (C2) [v:ghost, position=0:0mm from C] {};
				
				\node (C3) [v:ghost, position=0:40mm from C] {};

				
				
				\node (a) [v:main,position=0:0mm from C1] {};
				\node (ain) [v:ghost,position=0:2.25mm from a] {};
				
				\node (b) [v:main,position=0:13mm from a,fill=white] {};
				\node (e) [v:main,position=270:13mm from a,fill=white] {};
				
				\node (c) [v:main,position=0:13mm from b] {};
				\node (cin) [v:ghost,position=270:2.25mm from c] {};
				\node (f) [v:main,position=270:13mm from b] {};
				\node (fin) [v:ghost,position=270:2.25mm from f] {};
				\node (i) [v:main,position=270:13mm from e] {};
				\node (iin) [v:ghost,position=90:2.25mm from i] {};
				
				\node (d) [v:main,position=0:13mm from c,fill=white] {};
				\node (g) [v:main,position=270:13mm from c,fill=white] {};
				\node (j) [v:main,position=0:13mm from i,fill=white] {};
				\node (m) [v:main,position=270:13mm from i,fill=white] {};
				
				\node (h) [v:main,position=0:13mm from g] {};
				\node (hin) [v:ghost,position=90:2.25mm from h] {};
				\node (k) [v:main,position=270:13mm from g] {};
				\node (kin) [v:ghost,position=0:2.25mm from k] {};
				\node (n) [v:main,position=270:13mm from j] {};
				\node (nin) [v:ghost,position=180:2.25mm from n] {};
				
				\node (l) [v:main,position=0:13mm from k,fill=white] {};
				\node (o) [v:main,position=270:13mm from k,fill=white] {};
				
				\node (p) [v:main,position=0:13mm from o] {};
				\node (pin) [v:ghost,position=180:2.25mm from p] {};
				
				

				
				
				\node (ab) [v:main,fill=BostonUniversityRed,position=270:4mm from C3] {};
				\node (ei) [v:main,fill=BostonUniversityRed,position=247.5:16mm from ab] {};
				\node (fj) [v:main,fill=BostonUniversityRed,position=292.5:16mm from ab] {};
				\node (mn) [v:main,fill=BostonUniversityRed,position=270:32mm from ab] {};
				\node (cg) [v:main,fill=BostonUniversityRed,position=45:15mm from fj] {};
				\node (dh) [v:main,fill=BostonUniversityRed,position=0:15mm from cg] {};
				\node (kl) [v:main,fill=BostonUniversityRed,position=300:15mm from cg] {};
				\node (op) [v:main,fill=BostonUniversityRed,position=270:14.5mm from kl] {};
				
				

				
				
				\draw (b) [e:main] to (c);
				\draw (b) [e:main] to (f);
				
				\draw (e) [e:main] to (a);
				\draw (e) [e:main] to (f);
				
				\draw (d) [e:main] to (c);
				
				\draw (g) [e:main] to (f);
				\draw (g) [e:main] to (h);
				\draw (g) [e:main] to (k);
				
				\draw (j) [e:main] to (i);
				\draw (j) [e:main] to (k);
				\draw (j) [e:main] to (n);
				
				\draw (m) [e:main] to (i);
				
				\draw (l) [e:main] to (h);
				\draw (l) [e:main] to (p);
				
				\draw (o) [e:main] to (k);
				\draw (o) [e:main] to (n);

				
				
				
				
				\draw (ab) [e:main,->] to (fj);
				\draw (ab) [e:main,->] to (cg);
				
				\draw (ei) [e:main,->,bend left=15] to (fj);
				\draw (ei) [e:main,->] to (ab);
				
				\draw (fj) [e:main,->,bend left=15] to (ei);
				\draw (fj) [e:main,->] to (mn);
				\draw (fj) [e:main,->] to (kl);
				
				\draw (mn) [e:main,->] to (ei);
				
				\draw (cg) [e:main,->,bend right=15] to (dh);
				\draw (cg) [e:main,->] to (fj);
				\draw (cg) [e:main,->] to (kl);
				
				\draw (dh) [e:main,->,bend right=15] to (cg);
				
				\draw (kl) [e:main,->] to (dh);
				\draw (kl) [e:main,->,bend left=15] to (op);
				
				\draw (op) [e:main,->] to (mn);
				\draw (op) [e:main,->,bend left=15] to (kl);
				
				
				
			\end{pgfonlayer}
			

			\begin{pgfonlayer}{background}
				
				\draw (b) [e:coloredborder] to (a);
				\draw (e) [e:coloredborder] to (i);
				\draw (d) [e:coloredborder] to (h);
				\draw (g) [e:coloredborder] to (c);
				\draw (j) [e:coloredborder] to (f);
				\draw (m) [e:coloredborder] to (n);
				\draw (l) [e:coloredborder] to (k);
				\draw (o) [e:coloredborder] to (p);
				
				\draw (b) [e:colored,color=BostonUniversityRed] to (a);
				\draw (e) [e:colored,color=BostonUniversityRed] to (i);
				\draw (d) [e:colored,color=BostonUniversityRed] to (h);
				\draw (g) [e:colored,color=BostonUniversityRed] to (c);
				\draw (j) [e:colored,color=BostonUniversityRed] to (f);
				\draw (m) [e:colored,color=BostonUniversityRed] to (n);
				\draw (l) [e:colored,color=BostonUniversityRed] to (k);
				\draw (o) [e:colored,color=BostonUniversityRed] to (p);
				
			\end{pgfonlayer}	
			
			\begin{pgfonlayer}{foreground}

			\end{pgfonlayer}
		\end{tikzpicture}
	\end{center}
	\caption{Left: A bipartite graph $B$ with a perfect matching $M$. Right: The arising $M$-direction $\DirM{B}{M}$.}
	\label{fig:Mdirection}
\end{figure}

Clearly, once the colour classes $V_1$ and $V_2$ are uniquely identified with the direction of edges (in our case we say that edges go from $V_1$ to $V_2$) then the operation of forming the $M$-direction is invertible.
Indeed, one can obtain from every digraph a uniquely determined bipartite graph with a perfect matching and vice versa.
Moreover as it turns out, several structural properties of $B$ are reflect in $\DirM{B}{M}$ and vice versa.
These properties range from simple observations like `$B$ is connected if and only if $\DirM{B}{M}$ is weakly connected.' up to 'For each bipartite graph $J$ with a perfect matching there exists a unique family $\mathcal{J}$ of digraphs such that $B$ contains $J$ as a matching minor if and only if $\DirM{B}{M}$ contains a member of $\mathcal{J}$ as a butterfly minor.'.
This is proven and discussed in \cite{giannopoulou2021excluding}.
The second property in particular allows us to combine the theory of butterfly minors in digraphs and the theory of matching minors in bipartite graphs into one unified theory.
This angle of viewing digraphs and bipartite graphs with perfect matchings as related objects gives us a tool to have some sort of control over certain infinite antichains of the butterfly minor relation and, it also gives us a way to use the matching theoretic results of this paper to obtain structural results on digraphs while avoiding such problems as the absence of a directed Two Paths Theorem.

To see how this would be possible consider the example of the Directed $2$-Disjoint Paths Problem.
This problem is $\NP$-complete and by using the $M$-direction and its inverse one can easily see that the problem of finding two disjoint directed paths between given terminal vertices in a digraph can be translated into asking whether there exist disjoint internally $M$-conformal paths between given terminal vertices in the corresponding bipartite graph.
Hence, if we only slightly alter the definition of $t$-DAPP by not asking for the existence of a perfect matching for which a solution exists, but by insisting on the question whether there is solution for this particular perfect matching, the problem suddenly becomes $\NP$-hard.
So from the additional flexibility of being able to change the perfect matching possible tools can arise to deal with problems which are intractable on digraphs otherwise.
It would be interesting to see where the limits of this additional flexibility lie.

\bibliographystyle{alphaurl}
\bibliography{literature}

\end{document}